\documentclass[a4paper]{article}
\usepackage{amsfonts,amsthm,latexsym,amsmath,mathrsfs,amscd,geometry}
\usepackage{amssymb}
\usepackage{dsfont}
\usepackage{latexsym}
\usepackage{appendix}
\usepackage{enumerate}
\usepackage{color}
\usepackage{upgreek}
\numberwithin{equation}{section}
\newtheorem{theorem}{Theorem}[section]
\newtheorem{lemma}[theorem]{Lemma}
\newtheorem{proposition}[theorem]{Proposition}
\newtheorem{corollary}[theorem]{Corollary}

\newtheorem{remark}[theorem]{Remark}

\makeatletter%使用罗马数字

\newcommand{\Rmnum}[1]{\expandafter\@slowromancap\romannumeral #1@}
\makeatother

\allowdisplaybreaks

\usepackage[colorlinks,linkcolor=blue]{hyperref}

\begin{document}
\title{The vanishing diffusion limit for an Oldroyd-B model in $\mathbb{R}^2_+$}
\author{Yinghui Wang\footnote{School of Mathematics, South China University of Technology, Guangzhou 510641, China. E-mail: yhwangmath@163.com},\,\,\, Huanyao Wen
\footnote{Corresponding author. School of Mathematics, South China University of Technology, Guangzhou 510641, China. Email: mahywen@scut.edu.cn.}
}
\date{}
\maketitle
\begin{abstract}
We consider the initial-boundary value problem for an incompressible Oldroyd-B model with stress diffusion in two-dimensional upper half plane which describes the motion of viscoelastic polymeric fluids. From the physical point of view, the diffusive coefficient is several orders of magnitude smaller than other parameters in the model, and is usually assumed to be zero. However, the link between the diffusive model and the standard one (zero diffusion) via vanishing diffusion limit is still unknown from the mathematical point of view, in particular for the problem with boundary. Some numerical results \cite{Chupin} suggest that this should be true. In this work, we provide a rigorous justification for the vanishing diffusion in $L^\infty$-norm.
\end{abstract}
\bigbreak \textbf{{\bf Key Words}:}  An Oldroyd-B model; vanishing diffusion limit; boundary layers.
\bigbreak  {\textbf{AMS Subject Classification 2020:} 35Q35, 76A10, 76D10.

\tableofcontents

\section{Introduction}

The Oldroyd-B model is a widely used macroscopic model describing the motion of viscoelastic fluids. A typical feature for the viscoelastic fluid is the presence of an extra stress tensor which satisfies a nonlinear evolution equation. In the context of some physical considerations such as the shear and vorticity banding phenomena (\cite{Bhave2,Cates_2006,Dhont_2008,El-Kareh, Liu,Malek_etal_2018}), an additional term describing the polymeric stress diffusion is sometimes present in the equation of the extra stress tensor (diffusive model). However, it is known that this diffusion coefficient is significantly smaller than other effects (\cite{Bhave1}), and thus this quantity is usually neglected in some early works on the classical Oldroyd-B model (non-diffusive model) (\cite{Oldroyd_1958}). Moreover, from the mathematical point of view, the non-diffusive model (the stress tensor satisfying hyperbolic type equations) and the diffusive model (the stress tensor satisfying parabolic type equations) are significantly different. Therefore, a natural question is how to justify the vanishing diffusion limit. The link between the non-diffusive model and the diffusive model has been investigated from the numerical point of view for the steady case by Chupin, and Martin (\cite{Chupin}) via considering the vanishing diffusion process in the diffusive model. For the non-steady case, the limit process has been investigated for the 3-D Cauchy problem in a recent work \cite{Huang_Wang_Wen_Zi_2022} by Huang, the first author, the second author, and Zi.
Our aim in the present work is to provide a rigorous justification of the process in the spatial domain $\mathbb{R}^2_+$. To cover more cases, we consider a generalized incompressible Oldroyd-B model which can be derived via a macroscopic closure of a micro-macro model describing incompressible dilute polymeric fluids, i.e., the incompressible Navier--Stokes/Fokker--Planck system in a dumbbell Hookean setting, see \cite{Barrett_Lu_Suli_2017, Bris}, namely,
\begin{equation}\label{OB_incom}
    \begin{cases}
    u_t+u\cdot\nabla u +\nabla p -\mu \Delta u
    =\mathrm{div} \big(\mathbb{T}-(kL\eta+\mathfrak{z}\eta^{2})\mathbb{I}\big),\\
    \eta_t+ u\cdot\nabla \eta =\varepsilon\Delta \eta,\\
    \mathbb{T}_t+u\cdot \nabla \mathbb{T}-(\nabla u
    \mathbb{T}+\mathbb{T}\nabla^\top u)
    =\varepsilon\Delta \mathbb{T}+\frac{kA_0}{2\lambda}\eta\mathbb{I}-\frac{A_0}{2\lambda}\mathbb{T},\\
    \mathrm{div} u = 0,
    \end{cases}
    \end{equation}
    in $\mathbb{R}_+^2\times \mathbb{R}_+$ with $\mathbb{R}_+^2:= \{(x,y)\in \mathbb{R}^2| y>0\}$, where $u:=(u_1,u_2)(x,y,t) $ denotes the velocity field of fluid, and $p:=p(x,y,t)$ is the pressure function of the fluid. The scalar function $\eta:=\eta(x,y,t)\geq 0$ is  the  so-called polymer number density, and the extra stress tensor $\mathbb{T}: = (\mathbb{T}_{ij})(x,y,t)\in \mathbb{R}^{2\times 2}, 1\leq i,j\leq 2$
    is a nonnegative definite symmetric matrix. $\mathbb{I}$ is a $2\times2$ unit matrix for the case of two spatial dimensions. $\mu>0$ is the viscosity of the fluid, and  $\varepsilon>0$ corresponding to a spatial diffusion of the polymeric stresses is called the center-of-mass coefficient. The other physical parameters $k, \lambda, L, \mathfrak{z}, A_0$ are all positive numbers. In this paper, we equip \eqref{OB_main} with the  initial-boundary value conditions:
    \begin{align}
        &(u,\eta,\mathbb{T})(x,y,0) = (u_0,\eta_0,\mathbb{T}_0)(x,y),\label{OB_main_initial_origin} \\
        u(x,0&,t) = 0,~\partial_y\eta(x,0,t) = 0,~\partial_y\mathbb{T}(x,0,t) = 0.\label{OB_main_boundary_origin}
    \end{align}

% More specifically, we aim to study the vanishing diffusion limit $\varepsilon\to 0$. The boundary layer effect is the main challenge for the problem.

  Let us give a brief overview of some relevant works on the model. Here again, as Bhave, Armstrong, and Brown (\cite{Bhave1}) show that the center-of-mass coefficient is significantly smaller than other effects, thus $\varepsilon$ is usually not included in some early works of the model. As a matter of fact, the polymer number density $\eta$ is also assumed to be a positive constant, see the following model:
 \begin{equation}\label{OB_incom_classical}
        \begin{cases}
        u_t+u\cdot\nabla u +\nabla P^\prime -\mu \Delta u
        =\mathrm{div}  \tau,\\
       \tau_t+u\cdot \nabla \tau-Q(\nabla u,\tau)+ \gamma \tau
         = k\eta (\nabla u + \nabla^\top u),\\
        \mathrm{div} u = 0,
        \end{cases}
        \end{equation}
    which was first proposed by Oldroyd in  \cite{Oldroyd_1950, Oldroyd_1958}, see also \cite{Bird_etal_1,Bird_etal_2,Morrison_2001}. Here $Q(\nabla u,\tau)=\Omega\tau-\tau\Omega+b(Du\,\tau+\tau\,Du)$ for a parameter $b\in[-1,1]$, where $Du$ and $\Omega$ are respectively the symmetric and skew-symmetric parts of the velocity gradient. In fact, the model (\ref{OB_incom_classical}) with $b=1$ and $\tau=\mathbb{T}-k\eta \mathbb{I}$ is corresponding to (\ref{OB_incom}) for positive constant $\eta$. Some important works on the study of the non-diffusive model (\ref{OB_incom_classical}) have been achieved in the past few year. For the  Dirichlet problem in a bounded domain of $\mathbb{R}^3$, the global-in-time wellposedness with small initial data was obtained by Guillop\'e, and Saut (\cite{Guillope_Saut}) in $H^s$ setting , see the relaxation of a small coupling constant by Molinet, and Talhouk (\cite{Molinet_Talhouk_2004}). Fern\'andez-Cara, Guill\'en, and  Ortega (\cite{Fernandez_Guillen_Ortega}) extended the results in $L^p$ setting. For the problem in 3-D exterior domains, the global wellposedness of the small strong solution was obtained with small coupling constant by
    Hieber, Naito, and Shibata (\cite{Hieber_Naito_Shibata}), see  \cite{Fang_Hieber_Zi_2013}  by Fang, Hieber, and Zi for arbitrary coupling constant. Lions and Masmoudi (\cite{Lions_Masmoudi_2000}) obtained a global existence result for weak solutions of the corotational Oldroyd model with arbitrarily large initial data. Recently, Di Iorio, Marcati, and Spirito (\cite{DiIorio_Pierangelo_Spirito_ARMA}) investigated the splash singularities of the two-dimensional free boundary problem. Chemin, and Masmoudi (\cite{Chemin_Masmoudi_2001}) obtained the global small solutions in  critical Besov spaces together with some blow-up criteria. And the result in \cite{Chemin_Masmoudi_2001} was improved by Lei, Masmoudi, and Zhou (\cite{Lei_Masmoudi_Zhou_2010}), Sun and Zhang (\cite{Sun_Zhang_2011}) and by Zi, Fang, and Zhang (\cite{Zi_Fang_Zhang_2014}). Lei (\cite{Lei_2006}) investigated the connection between the compressible model and the incompressible model via the incompressible limit in periodic domains, see \cite{Fang_Zi_2014} by Fang, and Zi for the Cauchy problem in critical Besov spaces. Very recently, Hieber, the second author, and Zi  (\cite{Hieber_Wen_Zi_2019})
   investigated the long time behavior of the solutions, and proved that the decay rates are the same as those for the heat equation, see a recent work by Huang, the first author, the second author, and Zi (\cite{Huang_Wang_Wen_Zi_2022}) for some extensions to the diffusive model with zero viscosity. Zi (\cite{Zi_2021}) investigated the vanishing viscosity limit in analytic settings. For another macroscopic model for viscoelastic fluids based on the Lagrangian particle dynamics, we refer to the seminal work by Lin, Liu, and Zhang (\cite{Lin_Liu_Zhang_2005}) and other interesting works on wellposedness theory (\cite{Cai_2019, Chen_Zhang_2006, Hu_Lin_2016, Hu_Lin_Liu_2018, Lei_2010, Lei_Liu_Zhou_2008, Lin_2012, Lin_Zhang_2008, Zhang_Fang_2012}

 In view of physical and numerical background (\cite{Chupin, Chupin2, Malek_etal_2018, Rajagopal_2000, Ziegler_1987}), the diffusive model has been investigated extensively in recent years. Some interesting mathematical results of the diffusive model (\ref{OB_incom}) with constant polymer number density $\eta$ are achieved, i.e.,
   \begin{equation}\label{OB_incom sim}
    \begin{cases}
    u_t+u\cdot\nabla u +\nabla P^\prime -\mu \Delta u
    =\mathrm{div} \tau,\\
    \tau_t+u\cdot \nabla \tau-Q(\nabla u,\tau)+ \gamma \tau
         =\varepsilon\Delta \tau+ k\eta (\nabla u + \nabla^\top u),\\
    \mathrm{div} u = 0,
    \end{cases}
    \end{equation} where $\eta=const.>0$. For fixed $\varepsilon>0$ and $\mu>0$, the existence of global-in-time weak solutions was obtained by Barrett, and Boyaval (\cite{Barrett_Boyaval_2011}). In the case that $\varepsilon> 0$ is fixed and $\mu=0$, the global existence and uniqueness of regular solutions to the Cauchy problem for (\ref{OB_incom sim}) was obtained by Elgindi, and Rousset (\cite{Elgindi_Rousset_2015}) in  two dimensions for arbitrarily large initial data when $Q(\nabla u,\tau)=0$. In addition, the authors also obtained a global wellposedness result for $Q(\nabla u,\tau)\not=0$ provided that the initial data are small. Later on, Elgindi, and Liu (\cite{Elgindi_Liu_2015}) extended the result in \cite{Elgindi_Rousset_2015} to the three-dimensional case with small initial data in $H^3(\mathbb{R}^3)$-norm. See \cite{Constantin_Wu_Zhao_Zhu_2020} by Constantin, Wu, Zhao, and Zhu for the fractional dissipation case.
    The global wellposedness and optimal time-decay estimates of the solutions to the Cauchy problem for (\ref{OB_incom sim}) with vanishing viscosity (the case $\mu\geq 0$ and fixed $\varepsilon>0$) or vanishing diffusion (the case $\varepsilon\geq 0$ and fixed $\mu>0$)  in three dimensions were obtained by Huang, the first author, the second author, and Zi (\cite{Huang_Wang_Wen_Zi_2022}). When the polymer number density $\eta$ is varied and the associated diffusion term $\varepsilon\Delta \eta$ is ignored, Constantin and Kliegl (\cite{Constantin_Kliegl_2012}) obtained the global existence and uniqueness of strong solutions to the Cauchy problem for (\ref{OB_incom}) in two dimensions with fixed center-of-mass coefficient $\varepsilon>0$ and arbitrarily large initial data. La (\cite{La}) provided a rigorous derivation of the model (\ref{OB_incom}) with fixed viscosity $\mu>0$ and fixed center-of-mass coefficient $\varepsilon>0$ in two dimensions as a macroscopic closure of the micro-macro model. For compressible case, refer for instance to \cite{Barrett_Lu_Suli_2017, Liu_Lu_Wen_2021, Lu-Zhang, Wang-Wen} and references therein.

    The main results in \cite{Huang_Wang_Wen_Zi_2022} imply that the vanishing diffusion limit is justified in the whole space. However, for the boundary-value problem, the link between the non-diffusive model and the diffusive model is still unknown due to the boundary layer effect. Inspired by Chupin-Martin's numerical work \cite{Chupin} on the vanishing diffusion limit as $\varepsilon\rightarrow0$, we provide a rigorous justification for the initial-boundary value problem \eqref{OB_main}-\eqref{OB_main_boundary} in $\mathbb{R}^2_+$.

 \subsection{Reformulation of the problem}
    To simplify the following computations, it is convenient to reformulate the system \eqref{OB_incom} in a way as in \cite{Liu_Lu_Wen_2021}. More specifically, denoting $\tau_{ij}=\mathbb{T}_{ij}-k\eta \mathbb{I}_{ij}$, %we can reformulate as follows.
  then \eqref{OB_incom}, (\ref{OB_main_initial_origin}), and (\ref{OB_main_boundary_origin}) can be reformulated as
    \begin{equation}\label{OB_main}
        \begin{cases}
        u_t+u\cdot\nabla u +\nabla \tilde{p} -\mu \Delta u
        =\mathrm{div} \tau,\\
        \eta_t+ u\cdot\nabla \eta =\varepsilon\Delta \eta,\\
        \tau_t+u\cdot \nabla \tau-(\nabla u
        \tau+\tau \nabla^\top u) + \gamma \tau
        - \varepsilon\Delta \tau = k\eta (\nabla u + \nabla^\top u),\\
        \mathrm{div} u = 0,
        \end{cases}
    \end{equation}
    and
    \begin{align}
        &(u,\eta,\tau)(x,y,0) = (u_0,\eta_0,\tau_0)(x,y),\label{OB_main_initial} \\[2mm]
        u(x,0&,t) = 0,~\partial_y\eta(x,0,t) = 0,~\partial_y\tau(x,0,t) = 0,\label{OB_main_boundary}
    \end{align} for $(x,y)\in\mathbb{R}^2_+$ and $t>0$, where $ \gamma:= \frac{A_0}{2\lambda}>0$ and $\tilde{p}:= p +k(L-1)\eta+\mathfrak{z}\eta^2$.

    Formally, letting  $\varepsilon\to 0$ in \eqref{OB_main}, one can obtain the non-diffusive Oldroyd-B model:
    \begin{equation}\label{OB_limit_I}
        \begin{cases}
        u^{I,0}_t+u^{I,0}\cdot\nabla u^{I,0} +\nabla \tilde{p}^{I,0} -\mu \Delta u^{I,0}
        =\mathrm{div} \tau^{I,0},\\[2mm]
        \eta^{I,0}_t+ u^{I,0}\cdot\nabla \eta^{I,0} =0,\\[2mm]
        \tau^{I,0}_t+u^{I,0}\cdot \nabla \tau^{I,0}-\mathcal{Q}(\nabla u^{I,0}, \tau^{I,0} ) + \gamma \tau^{I,0}
         - \mathcal{B}(\eta^{I,0}, \nabla u^{I,0})=0,\\[2mm]
        \mathrm{div} u^{I,0} = 0,
        \end{cases}
    \end{equation}
    equipped with the initial-boundary value conditions:
    \begin{align}
        (u^{I,0},\eta^{I,0},\tau^{I,0})(x,y,0) ={}& (u_0,\eta_0,\tau_0)(x,y),\label{OB_limit_I_initial_boundary}
       \\ u^{I,0}(x,0,t) ={}& 0,\label{OB_limit_I_initial_boundary1}
    \end{align}
    where
    \begin{eqnarray*}
        \mathcal{Q}(\nabla u^{I,0}, \tau^{I,0} )  = \nabla u^{I,0}\tau^{I,0} + \tau^{I,0}\nabla^\top u^{I,0}, ~\mathcal{B}(\eta^{I,0}, \nabla u^{I,0}) = k\eta^{I,0}(\nabla u^{I,0} + \nabla^\top u^{I,0}).
    \end{eqnarray*}

    \subsection{Notation}
    In this paper, some standard notations are used. Let $C$ denote a generic positive constant which may depend on the initial data and some other known constants but independent of the variable parameter $\varepsilon$. When the dependence needs to be pointed out, we will   equip the constant with a corresponding subscript such as  $C_{\mu}$. Some other notations are stated as below: %We also use $B_i,D_i,$($i=0,1,2\cdots$) to denote the specific constants which are necessary to clearify the proofs.
    \begin{itemize}
    \item $A \lesssim B$ $\Leftrightarrow$ $A \leq CB$.

    \item The notation $\langle \cdot,\cdot\rangle$ means the $L^2$ inner product over $\mathbb{R}^2_+$.

    \item {For a scalar-valued function $\eta$, two vector-valued functions $u,v$ and a matrix-valued function $\mathbb{T}$, we use the following notations:
    \begin{align*}
        (\nabla \eta)_{i}& := \partial_i \eta,~(\nabla u)_{ij} := \partial_j u_i,~(\nabla \mathbb{T})_{ijk} := \partial_k \mathbb{T}_{ij},~(\nabla^2 u)_{ijk}:= \partial_j\partial_k u_i,\\
        &(u\otimes v)_{ij}:= u_iv_j,~(u\otimes \mathbb{T})_{ijk}:=u_i\mathbb{T}_{jk},~(\mathbb{T}\otimes u )_{ijk}:= \mathbb{T}_{ij}u_k.
    \end{align*}}

    \item $\langle \cdot \rangle := \sqrt{1+|\cdot|^2}.$

    \item  $L^p_{xy}$ and $H^s_{xy}$ denote the  usual Lebesuge and  Sobolev space over $\mathbb{R}_+^2: = \{(x,y)|(x,y)\in \mathbb{R}\times \mathbb{R}_+\}$ with corresponding norms $\|\cdot\|_{L^p_{xy}}$ and $\|\cdot\|_{H^s_{xy}}$, respectively.

    \item The following anisotropic Sobolev space is denoted as
     \begin{eqnarray*}
         H_x^mH_y^\ell:= \left\{f \in L^2(\mathbb{R}\times\mathbb{R}_+)\Big| \sum_{0\leq i \leq m, 0\leq j \leq \ell } \|\partial_x^{i}\partial_y^{j} f(x,y)\|_{L^2_{xy}}  < \infty \right\}
     \end{eqnarray*}
     with norm  $\|\cdot\|_{H^m_xH^\ell_y}$.

    \item $z = y/\sqrt{\varepsilon}$ for $\varepsilon>0$. The notations $L^p_{xz}$, $H^s_{xz}$, and  $ H^m_xH^\ell_z $ denote that their components
        are functions of $(x,z)$.

    \item $\|(u,v)\|_{X}^2:=\|u\|_{X}^2+\|v\|_{X}^2$ for Banach space $X$. The norm of  $L^q(0,T;X)$($1\leq p \leq \infty$)  is denoted by  $\|\cdot\|_{L^q_T X}$.

    \item Let $\varphi $  be a smooth function defined on $[0,+\infty)$ satisfying
     \begin{eqnarray}\label{eq_phi}
         \varphi(0) = 1,~~\varphi^\prime(0)=0, ~~~\varphi(z) = 0 \text{ for } z>1.
    \end{eqnarray}

    \end{itemize}

     \subsection{Main results}
   The main result of this paper is to investigate the non-diffusive limit of the problem \eqref{OB_main}-\eqref{OB_main_boundary}. For this purpose, we need the following wellposedness result of the limiting system \eqref{OB_limit_I}-\eqref{OB_limit_I_initial_boundary1}.
    \begin{proposition}\label{local_wellposedness_0}
        Assume that $(u_0,\eta_0,\tau_0) \in H^{14}_{xy}$ with $\mathrm{div}u_0 = 0$ and $\tau_0$ symmetric, and that the following compatibility conditions
        \begin{eqnarray}\label{eq_compatibility_u}
            \partial_t^{i}u^{I,0}(0)|_{y = 0} = 0,~~~0\leq i \leq 6,
        \end{eqnarray}
       hold, where $\partial_t^{i}u^{I,0}(0)$ is the $i$-th time derivative of $u^{I,0}$ at $\{t = 0\}$ which can be connected to the initial value $(u_0,\eta_0,\tau_0)$ by means of the system \eqref{OB_limit_I}. Then there exists a positive time $T_0$ such that \eqref{OB_limit_I}-\eqref{OB_limit_I_initial_boundary1} has a unique
        solution $(u^{I,0},\eta^{I,0},\tau^{I,0})$ on $[0,T_0]$ satisfying $\mathrm{div} u^{I,0} = 0$ and
        \begin{equation*}
            \begin{split}
                 & \partial_t^\ell u^{I,0} \in L^\infty(0,T_0; H^{14-2\ell}_{xy})\cap L^2(0,T_0; H^{15-2\ell}_{xy}),~~\ell=0,1,\cdots,{7}, \\
                 &\eta^{I,0}, \tau^{I,0}\in L^\infty(0,T_0; H^{14}_{xy}),~~\partial_t^j( \eta^{I,0},  \tau^{I,0})\in L^\infty(0,T_0; H^{15-2j}_{xy}),~~j =1,\cdots,{7}.\\
            \end{split}
        \end{equation*}
    \end{proposition}
    \begin{remark}\label{rem_u_eta_tau_I1}
        The proof of Proposition \ref{local_wellposedness_0} is standard. In \cite{Fang_Hieber_Zi_2013}, the authors proved the local well-posedness of \eqref{OB_limit_I}-\eqref{OB_limit_I_initial_boundary1} with $\eta = 0$ and $H^2$ initial data in 3-D exterior domains. After a
        slight modification, one can prove a similar result for \eqref{OB_limit_I}-\eqref{OB_limit_I_initial_boundary1} in $\mathbb{R}_+^2$. Moreover, the higher regularity can be proved by induction (cf. Chapter 7 of \cite{Evans_2010}).
    \end{remark}

    \medskip

    Now we are in a position to state our main result.

    \begin{theorem}\label{theorem_limit}
         In addition to the conditions of Proposition \ref{local_wellposedness_0},   we assume that $(u_0, \eta_0, \tau_0)$ satisfies the  additional strong compatibility conditions \eqref{eq_compatibility_eta_tau} and \eqref{compatibility_cond_u_I_2}.
         Then there exists a positive time $T_\star\leq T_{0}$ (which is defined in \eqref{def_T_star}) independent of $\varepsilon$ such that the problem \eqref{OB_main}-\eqref{OB_main_boundary} has a solution $(u^\varepsilon,\eta^\varepsilon,\tau^\varepsilon) \in C([0,T_{\star}];H^2_{xy})$
          with $\partial_x(u^\varepsilon,\eta^\varepsilon,\tau^\varepsilon) \in C([0,T_{\star}];H^2_{xy})$
         on $[0,T_\star]$,
         satisfying
        \begin{align}
            &\|u^\varepsilon(x,y,t) - u^{I,0}(x,y,t)    \|_{L^\infty_{T_\star}L^\infty_{xy}} \leq C\varepsilon,\label{eq_u_convergence_rate}\\
            &\|\eta^\varepsilon(x,y,t)   - \eta^{I,0}(x,y,t)   \|_{L^\infty_{T_\star}L^\infty_{xy}} \leq C\varepsilon^\frac{1}{2},\label{eq_eta_convergence_rate}\\
            &\|\tau^\varepsilon(x,y,t) - \tau^{I,0}(x,y,t) \|_{L^\infty_{T_\star}L^\infty_{xy}} \leq C\varepsilon^\frac{1}{2},\label{eq_tau_convergence_rate}
        \end{align}
        and
        \begin{align}
                &\|\partial_yu^\varepsilon(x,y,t) -\partial_y u^{I,0}(x,y,t)  \|_{L^\infty_{T_\star}L^\infty_{xy}}\leq   C\varepsilon^{\frac{1}{2}},  \label{eq_u_y_convergence_rate}\\
            &\|\partial_y\eta^\varepsilon(x,y,t)   - \partial_y\eta^{I,0}(x,y,t) - \partial_z\eta^{B,1}(x,\frac{y}{\sqrt{\varepsilon}},t) \|_{L^\infty_{T_\star}L^\infty_{xy}} \leq C\varepsilon^{\frac{1}{2}},\label{eq_eta_y_convergence_rate}\\
            &\|\partial_y\tau^\varepsilon(x,y,t) - \partial_y\tau^{I,0}(x,y,t)- \partial_z\tau^{B,1}(x,\frac{y}{\sqrt{\varepsilon}},t)\|_{L^\infty_{T_\star}L^\infty_{xy}} \leq C\varepsilon^{\frac{1}{2}},\label{eq_tau_y_convergence_rate}
        \end{align}
        where the positive constant $C$ is independent of $\varepsilon$,
         $\eta^{B,1}$ is a solution of \eqref{eq_eta_B1},
        and $\tau^{B,1}$ is  a solution of \eqref{eq_tau_B1_22},  \eqref{eq_tau_B1_12}, and \eqref{eq_tau_B1_11}. The well-posedness of $ \eta^{B,1},\tau^{B,1}$ are stated
        in Lemma \ref{lemma_eta_tau_b_1}.

    \end{theorem}

    \begin{remark}\label{rem_weak_boundary_l}
       Theorem \ref{theorem_limit} suggests that the boundary layer effect does not happen to the solution itself but to the $``\partial_y"$ of the solution as $\varepsilon\rightarrow0$. %This means tha $(\eta,\tau)$ enjoys the so called weak boundary layer effect.
    \end{remark}

    \begin{remark}
        The strong compatibility conditions \eqref{eq_compatibility_eta_tau} is used to prove the wellposedness of the higher order boundary layer profiles which is crucial to deduce the $L^\infty$  convergence estimates \eqref{eq_u_convergence_rate}-\eqref{eq_tau_convergence_rate}, see Section \ref{subsection_regularity} for details. In addition, to justify the weak boundary layer effects,  we need to involve the higher order profiles $(u^{I,2},\eta^{I,2},\tau^{I,2})$ (which satisfies \eqref{OB_outer_second}) in the approximate solution.
       And, \eqref{compatibility_cond_u_I_2} is the compatibility conditions for getting the higher regularities of $(u^{I,2},\eta^{I,2},\tau^{I,2})$.
    \end{remark}

  \begin{remark}
  For the convenience of presenting our method and technique, we state the main result in 2D. The 3D case will be reported in our forthcoming paper \cite{WangYH-Wen}.
  \end{remark}

    \begin{remark}
  It is also interesting to study the case that $\eta^{\varepsilon}$ and $\tau^{\varepsilon}$ equipped with the Dirichlet boundary conditions. This will be considered in the near future.
  \end{remark}

    \bigskip

   We will remove the superscript $\varepsilon$ of $(u^\varepsilon,\eta^\varepsilon,\tau^\varepsilon)$ for brevity throughout the rest of the paper when it does not cause any confusion.

    \subsection*{Main ideas}
    Let us sketch the main ideas in the present paper. To provide a rigorous justification from (\ref{OB_main}), (\ref{OB_main_initial}), and (\ref{OB_main_boundary}) to (\ref{OB_limit_I}), (\ref{OB_limit_I_initial_boundary}), and (\ref{OB_limit_I_initial_boundary1}) as $\varepsilon\rightarrow0$, the main challenge is the possible gap between $(\eta,\tau)$ and $(\eta^{I,0}, \tau^{I,0})$ at the boundary, compared with the Cauchy problem considered in \cite{Huang_Wang_Wen_Zi_2022}.
    Our main idea is to find some correctors to cancel the possible gap by employing the asymptotic matched expansion method (refer for instance to Chapter 4 in \cite{Holmes_2013}, see also \cite{Hou_Wang_2019}). More specifically, the solution of problem \eqref{OB_main}-\eqref{OB_main_boundary} is decomposed as below:
    \begin{align}
         u(x,y,t) ={}&  u^{I,0}(x,y,t) +  \varepsilon u^{I,2}(x,y,t) +\varepsilon \left(u_1^{B,2}(x,\frac{y}{\sqrt{\varepsilon}},t),0\right)^\top +\varepsilon^{\frac{3}{2}} u^{I,3}(x,y,t) +  \varepsilon^{\frac{3}{2}} u^{B,3}(x,\frac{y}{\sqrt{\varepsilon}},t) \notag\\
        &+ \varepsilon^{2} u^{B,4}(x,\frac{y}{\sqrt{\varepsilon}},t) + \varepsilon^{\frac{5}{2}}\left(0,u^{B,5}_2(x,\frac{y}{\sqrt{\varepsilon}},t)\right)^\top + \sqrt{\varepsilon}R^{u}(x,y),\notag\\
         \eta(x,y,t)  ={}& \eta^{I,0}(x,y,t) +\sqrt{\varepsilon} \eta^{B,1}(x,\frac{y}{\sqrt{\varepsilon}},t) + \varepsilon \eta^{I,2}(x,y,t)+ \varepsilon \eta^{B,2}(x,\frac{y}{\sqrt{\varepsilon}},t)  + \varepsilon^{\frac{3}{2}} \eta^{B,3}(x,\frac{y}{\sqrt{\varepsilon}},t)\notag\\
         & + \sqrt{\varepsilon} R^{\eta}(x,y),\notag\\
         \tau(x,y,t)  ={}& \tau^{I,0}(x,y,t) +\sqrt{\varepsilon} \tau^{B,1}(x,\frac{y}{\sqrt{\varepsilon}},t) + \varepsilon \tau^{I,2}(x,y,t)+ \varepsilon \tau^{B,2}(x,\frac{y}{\sqrt{\varepsilon}},t)  + \varepsilon^{\frac{3}{2}} \tau^{B,3}(x,\frac{y}{\sqrt{\varepsilon}},t)\notag\\
         & +\sqrt{\varepsilon}  R^{\tau}(x,y),\notag
    \end{align} where $(u^{I,0},\eta^{I,0},\tau^{I,0})$ is the solution to the initial-boundary value problem \eqref{OB_limit_I}-\eqref{OB_limit_I_initial_boundary1}, and the other quantities on the right-hand side except ($R^{u}, R^{\eta}, R^{\tau}$) are given via asymptotic matched expansion, see Section \ref{section_approx}. Indeed, the profiles $(u^{I,j},\eta^{I,j},\tau^{I,j})$ and $(u^{B,j},\eta^{B,j},\tau^{B,j})$ satisfy some linear equations which can be solved in sequence combined with the boundary condition \eqref{OB_main_boundary}. In addition, those profiles enjoy good regularities and decay properties in space, see Section \ref{subsection_regularity}. Then, the problem is converted to estimate the remainder terms $R^{u}(x,y),R^{\eta}(x,y)$ and $R^{\tau}(x,y),$ which relies on several tricky estimates including some delicate anisotropic estimates, see Section \ref{section_theorem_limit} for the details.
   Moreover, the higher order profiles $(u^{I,j},\eta^{I,j},\tau^{I,j})$ ($j\geq 2$) and $(u^{B,j},\eta^{B,j},\tau^{B,j})$ ($j\geq 3$) we introduce are very crucial to get some uniform estimates for high-order derivatives of the solutions such as \eqref{eq_eta_y_convergence_rate} and \eqref{eq_tau_y_convergence_rate}.

   \medskip

   The rest of the paper is organized as follows. In Section \ref{section_prelim}, we present some known inequalities which will be used in the proof of the main theorem. In Section \ref{section_approx}, we obtain some equations of boundary layer profiles and some other related quantities via asymptotic analysis together with wellposedness theorems. The details of the asymptotic analysis is given in Appendix \ref{section_derivation_eqs}. In addition, a system of the error terms is constructed. In Section \ref{section_theorem_limit}, we give a proof of Theorem \ref{theorem_limit}.

\section{Some inequalities}\label{section_prelim}
 In this section, we recall some useful inequalities which will be used later.
\begin{lemma}\label{lemma_inequality}
    Assume that $f,g,h \in H^1_{xy}$ and $u\in H^2_{xy}$. Then the following inequalities:
    \begin{eqnarray}\label{ieq_Lady}
        \|f\|_{L^4_{xy}}  \lesssim \|f\|_{L^2_{xy}}^{\frac{1}{2}}\|\nabla f\|_{L^2_{xy}}^{\frac{1}{2}},
    \end{eqnarray}
    \begin{equation}\label{eq_fgh}
        \int_0^{+\infty}\int_{-\infty}^{+\infty}|fgh|\mathrm{d}x\mathrm{d}y\leq C\|f\|_{L^2_{xy}}\|g\|_{L^2_{xy}}^{\frac{1}{2}}\|\partial_xg\|_{L^2_{xy}}^{\frac{1}{2}}
        \|h\|_{L^2_{xy}}^{\frac{1}{2}}\|\partial_yh\|_{L^2_{xy}}^{\frac{1}{2}},
       \end{equation}
       and
    \begin{equation}\label{eq_u_infty}
            \|u(x,y)\|_{L^\infty_{xy}}\lesssim \|\partial_xu\|_{L_{xy}^2}^{\frac{1}{2}}\|\partial_yu\|_{L_{xy}^2}^{\frac{1}{2}} + \|u\|_{L_{xy}^2}^{\frac{1}{2}}\|\partial_x\partial_yu\|_{L_{xy}^2}^{\frac{1}{2}},
        \end{equation}
        hold.
\end{lemma}
    \begin{proof}
      \eqref{ieq_Lady} is the well-known Ladyzhenskaya's inequality.
        For the proof of \eqref{eq_fgh}, one can refer for instance to Lemma 1 in \cite{Cao_Wu_2011}. Here we sketch the proof of \eqref{eq_u_infty}.
         For $u\in C^\infty( \mathbb{R}^2_+  ) \cap H^2(\mathbb{R}^2_+)$, it holds that
         \begin{eqnarray*}
             |u(x,y)|^2 = \int_y^{+\infty}\int_x^{+\infty}\partial_x\partial_y |u|^2(x,y) \mathrm{d}x \mathrm{d} y =2 \int_y^{+\infty}\int_x^{+\infty}(\partial_xu\cdot\partial_yu + u\cdot \partial_x\partial_yu ) (x,y)\mathrm{d}x \mathrm{d} y.
         \end{eqnarray*}
        Then, using H\"older's inequality, we have
        \begin{eqnarray*}
            \begin{aligned}
                \|u(x,y)\|_{L^\infty_{xy}}^2 \leq {}&  2 \int_0^{+\infty}\int_{-\infty}^{\infty}(|\partial_xu||\partial_yu| + |u| | \partial_x\partial_yu| ) (x,y)\mathrm{d}x \mathrm{d} y\\
                \leq {}&2\left( \|\partial_xu\|_{L^2_{xy}} \|\partial_yu\|_{L^2_{xy}} + \| u\|_{L^2_{xy}} \|\partial_x\partial_yu\|_{L^2_{xy}}  \right),
            \end{aligned}
        \end{eqnarray*}
        which, combined with a density argument, yields  \eqref{eq_u_infty}.
    \end{proof}
    \begin{lemma}[Page 271 in~\cite{Hou_Wang_2019}]\label{lemma_hou_wang}
        Denote $z = y/\sqrt{\varepsilon}$. Then,
        \begin{eqnarray*}
            \left\| \partial^m_y f\left( x, \frac{y}{\sqrt{\varepsilon}}, t\right) \right\|_{H^\ell_xL^2_y}
            = \varepsilon^{\frac{1}{4}-\frac{m}{2}}\|\partial^m_z f (x,z,t)\|_{H^\ell_xL^2_z},
        \end{eqnarray*} and
        \begin{equation*}
             \|  g ( x, 0, t)   \|_{H^\ell_x }^2
            \leq C \sum_{i =0}^\ell\int_{-\infty}^{\infty}  \|\partial^i_x  g (x,z,t)\|_{H^1_z}^2 \mathrm{d}x \leq C\| g(x,z,t)\|_{H^\ell_xH^1_z}^2,
        \end{equation*} hold for any $f(\cdot,\cdot,t) \in H^\ell_xH^m_z$ and $g(\cdot,\cdot,t)\in H^\ell_xH^1_z$ with $\ell,m\in \mathbb{N}$ and fixed $t>0$.

    \end{lemma}

\section{Construction of an approximate solution}\label{section_approx}
In this section, we present the equations of outer and inner layer profiles via asymptotic analysis whose derivations will be given in Appendix \ref{section_derivation_eqs}. With the boundary layer profiles, we can construct an approximate solution to the problem \eqref{OB_main}-\eqref{OB_main_boundary} which is useful in the proof of Theorem \ref{theorem_limit}.

 \subsection{Asymptotic analysis}\label{subsection_asym}
 In this subsection, we derive the equations of the outer and inner layer profiles by using the asymptotic matched expansion method (see e.g. \cite{Holmes_2013, Hou_Wang_2019}). Formally we introduce the following  Prandtl type boundary layer expansion:
 \begin{equation}\label{eq_approx}
    \begin{split}
        &u(x,y,t) = \sum_{j=0}^{\infty}\varepsilon^{\frac{1}{2}j}\Big(u^{I,j}(x,y,t) + u^{B,j}(x,z,t)\Big),~p(x,y,t) = \sum_{j=0}^{\infty}\varepsilon^{\frac{1}{2}j}\Big(p^{I,j}(x,y,t)+ p^{B,j}(x,z,t)\Big),\\
       &\eta(x,y,t) = \sum_{j=0}^{\infty}\varepsilon^{\frac{1}{2}j}\Big(\eta^{I,j}(x,y,t) + \eta^{B,j}(x,z,t)\Big),~\tau(x,y,t) = \sum_{j=0}^{\infty}\varepsilon^{\frac{1}{2}j}\Big(\tau^{I,j}(x,y,t) + \tau^{B,j}(x,z,t)\Big),
    \end{split}
\end{equation} where $z = y/\sqrt{\varepsilon}$.

We assume that
\begin{equation*}
    u^{B,j}(x,z,t)\to 0,~p^{B,j}(x,z,t)\to 0,~\eta^{B,j}(x,z,t)\to 0,~\tau^{B,j}(x,z,t)\to 0,~
\end{equation*}
   fast enough as $z\to +\infty$. Substituting \eqref{eq_approx} to \eqref{OB_limit_I}, \eqref{OB_limit_I_initial_boundary}, and (\ref{OB_limit_I_initial_boundary1}), and applying the matched asymptotic method, we can deduce the equations of outer and inner layer profiles in sequence, see Appendix \ref{section_derivation_eqs} for the details.

\subsubsection{The zeroth and the first order outer profiles}
    The leading order outer layer profile $(u^{I,0},\eta^{I,0},\tau^{I,0})(x,y,t)$ satisfies  \eqref{OB_limit_I}-\eqref{OB_limit_I_initial_boundary1}.
   The first order outer layer profile $(u^{I,1},\eta^{I,1},\tau^{I,1})(x,y,t)$ vanishes identically, see Lemma \ref{lemma_u_I_1}.

\subsubsection{The zeroth and the first order inner profiles}
    Due to the boundary condition \eqref{OB_main_boundary} and the analysis in Appendix \ref{section_derivation_eqs}, we find that the zeroth order and the first order inner layer profiles satisfy
    \begin{equation*}
    \begin{split}
    u^{B,0}=0,\,\,\,\,\tilde{p}^{B,0}=0,\,\,\,\eta^{B,0}=0,\,\,\,\,\tau^{B,0}=0,\,\,\,u^{B,1}=0,
    \end{split}
    \end{equation*}
     see  Lemma \ref{lemma_eq_u_eta_tau_B_0} and Corollary \ref{corollary_u} for the details.

    Next, from Lemma \ref{lemma_u_eta_tau_B_111}, the first order profile $\eta^{B,1}$ fulfills
    \begin{equation}\label{eq_eta_B1}
    \begin{cases}
        \partial_t \eta^{B,1} -\partial_z^2\eta^{B,1} = 0,\\[2mm]
        \eta^{B,1}(x,z,0) = 0,\,\,\partial_z\eta^{B,1}(x,0,t) = -\partial_y\eta^{I,0}(x,0,t).
    \end{cases}
 \end{equation}
 Furthermore, the equations for
        $\tau^{B,1} = \left(\begin{matrix}
            \tau^{B,1}_{11} & \tau^{B,1}_{12}\\[2mm]
            \tau^{B,1}_{12} & \tau^{B,1}_{22}
        \end{matrix}\right)$ are given as follows:
\begin{equation}\label{eq_tau_B1_22}
    \begin{cases}
        \partial_t \tau^{B,1}_{22} + \gamma\tau^{B,1}_{22} - \partial_z^2\tau^{B,1}_{22} = 0,\\[2mm]
        \tau^{B,1}_{22}(x,z,0) = 0,~\partial_z\tau^{B,1}_{22}(x,0,t) = -\partial_y\tau^{I,0}_{22}(x,0,t),
    \end{cases}
\end{equation}
\begin{equation}\label{eq_tau_B1_12}
    \begin{cases}
        \partial_t \tau^{B,1}_{12} + \gamma\tau^{B,1}_{12} - \partial_z^2\tau^{B,1}_{12} +  \frac{1}{\mu}(\tau_{22}^{I,0}(x,0,t) + k\eta^{I,0}(x,0,t))\tau^{B,1}_{12}    \\
         ~~= \partial_yu_1^{I,0}(x,0,t)\tau_{22}^{B,1} + k\eta^{B,1}\partial_y u_1^{I,0}(x,0,t),\\[2mm]
        \tau^{B,1}_{12}(x,z,0) = 0,~\partial_z\tau^{B,1}_{12}(x,0,t) = -\partial_y\tau^{I,0}_{12}(x,0,t),
    \end{cases}
\end{equation}
and
\begin{equation}\label{eq_tau_B1_11}
    \begin{cases}
        \partial_t \tau^{B,1}_{11}   + \gamma\tau^{B,1}_{11} - \partial_z^2\tau^{B,1}_{11}
        = 2\partial_yu_1^{I,0}(x,0,t)\tau_{12}^{B,1}-2\frac{1}{\mu}\tau_{12}^{I,0}(x,0,t)\tau^{B,1}_{12},\\[2mm]
        \tau^{B,1}_{11}(x,z,0) = 0,~\partial_z\tau^{B,1}_{11}(x,0,t) = -\partial_y\tau^{I,0}_{11}(x,0,t).
    \end{cases}
\end{equation}
Finally, from \eqref{eq_u_B2}, it holds that $$\tilde{p}^{B,1} = \tau_{22}^{B,1}.$$
\subsubsection{The second order inner profiles}
From    Lemma \ref{lemma_u_eta_tau_B_2222}, we find that
\begin{eqnarray}\label{eq_u_B_2_exp}
    u_1^{B,2} =\frac{1}{\mu} \int_z^{+\infty}\tau_{12}^{B,1}\mathrm{d} z,~~u_2^{B,2} = 0.
\end{eqnarray}
Next, $\eta^{B,2}$ solves
\begin{equation}\label{eq_eta_B2}
    \begin{cases}
        \partial_t \eta^{B,2} -\partial_z^2\eta^{B,2} =- u_1^{B,2}\partial_x\eta^{I,0}(x,0,t) -
              \frac{1}{2}z^2\partial_y^2u_2^{I,0}(x,0,t)\partial_z\eta^{B,1}\\
        ~~~~~~~~~~~~~~~~~~~~~~~~ -z\partial_yu_{1}^{I,0}(x,0,t)\partial_x\eta^{B,1},\\[2mm]
        \eta^{B,2}(x,z,0) = 0,~\partial_z\eta^{B,2}(x,0,t) =  0.
    \end{cases}
\end{equation}
The last component of $\tau^{B,2}$ satisfies
\begin{equation}\label{eq_tau_B2_22}
    \begin{cases}
        \partial_t \tau^{B,2}_{22} + \gamma \tau^{B,2}_{22} -\partial_z^2\tau^{B,2}_{22}   = \mathfrak{h}^{B,2}_{22},\\[2mm]
        \tau^{B,2}_{22}(x,z,0) = 0,~\partial_z\tau^{B,2}_{22}(x,0,t)=0,
    \end{cases}
\end{equation}
where
\begin{eqnarray*}
    \begin{aligned}
        \mathfrak{h}^{B,2}_{22} = &-u^{B,2}_1\partial_x\tau^{I,0}_{22}(x,0,t)+2 \tau^{I,0}_{22}(x,0,t)\partial_x u_1^{B,2}
        +2k\eta^{I,0}(x,0,t)\partial_xu_1^{B,2} \\
        &-\frac{1}{2}z^2\partial_y^2 u_2^{I,0}(x,0,t)\partial_z\tau^{B,1}_{22}- z\partial_y u_1^{I,0}(x,0,t)\partial_x\tau_{22}^{B,1} +2z\partial_y^2u_2^{I,0}(x,0,t)\tau^{B,1}_{22}\\
        & + 2kz\partial_y^2u_2^{I,0}(x,0,t)\eta^{B,1}.
    \end{aligned}
\end{eqnarray*}
The middle component of $\tau^{B,2}$ satisfies
\begin{equation}\label{eq_tau_B2_12}
    \begin{cases}
        \partial_t \tau^{B,2}_{12} + \gamma \tau^{B,2}_{12} -\partial_z^2\tau^{B,2}_{12} +[\frac{1}{\mu}\tau_{22}^{I,0}(x,0,t) +\frac{k}{\mu}\eta^{I,0}(x,0,t)]\tau^{B,2}_{12}   = \mathfrak{h}^{B,2}_{12},\\[2mm]
        \tau^{B,2}_{12}(x,z,0) = 0,~\partial_z\tau^{B,2}_{22}(x,0,t)=0,
    \end{cases}
\end{equation}
where
    \begin{align*}
        \mathfrak{h}^{B,2}_{12} = &  \partial_y u^{I,0}_1(x,0,t)\tau^{B,2}_{22}- u^{B,2}_1\partial_x\tau^{I,0}_{12}(x,0,t)  + \tau_{22}^{B,1} \partial_zu_1^{B,2}   +k\eta^{B,2}\partial_yu_1^{I,0}(x,0,t)
     +k  \eta^{B,1} \partial_z u_1^{B,2}\\
    & -\frac{1}{2}z^2\partial_y^2 u_2^{I,0}(x,0,t)\partial_z\tau^{B,1}_{12} - z\partial_y u_1^{I,0}(x,0,t)\partial_x\tau_{12}^{B,1} + z\partial_y^2u_1^{I,0}(x,0,t)\tau^{B,1}_{22}
    + z\partial_y\tau^{I,0}_{22}(x,0,t)\partial_zu^{B,2}_1 \\
     &
     +kz\partial_y\eta^{I,0}(x,0,t)\partial_zu^{B,2}_1 + kz\partial_y^2u^{I,0}_1(x,0,t)\eta^{B,1}+\frac{1}{\mu}[\tau_{22}^{I,0}(x,0,t) +k\eta^{I,0}(x,0,t)]\\
    & \times\int_{z}^{+\infty}(\partial_x\tau_{11}^{B,1} - \partial_x\tau_{22}^{B,1})\mathrm{d}z.
    \end{align*}
The first component of $\tau^{B,2}$ satisfies
\begin{equation}\label{eq_tau_B2_11}
    \begin{cases}
        \partial_t \tau^{B,2}_{11} + \gamma \tau^{B,2}_{11} -\partial_z^2\tau^{B,2}_{11}
        = \mathfrak{h}^{B,2}_{11},\\[2mm]
        \tau^{B,2}_{11}(x,z,0) = 0,~\partial_z\tau^{B,2}_{11}(x,0,t) =   0,
    \end{cases}
\end{equation}
where
    \begin{align*}
        \mathfrak{h}^{B,2}_{11} =  & -  \left[\frac{2}{\mu}\tau^{I,0}_{12}(x,0,t)-2 \partial_yu_1^{I,0}(x,0,t)\right]\tau^{B,2}_{12} +  \frac{2}{\mu}\tau^{I,0}_{12}(x,0,t)\int_{z}^{+\infty}(\partial_x\tau_{11}^{B,1} - \partial_x\tau_{22}^{B,1})\mathrm{d}z \\
        &-u^{B,2}_1\partial_x\tau^{I,0}_{11}(x,0,t)-\frac{1}{2}z^2\partial_y^2 u_2^{I,0}(x,0,t)\partial_z\tau^{B,1}_{11} -z\partial_y u_1^{I,0}(x,0,t)\partial_x\tau_{11}^{B,1} \\
        &+2\Big[
        \tau^{I,0}_{11}(x,0,t)\partial_xu^{B,2}_1    + \tau_{12}^{B,1} \partial_zu^{B,2}_1 + k\eta^{I,0}(x,0,t)\partial_x u^{B,2}_1 +z\partial_x\partial_yu^{I,0}_1(x,0,t)\tau^{B,1}_{11}  \\
        &
        +z\partial_y^2u^{I,0}_1(x,0,t)\tau^{B,1}_{12}+z\partial_y\tau^{I,0}_{12}(x,0,t)\partial_zu^{B,2}_1+ kz\partial_x\partial_yu^{I,0}_1(x,0,t)\eta^{B,1} \Big].
    \end{align*}
    Finally, from   \eqref{eq_u_B3_1}, we find that
\begin{equation}\label{eq_u_B3_p_B2}
        \tilde{p}^{B,2} = -\mu \partial_x u_1^{B,2} -\int_z^{+\infty}\partial_x\tau_{12}^{B,1}\mathrm{d}z + \tau_{22}^{B,2}.
\end{equation}
\subsubsection{The second and the third order outer profiles}
The second and the third order outer layer profiles satisfy the following linearized Oldroyd-B type system:
 \begin{equation}\label{OB_outer_second}
    \begin{cases}
    \partial_tu^{I,2} +u^{I,0}\cdot\nabla u^{I,2}   +u^{I,2}\cdot\nabla u^{I,0} +\nabla \tilde{p}^{I,2} -\mu \Delta u^{I,2}
    -\mathrm{div} \tau^{I,2} = 0,\\[2mm]
    \partial_t\eta^{I,2}+ u^{I,0}\cdot\nabla \eta^{I,2} + u^{I,2}\cdot\nabla \eta^{I,0} = \Delta \eta^{I,0},\\[2mm]
    \partial_t\tau^{I,2}+u^{I,0}\cdot \nabla \tau^{I,2} +u^{I,2}\cdot \nabla \tau^{I,0}+ \gamma\tau^{I,2}-\mathcal{Q}(\nabla u^{I,0}, \tau^{I,2})  \\
    ~~~-\mathcal{Q}(\nabla u^{I,2}, \tau^{I,0}) - \mathcal{B} (\eta^{I,2},\nabla u^{I,0}) - \mathcal{B} (\eta^{I,0},\nabla u^{I,2})=\Delta \tau^{I,0},\\[2mm]
    \mathrm{div} u^{I,2} = 0,\\[2mm]
    (u^{I,2},\eta^{I,2},\tau^{I,2})(x,y,0) = 0,~~u^{I,2}(x,0,t) = -u^{B,2}(x,0,t),\\
    \end{cases}
\end{equation}
where $\displaystyle u_1^{B,2}(x,0,t) = \frac{1}{\mu} \int_0^{+\infty}\tau_{12}^{B,1}\mathrm{d} z$  and $u_2^{B,2}(x,0,t) = 0,$ and
\begin{equation}\label{OB_outer_third}
    \begin{cases}
        \partial_tu^{I,3} +u^{I,0}\cdot\nabla u^{I,3}   +u^{I,3}\cdot\nabla u^{I,0} +\nabla \tilde{p}^{I,3} -\mu \Delta u^{I,3}
        -\mathrm{div} \tau^{I,3} = 0,\\[2mm]
        \partial_t\eta^{I,3}+ u^{I,0}\cdot\nabla \eta^{I,3} + u^{I,3}\cdot\nabla \eta^{I,0} =0,\\[2mm]
        \partial_t\tau^{I,3}+u^{I,0}\cdot \nabla \tau^{I,3} +u^{I,3}\cdot \nabla \tau^{I,0}+ \gamma\tau^{I,3}-\mathcal{Q}(\nabla u^{I,0}, \tau^{I,3})  \\
        ~~~-\mathcal{Q}(\nabla u^{I,3}, \tau^{I,0}) - \mathcal{B} (\eta^{I,3},\nabla u^{I,0}) - \mathcal{B} (\eta^{I,0},\nabla u^{I,3})=0,\\[2mm]
        \mathrm{div} u^{I,3} = 0,\\[2mm]
        (u^{I,3},\eta^{I,3},\tau^{I,3})(x,y,0) = 0,~~u^{I,3}(x,0,t) = -u^{B,3}(x,0,t),
        \end{cases}
    \end{equation}
    where $u^{B,3}(x,0,t)$ is determined by \eqref{eq_u_B_3333}:
    \begin{eqnarray}
        \begin{cases}
         u_1^{B,3}(x,0,t) =  \frac{1}{\mu}\int_0^{+\infty}\tau_{12}^{B,2}(x,z,t)\mathrm{d}z -\frac{1}{\mu}\int_0^{+\infty}\int_{z}^{+\infty}(\partial_x\tau_{11}^{B,1} - \partial_x\tau_{22}^{B,1})(x,\xi,t)\mathrm{d}\xi\mathrm{d}z,\\[3mm]
        u_2^{B,3}(x,0,t) = \int_0^{+\infty} \partial_x u_1^{B,2}(x,z,t)\mathrm{d}z.
    \end{cases}
    \end{eqnarray}
  \subsubsection{The third order inner profiles}
  From Lemma \ref{lemma_u_eta_tau_B_3333}, the third order inner profiles $u^{B,3}_1$ and $u^{B,3}_2$ are determined by
  \begin{eqnarray}\label{eq_u_B_3}
    \begin{cases}
    u_1^{B,3} =  \frac{1}{\mu}\int_z^{+\infty}\tau_{12}^{B,2}(x,z,t) \mathrm{d}z -\frac{1}{\mu}\int_{z}^{+\infty}\int_{z}^{+\infty}(\partial_x\tau_{11}^{B,1} - \partial_x\tau_{22}^{B,1})(x,\xi,t)\mathrm{d}\xi\mathrm{d}z,\\[3mm]
    u_2^{B,3} = \int_z^{+\infty} \partial_x u_1^{B,2}\mathrm{d}z.
\end{cases}
\end{eqnarray}
  And $\eta^{B,3}$ solves
  \begin{equation}\label{eq_eta_B3}
    \begin{cases}
        \partial_t \eta^{B,3} -\partial_z^2\eta^{B,3}  = \mathfrak{h}_{\eta}^{B,3},\\[2mm]
        \eta^{B,3}(x,z,0) = 0,~\partial_z\eta^{B,3}(x,0,t) = -\partial_y \eta^{I,2}(x,0,t),
    \end{cases}
\end{equation}
where
    \begin{align*}
        \mathfrak{h}_{\eta}^{B,3} = &- u_1^{B,3}\partial_x\eta^{I,0}(x,0,t) - u_2^{B,3}\partial_y\eta^{I,0}(x,0,t)
        + \partial_x^2\eta^{B,1}- [u_1^{I,2}(x,0,t) + u_1^{B,2}]\partial_x \eta^{B,1}  \\
        & - [u_2^{I,3}(x,0,t)+ u_2^{B,3}]\partial_z \eta^{B,1}
        - \frac{1}{6}z^3\partial_y^3u_2^{I,0}(x,0,t)\partial_z\eta^{B,1} -\frac{1}{2}z^2\partial_y^2u_{1}^{I,0}(x,0,t)\partial_x\eta^{B,1}\\
        ~& -\frac{1}{2}z^2\partial_y^2u_{2}^{I,0}(x,0,t)\partial_z\eta^{B,2} - z\partial_y\partial_x \eta^{I,0}(x,0,t) u_1^{B,2}
        -z\partial_y u_1^{I,0}(x,0,t)\partial_x\eta^{B,2}\\
        ~&  -z\partial_y u_2^{I,2}(x,0,t)\partial_z\eta^{B,1}.
    \end{align*}

For $\tau_{22}^{B,3}$, we have
\begin{equation}\label{eq_tau_B3_22}
    \begin{cases}
        \partial_t \tau^{B,3}_{22} + \gamma \tau^{B,3}_{22} -\partial_z^2\tau^{B,3}_{22}   = \mathfrak{h}^{B,3}_{22},\\[2mm]
        \tau^{B,3}_{22}(x,z,0) = 0,~\partial_z\tau^{B,3}_{22}(x,0,t)= -\partial_y \tau^{I,2}_{22}(x,0,t),
    \end{cases}
\end{equation}
where $\mathfrak{h}^{B,3}_{22}$ is given in \eqref{eq_tau_B3_22_h}.

For $\tau_{12}^{B,3}$, we have
\begin{equation}\label{eq_tau_B3_12}
    \begin{cases}
        \partial_t \tau^{B,3}_{12} + \gamma \tau^{B,3}_{12} -\partial_z^2\tau^{B,3}_{12} + \frac{1}{\mu}[\tau_{22}^{I,0}(x,0,t) + k\eta^{I,0}(x,0,t)]\tau_{12}^{B,3}  = \mathfrak{h}^{B,3}_{12},\\[2mm]
        \tau^{B,3}_{12}(x,z,0) = 0,~\partial_z\tau^{B,3}_{12}(x,0,t)= -\partial_y \tau^{I,2}_{12}(x,0,t),
    \end{cases}
\end{equation}
where $\mathfrak{h}^{B,3}_{12}$ is given in \eqref{eq_tau_B3_12_h}.

For $\tau_{11}^{B,3}$, we have
\begin{equation}\label{eq_tau_B3_11}
    \begin{cases}
        \partial_t \tau^{B,3}_{11} + \gamma \tau^{B,3}_{11} -\partial_z^2\tau^{B,3}_{11} + \frac{2}{\mu} \tau_{12}^{I,0}(x,0,t)  \tau_{11}^{B,3}  = \mathfrak{h}^{B,3}_{11},\\[2mm]
        \tau^{B,3}_{11}(x,z,0) = 0,~\partial_z\tau^{B,3}_{11}(x,0,t)= -\partial_y \tau^{I,2}_{11}(x,0,t),
    \end{cases}
\end{equation}
where $\mathfrak{h}^{B,3}_{11}$ is given in \eqref{eq_tau_B3_11_h}.

Finally, from \eqref{eq_u_B4}, we know that
\begin{eqnarray}\label{eq_u_B4_2}
    \begin{cases}
        {u_1^{B,4} =  \frac{1}{\mu}\int_z^{+\infty}\tau_{12}^{B,3}\mathrm{d}z  - \frac{1}{\mu}\int_z^{+\infty} \int_\xi^{+\infty}\left(\mu\partial_x^2u_1^{B,2} + \partial_x\tau_{11}^{B,2}  -\partial_t u_1^{B,2} -\partial_x\tilde{p}^{B,2} \right)(x,\zeta,t)\mathrm{d}\zeta\mathrm{d}\xi,}\\[3mm]
          u_2^{B,4} = \int_{z}^{+\infty} \partial_x u_1^{B,3}\mathrm{d}z, ~~u_2^{B,5} = \int_{z}^{+\infty} \partial_x u_1^{B,4}\mathrm{d}z,\\[2mm]
           \tilde{p}^{B,3} = \tau_{22}^{B,3} - \mu \partial_x u_1^{B,3} - \int_z^{+\infty} \partial_x \tau_{12}^{B,2}\mathrm{d} z.
    \end{cases}
\end{eqnarray}

\subsection{Regularity of outer and inner layer profiles}\label{subsection_regularity}
  To prove Theorem \ref{theorem_limit}, we need some regularity results on the outer and inner layer profiles.
   To begin with, we introduce the following Neumann problem for a linear parabolic equation which will be used to analyze the well-posedness of the inner layer profiles, i.e.,
    \begin{equation}\label{boundary_prototype}
        \begin{cases}
            \partial_t\theta (x,z,t)-\partial_z^2\theta(x,z,t) + \gamma \theta(x,z,t) + a(x,t)\theta(x,z,t)  = r(x,z,t),\\
            \theta(x,z,0) =0,\\[2mm]
             \partial_z\theta(x,0,t) = 0,
        \end{cases}
    \end{equation}
    where $\gamma$ is a nonnegative constant. Then, we have the following well-posedness results.
    \begin{proposition}\label{prop_boundary_prototype}
        For any given $T\in(0,\infty)$ and the integer $m \geq 2$, assume that $r(x,z,t)$ and $a(x,t)$ satisfy
        \begin{equation*}
            \langle z\rangle^\ell\partial_t^i r\in L^2(0,T;H^{m-2i}_xL^2_z),~~~i=0,1,\cdots,[\frac{m}{2}],
        \end{equation*}
         and
        \begin{equation*}
            \partial_t^j a\in L^2(0,T;H^{m-2j+1}_x),~~~j=0,1,\cdots,[\frac{m+1}{2}],
        \end{equation*}
        where $\ell\in \mathbb{N}$ and $\left[\frac{m}{2}\right]$ is the integer part of $\frac{m}{2}$, and that
        $r$ satisfies the compatibility conditions up to order $[\frac{m}{2}]-1$ for the initial-boundary value problem \eqref{boundary_prototype}, i.e., $\partial_t^kr\Big|_{t=0}=0$ for $k=0,1,\cdots,[\frac{m}{2}]-1$, then
         there exists a unique solution $\theta(x,z,t)$ of \eqref{boundary_prototype} on $[0,T]$ such that
        \begin{equation*}
            \begin{split}
                &   \langle z\rangle^\ell \partial_t^i\theta \in L^{\infty}(0,T;H^{m-2i}_xH^1_z)\cap L^2(0,T;H_x^{m-2i}H_z^2),~~~i=0,1,\cdots,[\frac{m}{2}],
            \end{split}
        \end{equation*}
        and
        \begin{eqnarray*}
            \langle z\rangle^\ell \partial_t^{[\frac{m}{2}]+1}\theta  \in  L^2(0,T; H^{m-2[\frac{m}{2}]}_xL^2_{z}).
        \end{eqnarray*}
    \end{proposition}
    \begin{remark}
        The well-posedness of problem \eqref{boundary_prototype} is standard, one can refer for instance to \cite{Ladyzenskaja_etal_1968}. The regularity stated in  Proposition \ref{prop_boundary_prototype} can be referred to \cite{Hou_Wang_2019} and the references therein.
    \end{remark}

    \bigskip

    Next, in order to analyze the well-posedness of the higher-order outer layer profiles, we introduce the following initial-boundary value problem:
    \begin{equation}\label{Outer_prototype}
        \begin{cases}
            \partial_t \bar{u} + \bar{v}\cdot\nabla \bar{u} + \bar{u}\cdot\nabla \bar{v} + \nabla \bar{\Pi} - \mu \Delta \bar{u} - \mathrm{div} \bar{\tau} = \mathfrak{f},\\[2mm]
            \partial_t \bar{\eta} + \bar{v}\cdot\nabla \bar{\eta} + \bar{u}\cdot\nabla \bar{\xi}  = \mathfrak{g}, \\[2mm]
            \partial_t\bar{\tau}+ \bar{v}\cdot \nabla \bar{\tau} +\bar{u}\cdot \nabla \bar{\sigma}+ \gamma\bar{\tau}- {\mathcal{Q}}(\nabla \bar{v}, \bar{\tau})  \\
        ~~~-\mathcal{Q}(\nabla \bar{u}, \bar{\sigma}) - \mathcal{B} (\bar{\eta},\nabla \bar{v}) - \mathcal{B} (\bar{\xi},\nabla \bar{u})= \mathfrak{j},\\[2mm]
        \mathrm{div} \bar{u} = 0,\\[2mm]
        (\bar{u},\bar{\eta},\bar{\tau})(x,y,0) =(\bar{u}_0,\bar{\eta}_0,\bar{\tau}_0)(x,y)  ,\, \bar{u}(x,0,t) = 0,
        \end{cases}
    \end{equation}
    whose well-posedness is stated as follows.
    \begin{proposition}\label{prop_outer_prototype}
   For any given $T \in(0, \infty)$ and  $m\in \mathbb{N}_+$, assume that
            $$\mathrm{div} \bar{v} = 0,~\bar{v}(x,0,t) = 0,~ (\bar{u}_0,\bar{\eta}_0,\bar{\tau}_0) \in H^{2m+1}_{xy},$$  and
        \begin{align*}
            & \partial_{t}^i \bar{v},~\partial_{t}^i \bar{\xi},~\partial_{t}^i \bar{\sigma} \in L^\infty_TH^{2m+1-2i}_{xy} \cap L^2_TH^{2m+2-2i}_{xy},~i = 0,1,2,\cdots,m,\\
            &\partial_t^{m}\mathfrak{f} \in  L^2_TL^{2}_{xy},~\partial_t^{j}\mathfrak{f}\in  L^\infty_TH^{2m-1-2j}_{xy} \cap L^2_TH^{2m-2j}_{xy},~j =  0,1,2,\cdots,m-1,\\
            &\partial_t^{i}\mathfrak{g},~\partial_t^{i}\mathfrak{j} \in  L^\infty_TH^{2m-2i}_{xy} \cap L^2_TH^{2m + 1 -2i}_{xy},~i =  0,1,2,\cdots,m,
        \end{align*}
        and further that $\bar{v},\bar{\xi},\bar{\sigma},\mathfrak{f},\mathfrak{g},$ and $\mathfrak{j}$ satisfy   the compatibility conditions up to order ${m}$  for the initial-boundary problem \eqref{Outer_prototype},  i.e.,
        \begin{eqnarray*}
            \partial_t^{\ell} \bar{u}(0)|_{y = 0} = 0,~~0 \leq \ell \leq m,
        \end{eqnarray*}
   then
         there exists a unique solution $(\bar{u},\bar{\eta},\bar{\tau})(x,z,t)$ of \eqref{Outer_prototype} on $[0,T]$ such that
         \begin{align*}
            &\partial_t^{m+1}\bar{u}\in L^2_TL_{xy}^2,~\partial_t^i\bar{u}\in L_T^{\infty}H_{xy}^{2m+1-2i}\cap L_T^2H_{xy}^{2m+2-2i},~i = 0,1,\cdots,m,\\[2mm]
            &\bar{\eta},\bar{\tau} \in L_T^\infty H_{xy}^{2m+1},~\partial_t^k\bar{\eta},\partial_t^k\bar{\tau}\in  L_T^\infty H_{xy}^{2m+2-2k},~k=1,2,\cdots,m+1.
         \end{align*}
    \end{proposition}
    \begin{remark}
        Since the system \eqref{Outer_prototype}  is linear, the proof of Proposition \ref{prop_outer_prototype}
        is standard for the higher regularity of the given coefficients and initial data.
        One of the ways is to make each equation of the system \eqref{Outer_prototype} decoupled from each other in a way like “linearization”.  Then one can use the standard
iteration arguments or the fixed-point
arguments to obtain the local existence and uniqueness, see for instance \cite{Fang_Hieber_Zi_2013}.
Due to the linearity of system \eqref{Outer_prototype}, one can deduce some global {\it a priori} estimates by the standard energy method.
    Finally,  the global wellposedness can be proved via local existence + global {\it a priori}. We omit the details here for brevity.
    \end{remark}

    \bigskip

     Based on Propositions \ref{prop_boundary_prototype} and \ref{prop_outer_prototype}, we can establish wellposedness of the inner and outer layer profiles successively. To begin with, for the higher order regularities of the problems \eqref{eq_eta_B1}, \eqref{eq_tau_B1_22}, \eqref{eq_tau_B1_12} and \eqref{eq_tau_B1_11}, we need the following strong compatibility conditions:
    \begin{eqnarray} \label{eq_compatibility_eta_tau}
        (\partial_t^{j}\partial_y\eta^{I,0},\partial_t^{j}\partial_y\tau^{I,0})(0)|_{y=0} = 0,~~ j= 1\cdots,5.
   \end{eqnarray}
   Then, we have the following results.
    \begin{lemma}\label{lemma_eta_tau_b_1}
         Under the assumptions of Theorem \ref{theorem_limit}, for any fixed $T\in(0,T_0)$, there exist unique
        solutions $\eta^{B,1},\tau^{B,1}_{22},\tau^{B,1}_{12},$ $\tau^{B,1}_{11}$ of the problems \eqref{eq_eta_B1}, \eqref{eq_tau_B1_22}, \eqref{eq_tau_B1_12}  and \eqref{eq_tau_B1_11} on $[0,T]$,
        satisfying
        \begin{eqnarray}\label{eq_eta_tau_B_1_1}
            \begin{aligned}
                \langle z \rangle^\ell \partial_t^j (\eta^{B,1},
                  \tau_{22}^{B,1},  \tau_{12}^{B,1},  \tau_{11}^{B,1})    \in L^\infty(0,T;H^{11-2j}_xH^1_z) \cap L^2(0,T;H^{{ 11}-2j}_xH^2_z),
            \end{aligned}
        \end{eqnarray} for all $\ell \in \mathbb{N}$ and $j = 0,1,\cdots,5$.
       Furthormore, using \eqref{eq_eta_B1}, \eqref{eq_tau_B1_22}, \eqref{eq_tau_B1_12}  and \eqref{eq_tau_B1_11},    we have
            \begin{align}
            &\langle z \rangle^\ell (\eta^{B,1},  \tau_{22}^{B,1},  \tau_{12}^{B,1},
             \tau_{11}^{B,1}) \in L^\infty(0,T;H^{{ 9}}_xH^3_z),\label{eq_eta_tau_B_1_2} \\[2mm]
            &\langle z \rangle^\ell \partial_t (\eta^{B,1}, \tau_{22}^{B,1}, \tau_{12}^{B,1},
             \tau_{11}^{B,1}) \in L^\infty(0,T;H^{{7}}_xH^3_z).\label{eq_eta_tau_B_1_t_2}
            \end{align}
        As a consequence, for $u_1^{B,2}$ and $u_2^{B,3}$, we have
            \begin{align}
                &\langle z \rangle^\ell \partial_t^j u_1^{B,2}   \in L^\infty(0,T;H^{11-2j}_xH^2_z) \cap L^2(0,T;H^{{11}-2j}_xH^3_z),~~j = 0,1,\cdots,5, \label{eq_u_B_2_1_1}\\[2mm]
                &\langle z \rangle^\ell \partial_t^j u_2^{B,3}   \in L^\infty(0,T;H^{10-2j}_xH^2_z) \cap L^2(0,T;H^{{10}-2j}_xH^3_z),~~j = 0,1,\cdots,5. \label{eq_u_B_3_2_1}
            \end{align}
    \end{lemma}

\begin{proof}
         Letting $\varphi$ be cut-off function defined in \eqref{eq_phi},
         and defining $$\tilde{\eta}^{B,1}(x,z,t) = \eta^{B,1}(x,z,t) + z\varphi(z)\partial_y\eta^{I,0}(x,0,t),$$ we deduce from \eqref{eq_eta_B1} that
         \begin{equation}\label{eq_eta_B1_tilde}
            \begin{cases}
                \partial_t \tilde{\eta}^{B,1} -\partial_z^2\tilde{\eta}^{B,1} = r^{\eta,1}(x,z,t),\\[2mm]
                \tilde{\eta}^{B,1}(x,z,0) = 0,\,\,\partial_z\tilde{\eta}^{B,1}(x,0,t) = 0,
            \end{cases}
         \end{equation}
         where $$r^{\eta,1}(x,z,t) =  z\varphi(z)\partial_t \partial_y\eta^{I,0}(x,0,t) -\partial_z^2(z\varphi(z))\partial_y\eta^{I,0}(x,0,t).$$
         A direct estimate together with Proposition \ref{local_wellposedness_0} and Lemma \ref{lemma_hou_wang} yields that
    \begin{eqnarray*}
        \begin{aligned}
            &\| \langle z\rangle^\ell \partial_t^jr^{\eta,1} \|_{L^2_TH^{{11}-2j}_xL^2_z} \\
            \lesssim& \|\partial_t^{j + 1}\partial_y\eta^{I,0}(x,0,t)\|_{L^2_TH^{11-2j}_x}\|\langle z \rangle^{\ell +1}\varphi(z) \|_{L^2_z}
            + \|\partial_t^{j}\partial_y\eta^{I,0}(x,0,t)\|_{L^2_TH^{11-2j}_x}\|\langle z \rangle^{\ell}\partial_z^2(z\varphi(z)) \|_{L^2_z} \\
            \lesssim &\|\partial_t^{j + 1} \eta^{I,0} \|_{L^2_TH^{15-2(j+1)}_{xy}}
            + \|\partial_t^{j} \eta^{I,0} \|_{L^2_TH^{13-2j}_{xy}}
            \leq C,\\
        \end{aligned}
    \end{eqnarray*}
    for $j = 0, 1,\cdots,5.$  Moreover, from \eqref{eq_compatibility_eta_tau}, it is not difficult to check that $r^{\eta,1}$ satisfies the fourth order compatibility condition stated in Proposition \ref{prop_boundary_prototype}.
     Therefore, using Proposition \ref{prop_boundary_prototype} with  $m={11}$, we get
    \begin{eqnarray*}
        \langle z\rangle^\ell \partial_t^j\tilde{\eta}^{B,1} \in L^{\infty}(0,T;H^{11-2j}_xH^1_z)\cap L^2(0,T;H_x^{11-2j}H_z^2),~~~j=0,1,\cdots,5.
    \end{eqnarray*}
    Noticing that $$\langle z\rangle^\ell  z\varphi(z)\partial_y\partial_t^j\eta^{I,0}(x,0,t) \in  L^{\infty}(0,T;H^{11-2j}_xH^1_z)\cap L^2(0,T;H_x^{11-2j}H_z^2),~~~j=0,1,\cdots,5,$$ we have
    \begin{eqnarray*}
        \langle z\rangle^\ell \partial_t^j \eta^{B,1} \in L^{\infty}(0,T;H^{11-2j}_xH^1_z)\cap L^2(0,T;H_x^{11-2j}H_z^2),~~~j=0,1,\cdots,5,
    \end{eqnarray*}
    which together with   \eqref{eq_eta_B1} yields the desired regularity for $\eta^{B,1}$ in \eqref{eq_eta_tau_B_1_1}, \eqref{eq_eta_tau_B_1_t_2} and \eqref{eq_eta_tau_B_1_t_2}.
    A similar argument shows that $\tau_{22}^{B,1}$ enjoys the desired regularity.

     Next, denoting
    $\tilde{\tau}_{12}^{B,1}(x,z,t) = \tau_{12}^{B,1}(x,z,t) + z\varphi(z)\partial_y\tau_{12}^{I,0}(x,0,t)$, and using \eqref{eq_tau_B1_12}, we have
    \begin{equation}\label{eq_tau_B1_12_tilde}
        \begin{cases}
            \partial_t \tilde{\tau}^{B,1}_{12} - \partial_z^2\tilde{\tau}^{B,1}_{12}+ \gamma\tilde{\tau}^{B,1}_{12}  +  a^{\tau,1}_{12}\tilde{\tau}^{B,1}_{12}
            = r^{\tau,1}_{12}, \\[2mm]
             \tilde{\tau}^{B,1}_{12}(x,z,0) = 0,~\partial_z\tilde{\tau}^{B,1}_{12}(x,0,t) = 0,
        \end{cases}
    \end{equation}
    where
    \begin{eqnarray*}
        \begin{aligned}
            a^{\tau,1}_{12}(x,t) = ~& \frac{1}{\mu}[ \tau_{22}^{I,0}(x,0,t) + k\eta^{I,0}(x,0,t) ],\\
            r^{\tau,1}_{12}(x,z,t) =~ & \partial_yu_1^{I,0}(x,0,t)\tau_{22}^{B,1} + k\eta^{B,1}\partial_y u_1^{I,0}(x,0,t) + z\varphi(z)\partial_t \partial_y\tau_{12}^{I,0}(x,0,t)  + \gamma z\varphi(z) \partial_y\tau_{12}^{I,0}(x,0,t)\\
            &+ \frac{1}{\mu}[ \tau_{22}^{I,0}(x,0,t) + k\eta^{I,0}(x,0,t) ]z\varphi(z) \partial_y\tau_{12}^{I,0}(x,0,t)
        -\partial_z^2(z\varphi(z))\partial_y\tau_{12}^{I,0}(x,0,t). \\
        \end{aligned}
    \end{eqnarray*}
    From Proposition \ref{local_wellposedness_0}, Lemma \ref{lemma_hou_wang}, and the regularity of $\eta^{B,1}$ and $\tau_{22}^{B,1},$ we get
    \begin{eqnarray*}
        \begin{aligned}
            \| \partial_t^j a^{\tau,1}_{12}(x,t)\|_{L^2_TH^{12-2j}_x}
            \lesssim{}  & \| \partial_t^j \tau_{22}^{I,0}(x,0,t)\|_{L^2_TH^{12-2j}_x} + \| \partial_t^j \eta^{I,0}(x,0,t)\|_{L^2_TH^{12-2j}_x}\\
            \lesssim {}& \| \partial_t^j \tau_{22}^{I,0} \|_{L^2_TH^{13-2j}_{xy}} + \| \partial_t^j \eta^{I,0} \|_{L^2_TH^{13-2j}_{xy}}  \leq C,
        \end{aligned}
    \end{eqnarray*}
and
        \begin{align*}
            &\| \langle z\rangle^\ell \partial_t^jr^{\tau,1}_{12} \|_{L^2_TH^{11-2j}_xL^2_z} \\
            \lesssim& \sum_{i = 0}^{j}\|\partial_t^{i}\partial_y u_1^{I,0}(x,0,t)\|_{L^\infty_T H^{12-2j}_x  }
             \|\langle z \rangle^{\ell}\partial_t^{j-i}(\tau_{22}^{B,1},\eta^{B,1})\|_{L^2_T H^{11-2j}_x  L^2_z}
              \\
            & + \Biggl( \sum_{i = 0}^{j} \|\partial_t^{i}(\tau_{22}^{I,0},\eta^{I,0})(x,0,t)\|_{L^\infty_T    H^{12-2j}_x} \|\partial_t^{j-i}\partial_y\tau_{12}^{I,0}(x,0,t)\|_{L^2_TH^{ 11 -2j}_x}
             \\
            & +  \|\partial_t^{j + 1}\partial_y\tau_{12}^{I,0}(x,0,t)\|_{L^2_TH^{ 11 -2j}_x}
            + \|\partial_t^{j}\partial_y\tau_{12}^{I,0}(x,0,t)\|_{L^2_TH^{ 11 -2j}_x}\Biggr )\|\langle z \rangle^{\ell +1}\varphi(z) \|_{L^2_z}\\
            &+ \|\partial_t^{j}\partial_y\tau_{12}^{I,0}(x,0,t)\|_{L^2_TH^{ 11 -2j}_x}\|\langle z \rangle^{\ell}\partial_z^2(z\varphi(z)) \|_{L^2_z} \\
            \lesssim &  \sum_{i = 0}^{j}\|\partial_t^{i} u_1^{I,0} \|_{L^\infty_T{   H^{14-2i}_{xy}}} \sum_{i = 0}^{j}    \Bigl(\|\langle z \rangle^{\ell}\partial_t^{i}\tau_{22}^{B,1}\|_{L^2_T  H^{ 11 -2i}_x  L^2_z}
            +\|\langle z \rangle^{\ell}\partial_t^{i}\eta^{B,1}\|_{L^2_T   H^{ 11 -2i}_x   L^2_z}\Bigr) \\
            &+  \sum_{i = 0}^{j}\Bigl(\|\partial_t^{i}\tau_{22}^{I,0} \|_{L^\infty_T{   H^{13-2i}_{xy}}} + \|\partial_t^{i}\eta^{I,0} \|_{L^\infty_T    H^{13-2i}_{xy} } \Bigr) \sum_{i = 0}^{j}  \|\partial_t^{i} \tau_{12}^{I,0} \|_{L^2_T    H^{13-2i}_{xy}  }\\
            &+\|\partial_t^{j + 1} \tau_{12}^{I,0} \|_{L^2_TH^{ 15 -2(j+1)}_{xy}}
            + \|\partial_t^{j} \tau_{12}^{I,0} \|_{L^2_TH^{ 13 -2j}_x}
            \leq C,
        \end{align*}
    for $j = 0, 1,\cdots ,5.$
   Using   \eqref{eq_compatibility_eta_tau}, we obtain that $r^{\tau,1}_{12}$ satisfies the fourth order compatibility condition stated in Proposition \ref{prop_boundary_prototype}.
    Therefore, using Proposition \ref{prop_boundary_prototype} with  $m={ 11}$, we get
    \begin{eqnarray*}
        \langle z\rangle^\ell \partial_t^j\tilde{\tau}_{12}^{B,1} \in L^{\infty}(0,T;H^{ 11 -2j}_xH^1_z)\cap L^2(0,T;H_x^{ 11 -2j}H_z^2),~~~j=0,1,\cdots,5.
    \end{eqnarray*}
    Similar to the case of $\eta^{B,1},$ one can easily obtain that $ \tau_{12}^{B,1}$ enjoys the desired regularity. The proof of the regularity for $\tau_{11}^{B,1}$ is similar to $\tau_{12}^{B,1}$, we omit here for brevity.

    Furthermore, using \eqref{eq_u_B2}, we have
   \begin{eqnarray*}
        \begin{aligned}
      \|\langle z \rangle^\ell \partial_t^j u_1^{B,2}\|_{L^\infty_TH^{ 11 -2j}_xH^2_z}^2 ={}& \frac{1}{\mu^2} \left\|\langle z \rangle^\ell \partial_t^j\int_z^{+\infty}\tau_{12}^{B,1} \mathrm{d} z\right\|_{L^\infty_TH^{ 11 -2j}_xH^2_z}^2\\
            \lesssim {}& \left\|\langle z \rangle^\ell \partial_t^j\int_z^{+\infty}\tau_{12}^{B,1} \mathrm{d} z\right\|_{L^\infty_TH^{ 11 -2j}_xL^2_z}^2 +
            \|\langle z \rangle^\ell \partial_t^j \tau_{12}^{B,1}\|_{L^\infty_TH^{ 11 -2j}_xH^1_z}^2.
        \end{aligned}
    \end{eqnarray*}
    The first term on the right-hand side can be estimated as follows:
        \begin{align*}
             &\left\|\langle z \rangle^\ell \partial_t^j\int_z^{+\infty}\tau_{12}^{B,1} \mathrm{d} z\right\|_{L^\infty_TH^{ 11 -2j}_xL^2_z}^2 \notag \\
            \lesssim {}& \left\|\sum_{j^\prime = 0 }^{ 11 -2j}\int_0^{+\infty}\frac{1}{\langle z\rangle^2}      \langle z\rangle^{2\ell+2}   \left(\int_z^{+\infty}\|\partial_t^j \partial_x^{j^{\prime}}\tau_{12}^{B,1}(x,\zeta,t)\|_{ L_x^{2}}\mathrm{d}\zeta\right)^2\mathrm{d}z \right\|_{L_T^\infty}\\
            \lesssim {}& \left\| \sum_{j^\prime = 0 }^{ 11 -2j} \int_0^{+\infty}\frac{1}{\langle z\rangle^2} \mathrm{d}z\left(\int_0^{+\infty}
            \frac{1}{\langle \zeta\rangle}   \|    \langle \zeta\rangle^{\ell+2}   \partial_t^j \partial_x^{j^{\prime}} \tau_{12}^{B,1}(x,\zeta,t)\|_{L_x^{2}}\mathrm{d}\zeta\right)^2\mathrm{d}z\right\|_{L_T^\infty}\\
            \lesssim{} & \|    \langle z \rangle^{\ell+ 2}     \partial_t^j \tau_{12}^{B,1}\|_{L^\infty_TH^{ 11 -2j}_xL^2_z}^2.
        \end{align*}
    Therefore, we have
    \begin{eqnarray*}
        \begin{aligned}
            \|\langle z \rangle^\ell \partial_t^j u_1^{B,2}\|_{L^\infty_T    H^{ 11 -2j}_x   H^2_z}^2
            \lesssim & \|   \langle z \rangle^{\ell+ 2}    \partial_t^j \tau_{12}^{B,1}\|_{L^\infty_T   H^{ 11 -2j}_x    H^1_z}^2,
        \end{aligned}
    \end{eqnarray*}
    for $j = 0, 1,\cdots,5.$
    Next, noticing that
    \begin{eqnarray*}
        u_2^{B,3}(x,z,t) = \int_{z}^{+\infty}\partial_x u_1^{B,2}(x,\zeta,t)\mathrm{d}\zeta,
    \end{eqnarray*}
    one can easily  prove \eqref{eq_u_B_3_2_1}.
    The proof is complete.
    \end{proof}

    \bigskip

   For the second order outer layer profiles $(u^{I,2},\eta^{I,2},\tau^{I,2})$, we have the following result.
    \begin{lemma}\label{local_wellposedness_outer_2}
        Under the assumptions of Theorem \ref{theorem_limit}, for any fixed   $T\in (0,T_0]$, there exists a unique
        solution $(u^{I,2},\eta^{I,2},\tau^{I,2})$ of the problem \eqref{OB_outer_second} on $[0,T]$,
        satisfying that
        \begin{eqnarray*}
          &    \partial_t^{5}u^{I,2}\in L^2_TL_{xy}^2,~\partial_t^iu^{I,2}\in L_T^{\infty}H_{xy}^{9-2i}\cap L_T^2H_{xy}^{10-2i},~i = 0,1,\cdots,4,
            \\[2mm]
            &     \eta^{I,2},\tau^{I,2} \in L_T^\infty H_{xy}^{9},~\partial_t^j\eta^{I,2},\partial_t^j\tau^{I,2}\in  L_T^\infty H_{xy}^{10-2j},~j=1,2,\cdots,5.
         \end{eqnarray*}
    \end{lemma}
    \begin{proof}
        Letting $\varphi$ be the cut-off function in \eqref{eq_phi}, and
        defining
        \begin{align*}
            \tilde{u}^{I,2}(x,y,t) =&
           u^{I,2}(x,y,t) +
           \left( \begin{matrix}
                   \frac{1}{\mu}\varphi^2(y)\int_0^{+\infty}\tau_{12}^{B,1}(x,z,t)\mathrm{d}z + \frac{1}{\mu}\varphi^\prime(y)\int_0^y
                      \varphi(y)   \mathrm{d}y\int_0^{+\infty}\tau_{12}^{B,1}(x,z,t)\mathrm{d}z \\[3mm]
                -  \frac{1}{\mu}\varphi(y)\int_0^y  \varphi(y)   \mathrm{d}y\int_0^{+\infty}\partial_x\tau_{12}^{B,1}(x,z,t)\mathrm{d}z
            \end{matrix}\right) \\
            =:&u^{I,2}(x,y,t) + \mathfrak{U}^2(x,y,t),
        \end{align*}
         then, we deduce from \eqref{OB_outer_second} that
        \begin{equation}\label{OB_outer_second_tilde}
            \begin{cases}
            \partial_t\tilde{u}^{I,2} +u^{I,0}\cdot\nabla \tilde{u}^{I,2}   +\tilde{u}^{I,2}\cdot\nabla u^{I,0} +\nabla \tilde{p}^{I,2} -\mu \Delta \tilde{u}^{I,2}
            -\mathrm{div} \tau^{I,2} = \mathfrak{f}^{I,2},\\[2mm]
            \partial_t\eta^{I,2}+ u^{I,0}\cdot\nabla \eta^{I,2} + \tilde{u}^{I,2}\cdot\nabla \eta^{I,0} = \mathfrak{g}^{I,2},\\[2mm]
            \partial_t\tau^{I,2}+u^{I,0}\cdot \nabla \tau^{I,2} +\tilde{u}^{I,2}\cdot \nabla \tau^{I,0}+ \gamma\tau^{I,2}-\mathcal{Q}(\nabla u^{I,0}, \tau^{I,2})  \\
            ~~~-\mathcal{Q}(\nabla \tilde{u}^{I,2}, \tau^{I,0}) - \mathcal{B} (\eta^{I,2},\nabla u^{I,0}) - \mathcal{B} (\eta^{I,0},\nabla \tilde{u}^{I,2})=\mathfrak{j}^{I,2},\\[2mm]
            \mathrm{div} \tilde{u}^{I,2} = 0,\\[2mm]
            (\tilde{u}^{I,2},\eta^{I,2},\tau^{I,2})(x,y,0) = 0,~~ \tilde{u}^{I,2}(x,0,t)   =0,
            \end{cases}
        \end{equation}
        where
        \begin{align*}
            \mathfrak{f}^{I,2} =& \partial_t\mathfrak{U}^2 + u^{I,0}\cdot\nabla\mathfrak{U}^2 + \mathfrak{U}^2\cdot\nabla u^{I,0} - \mu \Delta \mathfrak{U}^2,~
            \mathfrak{g}^{I,2} =  \mathfrak{U}^2\cdot\nabla\eta^{I,0} + \Delta\eta^{I,0},\\
            \mathfrak{j}^{I,2} =&  \mathfrak{U}^2\cdot\nabla\tau^{I,0} -\mathcal{Q}(\nabla\mathfrak{U}^2,\tau^{I,0})-\mathcal{B}(\eta^{I,0},\nabla\mathfrak{U}^2) +\Delta\tau^{I,0}.
        \end{align*}
        A direct calculation yields that
        \begin{eqnarray*}
             \partial_t^j\mathfrak{f}^{I,2}\in L^\infty_TH^{8-2j}_{xy},\,\,\,\,\,\, \partial_t^j(\mathfrak{g}^{I,2},\mathfrak{j}^{I,2}) \in L^\infty_TH^{9-2j}_{xy},
        \end{eqnarray*}   for  $j = 0,1,2,3,4.$
          In view of \eqref{compatibility_cond_u_I_2}, it is not difficult to check that the problem \eqref{OB_outer_second_tilde} satisfies the fourth order compatibility condition stated in Proposition \ref{prop_outer_prototype}. Thus, using Proposition \ref{prop_outer_prototype} with $m = 4$,  we finish the proof of Lemma \ref{local_wellposedness_outer_2}.
    \end{proof}
    \begin{remark}\label{remark_compatibility}
        To guarantee that the problem \eqref{OB_outer_second_tilde} satisfies the fourth order compatibility condition, we need to propose the following conditions on the initial data:
         \begin{equation}\label{compatibility_cond_u_I_2}
             \begin{cases}
                \mathbb{P}\mathrm{div}\Delta \tau^{I,0}(0)|_{y=0} =0, \\[2mm]
                \mathbb{P}\Big(-u^{I,0}\cdot\nabla(\mathbb{P}\mathrm{div}\Delta \tau^{I,0}) - \mathbb{P}\mathrm{div}\Delta \tau^{I,0}\cdot\nabla u^{I,0} + \mu \Delta\mathbb{P}\mathrm{div}\Delta  \tau^{I,0}\Big)(0)|_{y = 0}\\[2mm]
                + \mathbb{P}\mathrm{div}\Big( - u^{I,0}\cdot\nabla \Delta \tau^{I,0}   -\gamma \Delta \tau^{I,0} + \mathcal{Q}(\nabla u^{I,0},\Delta \tau^{I,0}) + \mathcal{B}(\Delta \eta^{I,0},\nabla u^{I,0}) \\[2mm]
                + \Delta\partial_t \tau^{I,0}\Big)(0)|_{y = 0} = 0,\\[2mm]
                \mathbb{P}\Bigl(-3\partial_t u^{I,0}\cdot\nabla (\mathbb{P}\mathrm{div}\Delta \tau^{I,0})
                -3 (\mathbb{P}\mathrm{div}\Delta \tau^{I,0})\cdot\nabla\partial_t u^{I,0}
                \Bigr)(0)|_{y = 0} + \mu \mathbb{P}\Delta \partial_t^3 \tilde{u}^{I,2}(0)|_{y = 0}\\[2mm]
                +\mathbb{P}\mathrm{div}\Bigl( -2\partial_{t}u^{I,0}\cdot\nabla\Delta\tau^{I,0} - u^{I,0}\cdot\nabla \partial_t^2\tau^{I,2}
                - \mathbb{P}\mathrm{div}\Delta \tau^{I,0}\cdot\nabla\tau^{I,0} - \gamma \partial_t^2\tau^{I,2}\Bigr)(0)|_{y = 0} \\[2mm]
                +\mathbb{P}\mathrm{div}\Bigl( 2\mathcal{Q}(\nabla\partial_t u^{I,0}, \Delta\tau^{I,0}) + \mathcal{Q}(\nabla u^{I,0}, \partial_t^2\tau^{I,2}) + \mathcal{Q}(\nabla \mathbb{P}\mathrm{div} \Delta\tau^{I,0},  \tau^{I,0})\Bigr)(0)|_{y = 0}\\[2mm]
                +\mathbb{P}\mathrm{div}\Bigl( 2\mathcal{B}(\Delta\eta^{I,0},\nabla\partial_t u^{I,0}) + \mathcal{B}(\partial_t^2\eta^{I,2},\nabla u^{I,0} ) + \mathcal{B}(\eta^{I,0},\nabla \mathbb{P}\mathrm{div}\partial_t\Delta\tau^{I,0})\Bigr)(0)|_{y = 0}\\[2mm]
                +{ \mathbb{P}\mathrm{div} }\Delta\partial_t^2\tau^{I,0}(0)|_{y = 0}
                { +\mathbb{P} \Bigl( -u^{I,0}\cdot\nabla \partial_t^3\tilde{u}^{I,2} -\partial_t^3 \tilde{u}^{I,2}\cdot\nabla {u}^{I,0}\Bigr) } = 0,
             \end{cases}
         \end{equation}
        where
        \begin{align*}
            \partial_t^3 \tilde{u}^{I,2}(0) ={}& \mathbb{P}\Big(-u^{I,0}\cdot\nabla(\mathbb{P}\mathrm{div}\Delta  \tau^{I,0}) - \mathbb{P}\mathrm{div}\Delta  \tau^{I,0}\cdot\nabla u^{I,0} + \mu \Delta\mathbb{P}\mathrm{div}\Delta  \tau^{I,0}\Big)(0) \\[1mm]
            &+ \mathbb{P}\mathrm{div}\Big(- u^{I,0}\cdot\nabla \Delta \tau^{I,0} - \gamma \Delta \tau^{I,0} + \mathcal{Q}(\nabla u^{I,0},\Delta \tau^{I,0}) + \mathcal{B}(\Delta \eta^{I,0},\nabla u^{I,0}) \\[1mm]
            &+ \Delta\partial_t \tau^{I,0}\Big)(0) ,\\[1mm]
            \partial_t^2\tau^{I,2}(0) ={}&   \Big(- u^{I,0}\cdot\nabla \Delta \tau^{I,0} - \gamma \Delta \tau^{I,0} + \mathcal{Q}(\nabla u^{I,0},\Delta \tau^{I,0}) + \mathcal{B}(\Delta \eta^{I,0},\nabla u^{I,0})
             + \Delta\partial_t \tau^{I,0}\Big)(0) \\
            \partial_t^2\eta^{I,2}(0) ={}& \Big(    -    u^{I,0}\cdot\nabla\Delta\eta^{I,0} + \Delta\partial_t\eta^{I,0}\Big)(0).
        \end{align*}
    \end{remark}

\medskip

    Based on the previous results, the well-posedness of the problems \eqref{eq_eta_B2}, \eqref{eq_tau_B2_22}, \eqref{eq_tau_B2_12}  and \eqref{eq_tau_B2_11} can be established as follows.
    \begin{lemma}\label{lemma_eta_tau_b_2}
        Suppose that $(u^{I,0},\eta^{I,0},\tau^{I,0})$ and $(\eta^{B,1},\tau^{B,1}_{22},\tau^{B,1}_{12},\tau^{B,1}_{11})$ are the solutions stated in Proposition \ref{local_wellposedness_0} and Lemma \ref{lemma_eta_tau_b_1}, respectively. Under the assumptions of Theorem \ref{theorem_limit}, then, for any fixed $T\in(0, T_0]$, there exist unique
        solutions $\eta^{B,2},\tau^{B,2}_{22},\tau^{B,2}_{12},\tau^{B,2}_{11}$ of the problems \eqref{eq_eta_B2}, \eqref{eq_tau_B2_22}, \eqref{eq_tau_B2_12}  and \eqref{eq_tau_B2_11} on $[0,T]$,
        satisfying
        \begin{eqnarray}\label{eq_eta_tau_B_2_1}
            \begin{aligned}
                \langle z \rangle^\ell \partial_t^j (\eta^{B,2},
                  \tau_{22}^{B,2},  \tau_{12}^{B,2},   \tau_{11}^{B,2})  \in L^\infty(0,T;H^{10-2j}_xH^1_z) \cap L^2(0,T;H^{10-2j}_xH^2_z),
            \end{aligned}
        \end{eqnarray} for all $\ell \in \mathbb{N}$ and $j = 0,1,\cdots,5$.

        Furthermore, using \eqref{eq_eta_B2}, \eqref{eq_tau_B2_22}, \eqref{eq_tau_B2_12}  and \eqref{eq_tau_B2_11},    we have
        \begin{eqnarray}\label{eq_eta_tau_B_2_t_2}
            \begin{aligned}
            &\langle z \rangle^\ell (\eta^{B,2}, \tau_{22}^{B,2},  \tau_{12}^{B,2},
              \tau_{11}^{B,2}) \in L^\infty(0,T;H^{8}_xH^3_z),~
             \langle z \rangle^\ell \partial_t(\eta^{B,2},  \partial_t\tau_{22}^{B,2},   \tau_{12}^{B,2},
             \tau_{11}^{B,2}) \in L^\infty(0,T;H^{6}_xH^3_z).
            \end{aligned}
        \end{eqnarray}
    \end{lemma}
    \begin{proof}
      The proof is similar to that of Lemma \ref{lemma_eta_tau_b_2} by applying Proposition \ref{prop_boundary_prototype} with $m=10$. We omit it for brevity.
    \end{proof}

    \begin{lemma}\label{lemma_p_g}
        Under the assumptions of Theorem \ref{theorem_limit}, we have
        \begin{align*}
            &\langle z \rangle^\ell \partial_t^j  \tilde{p}^{B,1}   \in L^\infty(0,T;H^{11-2j}_xH^1_z),
             ~\langle z \rangle^\ell \partial_t^j  \tilde{p}^{B,2}   \in L^\infty(0,T;H^{10-2j}_xH^1_z), ~j = 0,1,\cdots,5,
        \end{align*} for any fixed $T\in(0, T_0]$ and all $\ell\in\mathbb{N}.$
    \end{lemma}
    \begin{proof}
        Using the fact that $\tilde{p}^{B,1} = \tau^{B,1}_{22}$,  \eqref{eq_u_B3_p_B2},  Lemmas \ref{lemma_eta_tau_b_1}  and \ref{lemma_eta_tau_b_2}, we obtain the regularities for $\tilde{p}^{B,1}$ and $\tilde{p}^{B,2}$ via a direct calculation.
    \end{proof}
    As a consequence,  using \eqref{eq_u_B_3}, for  $u_1^{B,3}$ and $u_2^{B,4}$, we have the following results.
    \begin{lemma}\label{lemma_u_B_3}
        Under the assumptions of Theorem \ref{theorem_limit}, we have
    \begin{align}
        &\langle z \rangle^\ell \partial_t^i u_1^{B,3}   \in L^\infty(0,T;H^{10-2i}_xH^2_z),~\langle z \rangle^\ell \partial_t^j u_2^{B,4}  \in L^\infty(0,T;H^{9-2j}_xH^3_z),    \label{eq_u_B_3_1_1}
    \end{align}
    for $i =0,1,\cdots,5$, and $j = 0,1,\cdots,4.$
\end{lemma}
    Next, for the third order outer layer profiles $(u^{I,3},\eta^{I,3},\tau^{I,3})$, we have the following result.
    \begin{lemma}\label{local_wellposedness_outer_3}
        Suppose that $(u^{I,0},\eta^{I,0},\tau^{I,0})$, $u_1^{B,2}$, and $u_1^{B,3}$ are the solutions stated in Proposition \ref{local_wellposedness_0}, and Lemmas \ref{lemma_eta_tau_b_1}
        and \ref{lemma_u_B_3}, respectively. Under the assumptions of Theorem \ref{theorem_limit}, then, for any fixed $T\in(0,T_0]$, there exists a unique solution $(u^{I,3},\eta^{I,3},\tau^{I,3})$ of the problem \eqref{OB_outer_third} on $[0,T]$, satisfying
       \begin{eqnarray*}
            &\partial_t^{3}u^{I,3}\in L^2_TL_{xy}^2,~\partial_t^ju^{I,3}\in L_T^{\infty}H_{xy}^{5-2j}\cap L_T^2H_{xy}^{6-2j},~j = 0,1,2,\\[2mm]
            & \eta^{I,3},\tau^{I,3} \in L_T^\infty H_{xy}^{5},~\partial_t^j\eta^{I,3},\partial_t^j\tau^{I,3}\in  L_T^\infty H_{xy}^{6-2j},~j=1,2,3.
         \end{eqnarray*}
        \end{lemma}
    \begin{proof}
        Letting $\varphi$ be the cut-off function in \eqref{eq_phi}, and defining
      \begin{align*}
            \tilde{u}^{I,3}(x,y,t) ={}&
           u^{I,3}(x,y,t) +
           \left( \begin{matrix}
                   \varphi^2(y)\, u_{1}^{B,3}(x,0,t) +  \varphi^\prime(y)\int_0^y  \varphi(y)   \mathrm{d}y\, u_{1}^{B,3}(x,0,t)\\[2mm]
                -  \varphi(y)\int_0^y    \varphi(y)  \mathrm{d}y\,    \partial_x u_{1}^{B,3}(x,0,t)
            \end{matrix}\right) \\[2mm]
            &+\left(\begin{matrix}
            -\varphi^{\prime}(y)\int_{0}^{+\infty}    u_1^{B,2}(x,z,t)\mathrm{d}z  \\[2mm]
            \varphi(y)\int_{0}^{+\infty}\partial_x   u_1^{B,2}(x,z,t)\mathrm{d}z
            \end{matrix}\right)
            =: u^{I,3}(x,y,t) + \mathfrak{U}^3(x,y,t),
        \end{align*}
        where
    \begin{eqnarray*}
        u_{1}^{B,3}(x,0,t) = \frac{1}{\mu}\int_0^{+\infty}\tau_{12}^{B,2}(x,z,t)\mathrm{d}z -\frac{1}{\mu}\int_{0}^{+\infty}\int_{\zeta}^{+\infty}\left(\partial_x\tau_{11}^{B,1} - \partial_x\tau_{22}^{B,1}\right)(x,z,t)\mathrm{d}z\mathrm{d}\zeta,
    \end{eqnarray*}
    then, we deduce from \eqref{OB_outer_third} that
        \begin{equation}\label{OB_outer_third_tilde}
            \begin{cases}
            \partial_t\tilde{u}^{I,3} +u^{I,0}\cdot\nabla \tilde{u}^{I,3}   +\tilde{u}^{I,3}\cdot\nabla u^{I,0} +\nabla \tilde{p}^{I,3} -\mu \Delta \tilde{u}^{I,3}
            -\mathrm{div} \tau^{I,3} = \mathfrak{f}^{I,3},\\[2mm]
            \partial_t\eta^{I,3}+ u^{I,0}\cdot\nabla \eta^{I,3} + \tilde{u}^{I,3}\cdot\nabla \eta^{I,0} = \mathfrak{g}^{I,3},\\[2mm]
            \partial_t\tau^{I,3}+u^{I,0}\cdot \nabla \tau^{I,3} +\tilde{u}^{I,3}\cdot \nabla \tau^{I,0}+ \gamma\tau^{I,3}-\mathcal{Q}(\nabla u^{I,0}, \tau^{I,3})  \\
            ~~~-\mathcal{Q}(\nabla \tilde{u}^{I,3}, \tau^{I,0}) - \mathcal{B} (\eta^{I,3},\nabla u^{I,0}) - \mathcal{B} (\eta^{I,0},\nabla \tilde{u}^{I,3})=\mathfrak{j}^{I,3},\\[2mm]
            \mathrm{div} \tilde{u}^{I,3} = 0,\\[2mm]
            (\tilde{u}^{I,3},\eta^{I,3},\tau^{I,3})(x,y,0) = 0,~~\tilde{u}^{I,3}(x,0,t) =0,
            \end{cases}
        \end{equation}
        where
        \begin{align*}
            \mathfrak{f}^{I,3} ={}& \partial_t\mathfrak{U}^3 + u^{I,0}\cdot\nabla\mathfrak{U}^3 + \mathfrak{U}^3  \cdot\nabla u^{I,0} -   \mu \Delta \mathfrak{U}^3,
            ~~\mathfrak{g}^{I,3} =  \mathfrak{U}^3\cdot\nabla\eta^{I,0},\\[2mm]
            \mathfrak{j}^{I,3} ={}&  \mathfrak{U}^3\cdot\nabla\tau^{I,0} -\mathcal{Q}(\nabla\mathfrak{U}^3,\tau^{I,0})-\mathcal{B}(\eta^{I,0},\nabla  \mathfrak{U}^3  ).
        \end{align*}
        A direct calculation yields that
        \begin{eqnarray*}
             \partial_t^j\mathfrak{f}^{I,3} \in L^\infty_TH^{6-2j}_{xy},\,\,\, \partial_t^j(\mathfrak{g}^{I,3},\mathfrak{j}^{I,3}) \in L^\infty_TH^{7-2j}_{xy},
        \end{eqnarray*} for  $j=0,1,2,3$. From \eqref{eq_compatibility_eta_tau}, it is not difficult to check that the problem \eqref{OB_outer_third_tilde} satisfies the   second order compatibility condition stated in Proposition \ref{prop_outer_prototype}. Thus, using Proposition \ref{prop_outer_prototype} with $m = 2$,   we finish the proof of Lemma \ref{local_wellposedness_outer_3}.
    \end{proof}

    Based on the previous results, the well-posedness of the problems  \eqref{eq_eta_B3}, \eqref{eq_tau_B3_22}, \eqref{eq_tau_B3_12}  and \eqref{eq_tau_B3_11} can be established as follows.
    \begin{lemma}\label{lemma_eta_tau_b_3}
        Suppose that $(u^{I,0},\eta^{I,0},\tau^{I,0})$ be the solution stated in Proposition \ref{local_wellposedness_0}. Under the assumptions of Theorem \ref{theorem_limit}, then, for any fixed $T\in(0,T_0]$, there exist unique
        solutions $\eta^{B,3},\tau^{B,3}_{22},\tau^{B,3}_{12},\tau^{B,3}_{11}$ of the problems \eqref{eq_eta_B3}, \eqref{eq_tau_B3_22}, \eqref{eq_tau_B3_12}  and \eqref{eq_tau_B3_11} on $[0,T]$,
        satisfying
        \begin{eqnarray}\label{eq_eta_tau_B_3_1}
            \begin{aligned}
                \langle z \rangle^\ell \partial_t^j (\eta^{B,3},
                  \tau_{22}^{B,3},  \tau_{12}^{B,3},  \tau_{11}^{B,3})  \in L^\infty(0,T;H^{5-2j}_xH^1_z) \cap L^2(0,T;H^{5-2j}_xH^2_z), \\
            \end{aligned}
        \end{eqnarray} for all $\ell \in \mathbb{N}$ and $j = 0,1,2$.

        Furthermore, using \eqref{eq_eta_B3}, \eqref{eq_tau_B3_22}, \eqref{eq_tau_B3_12}  and \eqref{eq_tau_B3_11},    we have
        \begin{eqnarray}\label{eq_eta_tau_B_3_t_2}
            \begin{aligned}
            &\langle z \rangle^\ell (\eta^{B,3},  \tau_{22}^{B,3},  \tau_{12}^{B,3},
              \tau_{11}^{B,3}) \in L^\infty(0,T;H^{3}_xH^3_z),~
            \langle z \rangle^\ell \partial_t(\eta^{B,3}, \tau_{22}^{B,3}, \tau_{12}^{B,3},
             \tau_{11}^{B,3}) \in L^\infty(0,T;H^{1}_xH^3_z).
            \end{aligned}
        \end{eqnarray}
      As a consequence, using \eqref{eq_u_B4_2}, we have
         \begin{align}
                &\langle z \rangle^\ell \partial_t^j  u_1^{B,4}    \in L^\infty(0,T;H^{5-2j}_xH^2_z),~\langle z \rangle^\ell \partial_t^j u_2^{B,5}   \in L^\infty(0,T;H^{4-2j}_xH^2_z), \label{eq_u_B_4_1_1}
            \end{align}
            and
            \begin{eqnarray}
                \langle z \rangle^\ell \partial_t^j  \tilde{p}^{B,3} \in L^\infty(0,T;H^{5-2j}_xH^1_z),
            \end{eqnarray}
            for $j = 0,1,2$.
    \end{lemma}

    \subsection{An approximate solution}
    Based on the analysis in Section \ref{subsection_asym}, we define an approximate solution for the system (\ref{OB_main}) as follows:
         \begin{align*}
             &u^a(x,y,t) = u^I(x,y,t) + \varepsilon u^B(x,\frac{y}{\sqrt{\varepsilon}},t) +  \varepsilon^2 w (x,y,t), ~~\tilde{p}^a(x,y,t) = \tilde{p}^I(x,y,t) +  \sqrt{\varepsilon}\tilde{p}^B(x,\frac{y}{\sqrt{\varepsilon}},t),\\
             &\eta^a(x,y,t) = \eta^I(x,y,t) +   \sqrt{\varepsilon}\eta^B(x,\frac{y}{\sqrt{\varepsilon}},t),~~ \tau^a(x,y,t) = \tau^I(x,y,t) +   \sqrt{\varepsilon}\tau^B(x,\frac{y}{\sqrt{\varepsilon}},t),
         \end{align*}
     where
     \begin{align*}
             u^I(x,y,t) &=u^{I,0}(x,y,t) + \varepsilon u^{I,2}(x,y,t) + \varepsilon^{\frac{3}{2}} u^{I,3}(x,y,t),\\
             \tilde{p}^I(x,y,t) &=\tilde{p}^{I,0}(x,y,t)  + \varepsilon\tilde{p}^{I,2}(x,y,t), \\
             \eta^I(x,y,t) &=\eta^{I,0}(x,y,t)  + \varepsilon \eta^{I,2}(x,y,t),\\
             \tau^I(x,y,t) &=\tau^{I,0}(x,y,t)  +\varepsilon \tau^{I,2}(x,y,t),
         \end{align*}
         \begin{align*}
             u^B(x,\frac{y}{\sqrt{\varepsilon}},t) =&
             \left( \begin{matrix}
                  u_1^{B,2}(x,\frac{y}{\sqrt{\varepsilon}},t)\\
                  0
             \end{matrix}\right)
              +
             \sqrt{\varepsilon}
             \left(\begin{matrix}
                   u_1^{B,3}(x,\frac{y}{\sqrt{\varepsilon}},t)\\
                 u_2^{B,3}(x,\frac{y}{\sqrt{\varepsilon}},t)
             \end{matrix}\right)
             +  \varepsilon
             \left(\begin{matrix}
                u_1^{B,4}(x,\frac{y}{\sqrt{\varepsilon}},t)\\
                 u_2^{B,4}(x,\frac{y}{\sqrt{\varepsilon}},t)
             \end{matrix}\right)
             +
             \varepsilon^{\frac{3}{2}}
             \left(\begin{matrix}
                0\\
                 u_2^{B,5}(x,\frac{y}{\sqrt{\varepsilon}},t)
             \end{matrix}\right),
         \end{align*}
          \begin{align*}
           w(x,y,t) =&{}
            \left(\begin{matrix}
               \varphi^{\prime}(y)\int_0^{+\infty} u_1^{B,3}(x,z,t)\mathrm{d}z -\left[ \varphi^2(y) + \varphi^{\prime}(y)\int_0^{y}\varphi(\xi)\mathrm{d}\xi\right] u_1^{B,4}(x,0,t)\\[2mm]
                -\varphi(y)\int_0^{+\infty}\partial_x u_1^{B,3}(x,z,t)\mathrm{d}z + \varphi(y)\int_0^{y} \varphi(\xi)\mathrm{d}\xi \partial_x u_1^{B,4}(x,0,t)
            \end{matrix}\right)\\
            &+
           \sqrt{ \varepsilon }
            \left(\begin{matrix}
               \varphi^{\prime}(y) \int_0^{+\infty} u_1^{B,4}(x,z,t) \mathrm{d} z\\[2mm]
                -\varphi (y) \int_0^{+\infty} \partial_x u_1^{B,4}(x,z,t) \mathrm{d} z
            \end{matrix}\right),
        \end{align*}
         \begin{align*}
         \tilde{p}^{B}(x,\frac{y}{\sqrt{\varepsilon}},t)&=\tilde{p}^{B,1}(x,\frac{y}{\sqrt{\varepsilon}},t) + \sqrt{\varepsilon}\tilde{p}^{B,2}(x,\frac{y}{\sqrt{\varepsilon}},t)
         + \varepsilon \tilde{p}^{B,3}(x,\frac{y}{\sqrt{\varepsilon}},t),\\
         \eta^{B}(x,\frac{y}{\sqrt{\varepsilon}},t) &= \eta^{B,1}(x,\frac{y}{\sqrt{\varepsilon}},t) + \sqrt{\varepsilon}\eta^{B,2}(x,\frac{y}{\sqrt{\varepsilon}},t)
         + \varepsilon \eta^{B,3}(x,\frac{y}{\sqrt{\varepsilon}},t),\\
         \tau^{B}(x,\frac{y}{\sqrt{\varepsilon}},t) &= \tau^{B,1}(x,\frac{y}{\sqrt{\varepsilon}},t) + \sqrt{\varepsilon} \tau^{B,2}(x,\frac{y}{\sqrt{\varepsilon}},t)
         + \varepsilon \tau^{B,3}(x,\frac{y}{\sqrt{\varepsilon}},t),
         \end{align*}
     and $\varphi$ is the cut-off function defined by \eqref{eq_phi}.
     Letting $(u,\tilde{p},\eta,\tau)$ be the solution of problem  \eqref{OB_main}-\eqref{OB_main_boundary}, we define the error terms as below:
     \begin{eqnarray*}
         \begin{aligned}
                U^{\varepsilon}(x,y,t):=  \varepsilon^{-\frac{1}{2}}(u-u^a)(x,y,t), ~~ P^{\varepsilon}(x,y,t):=  \varepsilon^{-\frac{1}{2}}(\tilde{p}-\tilde{p}^a)(x,y,t),\\
                H^{\varepsilon}(x,y,t):=  \varepsilon^{-\frac{1}{2}}(\eta-\eta^a)(x,y,t), ~~ \varTheta^{\varepsilon}(x,y,t):=  \varepsilon^{-\frac{1}{2}}(\tau-\tau^a)(x,y,t).
         \end{aligned}
     \end{eqnarray*}
     \begin{remark}\label{rem_wellposedness_error}
     The wellposedness of $(u,\tilde{p},\eta,\tau)$ is equivalent to that of $(\varepsilon^\frac{1}{2}U^\varepsilon,\varepsilon^\frac{1}{2}P^\varepsilon,\varepsilon^\frac{1}{2}H^\varepsilon,\varepsilon^\frac{1}{2}\varTheta^\varepsilon)$, since $(u^a,\tilde{p}^a,\eta^a,\tau^a)$ is known due to Proposition \ref{local_wellposedness_0} and  Section  \ref{subsection_regularity} where the existence time is independent of $\epsilon$.
        See Corollary \ref{corollary_error} for the details.
     \end{remark}
     Substituting $(u^a,\tilde{p}^a,\eta^a,\tau^a)$ into \eqref{OB_main}-\eqref{OB_main_boundary}, and using the equations of inner and outer layer profiles in Section \ref{subsection_asym}, one can find that the error terms
     $(U^\varepsilon,H^\varepsilon,P^\varepsilon,\varTheta^\varepsilon)$ satisfies the following system:
     \begin{eqnarray}\label{eq_UHT_varepsilon}
         \left\{
         \begin{aligned}
             \partial_t U^\varepsilon &+ u^a\cdot\nabla U^\varepsilon   + U^\varepsilon\cdot\nabla u^a
          +  \sqrt{\varepsilon} U^\varepsilon\cdot \nabla U^\varepsilon
           + \nabla P^\varepsilon -\mu\Delta U^\varepsilon-\mathrm{div} \varTheta^\varepsilon = \mathfrak{F},\\
         \partial_t H^\varepsilon &+ u^a\cdot\nabla H^\varepsilon   + U^\varepsilon\cdot\nabla \eta^a
           +  \sqrt{\varepsilon}U^\varepsilon\cdot \nabla H^\varepsilon -\varepsilon\Delta H^\varepsilon
            = \mathfrak{G},\\
            \partial_t \varTheta^\varepsilon  &+ u^a\cdot\nabla \varTheta^\varepsilon  + U^\varepsilon\cdot\nabla  \tau^a
              +  \sqrt{\varepsilon}U^\varepsilon\cdot \nabla \varTheta^\varepsilon +\gamma \varTheta^\varepsilon
            -\varepsilon\Delta \varTheta^\varepsilon\\
             &-\big[\mathcal{Q}(\nabla u^a,\varTheta^\varepsilon)
            + \mathcal{Q}(\nabla U^\varepsilon,\tau^a)
              + \sqrt{\varepsilon}\mathcal{Q}(\nabla U^\varepsilon,\varTheta^\varepsilon)\big]\\
          &-\big[\mathcal{B}(\eta^a,\nabla U^\varepsilon)
            +\mathcal{B}(H^\varepsilon,\nabla u^a )
            + \sqrt{\varepsilon}\mathcal{B}(H^\varepsilon,\nabla U^\varepsilon )\big]
             = \mathfrak{J}, \\
             \mathrm{div} U^\varepsilon & = 0,
         \end{aligned}
         \right.
     \end{eqnarray}
     where
     \begin{eqnarray*}
        & \mathfrak{F}= \mathfrak{F}_1+\mathfrak{F}_2,~~\mathfrak{G} = \mathfrak{G}_1+\mathfrak{G}_2,~~\mathfrak{J}=\mathfrak{J}_1+\mathfrak{J}_2,
     \end{eqnarray*}
       with
         \begin{align*}
             \mathfrak{F}_1 =& -\sqrt{\varepsilon} \Big(\partial_t u^B  +  u^a\cdot\nabla u^B + u^B\cdot\nabla (u^I + \varepsilon^2 w)  -\mu\Delta u^B \Big) - \nabla \tilde{p}^B + \mathrm{div} \tau^B ,\\
             \mathfrak{F}_2 =&-\varepsilon   \Big( \partial_t (u^{I,3} +\sqrt{\varepsilon} w) + (u^{I,0} + \varepsilon u^{I,2})\cdot\nabla u^{I,3} + u^{I,3}\cdot\nabla u^I + \sqrt{\varepsilon}(u^{I}\cdot\nabla w + w\cdot\nabla u^I)   \\
             &+ \varepsilon^{\frac{5}{2}}w\cdot\nabla w- \mu\Delta (u^{I,3} +\sqrt{\varepsilon} w)+ \sqrt{\varepsilon}u^{I,2}\cdot\nabla u^{I,2}\Big),\\
             \mathfrak{G}_1= &  - (\partial_t \eta^B +   u^a\cdot\nabla \eta^B  + \sqrt{\varepsilon}u^B\cdot\nabla \eta^I   -\varepsilon \Delta \eta^B),\\
             \mathfrak{G}_2=  & - \varepsilon u^{I,3}\cdot\nabla\eta^{I } - \varepsilon^{\frac{3}{2}} (  u^{I,2}\cdot\nabla\eta^{I,2}  + w\cdot\nabla\eta^I -  \Delta \eta^{I,2}),   \\
             \mathfrak{J}_1 =& -  \Big(\partial_t \tau^B + \gamma \tau^B - \varepsilon\Delta \tau^B  +   u^a\cdot\nabla \tau^B   + \sqrt{\varepsilon}u^B\cdot\nabla \tau^I
                                  -\mathcal{Q}(\nabla u^{a}, \tau^{B})  - \sqrt{\varepsilon} \mathcal{Q}(\nabla u^{B}, \tau^{I })  \\
                             &  -   \mathcal{B}(\eta^{B},\nabla u^{a})  - \sqrt{\varepsilon}\mathcal{B}(\eta^{I},\nabla u^{B})  \Big), \\
             \mathfrak{J}_2 =& -  \varepsilon \Big(  u^{I,3}\cdot\nabla \tau^I  - \mathcal{Q}(\nabla u^{I,3}, \tau^{I }) -\mathcal{B}(\eta^{I},\nabla u^{I,3}) \Big) - \varepsilon^{\frac{3}{2}}\Big(    u^{I,2}\cdot\nabla \tau^{I,2}  + w\cdot\nabla \tau^I   \\
             &-
               \mathcal{Q}(\nabla u^{I,2}, \tau^{I,2}) -   \mathcal{Q}(\nabla w, \tau^{I})
             - \mathcal{B}(\eta^{I,2},\nabla u^{I,2}) - \mathcal{B}(\eta^{I},\nabla w) - \Delta \tau^{I,2}   \Big).
         \end{align*}

     Using the equations of the boundary layer profiles and some basic calculations, we can rewrite $\mathfrak{F}_1, \mathfrak{G}_1 $ and $\mathfrak{J}_1$ as follows:
     \begin{eqnarray*}
         \mathfrak{F}_1= \mathfrak{F}_{11} + \mathfrak{F}_{12},~~\mathfrak{G}_1= \mathfrak{G}_{11} + \mathfrak{G}_{12},
         ~~\mathfrak{J}_1 = \mathfrak{J}_{11} + \mathfrak{J}_{12},
     \end{eqnarray*}
     where
         \begin{align*}
         &\mathfrak{F}_{11} =\left(\begin{matrix}
             \mathfrak{F}_{11, 1}  \\
             \mathfrak{F}_{11, 2}
         \end{matrix}\right),
         ~~
         \mathfrak{F}_{12} =\left(\begin{matrix}
             \mathfrak{F}_{12, 1}  \\
             \mathfrak{F}_{12, 2}
         \end{matrix}\right),
         ~~
          \mathfrak{J}_{11}
         =\left(\begin{matrix}
             \mathfrak{J}_{11,11} & \mathfrak{J}_{11,12}\\
             \mathfrak{J}_{11,12} & \mathfrak{J}_{11,22}
         \end{matrix}\right),
         ~~
          \mathfrak{J}_{12}
         =\left(\begin{matrix}
             \mathfrak{J}_{12,11} & \mathfrak{J}_{12,12}\\
             \mathfrak{J}_{12,12} & \mathfrak{J}_{12,22}
         \end{matrix}\right),
     \end{align*}
     are stated in Appendix \ref{section_source_terms}.
     Furthermore, the system \eqref{eq_UHT_varepsilon} satisfies the initial boundary conditions
     \begin{equation}\label{eq_UHT_int_bo}
             (U^\varepsilon,H^\varepsilon,\varTheta^\varepsilon)(x,y,0) = 0,~(U^\varepsilon,\partial_yH^\varepsilon,\partial_y\varTheta^\varepsilon)(x,0,t) = 0.
     \end{equation}
    \section{Proof of Theorem \ref{theorem_limit}} \label{section_theorem_limit}
    In the sequel, all constants used in the estimates are independent of $\varepsilon,$ but may  depend  on  $\mu,\gamma,k$ or the bounds of the inner and outer profiles obtained in Proposition \ref{local_wellposedness_0} and Lemmas \ref{lemma_eta_tau_b_1}-\ref{lemma_p_g}.
    Moreover, the regularity results in Section \ref{subsection_regularity} will be frequently used.
    Since we are interested in the limit process, without loss of generality, we assume that $\varepsilon \in(0, 1]$ throughout the rest of the paper.
  \subsection{Estimates for the source terms }
    To begin with, we give some basic estimates involving the inner and outer profiles which will be frequently used later.
    \begin{lemma}\label{lemma_inner_outer_layer_est}
       Under the assumptions of Theorem \ref{theorem_limit}, then, for any $T\in(0, T_{0}]$, there exists a constant $C$ independent of $\varepsilon$ such that
        \begin{eqnarray}\label{eq_w_t_tt}
            \|w\|_{L^\infty_TH^4_{xy}} + \|\partial_tw\|_{L^\infty_TH^2_{xy}} + \|\partial_t^2w\|_{L^\infty_TL^2_{xy}}
            \leq C,~
        \end{eqnarray}
        \begin{eqnarray}\label{eq_u_eta_tau_I_t}
            \|(u^I,\eta^I,\tau^I)\|_{L^\infty_TH^4_{xy}} + \|\partial_t(u^I,\eta^I,\tau^I)\|_{L^\infty_TH^2_{xy}}
            \leq C,~
        \end{eqnarray}
        \begin{eqnarray}\label{eq_u_eta_tau_b_infty_2}
           \|(u^B,\eta^B,\tau^B)\|_{L^\infty_TH^3_xL^2_{z}} + \|(u^B,\eta^B,\tau^B)\|_{L^\infty_TH^2_xH^1_{z}}
           + \|(u^B,\eta^B,\tau^B)\|_{L^\infty_TH^1_xH^2_{z}}
            \leq C,~
        \end{eqnarray}
        and
        \begin{eqnarray}\label{eq_u_eta_tau_b_infty_3}
             \|\partial_t(u^B, \eta^B,\tau^B)\|_{L^\infty_TH^1_xL^2_{z}}
             + \|\partial_t u^B \|_{L^\infty_TH^2_xL^2_{z}}
             + \|\partial_t u^B \|_{L^\infty_TH^1_xH^1_{z}}
             + \|\partial_t^2u^B\|_{L^\infty_TL^2_{xz}} \leq C.
        \end{eqnarray}
\end{lemma}
\begin{proof}
    Using \eqref{eq_u_B_3_1_1}, \eqref{eq_u_B_4_1_1}, the expression of $w,$ and a direct calculation, we have
    \begin{align*}
        \|w\|_{L^\infty_TH^4_{xy}} \lesssim{}& \Bigl\|\int_0^{+\infty} \langle z\rangle^{-1} \|\langle z\rangle(u_1^{B,3},u_1^{B,4})(x,z,t)\|_{H^5_x}\mathrm{d}z + \|u_1^{B,4}(x,0,t)\|_{H^5_x} \Bigr\|_{L^\infty_T}\\
        \lesssim{}&\|\langle z\rangle(u_1^{B,3},u_1^{B,4})(x,z,t)\|_{L^\infty_TH^5_xH^1_z} \leq C.
    \end{align*}
    Similarly, we have
    \begin{align*}
       & \|\partial_t w\|_{L^\infty_TH^2_{xy}} + \|\partial_t^2 w\|_{L^\infty_TL^2_{xy}}\\
          \lesssim {} &\|\langle z\rangle \partial_t(u_1^{B,3},u_1^{B,4})(x,z,t)\|_{L^\infty_TH^3_xH^1_z}
          + \|\langle z\rangle \partial_t^2(u_1^{B,3},u_1^{B,4})(x,z,t)\|_{L^\infty_TH^1_xH^1_z} \leq C.
    \end{align*}
    Thus, we get \eqref{eq_w_t_tt}.  Furthermore,
    \eqref{eq_u_eta_tau_I_t}, \eqref{eq_u_eta_tau_b_infty_2} and \eqref{eq_u_eta_tau_b_infty_3}
    can be obtained by directly using Proposition \ref{local_wellposedness_0}, and
     Lemmas   \ref{lemma_inequality}, \ref{lemma_hou_wang}, \ref{lemma_eta_tau_b_1}, \ref{local_wellposedness_outer_2}, \ref{lemma_eta_tau_b_2}, \ref{lemma_u_B_3}, \ref{local_wellposedness_outer_3}, \ref{lemma_eta_tau_b_3}. We omit the details for brevity.
\end{proof}
\begin{remark}\label{remark_profiles_w}
    For readers' convenience,  we collect all the necessary regularities of the boundary layer profiles used in the proof of Lemma \ref{lemma_inner_outer_layer_est}, i.e.,
    \begin{align*}
        &\partial_t^i(u^{I,0},\eta^{I,0},\tau^{I,0},u^{I,2},\eta^{I,2},\tau^{I,2},u^{I,3}) \in L^\infty_TH^{4-2i}_{xy},~
         \langle z\rangle\partial_t^j(u_1^{B,3},u_1^{B,4})  \in L^\infty_TH^{5-2j}_xH^1_z,~\\
        &(u_1^{B,2},u_2^{B,k+2}, \eta^{B,k},\tau^{B,k}) \in L^\infty_TH^{2}_xH^1_z,~
         \partial_t^i(u_1^{B,2},u_2^{B,k+2}, \eta^{B,k},\tau^{B,k}) \in L^\infty_TH^{3-2i}_xL^2_z,~ \\
        &(u_1^{B,k+1},u_2^{B,k+2},\eta^{B,k},\tau^{B,k}) \in L^\infty_TH^1_xH^{2}_z,~\partial_t(u_1^{B,2},u_2^{B,k+2}) \in L^\infty_TH^1_xH^1_z\cap L^\infty_TH^2_xL^2_z,
    \end{align*}
    for $i=0,1,~j=0,1,2,~k=1,2,3$.
\end{remark}
Next, for the source terms, we have the following estimates.
   \begin{lemma}\label{lemma_source_term_1}
    Under the assumptions of Theorem \ref{theorem_limit}, then, for any $T\in(0, T_{0}]$, there exists a constant $C$ independent of $\varepsilon$ such that
    \begin{eqnarray}\label{eq_source_term}
        \|(\mathfrak{F},\mathfrak{G},\mathfrak{J})\|_{L^\infty_TH^1_{x}L^2_{y}}   +
        \|  (\mathfrak{G}_{2},\mathfrak{J}_{2})\|_{L^\infty_TH^1_{xy}}
        \leq C\varepsilon,~ ~\|(\mathfrak{G}_1,\mathfrak{J}_1)\|_{L^\infty_TH^1_{x}L^2_y}  \leq C \varepsilon^{\frac{7}{4}},
    \end{eqnarray}
    and
    \begin{eqnarray}\label{eq_source_term_t}
        \|(\partial_t\mathfrak{F},\partial_t\mathfrak{G},\partial_t\mathfrak{J})\|_{L^\infty_TL^2_{xy}}  \leq C \varepsilon.
    \end{eqnarray}
   \end{lemma}
   \begin{proof}
    To begin with, using the expressions in Appendix \ref{section_source_terms} and
     Taylor's formula, we have
        \begin{align}\label{eq_est_F_11_2}
           &\| \mathfrak{F}_{11,2}\|_{L^\infty_TL^2_{xy}}  \notag\\
           ={} & \varepsilon^{\frac{1}{2}}\left\|\frac{ u_2^{I,0}(x,y,t) -   u_2^{I,0}(x,0,t)   }{y}\varepsilon^{\frac{1}{2}}z\partial_zu_2^{B,3}
            + \varepsilon^{\frac{1}{2}}zu_1^{B,2}\frac{ \partial_x u_2^{I,0}(x,y,t) -   \partial_x u_2^{I,0}(x,0,t)   }{y}\right\|_{L^\infty_TL^2_{xy}}\notag \\
           \leq{} & \varepsilon^{\frac{5}{4}}\|\partial_y u_2^{I,0}\|_{L^\infty_TL^\infty_{xy}} \|\langle z \rangle\partial_zu_2^{B,3} \|_{L^\infty_TL^2_{xz}}
           +\varepsilon^{\frac{5}{4}}\|\langle z \rangle u_1^{B,2} \|_{L^\infty_TL^2_{xz}}\|\partial_y\partial_x u_2^{I,0}\|_{L^\infty_TL^\infty_{xy}} \\
           \leq{} &C\varepsilon^{\frac{5}{4}}\|  u_2^{I,0}\|_{L^\infty_TH^3_{xy}} \|\langle z \rangle\partial_zu_2^{B,3} \|_{L^\infty_TL^2_{xz}}
           + C\varepsilon^{\frac{5}{4}}\|\langle z \rangle u_1^{B,2} \|_{L^\infty_TL^2_{xz}}\|  u_2^{I,0}\|_{L^\infty_TH^4_{xy}}  \leq C\varepsilon^{\frac{5}{4}}.\notag
        \end{align}
        Noticing that
        \begin{eqnarray*}
            \|fg\|_{H^1_{x}L^2_y} \leq \|f\|_{W_{xy}^{1,\infty}}\|g\|_{H^1_{x}L^2_{y}},~~\text{ for all } f\in W_{xy}^{1,\infty},~g\in H^1_x L^2_y,
        \end{eqnarray*}
        similar to \eqref{eq_est_F_11_2}, we have
        \begin{eqnarray}\label{eq_est_F_11_2_x}
            \begin{aligned}
                \| \mathfrak{F}_{11,2}\|_{L^\infty_TH^1_{x}L^2_{y}}\lesssim{}& \varepsilon^{\frac{5}{4}}\|\partial_y u_2^{I,0}\|_{L^\infty_TW^{1,\infty}_{xy}} \|\langle z \rangle\partial_zu_2^{B,3} \|_{L^\infty_TH^1_{x}L^2_{z}}
            +\varepsilon^{\frac{5}{4}}\|\langle z \rangle u_1^{B,2} \|_{L^\infty_TH^1_{x}L^2_{z}}\|\partial_y\partial_x u_2^{I,0}\|_{L^\infty_TW^{1,\infty}_{xy}} \\
            \leq{}& C\varepsilon^{\frac{5}{4}}\|  u_2^{I,0}\|_{L^\infty_TH^5_{xy}} \Bigl( \|\langle z \rangle\partial_zu_2^{B,3} \|_{L^\infty_TH^1_{x}L^2_{z}}
            + \|\langle z \rangle u_1^{B,2} \|_{L^\infty_TH^1_{x}L^2_{z}} \Bigr) \leq C\varepsilon^{\frac{5}{4}}.
            \end{aligned}
        \end{eqnarray}

   Similar to \eqref{eq_est_F_11_2} and \eqref{eq_est_F_11_2_x}, we have, for $\mathfrak{F}_{11,1}$, that
       \begin{align}\label{eq_est_F_11_1}
        \|\mathfrak{F}_{11,1}\|_{L^\infty_TH^1_xL^2_{y}} \lesssim{} &  \varepsilon^{\frac{5}{4}}\Big( \|u_1^{I,0}\|_{L^\infty_TH^5_{xy}}\|\langle z\rangle u_1^{B,2}\|_{L^\infty_TH^2_xL^2_z} +  \|u_2^{I,0}\|_{L^\infty_TH^5_{xy}} \|\langle z\rangle^2 (u_1^{B,2},u_1^{B,3})\|_{L^\infty_TH^1_xH^1_z} \Big) \\
        \leq{}& C\varepsilon^{\frac{5}{4}}\notag.
       \end{align}

   From the expression of $\mathfrak{F}_{12,1}$, we have
   \begin{align}\label{eq_est_F_12_1}
     & \|\mathfrak{F}_{12,1}\|_{L^\infty_TH^1_xL^2_{y}}  \notag \\
     \lesssim{} &  \varepsilon^{\frac{5}{4}} \|\partial_t (u_1^{B,3},u_1^{B,4})\|_{L^\infty_TH^1_xL^2_{z}}
     + \varepsilon^{\frac{7}{4}}  \|(u_1^{I,2},u_1^{I,3})\|_{L^\infty_TH^3_{xy}}\|u_1^{B,2}\|_{L^\infty_TH^2_xL^2_z}
     + \varepsilon^{\frac{5}{4}}  \|(u_2^{I,2},u_2^{I,3})\|_{L^\infty_TH^3_{xy}}\notag\\
     &\times\|(u_1^{B,2},u_1^{B,3})\|_{L^\infty_TH^1_xH^1_z}
     + \varepsilon^{\frac{5}{4}}\|u^I_1\|_{L^\infty_TH^3_{xy}}\|(u_1^{B,3},u_1^{B,4})\|_{L^\infty_TH^2_xL^2_z}
     +\varepsilon^{\frac{5}{4}}\|u^I_2\|_{L^\infty_TH^3_{xy}}\| u_1^{B,4} \|_{L^\infty_TH^1_xH^1_z}  \notag\\
     &+ \varepsilon^{\frac{5}{4}}\|(u_1^{B,3},u_1^{B,4},u_2^{B,3},u_2^{B,4},u_2^{B,5})\|_{L^\infty_TH^1_xL^2_z}\|u_1^{I,0}\|_{L^\infty_TH^4_{xy}}
     + \varepsilon^{\frac{7}{4}}\|u^B\|_{L^\infty_TH^1_xL^2_z}\|(u_1^{I,2},u_1^{I,3})\|_{L^\infty_TH^4_{xy}} \notag\\
     &+ \varepsilon^{\frac{7}{4}}\|u^B_1\|_{L^\infty_TH^2_{xz}}^2 + \varepsilon^{\frac{5}{4}}\|u^B_2\|_{L^\infty_TH^2_{xz}} \|u^B_1\|_{L^\infty_TH^2_{xz}}
     +  \varepsilon^{\frac{7}{4}}\|u^B\|_{L^\infty_TH^2_{xz}}\|w\|_{L^\infty_TH^2_{xy}}+\varepsilon^{\frac{5}{4}}\|\tilde{p}^{B,3}\|_{L^\infty_TH^2_xL^2_z}  \\
     &+ \varepsilon^{\frac{5}{4}}\|(u_1^{B,3},u_1^{B,4})\|_{L^\infty_TH^3_xL^2_z} +    \varepsilon^{\frac{5}{4}}\|\tau^{B,3}_{11}\|_{L^\infty_TH^2_{x}L^2_z} \notag\\
     \lesssim{}&\varepsilon^{\frac{5}{4}}\Bigl( \|\partial_t (u_1^{B,3},u_1^{B,4})\|_{L^\infty_TH^1_xL^2_{z}}
     +\|(u_1^{B,3},u_1^{B,4})\|_{L^\infty_TH^3_xL^2_z}
     +\|(u_1^{B,2},u_1^{B,3},u_1^{B,4},u_2^{B,3},u_2^{B,4},u_2^{B,5})\|_{L^\infty_TH^2_xH^1_z}  \notag\\
     &\times\|(u^{I,0},u^{I,2},u^{I,3})\|_{L^\infty_TH^4_{xy}}
     +  \|u^B\|_{L^\infty_TH^2_{xz}}^2    + \|w\|_{L^\infty_TH^2_{xy}}^2+ \|\tilde{p}^{B,3}\|_{L^\infty_TH^2_xL^2_z}
     +\|\tau^{B,3}_{11}\|_{L^\infty_TH^2_{x}L^2_z} \Bigr)  \notag\\
     \leq{}&C  \varepsilon^{\frac{5}{4}}.\notag
    \end{align}

    Similarly, for $\mathfrak{F}_{12,2}$, we have
    \begin{align}\label{eq_est_F_12_2}
        &\|\mathfrak{F}_{12,2}\|_{L^\infty_TH^1_xL^2_{y}} \notag\\
        \lesssim{}& \varepsilon^{\frac{5}{4}}\Bigl(  \|\partial_t (u_2^{B,3},u_2^{B,4},u_2^{B,5})\|_{L^\infty_TH^1_xL^2_{z}}
        +\|(u_2^{B,3},u_2^{B,4},u_2^{B,5})\|_{L^\infty_TH^3_xL^2_z}
        +\|(u^{I,0},u^{I,2},u^{I,3})\|_{L^\infty_TH^4_{xy}} \notag\\
          &\times\|(u_1^{B,2},u_1^{B,3},u_1^{B,4},u_2^{B,3},u_2^{B,4},u_2^{B,5})\|_{L^\infty_TH^2_xH^1_z}
        +  \|u^B\|_{L^\infty_TH^2_{xz}}^2    + \|w\|_{L^\infty_TH^2_{xy}}^2 + \|u_2^{B,5}\|_{L^\infty_TH^1_xH^2_z}  \\
        &+\|\tau^{B,3}_{12}\|_{L^\infty_TH^2_{x}L^2_z} \Bigr)
        \leq{} C  \varepsilon^{\frac{5}{4}}.  \notag
    \end{align}

       Moreover,  for $\mathfrak{F}_{2}$, we have
       \begin{align}\label{eq_est_F_2}
        \|\mathfrak{F}_{2}\|_{L^\infty_TH^1_xL^2_{y}} \lesssim {}& \varepsilon\Bigl(  \|\partial_t (u^{I,3},w)\|_{L^\infty_T{H^1_{xy}}}  +  \|(u^{I,0},u^{I,2},u^{I,3},w)\|_{L^\infty_T H^2_{xy} }^2 + \|(u^{I,3},w)\|_{L^\infty_T H^3_{xy} }\Bigr) \leq C \varepsilon.
    \end{align}
    Combining \eqref{eq_est_F_11_2}-\eqref{eq_est_F_2}, we deduce that
    \begin{eqnarray*}
         \|\mathfrak{F} \|_{L^\infty_TH^1_xL^2_{y}}  \leq C \varepsilon.
    \end{eqnarray*}

    Similar to \eqref{eq_est_F_11_2_x}, for $\mathfrak{G}_{11}$, we have
    \begin{eqnarray}\label{eq_est_G_11_x}
        \begin{aligned}
            \| \mathfrak{G}_{11}\|_{L^\infty_TH^1_{x}L^2_{y}}\lesssim{}&  \varepsilon^{\frac{7}{4}}\Bigl( \|  u_1^{I,2}\|_{L^\infty_TH^4_{xy}} \|\langle z \rangle \eta^{B,1} \|_{L^\infty_TH^2_{x}L^2_{z}}
        +\|  u_1^{I,0}\|_{L^\infty_TH^4_{xy}} \|\langle z \rangle \eta^{B,3} \|_{L^\infty_TH^2_{x}L^2_{z}} \\
        &+ \|  u_2^{I,3}\|_{L^\infty_TH^4_{xy}}\|\langle z \rangle \eta^{B,1} \|_{L^\infty_TH^1_{x}H^1_{z}}
        + \|  u_2^{I,2}\|_{L^\infty_TH^4_{xy}}\|\langle z \rangle \eta^{B,2} \|_{L^\infty_TH^1_{x}H^1_{z}}\\
        &+ \|\langle z \rangle u^{B,3} \|_{L^\infty_TH^1_{x}L^2_{z}}\|  \eta^{I,0}\|_{L^\infty_TH^5_{xy}}
        +\|  u_1^{I,0}\|_{L^\infty_TH^6_{xy}} \|\langle z \rangle^3 \eta^{B,1} \|_{L^\infty_TH^2_{x}L^2_{z}}\\
        &+\|  u_2^{I,0}\|_{L^\infty_TH^6_{xy}} \|\langle z \rangle^3 \eta^{B,2} \|_{L^\infty_TH^1_{x}H^1_{z}}
        +\|  u_1^{I,0}\|_{L^\infty_TH^5_{xy}} \|\langle z \rangle^2 \eta^{B,2} \|_{L^\infty_TH^2_{x}L^2_{z}}\\
        &+\|  u_2^{I,2}\|_{L^\infty_TH^5_{xy}} \|\langle z \rangle^2 \eta^{B,1} \|_{L^\infty_TH^1_{x}H^1_{z}}
         +\|  u_2^{I,0}\|_{L^\infty_TH^5_{xy}} \|\langle z \rangle^2 \eta^{B,3} \|_{L^\infty_TH^1_{x}H^1_{z}}\\
         &+ \|\langle z \rangle^2 u_1^{B,2} \|_{L^\infty_TH^1_{x}L^2_{z}}\|  \eta^{I,0}\|_{L^\infty_TH^6_{xy}}
         +\|  u_2^{I,0}\|_{L^\infty_TH^7_{xy}} \|\langle z \rangle^4 \eta^{B,1} \|_{L^\infty_TH^1_{x}H^1_{z}}\Bigr)\\
        \leq{}& C\varepsilon^{\frac{7}{4}}.
        \end{aligned}
    \end{eqnarray}

    Similar to \eqref{eq_est_F_12_1}, for $\mathfrak{G}_{12},$ we have
    \begin{align}\label{eq_est_G_12}
         &\|\mathfrak{G}_{12}\|_{L^\infty_TH^1_xL^2_{y}}  \notag\\
       \lesssim{}& \varepsilon^{\frac{7}{4}}\Bigl(
         \|(u^{I,0},u^{I,2},u^{I,3},w)\|_{L^\infty_TH^3_{xy}}\|(\eta^{B,1},\eta^{B,2},\eta^{B,3})\|_{L^\infty_TH^2_xH^1_z} + \|(\eta^{I,0},\eta^{I,2})\|_{L^\infty_TH^4_{xy}}
        \notag\\
        &\times \|(u_1^{B,2},u_1^{B,3},u_1^{B,4},u_2^{B,3},u_2^{B,4},u_2^{B,5})\|_{L^\infty_TH^1_xL^2_z}
        +\|( \eta^{B,2},\eta^{B,3})\|_{L^\infty_TH^3_{x}L^2_z}\\
          & +\big( \|(u_1^{B,2},u_1^{B,3},u_1^{B,4},u_2^{B,3},u_2^{B,4},u_2^{B,5})\|_{L^\infty_TH^2_{xz}} + \|w\|_{L^\infty_TH^2_{xy}}\big)\|(\eta^{B,1},\eta^{B,2},\eta^{B,3})\|_{L^\infty_TH^2_{xz}}
             \Bigr) \notag\\
    \leq{}            &C  \varepsilon^{\frac{7}{4}}.\notag
       \end{align}
       Next, for  $\mathfrak{G}_{2},$ we have
       \begin{align}\label{eq_est_G_2}
        \|\mathfrak{G}_{2}\|_{L^\infty_TH^1_{xy}}
         \lesssim {}& \varepsilon\Bigl(    \|(u^{I,0},u^{I,2},u^{I,3},w)\|_{L^\infty_T H^2_{xy} }\|(\eta^{I,0},\eta^{I,2})\|_{L^\infty_T H^2_{xy} }    + \|\eta^{I,2}\|_{L^\infty_T H^3_{xy} }\Bigr) \leq C \varepsilon.
    \end{align}
    Combining \eqref{eq_est_G_11_x}, \eqref{eq_est_G_12} and \eqref{eq_est_G_2}, we deduce that
    \begin{eqnarray*}
        \|\mathfrak{G} \|_{L^\infty_TH^1_xL^2_{y}}  + \|\mathfrak{G}_2 \|_{L^\infty_TH^1_{xy}}  \leq C \varepsilon,
        ~~\|\mathfrak{G}_1\|_{L^\infty_TH^1_{x}L^2_y}  \leq C \varepsilon^{\frac{7}{4}}.
    \end{eqnarray*}

    The first term on the right-hand side of $\mathfrak{J}_{11,11}$ can be estimated in the same way as  (\ref{eq_est_G_11_x}). The rest of the estimates is similar to \eqref{eq_est_F_11_2_x}, namely,
        \begin{align}\label{eq_est_J_11_11_x}
            \| \mathfrak{J}_{11,11}\|_{L^\infty_TH^1_{x}L^2_{y}}\lesssim{}&  \varepsilon^{\frac{7}{4}}+\varepsilon^{\frac{7}{4}}\Big(  \|\langle z \rangle u_1^{B,3} \|_{L^\infty_TH^2_{x}L^2_{z}}\|  \tau_{11}^{I,0}\|_{L^\infty_TH^4_{xy}}+\|\langle z \rangle u_1^{B,4} \|_{L^\infty_TH^1_{x}H^1_{z}}\|  \tau_{12}^{I,0}\|_{L^\infty_TH^4_{xy}}
        \notag\\
        & +\|\langle z \rangle u_1^{B,2} \|_{L^\infty_TH^1_{x}H^1_{z}}\|  \tau_{12}^{I,2}\|_{L^\infty_TH^4_{xy}} +\| u_1^{I,2} \|_{L^\infty_TH^5_{xy}}\|  \langle z \rangle \tau_{11}^{B,1}\|_{L^\infty_TH^1_{x}L^2_z}
       \notag\\
        & +\| u_1^{I,0} \|_{L^\infty_TH^5_{xy}}\|  \langle z \rangle \tau_{11}^{B,3}\|_{L^\infty_TH^1_{x}L^2_z} \notag +\| u_1^{I,2} \|_{L^\infty_TH^5_{xy}}\|  \langle z \rangle \tau_{12}^{B,1}\|_{L^\infty_TH^1_{x}L^2_z}
       \notag \\
        &+\| u_1^{I,0} \|_{L^\infty_TH^5_{xy}}\|  \langle z \rangle \tau_{12}^{B,3}\|_{L^\infty_TH^1_{x}L^2_z} +\| \eta^{I,0} \|_{L^\infty_TH^4_{xy}}\|  \langle z \rangle u_1^{B,3}\|_{L^\infty_TH^2_{x}L^2_z}
        \notag\\
        &+\|  \langle z \rangle \eta^{B,1}\|_{L^\infty_TH^1_{x}L^2_z}\| u_1^{I,2} \|_{L^\infty_TH^5_{xy}}+\|  \langle z \rangle \eta^{B,3}\|_{L^\infty_TH^1_{x}L^2_z}\| u_1^{I,0} \|_{L^\infty_TH^5_{xy}}
       \notag\\
        & +\|  \tau_{12}^{I,0}\|_{L^\infty_TH^6_{xy}} \|\langle z \rangle^3 u_1^{B,2} \|_{L^\infty_TH^1_{x}H^1_{z}} + \|  u_1^{I,0}\|_{L^\infty_TH^7_{xy}}  \|\langle z \rangle^3 \tau_{11}^{B,1} \|_{L^\infty_TH^1_{x}L^2_{z}}
        \notag\\
        & + \|  u_1^{I,0}\|_{L^\infty_TH^7_{xy}}  \|\langle z \rangle^3 \tau_{12}^{B,1} \|_{L^\infty_TH^1_{x}L^2_{z}}    + \|  u_1^{I,0}\|_{L^\infty_TH^7_{xy}} \|\langle z \rangle^3 \eta^{B,1} \|_{L^\infty_TH^1_{x}L^2_{z}}
         \\&
         +\|  \tau_{11}^{I,0}\|_{L^\infty_TH^5_{xy}}\|\langle z \rangle^2 u_1^{B,2} \|_{L^\infty_TH^2_{x}L^2_{z}}+\|  \tau_{12}^{I,0}\|_{L^\infty_TH^5_{xy}}\|\langle z \rangle^2 u_1^{B,3} \|_{L^\infty_TH^1_{x}H^1_{z}}
         \notag\\
         &+\|  u_1^{I,0}\|_{L^\infty_TH^6_{xy}} \|\langle z \rangle^2 \tau_{11}^{B,2} \|_{L^\infty_TH^1_{x}L^2_{z}}+\|  u_1^{I,0}\|_{L^\infty_TH^6_{xy}} \|\langle z \rangle^2 \tau_{12}^{B,2} \|_{L^\infty_TH^1_{x}L^2_{z}}
         \notag\\
         &+\|  u_1^{I,0}\|_{L^\infty_TH^6_{xy}} \|\langle z \rangle^2 \eta^{B,2} \|_{L^\infty_TH^1_{x}L^2_{z}}  +\|  \eta^{I,0}\|_{L^\infty_TH^5_{xy}} \|\langle z \rangle^2 u_{1}^{B,2} \|_{L^\infty_TH^2_{x}L^2_{z}}\Big)\notag\\
         \lesssim{}&\varepsilon^{\frac{7}{4}}+\varepsilon^{\frac{7}{4}}\Bigl(\|u^{I,0}\|_{L^\infty_TH^7_{xy}} + \|(\eta^{I,0},\tau^{I,0})\|_{L^\infty_TH^6_{xy}}
         +\|u^{I,2}\|_{L^\infty_TH^5_{xy}} + \|\tau^{I,2}\|_{L^\infty_TH^4_{xy}}\Bigr) \notag \\
         &\times \|\langle z \rangle^3(u_1^{B,2},u_1^{B,3},u_1^{B,4},\eta^{B,1},\eta^{B,2},\eta^{B,3},\tau^{B,1},\tau^{B,2},\tau^{B,3})\|_{L^\infty_TH^2_{x}H^1_{z}} \notag\\
         \leq{}& C\varepsilon^{\frac{7}{4}}. \notag
        \end{align}

    Similar to \eqref{eq_est_J_11_11_x}, for $\mathfrak{J}_{11,12}$ and $\mathfrak{J}_{11,22}$, we have
    \begin{align}\label{eq_est_J_11_12_x}
        &\| \mathfrak{J}_{11,12}\|_{L^\infty_TH^1_{x}L^2_{y}} + \| \mathfrak{J}_{11,22}\|_{L^\infty_TH^1_{x}L^2_{y}} \notag\\
        \lesssim{}&\varepsilon^{\frac{7}{4}}+\varepsilon^{\frac{7}{4}}\Bigl(\|u^{I,0}\|_{L^\infty_TH^7_{xy}} + \|(\eta^{I,0},\tau^{I,0})\|_{L^\infty_TH^6_{xy}}
         +\|u^{I,2}\|_{L^\infty_TH^5_{xy}} + \|(\eta^{I,2},\tau^{I,2})\|_{L^\infty_TH^4_{xy}}\Bigr)  \\
         &\times \|\langle z \rangle^3(u_1^{B,2},u_1^{B,3},u_1^{B,4},u_2^{B,3},u_2^{B,4},\eta^{B,1},\eta^{B,2},\eta^{B,3},\tau^{B,1},\tau^{B,2},\tau^{B,3})\|_{L^\infty_TH^2_{x}H^1_{z}} \notag\\
         \leq{}& C\varepsilon^{\frac{7}{4}}. \notag
    \end{align}

    The first term on the right-hand side of $\mathfrak{J}_{12,11},$ $\mathfrak{J}_{12,12}$, and $\mathfrak{J}_{12,22}$ can be estimated in the same way as  (\ref{eq_est_G_12}). The rest of the estimates is similar to \eqref{eq_est_F_12_1}, namely,
    \begin{align}\label{eq_est_J_12}
        &\|\mathfrak{J}_{12}\|_{L^\infty_TH^1_xL^2_{y}}\\
        \leq{} & \|\mathfrak{J}_{12,11}\|_{L^\infty_TH^1_xL^2_{y}}+ \|\mathfrak{J}_{12,12}\|_{L^\infty_TH^1_xL^2_{y}} + \|\mathfrak{J}_{12,22}\|_{L^\infty_TH^1_xL^2_{y}}\notag\\
      \lesssim{}&\varepsilon^{\frac{7}{4}}\Bigl( \|(u_1^{B,2},u_1^{B,3},u_1^{B,4},u_2^{B,3},u_2^{B,4},u_2^{B,5},\eta^{B,1},\eta^{B,2},\eta^{B,3},\tau^{B,1},\tau^{B,2},\tau^{B,3})\|_{L^\infty_TH^2_{x}H^1_{z}} \notag\\
        &\times\|(u^{I,0},u^{I,2},u^{I,3},w,\eta^{I,0},\eta^{I,2},\tau^{I,0},\tau^{I,2})\|_{L^\infty_TH^4_{xy}}
        +\|(\tau^{B,2},\tau^{B,3})\|_{L^\infty_TH^3_{x}L^2_z}
       \\
         &\big( \| (u_1^{B,2},u_1^{B,3},u_1^{B,4},u_2^{B,3},u_2^{B,4},u_2^{B,5})\|_{L^\infty_TH^2_{xz}} + \|w\|_{L^\infty_TH^2_{xy}}\big) \notag\\
         &\times\|(\eta^{B,1},\eta^{B,2},\eta^{B,3},\tau^{B,1},\tau^{B,2},\tau^{B,3})\|_{L^\infty_TH^2_{xz}}   \Bigr) \leq{} C  \varepsilon^{\frac{7}{4}}. \notag
      \end{align}

    Finally, for $ \mathfrak{J}_{2}$, we have
    \begin{align}\label{eq_est_J_2}
        \|\mathfrak{J}_{2}\|_{L^\infty_TH^1_{xy}}
         \lesssim {}& \varepsilon\Bigl(    \|(u^{I,0},u^{I,2},u^{I,3},w)\|_{L^\infty_T H^2_{xy} }\|(\eta^{I,0},\eta^{I,2},\tau^{I,0},\tau^{I,2})\|_{L^\infty_T H^2_{xy} }    + \| \tau^{I,2}\|_{L^\infty_T H^3_{xy} }\Bigr) \leq C \varepsilon.
    \end{align}
    Combining \eqref{eq_est_J_11_11_x}, \eqref{eq_est_J_11_12_x}, \eqref{eq_est_J_12} and \eqref{eq_est_J_2}, we deduce that
    \begin{eqnarray*}
        \|\mathfrak{J} \|_{L^\infty_TH^1_xL^2_{y}}  + \|\mathfrak{J}_2 \|_{L^\infty_TH^1_{xy}}  \leq C \varepsilon,
        ~~\|\mathfrak{J}_1\|_{L^\infty_TH^1_{x}L^2_y}  \leq C \varepsilon^{\frac{7}{4}}.
    \end{eqnarray*}
    The estimates for $\partial_t\mathfrak{F},\partial_t\mathfrak{G},\partial_t\mathfrak{J}$ can be obtained in a similar way, we omit the details here for brevity.
\end{proof}
\begin{remark}\label{remark_source_terms}
    For readers' convenience, we remark that, in the proof of  Lemma \ref{lemma_source_term_1}, the regularities stated in Lemma \ref{lemma_inner_outer_layer_est} and the following additional regularities:
    \begin{align*}
        & \tilde{p}^{B,3} \in L^\infty_T H^2_xL^2_z,
         ~\partial_t \tilde{p}^{B,3} \in L^\infty_T H^1_xL^2_z,~\langle z \rangle^4(u_1^{B,i},u_2^{B,i+1},\eta^{B,j},\tau^{B,j}) \in L^\infty_TH^3_xH^1_z,\\[2mm]
         & \langle z \rangle^4 \partial_t(u_1^{B,i},u_2^{B,i+1},\eta^{B,j},\tau^{B,j}) \in L^\infty_TH^2_xH^1_z,~
         (u_1^{B,i},u_2^{B,i+1},\eta^{B,j},\tau^{B,j}) \in L^\infty_T H^2_{xz},  \\[2mm]
        & u_2^{B,5} \in L^\infty_TH^1_xH^2_z,~\partial_tu_2^{B,5} \in L^\infty_TL^2_xH^2_z,~\partial_t^2(u_1^{B,3},u_1^{B,4},u_2^{B,i+1})\in L^\infty_T L^2_{xz},
        \end{align*}
      and
         \begin{align*}
         &u^{I,0}\in L^\infty_TH^7_{xy},~\partial_t u^{I,0}\in L^\infty_TH^6_{xy},
        ~(\eta^{I,0},\tau^{I,0})\in L^\infty_TH^6_{xy},~(\partial_t \eta^{I,0},\partial_t\tau^{I,0})\in L^\infty_TH^5_{xy},\\[2mm]
        &u^{I,2}\in L^\infty_TH^5_{xy},~\partial_tu^{I,2}\in L^\infty_TH^4_{xy},
        ~(\eta^{I,2},\tau^{I,2}) \in L^\infty_TH^4_{xy},~(\partial_t\eta^{I,2},\partial_t\tau^{I,2}) \in L^\infty_TH^3_{xy},\\[2mm]
        &u^{I,3}\in L^\infty_TH^4_{xy},~\partial_tu^{I,3}\in L^\infty_TH^3_{xy},~ \partial_t^2u^{I,3}\in L^\infty_TL^2_{xy},
    \end{align*}
  are used, for $i=2,3,4,$ $j=1,2,3.$
\end{remark}
\subsection{Wellposedness of the error system}
In this section, we establish the wellposedness of $(\varepsilon^{\frac{1}{2}}U^\varepsilon,\varepsilon^{\frac{1}{2}}H^\varepsilon,\varepsilon^{\frac{1}{2}}\varTheta^\varepsilon)$
  for any fixed $\varepsilon\in (0,1].$ To begin with,
we denote
\begin{eqnarray}\label{def_UHT_tilde}
    \tilde{U^\varepsilon}:=\varepsilon^{\frac{1}{2}}U^\varepsilon,~\tilde{H^\varepsilon}:=\varepsilon^{\frac{1}{2}}H^\varepsilon,  \text{ and }\tilde{\varTheta^\varepsilon}:=\varepsilon^{\frac{1}{2}}\varTheta^\varepsilon.
\end{eqnarray}
Then, from \eqref{eq_UHT_varepsilon} and \eqref{eq_UHT_int_bo}, we obtain that  $(\tilde{U^\varepsilon},\tilde{H^\varepsilon},\tilde{\varTheta^\varepsilon})$ satisfies the following system:
\begin{eqnarray}\label{eq_UHT_varepsilon_tilde}
    \left\{
    \begin{aligned}
        \partial_t \tilde{U^\varepsilon} &+ u^a\cdot\nabla \tilde{U^\varepsilon}   + \tilde{U^\varepsilon}\cdot\nabla u^a
     +    \tilde{U^\varepsilon}\cdot \nabla \tilde{U^\varepsilon}
      + \nabla (\varepsilon^{\frac{1}{2}}P^\varepsilon) -\mu\Delta \tilde{U^\varepsilon}-\mathrm{div} \tilde{\varTheta^\varepsilon} = \varepsilon^{\frac{1}{2}}\mathfrak{F},\\
    \partial_t \tilde{H^\varepsilon} &+ u^a\cdot\nabla \tilde{H^\varepsilon}   + \tilde{U^\varepsilon}\cdot\nabla \eta^a
      +   \tilde{U^\varepsilon}\cdot \nabla \tilde{H^\varepsilon} -\varepsilon\Delta \tilde{H^\varepsilon}
       = \varepsilon^{\frac{1}{2}}\mathfrak{G},\\
       \partial_t \tilde{\varTheta^\varepsilon}  &+ u^a\cdot\nabla \tilde{\varTheta^\varepsilon}  + \tilde{U^\varepsilon}\cdot\nabla  \tau^a
         +   \tilde{U^\varepsilon}\cdot \nabla \tilde{\varTheta^\varepsilon} +\gamma \tilde{\varTheta^\varepsilon}
       -\varepsilon\Delta \tilde{\varTheta^\varepsilon}\\
        &-\big[\mathcal{Q}(\nabla u^a,\tilde{\varTheta^\varepsilon})
       + \mathcal{Q}(\nabla \tilde{U^\varepsilon},\tau^a)
         +  \mathcal{Q}(\nabla \tilde{U^\varepsilon},\tilde{\varTheta^\varepsilon})\big]\\
     &-\big[\mathcal{B}(\eta^a,\nabla \tilde{U^\varepsilon})
       +\mathcal{B}(\tilde{H^\varepsilon},\nabla u^a )
       +  \mathcal{B}(\tilde{H^\varepsilon},\nabla \tilde{U^\varepsilon} )\big]
        = \varepsilon^{\frac{1}{2}}\mathfrak{J}, \\
        {}\mathrm{div} \tilde{U^\varepsilon} & = 0,
    \end{aligned}
    \right.
\end{eqnarray}
with  the initial boundary condition:
\begin{equation}\label{eq_UHT_int_bo_tilde}
        (\tilde{U^\varepsilon},\tilde{H^\varepsilon},\tilde{\varTheta^\varepsilon})(x,y,0) = 0,~(\tilde{U^\varepsilon},\partial_y\tilde{H^\varepsilon},\partial_y\tilde{\varTheta^\varepsilon})(x,0,t) = 0.
\end{equation}

We will state a local-in-time wellposedness result for the system \eqref{eq_UHT_varepsilon_tilde} equipped with more general initial-boundary conditions:
\begin{equation}\label{eq_UHT_int_bo_tilde_general}
    (\tilde{U^\varepsilon},\tilde{H^\varepsilon},\tilde{\varTheta^\varepsilon})(x,y,0) = (\tilde{U^\varepsilon_0},\tilde{H^\varepsilon_0},\tilde{\varTheta^\varepsilon_0})(x,y)
     ,~(\tilde{U^\varepsilon},\partial_y\tilde{H^\varepsilon},\partial_y\tilde{\varTheta^\varepsilon})(x,0,t) = 0.
\end{equation}

\begin{proposition}\label{prop_local_well_posedness_tilde}
    Assume that the conditions of Theorem \ref{theorem_limit} hold, and that
    $$(\tilde{U^\varepsilon_0},\tilde{H^\varepsilon_0},\tilde{\varTheta^\varepsilon_0})\in H^2_{xy},~\partial_x(\tilde{U^\varepsilon_0},\tilde{H^\varepsilon_0},\tilde{\varTheta^\varepsilon_0})\in H^2_{xy},~    \mathrm{div}\, \tilde{U^\varepsilon_0} = 0,~\,(\tilde{U^\varepsilon_0},\partial_y\tilde{H^\varepsilon_0},\partial_y\tilde{\varTheta^\varepsilon_0})(x,0)=0.$$
    Then, for any given $\varepsilon\in (0,1]$, there exists a positive constant $T_{\varepsilon} \leq T_0$ such that the initial-boundary-value
    problem \eqref{eq_UHT_varepsilon_tilde} and \eqref{eq_UHT_int_bo_tilde_general} has a unique local solution $(\tilde{U^\varepsilon},\tilde{H^\varepsilon},\tilde{\varTheta^\varepsilon})$ satisfying
    $$(\tilde{U^\varepsilon},\tilde{H^\varepsilon},\tilde{\varTheta^\varepsilon}) \in C([0,T_\varepsilon],H_{xy}^{2})
     \text{ ~and~ }   \partial_x(u^\varepsilon,\eta^\varepsilon,\tau^\varepsilon) \in C([0,T_{\varepsilon}];H^2_{xy}).$$
\end{proposition}
\begin{remark}
    From Lemma \ref{lemma_source_term_1}, we know that $(\mathfrak{F},\mathfrak{G},\mathfrak{J}) \in L^\infty_{T_0}H^1_{x}L^2_y$ and  $\partial_t(\mathfrak{F},\mathfrak{G},\mathfrak{J}) \in L^\infty_{T_0}L^2_{xy}$. Therefore, the proof of Proposition \ref{prop_local_well_posedness_tilde} can be established via  the standard iteration arguments or the fixed-point argument. Refer, for
    instance, to \cite{Fang_Hieber_Zi_2013,Heywood_1980}.
\end{remark}
Note that the existence time $T_\varepsilon$ in Proposition \ref{prop_local_well_posedness_tilde} may depend on $\varepsilon$. However, we are interested in the limit process $\varepsilon\to 0.$ Hence, we need to show that $T_\varepsilon$ ($0<\varepsilon\leq 1$) has a positive lower bound uniform  for  $\varepsilon.$
For this purpose, we firstly derive the following uniform estimates, and the proof will be given in Sections \ref{section_est_error_lemma} and \ref{section_est_error_prop}.
\begin{lemma}\label{lemma_a_priori}
   Suppose that
    $(\tilde{U^\varepsilon},\tilde{H^\varepsilon},\tilde{\varTheta^\varepsilon})$ is the solution of the problem \eqref{eq_UHT_varepsilon_tilde}-\eqref{eq_UHT_int_bo_tilde} on $[0,T_\varepsilon].$ Under the assumptions of Theorem \ref{theorem_limit},
    then, there  exists  a positive time $T_{\star}$ (defined in \eqref{def_T_star})   which is independent of $\varepsilon$, such that
    \begin{equation*}
        \|(\tilde{U^\varepsilon},\tilde{H^\varepsilon},\tilde{\varTheta^\varepsilon})\|_{C([0,T];L_{xy}^2)}^2 \leq \varepsilon,
    \end{equation*}
    for all $T\in (0,\min\{T_\star,T_\varepsilon\}].$
\end{lemma}
\begin{remark}
     The proof of Lemma \ref{lemma_a_priori} is given in Section \ref{section_est_error_lemma}.
\end{remark}
Based on Lemma \ref{lemma_a_priori}, we have the following further higher order estimates.

\begin{proposition}\label{prop_a_priori_tilde}
   Suppose that
    $(\tilde{U^\varepsilon},\tilde{H^\varepsilon},\tilde{\varTheta^\varepsilon})$ is the solution of the problem \eqref{eq_UHT_varepsilon_tilde}-\eqref{eq_UHT_int_bo_tilde} on $[0,T_\varepsilon].$ Under the assumptions of  Theorem \ref{theorem_limit},
    then, there  exists  a positive time $T_{\star}$ (defined in \eqref{def_T_star})  which is independent of $\varepsilon$,  such that
    \begin{eqnarray}\label{eq_UHT_infty_nabla_L_2_tilde}
        \|(\tilde{U^\varepsilon},\tilde{H^\varepsilon},\tilde{\varTheta^\varepsilon})\|_{L^\infty_TL^2_{xy}}^2 +  \| \nabla \tilde{U^\varepsilon} \|_{L^2_TL^2_{xy}}^2
        + \varepsilon \| (\nabla \tilde{H^\varepsilon},\nabla\tilde{\varTheta^\varepsilon}) \|_{L^2_TL^2_{xy}}^2\leq C\varepsilon^3,
    \end{eqnarray}
    \begin{eqnarray}\label{eq_UHT_t_infty_tilde}
        \begin{aligned}
            &  \|\partial_t (\tilde{U^\varepsilon},\tilde{H^\varepsilon},\tilde{\varTheta^\varepsilon})\|_{L^\infty_TL^2_{xy}}^2
        +  \|\nabla (\tilde{U^\varepsilon},\tilde{H^\varepsilon},\tilde{\varTheta^\varepsilon})\|_{L^\infty_TL^2_{xy}}^2
    \leq  C  \varepsilon^{\frac{5}{2}} ,
        \end{aligned}
    \end{eqnarray}
    \begin{eqnarray}\label{eq_UHT_t_nabla_L_2_tilde}
        \begin{aligned}
             &\|\nabla\partial_t \tilde{U^\varepsilon}\|_{L^2_TL^2_{xy}}^2  +   \|\tilde{U^\varepsilon}\|_{L^2_TH^2_{xy}}^2
            +  \varepsilon \|\nabla\partial_t (\tilde{H^\varepsilon}, \tilde{\varTheta^\varepsilon} )\|_{L^2_TL^2_{xy}}^2
              +\varepsilon \|\nabla^2 ( \tilde{H^\varepsilon}, \tilde{\varTheta^\varepsilon} )\|_{L^2_TL^2_{xy}}^2  \leq  C  \varepsilon^{\frac{5}{2}}.
        \end{aligned}
    \end{eqnarray}
    \begin{eqnarray}\label{eq_U_H_2_tilde}
        \|\tilde{U^\varepsilon}\|_{L^\infty_TH^2_{xy}}^2 \leq C \varepsilon^{\frac{5}{2}},~~
           \|\nabla\partial_x \tilde{\varTheta^\varepsilon}\|_{L^\infty_TL^2_{xy}}^2 + \|\nabla\partial_x \tilde{H^\varepsilon}\|_{L^\infty_TL^2_{xy}}^2
        \leq C\varepsilon^{\frac{3}{2}},
     \end{eqnarray}
     \begin{eqnarray}\label{eq_HT_t_nabla_H_2_tilde}
        \varepsilon \Big(\|\nabla^2\partial_x \tilde{H^\varepsilon}\|_{L^2_TL^2_{xy}}^2
        +\|\nabla^2\partial_x \tilde{\varTheta^\varepsilon}\|_{L^2_TL^2_{xy}}^2\Big)\leq C\varepsilon^{\frac{3}{2}}.
     \end{eqnarray}
     \begin{eqnarray}\label{eq_HT_H_2_tilde}
           \|\tilde{H^\varepsilon}\|_{L^\infty_TH^2_{xy}}^2  +   \|\tilde{\varTheta^\varepsilon}\|_{L^\infty_TH^2_{xy}}^2 \leq C \varepsilon^{\frac{1}{2}},
     \end{eqnarray}
     \begin{eqnarray}\label{eq_UHT_x_t_L_2_tilde}
        \|\partial_t \partial_x (\tilde{U^\varepsilon},\tilde{H^\varepsilon},\tilde{\varTheta^\varepsilon})\|_{L^\infty_TL^2_{xy}}^2 +\|\nabla\partial_t\partial_x \tilde{U^\varepsilon}\|_{L^2_tL^2_{xy}}^2  +   \varepsilon\|\nabla\partial_t\partial_x (\tilde{H^\varepsilon},\tilde{\varTheta^\varepsilon})\|_{L^2_TL^2_{xy}}^2
          \leq C  \varepsilon^{\frac{3}{2}},
     \end{eqnarray}
     and
     \begin{eqnarray}\label{eq_UHT_x_infty_H_2_tilde}
        \| \partial_x\tilde{U^\varepsilon} \|_{L^\infty_TH^2_{xy}}^2
        \leq C\varepsilon^{\frac{3}{2}},~~
          \varepsilon^{\frac{1}{2}}  \|\partial_x( \tilde{H^\varepsilon},\tilde{\varTheta^\varepsilon})\|_{L^\infty_TH^2_{xy}}^2
        \leq C,
     \end{eqnarray}
    for all $T\in (0,\min\{T_\star,T_\varepsilon\}]$ and some positive constant $C$ independent of $\varepsilon$.

      In particular, we have
     \begin{eqnarray}\label{eq_UHT_infty_H2_copy}
        \|(\tilde{U^\varepsilon},\tilde{H^\varepsilon},\tilde{\varTheta^\varepsilon})\|_{L^\infty_TH^2_{xy}}^2
      \leq C,
     \end{eqnarray}
     for all $T\in (0,\min\{T_\star,T_\varepsilon\}]$ and some positive constant $C$ independent of $\varepsilon$.
\end{proposition}

     \begin{remark}
        The proof of Proposition \ref{prop_a_priori_tilde} is divided into several parts,
        which are  stated in terms of $(U^\varepsilon,H^\varepsilon,\varTheta^\varepsilon)$, in Section \ref{section_est_error_prop}.
        More precisely, \eqref{eq_UHT_infty_nabla_L_2_tilde}, \eqref{eq_UHT_t_infty_tilde} and \eqref{eq_UHT_t_nabla_L_2_tilde}, \eqref{eq_U_H_2_tilde} and \eqref{eq_HT_t_nabla_H_2_tilde}, \eqref{eq_HT_H_2_tilde} and \eqref{eq_UHT_x_t_L_2_tilde}, and \eqref{eq_UHT_x_infty_H_2_tilde} are proved in Lemmas \ref{lemma_UHT_infty_L2}, \ref{lemma_UHT_nabla_t}, \ref{lemma_HT_nabla_x}, \ref{lemma_UHT_x_t}, and \ref{lemma_UHT_x_H_2}, respectively.
     \end{remark}
     Next, for any given $\varepsilon\in (0,1]$, denote by $\tilde{T_\varepsilon}$
     the maximal time for the existence of the solution stated in Proposition \ref{prop_local_well_posedness_tilde} with zero initial data.
     In other words, the solution of the problem \eqref{eq_UHT_varepsilon_tilde}-\eqref{eq_UHT_int_bo_tilde} exists on $[0,T_\varepsilon]$ for $T_\varepsilon\in (0,\tilde{T}_\varepsilon),$
     but does not exist on $ [0,\tilde{T}_\varepsilon].$
      In fact, Section \ref{section_est_error_prop} indicates that Proposition \ref{prop_a_priori_tilde}  holds for all $T\in (0,\min\{T_\star,T_\varepsilon\}] = (0,T_\star]\cap(0,T_\varepsilon]$, where $(0,T_\varepsilon]$ is the existence time interval. Therefore, it still holds with $(0,T_\varepsilon]$ in Proposition \ref{prop_a_priori_tilde} replaced by $(0,\tilde{T}_\varepsilon)$.

     With Proposition \ref{prop_a_priori_tilde}, we will show that the solution of the problem \eqref{eq_UHT_varepsilon_tilde}-\eqref{eq_UHT_int_bo_tilde}
     exists on $[0,T_\star]$  for all $\varepsilon\in (0,1].$ More precisely, we have the following result.
     \begin{corollary}\label{corollary_error}
        Under the conditions of Theorem \ref{theorem_limit}, the problem \eqref{eq_UHT_varepsilon_tilde}-\eqref{eq_UHT_int_bo_tilde} has a unique   solution
        $(\tilde{U^\varepsilon},\tilde{H^\varepsilon},\tilde{\varTheta^\varepsilon}) \in C([0,T_\star],H_{xy}^{2})$ with $\partial_x (\tilde{U^\varepsilon},\tilde{H^\varepsilon},\tilde{\varTheta^\varepsilon})  \in C([0,T_{\star}];H^2_{xy})$.
        Furthermore,  \eqref{eq_UHT_infty_nabla_L_2_tilde}-\eqref{eq_UHT_infty_H2_copy} hold for all $T\in (0,T_\star].$
    \end{corollary}
    \begin{proof}
        To prove the Corollary, we need to prove that $\tilde{T}_\varepsilon > T_\star$ for any $\varepsilon\in (0,1]$.
     If $\tilde{T}_\varepsilon \leq  T_\star$ for some $\varepsilon\in (0,1]$,
        then, from \eqref{eq_UHT_infty_H2_copy},  we have
        \begin{eqnarray}\label{H2estimates}
            \|(\tilde{U^\varepsilon},\tilde{H^\varepsilon},\tilde{\varTheta^\varepsilon}) \|_{C([0,T] H^2_{xy})}^2\leq C,~~\text{ for all } T\in (0,\tilde{T_\varepsilon}),
         \end{eqnarray}
         where $C$ is independent of $\varepsilon$. In view of the local existence results specified in Proposition \ref{prop_local_well_posedness_tilde}, it is
easy to check (see for instance \cite{Sun-Zhang_book}) that (\ref{H2estimates}) implies the strong solution
can be extended beyond $\tilde{T}_\varepsilon$, which leads to a contradiction with the definition of $\tilde{T}_\varepsilon.$  The proof of Corollary \ref{corollary_error} is complete.

    \end{proof}

\subsection{Proof of Lemma \ref{lemma_a_priori}}\label{section_est_error_lemma}
To begin with, we get the following $L^2_{xy}$ estimate for $( U^\varepsilon,H^\varepsilon,\varTheta^\varepsilon).$
 And the corresponding estimate for $(\tilde{U^\varepsilon},\tilde{H^\varepsilon},\tilde{\varTheta^\varepsilon})$
   can be easily deduced via the relation \eqref{def_UHT_tilde}.
   \begin{lemma}\label{lemma_UHT_infty_L2_cont}
   Suppose $(U^\varepsilon,H^\varepsilon,\varTheta^\varepsilon)$ is the solution of  problem \eqref{eq_UHT_varepsilon}-\eqref{eq_UHT_int_bo}
     on $[0,T_\varepsilon].$ Under the assumptions of  Theorem \ref{theorem_limit}, there exists a positive time $T_{\star}$ (defined in \eqref{def_T_star}) which is independent of $\varepsilon$, such that
     if
     \begin{equation}\label{eq_smallness_asumption}
         \|(U^\varepsilon,H^\varepsilon,\varTheta^\varepsilon)\|_{C([0,T];L_{xy}^2)}^2 \leq 1,
     \end{equation}
    then, the following estimate:
     \begin{equation}\label{eq_UHT_infty_L2}
         \|(U^\varepsilon,H^\varepsilon,\varTheta^\varepsilon)\|_{C([0,T];L_{xy}^2)}^2  \leq \frac{1}{2},
     \end{equation}
    holds for all $T\in (0,\min\{T_\star,T_\varepsilon\}]$.
    \end{lemma}
    \begin{proof}
        To begin with, taking the  $L_{xy}^2$ inner product of \eqref{eq_UHT_varepsilon}$_1$ with $U^\varepsilon$, and using integration by parts, we have
        \begin{align*}
                \frac{1}{2}\frac{\mathrm{d}}{\mathrm{d}t}\|U^\varepsilon\|_{L_{xy}^2}^2 + \mu\|\nabla U^\varepsilon\|_{L^2_{xy}}^2
                = {}&  \langle u^a \otimes U^\varepsilon  ,\nabla U^\varepsilon\rangle
                -\langle \varTheta^\varepsilon,\nabla U^\varepsilon\rangle  + \langle \mathfrak{F} , U^\varepsilon\rangle \\
                \leq {}&  \|  u^a\|_{L^\infty_{xy}}\|  U^\varepsilon\|_{L^2_{xy}} \| \nabla U^\varepsilon\|_{L^2_{xy}}  +\|  \varTheta^\varepsilon\|_{L^2_{xy}}\|\nabla U^\varepsilon\|_{L^2_{xy}}\\
                &+ \|\mathfrak{F}  \|_{L^2_{xy}}\|  U^\varepsilon\|_{L^2_{xy}}  \\
                \leq {}& \frac{1}{2}\mu\|\nabla U^\varepsilon\|_{L^2_{xy}}^2 +C\left(\|  u^a\|_{L^\infty_{xy}}^2+ 1\right)\|  (U^\varepsilon,\varTheta^\varepsilon)\|_{L^2_{xy}}^2   + C\|\mathfrak{F}  \|_{L^2_{xy}}^2,
        \end{align*}
        which implies
        \begin{equation}\label{eq_U_varepsilon_L2}
            \frac{\mathrm{d}}{\mathrm{d}t}\|U^\varepsilon\|_{L_{xy}^2}^2 + \mu\|\nabla U^\varepsilon\|_{L^2_{xy}}^2 \lesssim   \|  \varTheta^\varepsilon\|_{L^2_{xy}}^2 +  \|  U^\varepsilon\|_{L^2_{xy}}^2  + \varepsilon^2,
        \end{equation}
        where  Lemma \ref{lemma_inner_outer_layer_est}  and (\ref{eq_source_term}) are used.
        Next, taking the  $L_{xy}^2$ inner product of \eqref{eq_UHT_varepsilon}$_2$ with $H^\varepsilon$, we have from integration by parts that
            \begin{align*}
               & \frac{1}{2}\frac{\mathrm{d}}{\mathrm{d}t}\|H^\varepsilon\|_{L_{xy}^2}^2 + \varepsilon\|\nabla H^\varepsilon\|_{L^2_{xy}}^2\\
                = {}& -\langle U^\varepsilon\cdot\nabla \eta^{I}, H^\varepsilon \rangle +\varepsilon^{\frac{1}{2}}\langle \eta^B  U^\varepsilon  ,\nabla H^\varepsilon\rangle
                 + \langle \mathfrak{G}, H^\varepsilon\rangle\\
                \leq {}& \|\nabla \eta^I\|_{L^\infty_{xy}}\|  U^\varepsilon\|_{L^2_{xy}}\|  H^\varepsilon\|_{L^2_{xy}}  +  \varepsilon^{\frac{1}{2}}\|  \eta^B\|_{L^\infty_{xy}}\|  U^\varepsilon\|_{L^2_{xy}}\| \nabla H^\varepsilon\|_{L^2_{xy}}
                  + \|\mathfrak{G}\|_{L^2_{xy}}\|  H^\varepsilon\|_{L^2_{xy}}\\
                \leq{} & \frac{1}{2}\varepsilon\|\nabla H^\varepsilon\|_{L^2_{xy}}^2+ C\left(\|  \eta^I\|_{H^3_{xy}}+   \varepsilon^{\frac{1}{2}}\|  \eta^B\|_{H^1_xH^1_z}^2 + 1\right)\|  (U^\varepsilon,H^\varepsilon)\|_{L^2_{xy}}^2
                + C    \|\mathfrak{G}\|_{L^2_{xy}}^2,
                  \end{align*}
        which, combined with   Lemma  \ref{lemma_inner_outer_layer_est}, and (\ref{eq_source_term}), yields
        \begin{equation}\label{eq_H_varepsilon_L2}
            \frac{\mathrm{d}}{\mathrm{d}t}\|H^\varepsilon\|_{L_{xy}^2}^2 + \varepsilon\|\nabla H^\varepsilon\|_{L^2_{xy}}^2 \lesssim   \|  H^\varepsilon\|_{L^2_{xy}}^2 +  \|  U^\varepsilon\|_{L^2_{xy}}^2  + \varepsilon^2.
        \end{equation}
        Taking the  $L_{xy}^2$ inner product of \eqref{eq_UHT_varepsilon}$_3$ with $\varTheta^\varepsilon$, we have from integration by parts that
        \begin{eqnarray}\label{eq_varT_L2}
            \begin{aligned}
                &\frac{1}{2}\frac{\mathrm{d}}{\mathrm{d}t}\|\varTheta^\varepsilon\|_{L_{xy}^2}^2 +\gamma\|\varTheta^\varepsilon\|_{L_{xy}^2}^2+ \varepsilon\|\nabla \varTheta^\varepsilon\|_{L^2_{xy}}^2 \\
                ={}&\langle \mathfrak{J},\varTheta^\varepsilon\rangle -\langle U^\varepsilon\cdot\nabla \tau^{I}, \varTheta^\varepsilon \rangle + \sqrt{\varepsilon}\langle \tau^B \otimes U^\varepsilon  ,\nabla\varTheta^\varepsilon\rangle
                + \langle \mathcal{Q} (\nabla u^I ,\varTheta^\varepsilon)  , \varTheta^\varepsilon\rangle
                + \varepsilon\langle  \mathcal{Q} ( \nabla u^B,\varTheta^\varepsilon ), \varTheta^\varepsilon\rangle
                \\
               &+ \langle  \mathcal{Q} (\nabla U^\varepsilon,\tau^I  ), \varTheta^\varepsilon\rangle
                + \sqrt{\varepsilon}\langle  \mathcal{Q} (\nabla U^\varepsilon,\tau^B ),\varTheta^\varepsilon\rangle
                +  \langle  \mathcal{B} (\eta^I,\nabla U^\varepsilon  ), \varTheta^\varepsilon\rangle
                +\sqrt{\varepsilon}\langle\mathcal{B} ( \eta^B,\nabla U^\varepsilon ), \varTheta^\varepsilon\rangle
                 \\
                & +\langle \mathcal{B} (H^\varepsilon,\nabla u^I ), \varTheta^\varepsilon\rangle
                  + \varepsilon\langle \mathcal{B} (H^\varepsilon,\nabla u^B ), \varTheta^\varepsilon\rangle
                + \varepsilon^2 \langle \mathcal{Q} ( \nabla w,\varTheta^\varepsilon) , \varTheta^\varepsilon\rangle
                + \varepsilon^2 \langle \mathcal{B} (H^\varepsilon, \nabla w) , \varTheta^\varepsilon\rangle
                \\
                &  +\sqrt{\varepsilon}\langle \mathcal{Q} ( \nabla U^\varepsilon,\varTheta^\varepsilon) , \varTheta^\varepsilon\rangle
                + \sqrt{\varepsilon} \langle\mathcal{B} ( H^\varepsilon,\nabla U^\varepsilon  ),\varTheta^\varepsilon\rangle \\
               =:{} & \mathcal{I}_1 + \mathcal{I}_2 + \mathcal{I}_3 + \mathcal{I}_4,
            \end{aligned}
        \end{eqnarray}
         where $\mathcal{I}_i$ represents the whole $i$-th line on the right hand side of the first equality sign in \eqref{eq_varT_L2}.

     $\mathcal{I}_1, \mathcal{I}_2$ and $\mathcal{I}_3$ are estimated as below:
            \begin{align*}
                &\mathcal{I}_1 + \mathcal{I}_2  + \mathcal{I}_3\\
                \leq {}&  \|\mathfrak{J}\|_{L_{xy}^2}\|\varTheta^\varepsilon\|_{L_{xy}^2} + \|U^\varepsilon\|_{L_{xy}^2} \|\nabla \tau^I \|_{L^\infty_{xy}}\|\varTheta^\varepsilon\|_{L_{xy}^2}
                  + \varepsilon^{\frac{1}{2}}\| \tau^B \|_{L^\infty_{xy}}\|U^\varepsilon\|_{L_{xy}^2}  \|\nabla\varTheta^\varepsilon\|_{L_{xy}^2}\\
                  &+  \|\nabla u^I \|_{L^\infty_{xy}} \|(H^\varepsilon,\varTheta^\varepsilon)\|_{L_{xy}^2}^2
                  +\varepsilon \|\nabla u^B \|_{L^2_{xy}}\|(H^\varepsilon,\varTheta^\varepsilon)\|_{L_{xy}^2} \|\nabla(H^\varepsilon,\varTheta^\varepsilon)\|_{L_{xy}^2} \\
                  &+ \|\nabla U^\varepsilon\|_{L_{xy}^2}  \|(\eta^I,\tau^I) \|_{L^\infty_{xy}}  \|\varTheta^\varepsilon\|_{L_{xy}^2}
                  + \varepsilon^{\frac{1}{2}}\|\nabla U^\varepsilon\|_{L_{xy}^2}\|(\eta^B,\tau^B) \|_{L^4_{xy}}\|\varTheta^\varepsilon\|_{L_{xy}^2}^{\frac{1}{2}}\|\nabla\varTheta^\varepsilon\|_{L_{xy}^2}^{\frac{1}{2}}\\
                  &+ \varepsilon^2 \|\nabla w \|_{L^\infty_{xy}} \|(H^\varepsilon,\varTheta^\varepsilon)\|_{L_{xy}^2}^2 \\
                  \leq{}&
                   C\left(1+ \| (u^I,w,\eta^I,\tau^I) \|_{H^3_{xy}}^2 + \varepsilon \|( \eta^B,\tau^B) \|_{L^4_{xz}}^4 +\|\tau^B\|_{H^1_xH^1_z}^2+ \varepsilon \|\nabla u^B \|_{L^2_{xy}}^2 \right)
                     \|(U^\varepsilon,H^\varepsilon,\varTheta^\varepsilon)\|_{L_{xy}^2}^2  \\
                  &+ C\|\mathfrak{J}\|_{L_{xy}^2}^2+ \frac{\varepsilon}{4} \|\nabla \varTheta^\varepsilon\|_{L_{xy}^2}^2 + \frac{\varepsilon}{8} \|\nabla H^\varepsilon\|_{L_{xy}^2}^2+  C \|\nabla U^\varepsilon\|_{L_{xy}^2}^2.
            \end{align*}
        By using Ladyzenskaya's inequality and the assumption $\|(H^\varepsilon,\varTheta^\varepsilon)\|_{L_{xy}^2}\leq 1$, we can estimate $\mathcal{I}_4$ as follows:
        \begin{eqnarray*}
            \begin{aligned}
                \mathcal{I}_4\leq{} &
                  \varepsilon^{\frac{1}{2}}\|\nabla U^\varepsilon\|_{L^2_{xy}}\left(\|\varTheta^\varepsilon\|_{L_{xy}^4}+\|H^\varepsilon\|_{L_{xy}^4} \right)\|\varTheta^\varepsilon\|_{L_{xy}^4} \\
                \leq{} & C \|\nabla U^\varepsilon\|_{L^2_{xy}}
                \left(\varepsilon^{\frac{1}{2}}\|H^\varepsilon\|_{L_{xy}^2} \|\nabla H^\varepsilon\|_{L_{xy}^2} +\varepsilon^{\frac{1}{2}}\|\varTheta^\varepsilon\|_{L_{xy}^2} \|\nabla\varTheta^\varepsilon\|_{L_{xy}^2} \right)\\
                \leq {}& \frac{\varepsilon}{4} \|\nabla \varTheta^\varepsilon\|_{L_{xy}^2}^2 + \frac{\varepsilon}{8} \|\nabla H^\varepsilon\|_{L_{xy}^2}^2 + C \|\nabla U^\varepsilon\|_{L_{xy}^2}^2.
            \end{aligned}
        \end{eqnarray*}
        Plugging the estimates of $\mathcal{I}_1$-$\mathcal{I}_4$ into \eqref{eq_varT_L2}, and using   Lemma  \ref{lemma_inner_outer_layer_est} and  (\ref{eq_source_term}), we get
        \begin{eqnarray}\label{eq_T_varepsilon_L2}
            \begin{aligned}
                &\frac{\mathrm{d}}{\mathrm{d}t}\|\varTheta^\varepsilon\|_{L_{xy}^2}^2 +\gamma\|\varTheta^\varepsilon\|_{L_{xy}^2}^2+ \varepsilon\|\nabla \varTheta^\varepsilon\|_{L^2_{xy}}^2 \\
                \leq{} & \frac{\varepsilon}{2} \|\nabla H^\varepsilon\|_{L_{xy}^2}^2+C \|\nabla U^\varepsilon\|_{L_{xy}^2}^2 + C \|  (U^\varepsilon,H^\varepsilon,\varTheta^\varepsilon)\|_{L^2_{xy}}^2     + C\varepsilon^2.
            \end{aligned}
            \end{eqnarray}
        Letting  $M$\eqref{eq_U_varepsilon_L2} + \eqref{eq_H_varepsilon_L2} + \eqref{eq_T_varepsilon_L2} for some $M>1$ large enough, we find that there exists a constant $\tilde{C}_1$ independent of $\varepsilon$ such that
        \begin{eqnarray*}
            \begin{aligned}
                &\frac{\mathrm{d}}{\mathrm{d}t}\left(M\|U^\varepsilon\|_{L_{xy}^2}^2 + \|H^\varepsilon\|_{L_{xy}^2}^2 + \|\varTheta^\varepsilon\|_{L_{xy}^2}^2 \right)  +\frac{\mu M}{2}\|\nabla U^\varepsilon\|_{L^2_{xy}}^2
                  + \frac{\varepsilon}{2}\|\nabla H^\varepsilon\|_{L^2_{xy}}^2+ \frac{\varepsilon}{2}\|\nabla \varTheta^\varepsilon\|_{L^2_{xy}}^2 \\
                \leq{} & \tilde{C}_1\left(M\|U^\varepsilon\|_{L_{xy}^2}^2 + \|H^\varepsilon\|_{L_{xy}^2}^2 + \|\varTheta^\varepsilon\|_{L_{xy}^2}^2 \right) +  \tilde{C}_1\varepsilon^2,
            \end{aligned}
        \end{eqnarray*}
        which, together with Gronwall's inequality, implies that
        \begin{equation}\label{eq_proof_UHT_L_2}
            \begin{split}
                &  M\|U^\varepsilon\|_{C([0,T];L_{xy}^2)}^2 + \|H^\varepsilon\|_{C([0,T];L_{xy}^2)}^2 + \|\varTheta^\varepsilon\|_{C([0,T];L_{xy}^2)}^2
                +\frac{\mu M}{2}\|\nabla U^\varepsilon\|_{L^2_TL^2_{xy}}^2   \\
                &
                  + \frac{\varepsilon}{2}\|\nabla H^\varepsilon\|_{L^2_TL^2_{xy}}^2+ \frac{\varepsilon}{2}\|\nabla \varTheta^\varepsilon\|_{L^2_TL^2_{xy}}^2
                  \leq{}  \tilde{C}_1T \exp(\tilde{C}_1T )\varepsilon^2.
            \end{split}
        \end{equation}

        Defining
        \begin{eqnarray}\label{def_T_star}
            T_\star :=\frac{1}{2e\tilde{C}_1},
        \end{eqnarray}
       we have, for any $T\in (0,\min\{T_\star,T_\varepsilon\}],$ that
        \begin{eqnarray*}
            \|U^\varepsilon\|_{C([0,T];L_{xy}^2)}^2 + \|H^\varepsilon\|_{C([0,T];L_{xy}^2)}^2 + \|\varTheta^\varepsilon\|_{C([0,T];L_{xy}^2)}^2
            \leq   \frac{1}{2}\varepsilon^2 \leq  \frac{1}{2},
        \end{eqnarray*}
        where the fact  that $\varepsilon\leq 1$ is used.
        The proof is complete.
    \end{proof}
    \begin{remark}\label{remark_lemma_UHT_infty_L2}
        For readers' convenience, we remark  that the following bounds:
        \begin{align*}
            &     \|(u^I,w,\eta^I,\tau^I)\|_{L^\infty_TH^3_{xy}} +  \|(u^B,\eta^B,\tau^B)\|_{L^\infty_TH^1_xH^1_z} \leq C,
          \end{align*} and \eqref{eq_source_term} are used in the proof of  Lemma \ref{lemma_UHT_infty_L2}.
    \end{remark}
    Using Proposition \ref{prop_local_well_posedness_tilde} and Lemma \ref{lemma_UHT_infty_L2_cont}, we can prove Lemma \ref{lemma_a_priori} as follows.
    \begin{proof}[{\bfseries Proof of Lemma \ref{lemma_a_priori}}]
        Note that $(U^\varepsilon,H^\varepsilon, \varTheta^\varepsilon)|_{t = 0} = 0$, and  the conclusion \eqref{eq_UHT_infty_L2} is stronger than the assumption \eqref{eq_smallness_asumption}.
        Then,  using  the relation \eqref{def_UHT_tilde}, one can prove that  the conclusion  in Lemma \ref{lemma_a_priori}
       holds true by the standard continuity argument for any $T\in (0, \min\{T_\star,T_\varepsilon\}]$. (see page 21 in \cite{Tao_2006}).
    \end{proof}
\subsection{Proof of Proposition \ref{prop_a_priori_tilde}}\label{section_est_error_prop}
To seek of simplicity, we divide  Proposition \ref{prop_a_priori_tilde} into several parts and restate the results in terms of $( U^\varepsilon,H^\varepsilon,\varTheta^\varepsilon)$, see Lemmas \ref{lemma_UHT_infty_L2},
\ref{lemma_UHT_nabla_t}, \ref{lemma_HT_nabla_x}, \ref{lemma_UHT_x_t} and \ref{lemma_UHT_x_H_2}.
Then we  prove them successively in this section.
     In fact, the corresponding estimates for $(\tilde{U^\varepsilon},\tilde{H^\varepsilon},\tilde{\varTheta^\varepsilon})$
        can be easily deduced via the relation \eqref{def_UHT_tilde}. To begin with, we have the following $L^2_{xy}$ estimates.
    \begin{lemma}\label{lemma_UHT_infty_L2}
       Suppose that $(U^\varepsilon,H^\varepsilon,\varTheta^\varepsilon)$ is  the solution of the problem \eqref{eq_UHT_varepsilon}-\eqref{eq_UHT_int_bo} on $[0,T]$ with $T\in (0,\min\{T_\star,T_\varepsilon\}]$. Under the assumptions of  Theorem \ref{theorem_limit}, there exists a generic positive constant C such that
         \begin{equation}\label{eq_UHT_infty_nabla_L^2}
            \|(U^\varepsilon,H^\varepsilon,\varTheta^\varepsilon)\|_{L^\infty_TL^2_{xy}}^2 +  \| \nabla U^\varepsilon \|_{L^2_TL^2_{xy}}^2
            + \varepsilon \| (\nabla H^\varepsilon,\nabla \varTheta^\varepsilon) \|_{L^2_TL^2_{xy}}^2\leq C\varepsilon^2.
        \end{equation}
    \end{lemma}
   \begin{proof}
        From Lemma \ref{lemma_a_priori}, we know that
        \begin{eqnarray*}
            \|(U^\varepsilon,H^\varepsilon,\varTheta^\varepsilon)\|_{C([0,T];L_{xy}^2)}^2 \leq 1,
        \end{eqnarray*}
        for all $T\in (0,\min\{T_\star,T_\varepsilon\}]$. Thus, we can prove Lemma \ref{lemma_UHT_infty_L2} by \eqref{eq_proof_UHT_L_2}.
   \end{proof}

    Next, we get the following estimates for $\partial_t (U^\varepsilon,H^\varepsilon, \varTheta^\varepsilon )$ and $\partial_t (U^\varepsilon,H^\varepsilon, \varTheta^\varepsilon ).$
     \begin{lemma}\label{lemma_UHT_nabla_t}
        Suppose that
        $(U^\varepsilon,H^\varepsilon,\varTheta^\varepsilon)$ is  the solution of the problem \eqref{eq_UHT_varepsilon}-\eqref{eq_UHT_int_bo} on $[0,T]$ with $T\in (0,\min\{T_\star,T_\varepsilon\}]$. Under the assumptions of  Theorem \ref{theorem_limit}, there exists a generic positive constant C such that
        \begin{eqnarray*}
            \begin{aligned}
                &  \|\partial_t (U^\varepsilon,H^\varepsilon, \varTheta^\varepsilon )\|_{L^\infty_TL^2_{xy}}^2
            +  \|\nabla (U^\varepsilon,H^\varepsilon, \varTheta^\varepsilon )\|_{L^\infty_TL^2_{xy}}^2
        \leq  C  \varepsilon^{\frac{3}{2}},
            \end{aligned}
        \end{eqnarray*}
        and
        \begin{eqnarray*}
            \begin{aligned}
                 &\|\nabla\partial_t U^\varepsilon\|_{L^2_TL^2_{xy}}^2  +   \|U^\varepsilon\|_{L^2_TH^2_{xy}}^2
                +  \varepsilon \|\nabla\partial_t (H^\varepsilon, \varTheta^\varepsilon )\|_{L^2_TL^2_{xy}}^2
                  +\varepsilon \|\nabla^2 ( H^\varepsilon, \varTheta^\varepsilon )\|_{L^2_TL^2_{xy}}^2  \leq  C  \varepsilon^{\frac{3}{2}}.
            \end{aligned}
        \end{eqnarray*}

    \end{lemma}
    \begin{proof}
        Multiplying $\partial_t$ \eqref{eq_UHT_varepsilon}$_1$ by $\partial_t U^\varepsilon$, integrating the result over $\mathbb{R}^2_{xy}$, and using integration by parts, Lemmas \ref{lemma_inequality}, \ref{lemma_inner_outer_layer_est}, \ref{lemma_source_term_1} and \ref{lemma_UHT_infty_L2}, we get
            \begin{align*}
                &\frac{1}{2}\frac{\mathrm{d}}{\mathrm{d}t} \|\partial_t U^\varepsilon\|_{L^2_{xy}}^2 +  \mu\|\nabla\partial_t U^\varepsilon\|_{L^2_{xy}}^2 \\
                ={}& \langle U^\varepsilon\otimes \partial_tu^a  ,\nabla\partial_t U^\varepsilon \rangle
                +\langle   u^a\otimes\partial_t U^\varepsilon,\nabla\partial_t U^\varepsilon \rangle
                +\langle \partial_tu^a  \otimes U^\varepsilon  , \nabla\partial_t U^\varepsilon \rangle
                -\sqrt{\varepsilon}\langle \partial_t U^\varepsilon \cdot\nabla U^\varepsilon  , \partial_t U^\varepsilon \rangle
                 \\
                &
                -\langle \partial_t \varTheta^\varepsilon, \nabla \partial_t U^\varepsilon \rangle + \langle \partial_t \mathfrak{F},\partial_t U^\varepsilon \rangle \\
                \leq{}& \|  U^\varepsilon\|_{L^2_{xy}}^{\frac{1}{2}}\|\partial_y U^\varepsilon\|_{L^2_{xy}}^{\frac{1}{2}} \|\partial_t u^a\|_{L^2_{xy}}^{\frac{1}{2}}\| \partial_t\partial_x u^a\|_{L^2_{xy}}^{\frac{1}{2}} \|\nabla\partial_t U^\varepsilon\|_{L^2_{xy}}
                 +  \|  u^a\|_{L^\infty_{xy}} \| \partial_t U^\varepsilon\|_{L^2_{xy}}\| \nabla\partial_t U^\varepsilon\|_{L^2_{xy}}\\
                & + \varepsilon^{\frac{1}{2}}\| \partial_t U^\varepsilon\|_{L^2_{xy}}\| \nabla\partial_t U^\varepsilon\|_{L^2_{xy}}\|\nabla  U^\varepsilon\|_{L^2_{xy}}
                +  \| \partial_t \varTheta^\varepsilon\|_{L^2_{xy}} \| \nabla\partial_t U^\varepsilon\|_{L^2_{xy}}
                + \|  \partial_t \mathfrak{F} \|_{L^2_{xy}}\|  \partial_t U^\varepsilon\|_{L^2_{xy}}\\
                \leq{} &  \frac{\mu}{2}\|\nabla\partial_t U^\varepsilon\|_{L^2_{xy}}^2
                +C\|\nabla U^\varepsilon\|_{L^2_{xy}}^2
                +C \|\partial_t u^a\|_{L^2_{xy}}^2\|\partial_t \partial_x u^a\|_{L^2_{xy}}^2  \|  U^\varepsilon\|_{L^2_{xy}}^2
                +C\Big(  1 + \|  u^a\|_{L^\infty_{xy}}^2\\
                &  + \|\nabla  U^\varepsilon\|_{L^2_{xy}}^2\Big) \| \partial_t U^\varepsilon\|_{L^2_{xy}}^2 +C\Big(   \| \partial_t \varTheta^\varepsilon\|_{L^2_{xy}}^2+ \|  \partial_t \mathfrak{F} \|_{L^2_{xy}}^2\Big)   \\
                \leq{} & \frac{\mu}{2}\|\nabla\partial_t U^\varepsilon\|_{L^2_{xy}}^2 + C\Big(\|\nabla  U^\varepsilon\|_{L^2_{xy}}^2\| \partial_t U^\varepsilon\|_{L^2_{xy}}^2
                 +  \| \partial_t (U^\varepsilon,\varTheta^\varepsilon)\|_{L^2_{xy}}^2  + \| \nabla U^\varepsilon\|_{L^2_{xy}}^2  \Big) + C  \varepsilon^2,
            \end{align*}
        which implies
        \begin{eqnarray}\label{eq_U_t}
            \begin{aligned}
                &\frac{\mathrm{d}}{\mathrm{d}t} \|\partial_t U^\varepsilon\|_{L^2_{xy}}^2 +  \mu\|\nabla\partial_t U^\varepsilon\|_{L^2_{xy}}^2
                \lesssim       \Big(\|\nabla  U^\varepsilon\|_{L^2_{xy}}^2 + 1\Big)\| \partial_t (U^\varepsilon,\varTheta^\varepsilon)\|_{L^2_{xy}}^2
                +\| \nabla U^\varepsilon\|_{L^2_{xy}}^2  +  \varepsilon^2.
            \end{aligned}
        \end{eqnarray}

       Differentiating \eqref{eq_UHT_varepsilon}$_2$ with respect to $t$, then taking $L^2_{xy}$ inner product with  $\partial_t H^\varepsilon$, and using integration by parts and the divergence-free conditions, we deduce
        \begin{eqnarray}\label{eq_H_t_1}
            \begin{aligned}
                &\frac{1}{2}\frac{\mathrm{d}}{\mathrm{d}t} \|\partial_t H^\varepsilon\|_{L^2_{xy}}^2 +  \varepsilon\|\nabla\partial_t H^\varepsilon\|_{L^2_{xy}}^2 \\
                =& -\langle \partial_tu^a  \cdot\nabla H^\varepsilon,\partial_t H^\varepsilon \rangle
                +\langle \sqrt{\varepsilon} H^\varepsilon\partial_t U^\varepsilon ,\nabla\partial_t H^\varepsilon \rangle
                -\langle\partial_t U^\varepsilon\cdot\nabla \eta^I,\partial_t H^\varepsilon \rangle -\langle U^\varepsilon \cdot\nabla  \partial_t\eta^I  , \partial_t H^\varepsilon \rangle\\
                &
                 +\sqrt{\varepsilon} \langle\eta^B\partial_t U^\varepsilon  ,\nabla\partial_t H^\varepsilon \rangle
                +\sqrt{\varepsilon}\langle\partial_t \eta^B  U^\varepsilon , \nabla\partial_t H^\varepsilon \rangle
                  + \langle \partial_t \mathfrak{G},\partial_t H^\varepsilon \rangle. \\
            \end{aligned}
        \end{eqnarray}
       Using Ladyzhenskaya's inequality, we have
        \begin{eqnarray}\label{eq_H_Ladyzhenskaya}
            \begin{aligned}
            \sqrt{\varepsilon}|\langle  H^\varepsilon\partial_t U^\varepsilon ,\nabla\partial_t H^\varepsilon \rangle|
            \leq{} &\sqrt{\varepsilon}\|\nabla\partial_t H^\varepsilon\|_{L^2_{xy}}\|H^\varepsilon\|_{L^4_{xy}}\|\partial_t U^\varepsilon\|_{L^4_{xy}}\\
            \lesssim {}&\sqrt{\varepsilon}\|\nabla\partial_t H^\varepsilon\|_{L^2_{xy}}\|H^\varepsilon\|_{L^2_{xy}}^{\frac{1}{2}}\|\nabla H^\varepsilon\|_{L^2_{xy}}^{\frac{1}{2}}
            \|\partial_t U^\varepsilon\|_{L^2_{xy}}^{\frac{1}{2}}\|\nabla \partial_t U^\varepsilon\|_{L^2_{xy}}^{\frac{1}{2}}\\
            \leq{}& \frac{\varepsilon}{ 8}\|\nabla\partial_t H^\varepsilon\|_{L^2_{xy}}^2  + \frac{\mu}{8}\|\nabla\partial_t U^\varepsilon\|_{L^2_{xy}}^2
            +C\|  H^\varepsilon\|_{L^2_{xy}}^2 \|\nabla H^\varepsilon\|_{L^2_{xy}}^2 \|  \partial_t U^\varepsilon\|_{L^2_{xy}}^2.
            \end{aligned}
        \end{eqnarray}
    Then,  by virtue of Lemmas \ref{lemma_inequality} and \ref{lemma_inner_outer_layer_est},  the first line  on  the right-hand side of \eqref{eq_H_t_1} can be  estimated  as follows:
             \begin{align}\label{ptua}
                \nonumber&-\langle \partial_tu^a \cdot\nabla H^\varepsilon,\partial_t H^\varepsilon \rangle
                +\sqrt{\varepsilon}\langle  H^\varepsilon\partial_t U^\varepsilon ,\nabla\partial_t H^\varepsilon \rangle
                -\langle\partial_t U^\varepsilon\cdot\nabla \eta^I,\partial_t H^\varepsilon \rangle
                -\langle U^\varepsilon \cdot\nabla  \partial_t\eta^I  , \partial_t H^\varepsilon \rangle \\
                \leq & \nonumber
                \varepsilon\|\partial_t u^B\|_{L^2_{xy}}^{\frac{1}{2}}\|\partial_t\partial_x u^B\|_{L^2_{xy}}^{\frac{1}{2}}\|\nabla H^\varepsilon\|_{L^2_{xy}}\|\partial_t H^\varepsilon\|_{L^2_{xy}}^{ \frac{1}{2}}\|\partial_y\partial_t H^\varepsilon\|_{L^2_{xy}}^{\frac{1}{2}}
                +  \|\partial_t (u^I,w)\|_{L^\infty_{xy}} \|\nabla H^\varepsilon\|_{L^2_{xy}}\|\partial_t H^\varepsilon\|_{L^2_{xy}} \\ \nonumber
                & +\|\nabla \eta^I\|_{L^\infty_{xy}}  \|\partial_t U^\varepsilon\|_{L^2_{xy}}  \|\partial_t H^\varepsilon\|_{L^2_{xy}}
                + \|\nabla \partial_t\eta^I\|_{L^4_{xy}}  \|  U^\varepsilon\|_{L^2_{xy}}^{\frac{1}{2}}\|  \nabla U^\varepsilon\|_{L^2_{xy}}^{\frac{1}{2}} \|\partial_t H^\varepsilon\|_{L^2_{xy}}
                  \\
                 & \nonumber+\sqrt{\varepsilon} |\langle  H^\varepsilon\partial_t U^\varepsilon ,\nabla\partial_t H^\varepsilon \rangle| \\
                \leq{} &\nonumber
                C\|  H^\varepsilon\|_{L^2_{xy}}^2 \|\nabla H^\varepsilon\|_{L^2_{xy}}^2 \|  \partial_t U^\varepsilon\|_{L^2_{xy}}^2
                +  C\Big(\|\partial_t(U^\varepsilon, H^\varepsilon)\|_{L^2_{xy}}^2
                 +   \|\nabla (U^\varepsilon,H^\varepsilon)\|_{L^2_{xy}}^2  + \varepsilon^2\Big)
                          \\
                &+\frac{\varepsilon}{4}\|\nabla\partial_t H^\varepsilon\|_{L^2_{xy}}^2   + \frac{\mu}{ 8}\|\nabla\partial_t U^\varepsilon\|_{L^2_{xy}}^2.
            \end{align}
          Similarly, the second line  on  the right-hand side of \eqref{eq_H_t_1} can be estimated as follows:
        \begin{eqnarray}\label{etaBptU}
            \begin{aligned}
                & \sqrt{\varepsilon} \langle\eta^B\partial_t U^\varepsilon  ,\nabla\partial_t H^\varepsilon \rangle
                +\sqrt{\varepsilon}\langle\partial_t \eta^B  U^\varepsilon , \nabla\partial_t H^\varepsilon \rangle
                  + \langle \partial_t \mathfrak{G},\partial_t H^\varepsilon \rangle \\
                \leq{}&\sqrt{\varepsilon}\|\nabla\partial_t H^\varepsilon\|_{L^2_{xy}} \Bigl( \|  \eta^B\|_{L^\infty_{xy}}\|\partial_t U^\varepsilon\|_{L^2_{xy}}
                +   \|  \partial_t \eta^B\|_{L^2_{xy}}^{\frac{1}{2}}\|  \partial_t\partial_x \eta^B\|_{L^2_{xy}}^{\frac{1}{2}}\|U^\varepsilon\|_{L^2_{xy}}^{\frac{1}{2}}\|\partial_yU^\varepsilon\|_{L^2_{xy}}^{\frac{1}{2}}\Bigr)\\
                 &
                  +\|\partial_t \mathfrak{G}\|_{L^2_{xy}}\|\partial_t H^\varepsilon\|_{L^2_{xy}}\\
                 \leq{}& \frac{\varepsilon}{4}\|\nabla\partial_t H^\varepsilon\|_{L^2_{xy}}^2 +
                 C\left(\|\partial_t (U^\varepsilon,H^\varepsilon)\|_{L^2_{xy}}^2
                +   \|  \nabla U^\varepsilon\|_{L^2_{xy}}^2 +   \varepsilon^2\right).
            \end{aligned}
        \end{eqnarray}
        Then, putting the above estimates (\ref{eq_H_Ladyzhenskaya}), (\ref{ptua}), and (\ref{etaBptU}) to (\ref{eq_H_t_1}), we get       \begin{eqnarray}\label{eq_H_t}
            \begin{aligned}
                &\frac{\mathrm{d}}{\mathrm{d}t} \|\partial_t H^\varepsilon\|_{L^2_{xy}}^2 +  \varepsilon\|\nabla\partial_t H^\varepsilon\|_{L^2_{xy}}^2 \\
                \leq{} &  \frac{\mu}{4}\|\nabla\partial_t U^\varepsilon\|_{L^2_{xy}}^2
                   +C\Big(   \|  \partial_t (U^\varepsilon,H^\varepsilon)\|_{L^2_{xy}}^2 +   \|\nabla(U^\varepsilon, H^\varepsilon)\|_{L^2_{xy}}^2
                 +  \varepsilon^2 \|\nabla H^\varepsilon\|_{L^2_{xy}}^2 \|  \partial_t U^\varepsilon\|_{L^2_{xy}}^2
                    + \varepsilon^2\Big).
            \end{aligned}
        \end{eqnarray}

       Differentiating \eqref{eq_UHT_varepsilon}$_3$ with respect to $t$, then taking $L^2_{xy}$ inner product with  $\partial_t \varTheta^\varepsilon$, and using integration by parts and the divergence-free conditions, we obtain
        \begin{eqnarray}\label{eq_T_t_1}
            \begin{aligned}
                &\frac{1}{2}\frac{\mathrm{d}}{\mathrm{d}t} \|\partial_t \varTheta^\varepsilon\|_{L^2_{xy}}^2 + \gamma  \|\partial_t \varTheta^\varepsilon\|_{L^2_{xy}}^2 + \varepsilon\|\nabla\partial_t \varTheta^\varepsilon\|_{L^2_{xy}}^2 \\
                ={}& \langle \partial_t \mathfrak{J},\partial_t \varTheta^\varepsilon \rangle-\langle \partial_tu^a  \cdot\nabla \varTheta^\varepsilon,\partial_t \varTheta^\varepsilon \rangle
                -\langle\partial_t (U^\varepsilon\cdot\nabla  \tau^I )   ,\partial_t \varTheta^\varepsilon \rangle
                +\sqrt{\varepsilon} \langle\partial_t ( \tau^B \otimes U^\varepsilon)   ,\nabla \partial_t \varTheta^\varepsilon \rangle
                  \\
                &
                  + \langle \partial_t\mathcal{Q}(\nabla   u^a, \varTheta^\varepsilon ) +\partial_t\mathcal{Q}(\nabla U^\varepsilon, \tau^a)  ,\partial_t \varTheta^\varepsilon\rangle
                  + \langle \partial_t\mathcal{B}(\eta^a,\nabla U^\varepsilon)+\partial_t\mathcal{B}(H^\varepsilon,\nabla u^a)   ,\partial_t \varTheta^\varepsilon\rangle
                 \\
                  &
                   +\sqrt{\varepsilon}\langle  \varTheta^\varepsilon\partial_t U^\varepsilon ,\nabla\partial_t \varTheta^\varepsilon \rangle
                  + \sqrt{\varepsilon} \langle \partial_t\mathcal{Q}(\nabla U^\varepsilon, \varTheta^\varepsilon),\partial_t \varTheta^\varepsilon\rangle
                   +\sqrt{\varepsilon}\langle \partial_t\mathcal{B}(H^\varepsilon,\nabla U^\varepsilon),\partial_t \varTheta^\varepsilon\rangle
                  \\
                  =:&~  \mathcal{K}_1 + \mathcal{K}_2 + \mathcal{K}_3,
             \end{aligned}
        \end{eqnarray}
     where $\mathcal{K}_i$ represents the whole $i$-th line on the right hand side of the first equality sign in \eqref{eq_T_t_1}.

  By virtue of Lemmas \ref{lemma_inequality}, \ref{lemma_inner_outer_layer_est}, \ref{lemma_source_term_1} and \ref{lemma_UHT_infty_L2},  $\mathcal{K}_1$ can be estimated as follows:
            \begin{align*}
                \mathcal{K}_1 \leq{}& \|\partial_t (u^{I},w)\|_{L^\infty_{xy}} \|\nabla \varTheta^\varepsilon\|_{L^2_{xy}} \|\partial_t\varTheta^\varepsilon\|_{L^2_{xy}}
                +\varepsilon\|\partial_t u^B\|_{L^2_{xy}}^{\frac{1}{2}}\|\partial_t \partial_x u^B\|_{L^2_{xy}}^{\frac{1}{2}}
                 \|\nabla \varTheta^\varepsilon\|_{L^2_{xy}} \|\partial_t\varTheta^\varepsilon\|_{L^2_{xy}}^{\frac{1}{2}}\|\partial_y\partial_t\varTheta^\varepsilon\|_{L^2_{xy}}^{\frac{1}{2}} \\
                &
                +\Big(\|\partial_t \mathfrak{J}\|_{L^2_{xy}}^2 + \|\partial_t U^\varepsilon\|_{L^2_{xy}}  \|\nabla\tau^I\|_{L^\infty_{xy}} \Big)\|\partial_t \varTheta^\varepsilon\|_{L^2_{xy}}
                + \|U^\varepsilon\|_{L^2_{xy}}^{\frac{1}{2}} \|\nabla U^\varepsilon\|_{L^2_{xy}}^{\frac{1}{2}} \|\nabla\partial_t \tau^I\|_{L^4_{xy}} \|\partial_t \varTheta^\varepsilon\|_{L^2_{xy}}\\
                &
                 + \sqrt{\varepsilon}\|\nabla \partial_t \varTheta^\varepsilon\|_{L^2_{xy}}  \Bigl( \| \tau^B\|_{L^\infty_{xy}}\|\partial_t U^\varepsilon\|_{L^2_{xy}}
                    + \| \partial_t\tau^B\|_{L^2_{xy}}^{\frac{1}{2}}\| \partial_t\partial_x\tau^B\|_{L^2_{xy}}^{\frac{1}{2}}\| U^\varepsilon\|_{L^2_{xy}}^{\frac{1}{2}}\| \partial_y U^\varepsilon\|_{L^2_{xy}}^{\frac{1}{2}} \Bigr)  \\
                 \leq{} & \frac{\varepsilon}{8}\|\nabla \partial_t \varTheta^\varepsilon\|_{L^2_{xy}}^2 +C\Bigl( \|\nabla (U^\varepsilon,\varTheta^\varepsilon)\|_{L^2_{xy}}^2 + \|\partial_t(U^\varepsilon,\varTheta^\varepsilon)\|_{L^2_{xy}}^2  + \varepsilon^2\Bigr).
            \end{align*}
        Similarly, we can estimate $\mathcal{K}_2 $ as follows:
            \begin{align*}
                 \mathcal{K}_2
                \leq{}  &    \|\nabla\partial_t (u^I,w)\|_{L^4_{xy}}  (\|  \varTheta^\varepsilon\|_{L^2_{xy}}^{\frac{1}{2}}\|  \nabla\varTheta^\varepsilon\|_{L^2_{xy}}^{\frac{1}{2}}
                +\|  H^\varepsilon\|_{L^2_{xy}}^{\frac{1}{2}}\|  \nabla H^\varepsilon\|_{L^2_{xy}}^{\frac{1}{2}}  \Big)  \|\partial_t\varTheta^\varepsilon\|_{L^2_{xy}}\\
                &+ \varepsilon\|\nabla\partial_t u^B\|_{L^2_{xy}}  \Big( \|  \varTheta^\varepsilon\|_{L^2_{xy}}^{\frac{1}{2}}\|  \nabla\varTheta^\varepsilon\|_{L^2_{xy}}^{\frac{1}{2}}
                +\|  H^\varepsilon\|_{L^2_{xy}}^{\frac{1}{2}}\|  \nabla H^\varepsilon\|_{L^2_{xy}}^{\frac{1}{2}}\Big)
                \|  \partial_t\varTheta^\varepsilon\|_{L^2_{xy}}^{\frac{1}{2}}\|  \nabla\partial_t\varTheta^\varepsilon\|_{L^2_{xy}}^{\frac{1}{2}} \\
                &+  \|\nabla   (u^I,w)\|_{L^\infty_{xy}} \| \partial_t (H^\varepsilon,\varTheta^\varepsilon)\|_{L^2_{xy}}^2
                  +  \varepsilon\|\nabla  u^B\|_{L^2_{xy}}
                \Big(\|  \partial_t\varTheta^\varepsilon\|_{L^2_{xy}}\|  \nabla\partial_t\varTheta^\varepsilon\|_{L^2_{xy}}\\
                &+\|  \partial_t H^\varepsilon\|_{L^2_{xy}} \|  \nabla\partial_tH^\varepsilon\|_{L^2_{xy}} \Big)
                +\|\nabla\partial_t U^\varepsilon\|_{L^2_{xy}} \|(\eta^a,\tau^a)\|_{L^\infty_{xy}} \|\partial_t\varTheta^\varepsilon\|_{L^2_{xy}}\\
                &
                  +\|\nabla  U^\varepsilon\|_{L^2_{xy}}  \|\partial_t(\eta^I,\tau^I)\|_{L^\infty_{xy}}
                  \|\partial_t\varTheta^\varepsilon\|_{L^2_{xy}}
                    +\sqrt{\varepsilon}  \|\nabla  U^\varepsilon\|_{L^2_{xy}}\|\partial_t(\eta^B,\tau^B)\|_{L^2_{xy}}^{\frac{1}{2}}\\
                    &\times \|\partial_t\partial_x(\eta^B,\tau^B)\|_{L^2_{xy}}^{\frac{1}{2}}
                    \|  \partial_t\varTheta^\varepsilon\|_{L^2_{xy}}^{\frac{1}{2}}\|  \partial_y\partial_t\varTheta^\varepsilon\|_{L^2_{xy}}^{\frac{1}{2}}\\
                \leq{}&\frac{\varepsilon}{8}\|\nabla \partial_t ( H^\varepsilon,\varTheta^\varepsilon)\|_{L^2_{xy}}^2+ \frac{\mu}{16}\|\nabla\partial_t U^\varepsilon\|_{L^2_{xy}}^2
                + C\Big(\|\partial_t(H^\varepsilon,\varTheta^\varepsilon)\|_{L^2_{xy}}^2 + \|\nabla (U^\varepsilon,  H^\varepsilon,\varTheta^\varepsilon )\|_{L^2_{xy}}^2 + \varepsilon^2\Big).
                \end{align*}
                 Similar to the estimate in \eqref{eq_H_Ladyzhenskaya}, using Ladyzhenskaya's  inequality,  we can estimates $\mathcal{K}_3$ as follows:
        \begin{eqnarray*}
            \begin{aligned}
                \mathcal{K}_3\leq{} & \sqrt{\varepsilon}\| \varTheta^\varepsilon\|_{L^2_{xy}}^{\frac{1}{2}}\|\nabla \varTheta^\varepsilon\|_{L^2_{xy}}^{\frac{1}{2}}
                \| \partial_t U^\varepsilon\|_{L^2_{xy}}^{\frac{1}{2}}\|\nabla\partial_t U^\varepsilon\|_{L^2_{xy}}^{\frac{1}{2}}                 \|\nabla\partial_t \varTheta^\varepsilon\|_{L^2_{xy}} \\
                  & + \sqrt{\varepsilon} \|\nabla\partial_t U^\varepsilon\|_{L^2_{xy}}\|\partial_t \varTheta^\varepsilon\|_{L^2_{xy}}^{\frac{1}{2}}
                 \|\nabla\partial_t \varTheta^\varepsilon\|_{L^2_{xy}}^{\frac{1}{2}}
                 \Big( \| \varTheta^\varepsilon\|_{L^2_{xy}}^{\frac{1}{2}}\|\nabla \varTheta^\varepsilon\|_{L^2_{xy}}^{\frac{1}{2}}
                 +\| H^\varepsilon\|_{L^2_{xy}}^{\frac{1}{2}}\|\nabla H^\varepsilon\|_{L^2_{xy}}^{\frac{1}{2}}\Big)\\
                 &+  \sqrt{\varepsilon} \|\nabla U^\varepsilon\|_{L^2_{xy}}
                  \Big(\|\partial_t\varTheta^\varepsilon\|_{L^2_{xy}} \|\nabla\partial_t\varTheta^\varepsilon\|_{L^2_{xy}}
                   +\|\partial_tH^\varepsilon\|_{L^2_{xy}} \|\nabla\partial_tH^\varepsilon\|_{L^2_{xy}}
                  \Big)  \\
             \leq{}
               & C \Big(\|\nabla  U^\varepsilon\|_{L^2_{xy}}^2 +
                \| \varTheta^\varepsilon\|_{L^2_{xy}}^{2}\|\nabla \varTheta^\varepsilon\|_{L^2_{xy}}^{2}
               +\varepsilon \| H^\varepsilon\|_{L^2_{xy}}^{2}\|\nabla H^\varepsilon\|_{L^2_{xy}}^{2}   \Big)\|\partial_t (U^\varepsilon,H^\varepsilon,\varTheta^\varepsilon)\|_{L^2_{xy}}^{2}\\
              & \frac{\varepsilon}{8}\Big(\|\nabla \partial_t \varTheta^\varepsilon\|_{L^2_{xy}}^2+ \|\nabla \partial_t H^\varepsilon\|_{L^2_{xy}}^2 \Big)
                +  \frac{\mu}{16}\|\nabla \partial_t U^\varepsilon\|_{L^2_{xy}}^2.
            \end{aligned}
        \end{eqnarray*}
        Substituting the above estimates for $\mathcal{K}_1,\mathcal{K}_2, \mathcal{K}_3$  into  \eqref{eq_T_t_1}, we get
        \begin{eqnarray}\label{eq_T_t}
            \begin{aligned}
        & \frac{\mathrm{d}}{\mathrm{d}t} \|\partial_t \varTheta^\varepsilon\|_{L^2_{xy}}^2 + \gamma  \|\partial_t \varTheta^\varepsilon\|_{L^2_{xy}}^2 + \varepsilon\|\nabla\partial_t \varTheta^\varepsilon\|_{L^2_{xy}}^2 \\
        \leq{} &  C \Big(\|\nabla  U^\varepsilon\|_{L^2_{xy}}^2 +
        \| \varTheta^\varepsilon\|_{L^2_{xy}}^{2}\|\nabla \varTheta^\varepsilon\|_{L^2_{xy}}^{2}
       +  \| H^\varepsilon\|_{L^2_{xy}}^{2}\|\nabla H^\varepsilon\|_{L^2_{xy}}^{2}+ 1  \Big)\|\partial_t (U^\varepsilon,H^\varepsilon,\varTheta^\varepsilon)\|_{L^2_{xy}}^{2} \\
         & +C\Big(\| \nabla (U^\varepsilon,H^\varepsilon,\varTheta^\varepsilon)\|_{L^2_{xy}}^{2}
          + \varepsilon^2\Big)
         +\frac{\varepsilon}{2}\|\nabla \partial_t H^\varepsilon\|_{L^2_{xy}}^2  + \frac{\mu}{4}\|\nabla \partial_t U^\varepsilon\|_{L^2_{xy}}^2
         . \\
     \end{aligned}
\end{eqnarray}

Summing \eqref{eq_U_t}, \eqref{eq_H_t} and \eqref{eq_T_t} up, we can get the following estimate:
\begin{eqnarray}\label{eq_U_H_T_t}
    \begin{aligned}
& \frac{\mathrm{d}}{\mathrm{d}t}  \|\partial_t (U^\varepsilon,H^\varepsilon, \varTheta^\varepsilon)\|_{L^2_{xy}}^2
 +   \mu\|\nabla\partial_t U^\varepsilon\|_{L^2_{xy}}^2
 +  \varepsilon \|\nabla\partial_t (H^\varepsilon, \varTheta^\varepsilon)\|_{L^2_{xy}}^2      \\
\leq{}
  &  C \Big(\|\nabla  U^\varepsilon\|_{L^2_{xy}}^2 +
\varepsilon\|\nabla (H^\varepsilon,\varTheta^\varepsilon)\|_{L^2_{xy}}^{2}
 +   1  \Big)
 \|\partial_t (U^\varepsilon,H^\varepsilon,\varTheta^\varepsilon)\|_{L^2_{xy}}^{2}
 +C \| \nabla (U^\varepsilon,H^\varepsilon,\varTheta^\varepsilon)\|_{L^2_{xy}}^{2}  \\
& + C  \varepsilon^2 . \\
\end{aligned}
\end{eqnarray}

Next, to handle some associated terms on the right-hand side of \eqref{eq_U_H_T_t}, we need to estimate $\|\nabla U^\varepsilon\|_{L^2_{xy}},$ $\|\nabla H^\varepsilon\|_{L^2_{xy}}$ and  $\|\nabla \varTheta^\varepsilon\|_{L^2_{xy}}$. More specifically, taking $L^2_{xy}$ inner product  of \eqref{eq_UHT_varepsilon}$_1$ with $\partial_t U^\varepsilon,$ and using integration by parts, Lemmas \ref{lemma_inequality}, \ref{lemma_inner_outer_layer_est}, \ref{lemma_source_term_1} and \ref{lemma_UHT_infty_L2}, we have
\begin{eqnarray}\label{eq_nabla_U}
    \begin{aligned}
        &\frac{\mu}{2}\frac{\mathrm{d}}{\mathrm{d}t}  \|\nabla U^\varepsilon\|_{L^2_{xy}}^2 +  \|\partial_t U^\varepsilon\|_{L^2_{xy}}^2\\
        ={}& \langle U^\varepsilon\otimes u^a  , \nabla\partial_t U^\varepsilon \rangle
         + \langle   u^a  \otimes U^\varepsilon, \nabla \partial_t U^\varepsilon \rangle
         + \sqrt{\varepsilon}\langle U^\varepsilon \otimes U^\varepsilon, \nabla\partial_t U^\varepsilon \rangle
          - \langle \varTheta^\varepsilon, \nabla \partial_t U^\varepsilon \rangle
         +\langle \mathfrak{F}, \partial_t U^\varepsilon \rangle \\
         \leq{} &C   \|  U^\varepsilon\|_{L^2_{xy}}^{\frac{1}{2}} \| \partial_y U^\varepsilon\|_{L^2_{xy}}^{\frac{1}{2}}\|   u^a\|_{L^2_{xy}}^{\frac{1}{2}} \| \partial_x u^a\|_{L^2_{xy}}^{\frac{1}{2}}   \| \nabla \partial_t U^\varepsilon\|_{L^2_{xy}}
         +\sqrt{\varepsilon}\|U^{\varepsilon}\|_{L^2_{xy}}\|\nabla U^{\varepsilon}\|_{L^2_{xy}} \| \nabla \partial_t U^\varepsilon\|_{L^2_{xy}}\\
         & +\| \varTheta^\varepsilon\|_{L^2_{xy}}\| \nabla \partial_t U^\varepsilon\|_{L^2_{xy}}   + \| \mathfrak{F}\|_{L^2_{xy}}\| \partial_t U^\varepsilon\|_{L^2_{xy}}\\
         \leq{} & \frac{\mu}{4}\| \nabla \partial_t U^\varepsilon\|_{L^2_{xy}}^2 + C\left(
           \| \nabla U^\varepsilon\|_{L^2_{xy}}^2 + \|\partial_t U^\varepsilon\|_{L^2_{xy}}^2+ \varepsilon^2\right).
    \end{aligned}
\end{eqnarray}
Taking $L^2_{xy}$ inner product of $\nabla\eqref{eq_UHT_varepsilon}_2$ with $\nabla H^\varepsilon,$ and using integration by parts and Lemma \ref{lemma_inequality}, we have
    \begin{align*}
        &\frac{1}{2}\frac{\mathrm{d}}{\mathrm{d}t}  \|\nabla H^\varepsilon\|_{L^2_{xy}}^2 +  \varepsilon \|\nabla^2 H^\varepsilon\|_{L^2_{xy}}^2\\
        ={}&-\langle\nabla (u^I + \varepsilon^2 w)  \cdot\nabla H^\varepsilon, \nabla H^\varepsilon \rangle
           + \varepsilon \langle  u^B\cdot\nabla H^\varepsilon, \Delta H^\varepsilon \rangle
         - \langle \nabla (U^\varepsilon \cdot\nabla  \eta^I ), \nabla H^\varepsilon \rangle  \\
        & + \sqrt{\varepsilon} \langle  U^\varepsilon \cdot\nabla    \eta^B , \Delta H^\varepsilon \rangle
          - \sqrt{\varepsilon}\langle  \nabla U^\varepsilon \cdot\nabla H^\varepsilon,  \nabla H^\varepsilon \rangle
         - \langle  \mathfrak{G}_1,  \Delta H^\varepsilon \rangle  + \langle \nabla \mathfrak{G}_2,   \nabla H^\varepsilon \rangle \\
         \lesssim{} &  \| \nabla (u^I,w)\|_{L^\infty_{xy}}  \|\nabla H^\varepsilon\|_{L^2_{xy}}^2
         +   \varepsilon\| u^B\|_{L^2_{xy}}^{\frac{1}{2}}\| \partial_x u^B\|_{L^2_{xy}}^{\frac{1}{2}} \|\nabla H^\varepsilon\|_{L^2_{xy}}^{\frac{1}{2}} \|\nabla\partial_y H^\varepsilon\|_{L^2_{xy}}^{\frac{1}{2}} \|\Delta H^\varepsilon\|_{L^2_{xy}}\\
         & + \|  \eta^I\|_{H^3_{xy}}   \|  U^\varepsilon\|_{H^1_{xy}}\|\nabla  H^\varepsilon\|_{L^2_{xy}}
         +\sqrt{\varepsilon}\|U^\varepsilon\|_{L^2_{xy}}^{\frac{1}{2}}\|\partial_y U^\varepsilon\|_{L^2_{xy}}^{\frac{1}{2}}
         \|\nabla\eta^B\|_{L^2_{xy}}^{\frac{1}{2}}\|\partial_x\nabla\eta^B\|_{L^2_{xy}}^{\frac{1}{2}}\|\Delta H^\varepsilon\|_{L^2_{xy}}  \\
         &+ \sqrt{\varepsilon}
         \| \nabla U^\varepsilon\|_{L^2_{xy}}    \|   \nabla H^\varepsilon\|_{L^2_{xy}}\| \nabla^2 H^\varepsilon\|_{L^2_{xy}}
        + \| \mathfrak{G}_1\|_{L^2_{xy}}\| \Delta H^\varepsilon\|_{L^2_{xy}}  +  \| \nabla\mathfrak{G}_2\|_{L^2_{xy}}\|\nabla H^\varepsilon\|_{L^2_{xy}} \\
           \leq{}& \frac{\varepsilon}{2}\| \nabla^2 H^\varepsilon\|_{L^2_{xy}}^2
         + C \|\partial_x\nabla\eta^B\|_{L^2_{xy}}  \| \nabla\eta^B\|_{L^2_{xy}} \|  U^\varepsilon\|_{H^1_{xy}}^2\\
            &  +  C\Big(\|   \nabla U^\varepsilon\|_{L^2_{xy}}^{2} \|   \nabla H^\varepsilon\|_{L^2_{xy}}^{2}
            + \|  U^\varepsilon\|_{H^1_{xy}}^2 + \|\nabla  H^\varepsilon\|_{L^2_{xy}}^2 + \frac{1}{\varepsilon} \|  \mathfrak{G}_1\|_{L^2_{xy}}^2 + \| \nabla \mathfrak{G}_2\|_{L^2_{xy}}^2\Big),
    \end{align*}
which implies
\begin{eqnarray}\label{eq_nabla_H}
    \begin{aligned}
        & \frac{\mathrm{d}}{\mathrm{d}t}  \|\nabla H^\varepsilon\|_{L^2_{xy}}^2 +  \varepsilon \|\nabla^2 H^\varepsilon\|_{L^2_{xy}}^2\\
        \leq{}
           &   C\Big(\|   \nabla U^\varepsilon\|_{L^2_{xy}}^{2} + 1\Big)  \|   \nabla (U^\varepsilon,H^\varepsilon )\|_{L^2_{xy}}^{2}  + C \Bigl( \varepsilon^{-\frac{1}{2}} \|  \nabla U^\varepsilon\|_{L^2_{xy}}^2 + \varepsilon^{ \frac{3}{2}}\Bigr).
       \end{aligned}
\end{eqnarray}

Similarly, taking $L^2_{xy}$ inner product of $\nabla\eqref{eq_UHT_varepsilon}_3$ with $\nabla \varTheta^\varepsilon,$ and using integration by parts, we have
\begin{eqnarray}\label{eq_nabla_T_1}
    \begin{aligned}
        &\frac{1}{2}\frac{\mathrm{d}}{\mathrm{d}t}  \|\nabla \varTheta^\varepsilon\|_{L^2_{xy}}^2 + \gamma  \|\nabla \varTheta^\varepsilon\|_{L^2_{xy}}^2  +  \varepsilon \|\nabla^2 \varTheta^\varepsilon\|_{L^2_{xy}}^2\\
        =&-\langle\nabla (u^I+\varepsilon^2w)\cdot\nabla \varTheta^\varepsilon, \nabla \varTheta^\varepsilon \rangle+  \langle \varepsilon u^B\cdot\nabla \varTheta^\varepsilon + \sqrt{\varepsilon}   U^\varepsilon \cdot\nabla    \tau^B, \Delta \varTheta^\varepsilon \rangle
        - \langle \nabla (  U^\varepsilon \cdot\nabla  \tau^I ) , \nabla \varTheta^\varepsilon \rangle
          \\
         &+\langle \nabla \mathcal{Q}(\nabla (u^I+\varepsilon^2w), \varTheta^\varepsilon ) +\nabla \mathcal{Q}(\nabla U^\varepsilon, \tau^I)
        +   \nabla \mathcal{B}(\eta^I,\nabla U^\varepsilon)  + \nabla\mathcal{B}(H^\varepsilon,\nabla (u^I+\varepsilon^2w)) ,  \nabla \varTheta^\varepsilon\rangle
        \\
        &- \langle \varepsilon   \mathcal{Q}(\nabla u^B, \varTheta^\varepsilon ) + \sqrt{\varepsilon}\mathcal{Q}(\nabla U^\varepsilon, \tau^B) , \Delta \varTheta^\varepsilon\rangle
         -  \langle     \sqrt{\varepsilon} \mathcal{B}(\eta^B,\nabla U^\varepsilon) + \varepsilon\mathcal{B}(H^\varepsilon,\nabla u^B),  \Delta \varTheta^\varepsilon\rangle
         \\
        &- \sqrt{\varepsilon}\langle \nabla  U^\varepsilon \cdot\nabla \varTheta^\varepsilon, \nabla \varTheta^\varepsilon \rangle
         - \sqrt{\varepsilon}\langle   \mathcal{Q}(\nabla U^\varepsilon, \varTheta^\varepsilon)+\mathcal{B}(H^\varepsilon,\nabla U^\varepsilon),\Delta\varTheta^\varepsilon\rangle
          -\langle  \mathfrak{J}_1,  \Delta \varTheta^\varepsilon \rangle + \langle  \nabla\mathfrak{J}_2,  \nabla \varTheta^\varepsilon \rangle  \\
          =:{}& \mathcal{M}_1 + \mathcal{M}_2 + \mathcal{M}_3 + \mathcal{M}_4,
     \end{aligned}
\end{eqnarray}
where $\mathcal{M}_i$ represents the whole $i$-th line on the right hand side of the first equality sign in \eqref{eq_nabla_T_1}.

For $\mathcal{M}_1$, using Lemmas \ref{lemma_inequality}, \ref{lemma_inner_outer_layer_est} and \ref{lemma_UHT_infty_L2}, we have
\begin{eqnarray*}
    \begin{aligned}
        \mathcal{M}_1 \leq{} & \|   \nabla (u^I,w)\|_{L^\infty_{xy}}\|\nabla\varTheta^\varepsilon\|_{L^2_{xy}}^2
        +   \varepsilon\| u^B\|_{L^2_{xy}}^{\frac{1}{2}}\| \partial_x u^B\|_{L^2_{xy}}^{\frac{1}{2}} \|\nabla \varTheta^\varepsilon\|_{L^2_{xy}}^{\frac{1}{2}} \|\nabla\partial_y \varTheta^\varepsilon\|_{L^2_{xy}}^{\frac{1}{2}} \|\Delta \varTheta^\varepsilon\|_{L^2_{xy}}\\
        & +\sqrt{\varepsilon}\|U^\varepsilon\|_{L^2_{xy}}^{\frac{1}{2}}\| \partial_y  U^\varepsilon\|_{L^2_{xy}}^{\frac{1}{2}}
        \|\nabla\tau^B\|_{L^2_{xy}}^{\frac{1}{2}}\|\partial_x\nabla\tau^B\|_{L^2_{xy}}^{\frac{1}{2}}\|\Delta \varTheta^\varepsilon\|_{L^2_{xy}}
       + \|  U^\varepsilon\|_{H^1_{xy}} \|  \tau^I \|_{H^3_{xy}} \|\nabla \varTheta^\varepsilon\|_{L^2_{xy}} \\
        \leq{}& \frac{\varepsilon}{8}\| \nabla^2  \varTheta^\varepsilon\|_{L^2_{xy}}^2 +C\Big(
           \|\nabla  (U^\varepsilon,\varTheta^\varepsilon)\|_{L^2_{xy}}^2 + \varepsilon^{-\frac{1}{2}} \|  \nabla U^\varepsilon\|_{L^2_{xy}}^2 + \varepsilon^{\frac{3}{2}} \Big).
    \end{aligned}
\end{eqnarray*}
In a similar way, we have, for $\mathcal{M}_2,\mathcal{M}_3$, that
    \begin{align*}
        \mathcal{M}_2\leq {}&  \|\nabla ( u^I,w)\|_{H^2_{xy}}\Big(\| \varTheta^\varepsilon\|_{H^1_{xy}}^2 + \| H^\varepsilon\|_{H^1_{xy}}^2\Big)
        +\|\nabla^2U^\varepsilon\|_{L^2_{xy}}   \| (\eta^I,\tau^I)\|_{L^\infty_{xy}}  \|\nabla\varTheta^\varepsilon\|_{L^2_{xy}} \\
        &+\|\nabla U^\varepsilon\|_{L^2_{xy}}    \|\nabla (\eta^I,\tau^I)\|_{L^\infty_{xy}}  \|\nabla\varTheta^\varepsilon\|_{L^2_{xy}} \\
        \leq{}& \frac{1}{16}\|\nabla^2U^\varepsilon\|_{L^2_{xy}}^2
         +C\Big(\| \nabla(U^\varepsilon,H^\varepsilon,\varTheta^\varepsilon)\|_{L^2_{xy}}^2   +     \varepsilon^2\Big),
    \end{align*}
and
    \begin{align*}
        \mathcal{M}_3\leq {}&  \varepsilon  \|(H^\varepsilon,\varTheta^\varepsilon)\|_{H^1_{xy}}\|\nabla u^B\|_{L^2_{xy}}^{\frac{1}{2}}\|\partial_x\nabla u^B\|_{L^2_{xy}}^{\frac{1}{2}}\|\Delta \varTheta^\varepsilon\|_{L^2_{xy}}
        + \sqrt{\varepsilon}\|\nabla U^\varepsilon\|_{L^2_{xy}}\|(\eta^B,\tau^B)\|_{L^\infty_{xy}} \|\Delta \varTheta^\varepsilon\|_{L^2_{xy}}\\
         \leq{}& \frac{\varepsilon}{8}\|  \nabla^2 \varTheta^\varepsilon\|_{L^2_{xy}}^2
        +C\Big(\varepsilon^{\frac{1}{2}}\|\nabla (H^\varepsilon, \varTheta^\varepsilon)\|_{L^2_{xy}}^2   + \|\nabla U^\varepsilon\|_{L^2_{xy}}^2  + \varepsilon^2\Big).
    \end{align*}
Similar to the estimates of $\mathcal{K}_3,$ we get
    \begin{align*}
        \mathcal{M}_4\leq {}&  \sqrt{\varepsilon}
        \|  \nabla U^\varepsilon\|_{L^2_{xy}} \| \nabla \varTheta^\varepsilon\|_{L^2_{xy}} \|   \nabla^2 \varTheta^\varepsilon\|_{L^2_{xy}}
        +  \sqrt{\varepsilon}\Big( \|  \nabla U^\varepsilon\|_{L^2_{xy}}^{\frac{1}{2}}\| \nabla^2 U^\varepsilon\|_{L^2_{xy}}^{\frac{1}{2}}
        \|  \varTheta^\varepsilon\|_{L^2_{xy}}^{\frac{1}{2}}\| \nabla\varTheta^\varepsilon\|_{L^2_{xy}}^{\frac{1}{2}}
         \\
        &+  \|  H^\varepsilon\|_{L^2_{xy}}^{\frac{1}{2}}\| \nabla H^\varepsilon\|_{L^2_{xy}}^{\frac{1}{2}}
        \|  \nabla U^\varepsilon\|_{L^2_{xy}}^{\frac{1}{2}}\| \nabla^2 U^\varepsilon\|_{L^2_{xy}}^{\frac{1}{2}}
          \Big)
        \|\Delta \varTheta^\varepsilon\|_{L^2_{xy}}
         +\|   \mathfrak{J}_1\|_{L^2_{xy}}\|\Delta \varTheta^\varepsilon\|_{L^2_{xy}} \\
        &+  \|  \nabla \mathfrak{J}_2\|_{L^2_{xy}}
        \|  \nabla\varTheta^\varepsilon\|_{L^2_{xy}} \\
        \leq{}& \frac{\varepsilon}{8}\|\nabla^2 \varTheta^\varepsilon\|_{L^2_{xy}}^2 + \frac{1}{16}\|\nabla^2U^\varepsilon\|_{L^2_{xy}}^2
        +  C\Big(1+ \|   \nabla U^\varepsilon\|_{L^2_{xy}}^{2} \Big)   \|   \nabla (H^\varepsilon,\varTheta^\varepsilon)\|_{L^2_{xy}}^{2}     + C\varepsilon^2.
    \end{align*}

Substituting the above estimates for $\mathcal{M}_1,\cdots, \mathcal{M}_4$  into  \eqref{eq_nabla_T_1}, we get
\begin{eqnarray}\label{eq_nabla_T}
    \begin{aligned}
        & \frac{\mathrm{d}}{\mathrm{d}t}  \|\nabla \varTheta^\varepsilon\|_{L^2_{xy}}^2 + \gamma  \|\nabla \varTheta^\varepsilon\|_{L^2_{xy}}^2  +  \varepsilon \|\nabla^2 \varTheta^\varepsilon\|_{L^2_{xy}}^2\\
        \leq {}& \frac{1}{4}\|\nabla^2U^\varepsilon\|_{L^2_{xy}}^2
         + C\Big(\|   \nabla U^\varepsilon\|_{L^2_{xy}}^{2} + 1\Big)  \|   \nabla (U^\varepsilon,H^\varepsilon,\varTheta^\varepsilon)\|_{L^2_{xy}}^{2}  + C \Bigl( \varepsilon^{-\frac{1}{2}} \|  \nabla U^\varepsilon\|_{L^2_{xy}}^2 + \varepsilon^{ \frac{3}{2}}\Bigr).
     \end{aligned}
\end{eqnarray}
Finally, to close the estimates deduced, we also need the dissipation estimate of $ \nabla^2 U^\varepsilon $. To this end,   using the Stokes estimate,  Lemmas \ref{lemma_inequality}, \ref{lemma_inner_outer_layer_est}, \ref{lemma_source_term_1}  and \ref{lemma_UHT_infty_L2}, we have
\begin{eqnarray*}
    \begin{aligned}
           \|U^\varepsilon\|_{H^2_{xy}}^2 \lesssim {}&\|\nabla \varTheta^\varepsilon\|_{L^2_{xy}}^2 + \|\mathfrak{F}\|_{L^2_{xy}}^2 +\|\partial_t U^\varepsilon\|_{L^2_{xy}}^2
           +  \|u^a\|_{L^\infty_{xy}}^2  \|U^\varepsilon\|_{H^1_{xy}}^2
           + \|U^\varepsilon\|_{H^1_{xy}}^2 \|\nabla (u^I,w)\|_{H^1_{xy}}^2 \\
           &  + \varepsilon^2  \|U^\varepsilon\|_{L_{xy}^2} \|U^\varepsilon\|_{H_{xy}^2} \|\nabla u^B\|_{L^2_{xy}}^2
           +\varepsilon\|U^\varepsilon\|_{L^2_{xy}}\|\nabla U^\varepsilon\|_{L^2_{xy}}^2\|\nabla^2 U^\varepsilon\|_{L^2_{xy}}
            \\
        \leq {}& \frac{1}{2}\|  U^\varepsilon\|_{H^2_{xy}}^2 + C\Big(\|\nabla \varTheta^\varepsilon\|_{L^2_{xy}}^2 + \|\mathfrak{F}\|_{L^2_{xy}}^2 +\|\partial_t U^\varepsilon\|_{L^2_{xy}}^2
        + \|U^\varepsilon\|_{H^1_{xy}}^2 + \|U^\varepsilon\|_{L^2_{xy}}^2\|\nabla U^\varepsilon\|_{L^2_{xy}}^4\Big),
     \end{aligned}
\end{eqnarray*}
which implies
\begin{eqnarray}\label{eq_nabla2_U}
    \begin{aligned}
           \|U^\varepsilon\|_{H^2_{xy}}^2 \lesssim{}   &    \|\nabla \varTheta^\varepsilon\|_{L^2_{xy}}^2 + \|\mathfrak{F}\|_{L^2_{xy}}^2 +\|\partial_t U^\varepsilon\|_{L^2_{xy}}^2
        + \|U^\varepsilon\|_{H^1_{xy}}^2 + \|U^\varepsilon\|_{L^2_{xy}}^2\|\nabla U^\varepsilon\|_{L^2_{xy}}^4.
     \end{aligned}
\end{eqnarray}

Summing \eqref{eq_nabla_U}, \eqref{eq_nabla_H}, \eqref{eq_nabla_T} and \eqref{eq_nabla2_U} up, we get
\begin{eqnarray} \label{eq_U_H_T_nabla}
    \begin{aligned}
        & \frac{\mathrm{d}}{\mathrm{d}t}  \left(\mu \|\nabla U^\varepsilon\|_{L^2_{xy}}^2 + \|\nabla (H^\varepsilon,\varTheta^\varepsilon)\|_{L^2_{xy}}^2 \right)
        +   \|U^\varepsilon\|_{H^2_{xy}}^2+\varepsilon   \|\nabla^2(H^\varepsilon, \varTheta^\varepsilon)\|_{L^2_{xy}}^2 \\
        \leq{} &  \frac{\mu}{4}\|\nabla\partial_t U^\varepsilon\|_{L^2_{xy}}^2 +  C\Big(\|   \nabla U^\varepsilon\|_{L^2_{xy}}^{2} + 1\Big)  \|   \nabla (U^\varepsilon,H^\varepsilon,\varTheta^\varepsilon)\|_{L^2_{xy}}^{2}  + C \Bigl( \|\partial_t U^\varepsilon\|_{L^2_{xy}}^2+ \varepsilon^{-\frac{1}{2}} \|  \nabla U^\varepsilon\|_{L^2_{xy}}^2 + \varepsilon^{ \frac{3}{2}}\Bigr).
     \end{aligned}
\end{eqnarray}

Putting \eqref{eq_U_H_T_t} and \eqref{eq_U_H_T_nabla} together, we get
\begin{equation}\label{eq_U_H_T_nabla_t}
    \begin{split}
& \frac{\mathrm{d}}{\mathrm{d}t} \left(    \|\partial_t (U^\varepsilon,H^\varepsilon, \varTheta^\varepsilon)\|_{L^2_{xy}}^2
+\mu \|\nabla U^\varepsilon\|_{L^2_{xy}}^2
+ \|\nabla (H^\varepsilon,\varTheta^\varepsilon)\|_{L^2_{xy}}^2\right)
+   \mu\|\nabla\partial_t U^\varepsilon\|_{L^2_{xy}}^2  \\
 &+   \|U^\varepsilon\|_{H^2_{xy}}^2
 +  \varepsilon\left(\|\nabla\partial_t (H^\varepsilon,\varTheta^\varepsilon)\|_{L^2_{xy}}^2  +\|\nabla^2( H^\varepsilon,\varTheta^\varepsilon)\|_{L^2_{xy}}^2 \right)\\
\lesssim{}
  & \left(\|\nabla  U^\varepsilon\|_{L^2_{xy}}^2 +
  \varepsilon\|\nabla (H^\varepsilon,\varTheta^\varepsilon)\|_{L^2_{xy}}^{2}
    + 1  \right) \left(\|\partial_t (U^\varepsilon, H^\varepsilon,\varTheta^\varepsilon)\|_{L^2_{xy}}^{2}
+\|\nabla (U^\varepsilon, H^\varepsilon,\varTheta^\varepsilon)\|_{L^2_{xy}}^{2} \right)     \\
&    + \varepsilon^{-\frac{1}{2}} \|  \nabla U^\varepsilon\|_{L^2_{xy}}^2
  + \varepsilon^{ \frac{3}{2}}.
\end{split}
\end{equation}
Noticing that $(\partial_tU^\varepsilon,\partial_t H^\varepsilon, \partial_t \varTheta^\varepsilon)|_{t=0} = 0$, using Gronwall's inequality and Lemma \ref{lemma_UHT_infty_L2}, we deduce
\begin{eqnarray*}
    \begin{aligned}
        &  \|\partial_t (U^\varepsilon,H^\varepsilon,\varTheta^\varepsilon)\|_{L^\infty_TL^2_{xy}}^2
+\mu \|\nabla U^\varepsilon\|_{L^\infty_TL^2_{xy}}^2
  + \|\nabla (H^\varepsilon,\varTheta^\varepsilon)\|_{L^\infty_TL^2_{xy}}^2+   \mu\|\nabla\partial_t U^\varepsilon\|_{L^2_TL^2_{xy}}^2\\
 &  +   \|U^\varepsilon\|_{L^2_TH^2_{xy}}^2
 +  \varepsilon\Big(\|\nabla\partial_t (H^\varepsilon,\varTheta^\varepsilon)\|_{L^2_TL^2_{xy}}^2
  +\|\nabla^2 (H^\varepsilon,\varTheta^\varepsilon)\|_{L^2_TL^2_{xy}}^2 \Big)
\leq   \tilde{C}_3 T
     \varepsilon^{\frac{3}{2}} \exp(\tilde{C}_3T).
    \end{aligned}
\end{eqnarray*}
The proof of Lemma \ref{lemma_UHT_nabla_t} is complete.
\end{proof}
\begin{remark}\label{remark_lemma_UHT_nabla_t}
    For readers' convenience, we remark that the following bounds:
    \begin{align*}
          & \|(u^I,w,\eta^I,\tau^I)\|_{L^\infty_TH^3_{xy}} +  \|(u^B,\eta^B,\tau^B)\|_{L^\infty_TH^1_xH^1_z}
          +  \|(u^B,\eta^B,\tau^B)\|_{L^\infty_TH^2_xL^2_z}    \leq C,\\
       & \| \partial_t (  u^I,w,\eta^I,\tau^I ) \|_{L^\infty_TH^2_{xy}}
       +\|\partial_t (  u^B,\eta^B,\tau^B)\|_{L^\infty_TH^1_xL^2_z}
       + \|\partial_t u^B\|_{L^\infty_T L^2_xH^1_z}
            \leq C,
    \end{align*} \eqref{eq_source_term} and \eqref{eq_source_term_t} are used in the proof of  Lemma \ref{lemma_UHT_nabla_t}.
\end{remark}
\begin{lemma}\label{lemma_HT_nabla_x}
    Suppose that $(U^\varepsilon,H^\varepsilon,\varTheta^\varepsilon)$ is the solution of the problem \eqref{eq_UHT_varepsilon}-\eqref{eq_UHT_int_bo} on $[0,T]$ with $T\in (0,\min\{T_\star,T_\varepsilon\}]$. Under the assumptions of  Theorem \ref{theorem_limit}, there exists a generic positive constant C such that
    \begin{eqnarray*}
        \|U^\varepsilon\|_{L^\infty_TH^2_{xy}}^2 \leq C \varepsilon^{\frac{3}{2}},~~
           \|\nabla\partial_x \varTheta^\varepsilon\|_{L^\infty_TL^2_{xy}}^2 + \|\nabla\partial_x H^\varepsilon\|_{L^\infty_TL^2_{xy}}^2
        \leq C\varepsilon^{\frac{1}{2}},
     \end{eqnarray*}
     and
     \begin{eqnarray*}
        \varepsilon \Big(\|\nabla^2\partial_x \varTheta^\varepsilon\|_{L^2_TL^2_{xy}}^2
        + \|\nabla^2\partial_x H^\varepsilon\|_{L^2_TL^2_{xy}}^2  \Big) \leq C\varepsilon^{\frac{1}{2}}.\\
     \end{eqnarray*}
\end{lemma}
\begin{proof}
    It follows from \eqref{eq_nabla2_U}, Lemmas \ref{lemma_inner_outer_layer_est},   \ref{lemma_UHT_infty_L2} and \ref{lemma_UHT_nabla_t} that
    \begin{eqnarray}\label{eq_U_varepsilon_H2}
        \begin{aligned}
               \|U^\varepsilon\|_{L^\infty_TH^2_{xy}}^2 \lesssim{}   &    \|\nabla \varTheta^\varepsilon\|_{L^\infty_TL^2_{xy}}^2 + \|\mathfrak{F}\|_{L^\infty_TL^2_{xy}}^2 +\|\partial_t U^\varepsilon\|_{L^\infty_TL^2_{xy}}^2\\
            &+ \|U^\varepsilon\|_{L^\infty_TH^1_{xy}}^2 + \|U^\varepsilon\|_{L^\infty_TL^2_{xy}}^2\|\nabla U^\varepsilon\|_{L^\infty_TL^2_{xy}}^4
            \lesssim   \varepsilon^{\frac{3}{2}}.
         \end{aligned}
    \end{eqnarray}

    Next, multiplying $\nabla\partial_x$ \eqref{eq_UHT_varepsilon}$_2$ by $\nabla\partial_x H^\varepsilon$, integrating by parts over $\mathbb{R}^2_{+}$, we get
    \begin{eqnarray}\label{eq_H_varepsilon_nabla_x_1}
        \begin{aligned}
            &\frac{1}{2}\frac{\mathrm{d}}{\mathrm{d}t} \|\nabla\partial_x H^\varepsilon\|_{L^2_{xy}}^2 +  \varepsilon\|\nabla^2\partial_x H^\varepsilon\|_{L^2_{xy}}^2 \\
            ={}& \langle \partial_x u^a \cdot\nabla H^\varepsilon,\Delta\partial_x H^\varepsilon \rangle
            -\langle\nabla u^a \cdot\nabla\partial_x H^\varepsilon,\nabla\partial_x H^\varepsilon \rangle
              +\langle \partial_x U^\varepsilon\cdot\nabla  \eta^a ,\Delta\partial_x H^\varepsilon \rangle\\
              &+ \langle  U^\varepsilon\cdot\nabla \partial_x \eta^a ,\Delta\partial_x H^\varepsilon \rangle
             +\sqrt{\varepsilon}\langle  \partial_x U^\varepsilon\cdot\nabla H^\varepsilon   ,\Delta\partial_x H^\varepsilon \rangle
            -\sqrt{\varepsilon}\langle \nabla U^\varepsilon\cdot\nabla\partial_x H^\varepsilon   ,\nabla\partial_x H^\varepsilon \rangle \\
            &-\langle  \partial_x \mathfrak{G},\Delta\partial_x H^\varepsilon \rangle  \\
              =:{}&\mathcal{N}_1 +\mathcal{N}_2 +\mathcal{N}_3 ,
        \end{aligned}
    \end{eqnarray}
   where $\mathcal{N}_i$ represents the whole $i$-th line on the right-hand side of the first equality in \eqref{eq_H_varepsilon_nabla_x_1}.

Using  Lemmas \ref{lemma_inequality}, \ref{lemma_inner_outer_layer_est}, \ref{lemma_source_term_1}, \ref{lemma_UHT_infty_L2}  and \ref{lemma_UHT_nabla_t}, we can estimate $\mathcal{N}_1,\mathcal{N}_2,$ and $\mathcal{N}_3$ as follows:
    \begin{align*}
        \mathcal{N}_1\leq{} &
         \|\partial_x (u^I,w)\|_{L^\infty_{xy}}  \|\nabla H^\varepsilon\|_{L^2_{xy}} \|\Delta\partial_x H^\varepsilon\|_{L^2_{xy}}
          + \|\nabla (u^I,w)\|_{L^\infty_{xy}}\|\nabla\partial_x H^\varepsilon\|_{L^2_{xy}}^2 \\
        &+ \varepsilon\|\partial_x u^B\|_{L^2_{xy}}^{\frac{1}{2}}\|\partial_y\partial_x u^B\|_{L^2_{xy}}^{\frac{1}{2}}
        \|\nabla H^\varepsilon\|_{L^2_{xy}}^{\frac{1}{2}}\|\nabla \partial_x H^\varepsilon\|_{L^2_{xy}}^{\frac{1}{2}}\|\Delta\partial_x H^\varepsilon\|_{L^2_{xy}}\\
          &+\varepsilon \|\nabla  u^B\|_{L^2_{xy}}  \|\nabla\partial_x H^\varepsilon\|_{L^2_{xy}}  \|\nabla^2 \partial_x H^\varepsilon\|_{L^2_{xy}}
         +\| \partial_x U^\varepsilon\|_{L^2_{xy}} \| \nabla \eta^I\|_{L^\infty_{xy}}
           \|\Delta\partial_x H^\varepsilon\|_{L^2_{xy}}\\
           &+ \sqrt{\varepsilon}\|\partial_x U^\varepsilon\|_{L^2_{xy}}^{\frac{1}{2}}\|\partial_y\partial_x U^\varepsilon\|_{L^2_{xy}}^{\frac{1}{2}}
           \|\nabla\eta^B\|_{L^2_{xy}}^{\frac{1}{2}}\|\nabla\partial_x\eta^B\|_{L^2_{xy}}^{\frac{1}{2}}    \|\Delta\partial_x H^\varepsilon\|_{L^2_{xy}} \\
         \leq{}&\frac{\varepsilon}{4}\|\nabla^2\partial_x H^\varepsilon\|_{L^2_{xy}}^2 + C\Big(\|\nabla\partial_x H^\varepsilon\|_{L^2_{xy}}^2 + \varepsilon^{\frac{1}{2}} \Big),
    \end{align*}
and
    \begin{align*}
        & \mathcal{N}_2+\mathcal{N}_3  \\
        \leq{}& \Big(\|U^\varepsilon\|_{L^2}\|\nabla\partial_x\eta^I\|_{L^\infty_{xy}}
         +   \varepsilon^{\frac{1}{2}} \|\nabla \partial_x \eta^B\|_{L^2_{xy}}^{\frac{1}{2}}\|\nabla\partial_x^2 \eta^B\|_{L^2_{xy}}^{\frac{1}{2}}\|  U^\varepsilon\|_{L^2_{xy}}^{\frac{1}{2}}\| \partial_y  U^\varepsilon\|_{L^2_{xy}}^{\frac{1}{2}}\Big) \|\Delta\partial_x H^\varepsilon\|_{L^2_{xy}}\\
          &+\varepsilon^{\frac{1}{2}}\|\partial_xU^\varepsilon\|_{L_{xy}^2}^{\frac{1}{2}}\|\partial_x\partial_yU^\varepsilon\|_{L_{xy}^2}^{\frac{1}{2}}
        \|\nabla H^\varepsilon\|_{L_{xy}^2}^{\frac{1}{2}}\|\nabla\partial_xH^\varepsilon\|_{L_{xy}^2}^{\frac{1}{2}}\|\Delta\partial_xH^\varepsilon\|_{L_{xy}^2}\\
        &+\varepsilon^{\frac{1}{2}} \|\nabla U^\varepsilon\|_{L_{xy}^2}
         \|\nabla\partial_xH^\varepsilon\|_{L_{xy}^2} \|\nabla^2\partial_xH^\varepsilon\|_{L_{xy}^2}
         +   \|  \partial_x \mathfrak{G}\|_{L_{xy}^2} \|\Delta\partial_x H^\varepsilon \|_{L_{xy}^2} \\
        \leq{}&\frac{\varepsilon}{4}\|\nabla^2\partial_x H^\varepsilon\|_{L^2_{xy}}^2 + C\Big(\|\nabla\partial_x H^\varepsilon\|_{L^2_{xy}}^2   + \varepsilon \Big).
    \end{align*}
Substituting the above estimates for $\mathcal{N}_1,\mathcal{N}_2, \mathcal{N}_3$  into  \eqref{eq_H_varepsilon_nabla_x_1}, we get
\begin{eqnarray}\label{eq_H_varepsilon_nabla_x}
    \begin{aligned}
        & \frac{\mathrm{d}}{\mathrm{d}t} \|\nabla\partial_x H^\varepsilon\|_{L^2_{xy}}^2 +  \varepsilon\|\nabla^2\partial_x H^\varepsilon\|_{L^2_{xy}}^2
            \lesssim   \|\nabla\partial_x H^\varepsilon\|_{L^2_{xy}}^2 + \varepsilon^{\frac{1}{2}}.
    \end{aligned}
\end{eqnarray}

Multiplying $\nabla\partial_x$ \eqref{eq_UHT_varepsilon}$_3$ by $\nabla\partial_x \varTheta^\varepsilon$, integrating by parts over  $\mathbb{R}^2_{+}$, we get
    \begin{align}\label{eq_T_varepsilon_nabla_x_1}
        &\frac{1}{2}\frac{\mathrm{d}}{\mathrm{d}t} \|\nabla\partial_x \varTheta^\varepsilon\|_{L^2_{xy}}^2 +  \varepsilon\|\nabla^2\partial_x \varTheta^\varepsilon\|_{L^2_{xy}}^2
        +\gamma\|\nabla\partial_x \varTheta^\varepsilon\|_{L^2_{xy}}^2\notag \\
        ={}&-\langle  \partial_x \mathfrak{J},\Delta\partial_x \varTheta^\varepsilon \rangle
        +\langle \partial_x  u^a \cdot\nabla \varTheta^\varepsilon,\Delta\partial_x \varTheta^\varepsilon \rangle
        -\langle\nabla u^a \cdot\nabla\partial_x  \varTheta^\varepsilon,\nabla\partial_x \varTheta^\varepsilon \rangle
         + \langle  \partial_xU^\varepsilon\cdot\nabla  \tau^a ,\Delta\partial_x \varTheta^\varepsilon \rangle\notag\\
         &+ \langle  U^\varepsilon\cdot\nabla \partial_x \tau^a ,\Delta\partial_x \varTheta^\varepsilon \rangle
         +\sqrt{\varepsilon}\langle  \partial_x U^\varepsilon\cdot\nabla \varTheta^\varepsilon   ,\Delta\partial_x \varTheta^\varepsilon \rangle
         -\sqrt{\varepsilon}\langle \nabla U^\varepsilon\cdot\nabla\partial_x \varTheta^\varepsilon   ,\nabla\partial_x \varTheta^\varepsilon \rangle \\
         &  - \langle \mathcal{Q}(\nabla \partial_x u^a, \varTheta^\varepsilon )
        + \mathcal{Q}(\nabla  u^a, \partial_x\varTheta^\varepsilon )  ,\Delta\partial_x \varTheta^\varepsilon\rangle
           - \langle \mathcal{Q}(\nabla \partial_xU^\varepsilon, \tau^a)
          +\mathcal{Q}(\nabla U^\varepsilon, \partial_x\tau^a)  ,\Delta\partial_x \varTheta^\varepsilon\rangle
          \notag\\
        &
         - \langle \mathcal{B}(\partial_x\eta^a,\nabla U^\varepsilon)
         +\mathcal{B}(\eta^a,\nabla\partial_x U^\varepsilon)  ,\Delta\partial_x \varTheta^\varepsilon\rangle
         -\langle \mathcal{B}(\partial_xH^\varepsilon,\nabla u^a)
           +\mathcal{B}(H^\varepsilon,\nabla \partial_xu^a)  ,\Delta\partial_x \varTheta^\varepsilon\rangle
           \notag\\
                  &  - \sqrt{\varepsilon} \langle \mathcal{Q}(\nabla \partial_xU^\varepsilon, \varTheta^\varepsilon) +  \mathcal{Q}(\nabla U^\varepsilon, \partial_x\varTheta^\varepsilon) ,\Delta\partial_x \varTheta^\varepsilon\rangle
                   -\sqrt{\varepsilon}\langle \mathcal{B}(\partial_xH^\varepsilon,\nabla U^\varepsilon),\Delta\partial_x \varTheta^\varepsilon\rangle
                   \notag\\
                  &-\sqrt{\varepsilon}\langle \mathcal{B}(H^\varepsilon,\nabla \partial_xU^\varepsilon),\Delta\partial_x \varTheta^\varepsilon\rangle
                   \notag\\
          =:&\sum_{i=1}^6\mathcal{O}_i,\notag
    \end{align}
    where $\mathcal{O}_i$ represents the whole $i$-th line on the right hand side of the first equality sign in \eqref{eq_T_varepsilon_nabla_x_1}.

 $\mathcal{O}_1 + \mathcal{O}_2 $ can be estimated similar to $\mathcal{N}_1 +\mathcal{N}_2+ \mathcal{N}_3$. Thus, we just list the conclusion:
 \begin{eqnarray*}
    \mathcal{O}_1 + \mathcal{O}_2  \leq \frac{\varepsilon}{ 8}\|\nabla^2\partial_x \varTheta^\varepsilon\|_{L^2_{xy}}^2 + C\Big(\|\nabla\partial_x \varTheta^\varepsilon\|_{L^2_{xy}}^2  + \varepsilon^{\frac{1}{2}} \Big).
 \end{eqnarray*}
 Using Lemmas \ref{lemma_inequality}, \ref{lemma_inner_outer_layer_est}, \ref{lemma_UHT_infty_L2}  and \ref{lemma_UHT_nabla_t},  we can estimate the rest terms as follows:
 \begin{align*}
    \mathcal{O}_3  \leq{} &
     \Big(\|\nabla\partial_x (u^I,w)\|_{L^\infty_{xy}}\|\varTheta^\varepsilon\|_{L^2_{xy}} + \|\nabla (u^I,w)\|_{L^\infty_{xy}}\| \partial_x \varTheta^\varepsilon\|_{L^2_{xy}} \Big) \|\Delta \partial_x \varTheta^\varepsilon\|_{L^2_{xy}}\\
     &+   \varepsilon\|\nabla \partial_x u^B\|_{L^2_{xy}}^{\frac{1}{2}}\|\nabla \partial_x^2 u^B\|_{L^2_{xy}}^{\frac{1}{2}}
     \|    \varTheta^\varepsilon\|_{L^2_{xy}}^{\frac{1}{2}} \|   \partial_y \varTheta^\varepsilon\|_{L^2_{xy}}^{\frac{1}{2}}  \|\Delta \partial_x \varTheta^\varepsilon\|_{L^2_{xy}}\\
     &+   \varepsilon\|\nabla   u^B\|_{L^2_{xy}}^{\frac{1}{2}}\|\nabla \partial_x  u^B\|_{L^2_{xy}}^{\frac{1}{2}}
     \|    \partial_x\varTheta^\varepsilon\|_{L^2_{xy}}^{\frac{1}{2}} \|   \partial_y\partial_x \varTheta^\varepsilon\|_{L^2_{xy}}^{\frac{1}{2}}  \|\Delta \partial_x \varTheta^\varepsilon\|_{L^2_{xy}}\\
    &+ \|   \nabla \partial_x U^\varepsilon\|_{L^2_{xy}} \|  \tau^a\|_{L^\infty_{xy}}  \|\Delta \partial_x \varTheta^\varepsilon\|_{L^2_{xy}}
    + \|   \nabla  U^\varepsilon\|_{L^2_{xy}}  \| \partial_x \tau^a\|_{L^\infty_{xy}}  \|\Delta \partial_x \varTheta^\varepsilon\|_{L^2_{xy}}\\
     \leq{}&\frac{\varepsilon}{8}\|\nabla^2\partial_x \varTheta^\varepsilon\|_{L^2_{xy}}^2 + C\Big(\|\nabla\partial_x \varTheta^\varepsilon\|_{L^2_{xy}}^2   + \varepsilon^{\frac{1}{2}}\Big).
\end{align*}
Similarly, we have
\begin{eqnarray*}
    \begin{aligned}
       \mathcal{O}_4
       \leq{}&\frac{\varepsilon}{8}\|\nabla^2\partial_x \varTheta^\varepsilon\|_{L^2_{xy}}^2 + C\Big(\|\nabla\partial_x H^\varepsilon\|_{L^2_{xy}}^2   +\varepsilon^{\frac{1}{2}} \Big).
   \end{aligned}
   \end{eqnarray*}
   For the last two terms, using \eqref{eq_u_infty}, we have
   \begin{eqnarray*}
    \begin{aligned}
      & \mathcal{O}_5 +\mathcal{O}_6 \\
       \leq {} &  \varepsilon^{\frac{1}{2}}\|\nabla\partial_xU^\varepsilon\|_{L_{xy}^2}\Big(\|\varTheta^\varepsilon\|_{H^1_{xy}} + \|\varTheta^\varepsilon\|_{L^2_{xy}}^{\frac{1}{2}}
         \|\partial_x\partial_y\varTheta^\varepsilon\|_{L^2_{xy}}^{\frac{1}{2}}
          +\|H^\varepsilon\|_{H^1_{xy}} + \|H^\varepsilon\|_{L^2_{xy}}^{\frac{1}{2}}
         \|\partial_x\partial_yH^\varepsilon\|_{L^2_{xy}}^{\frac{1}{2}}   \Big)\|\Delta\partial_x\varTheta^\varepsilon\|_{L_{xy}^2}\\
       &+\varepsilon^{\frac{1}{2}}\|\nabla U^\varepsilon\|_{L_{xy}^2}^{\frac{1}{2}}\|\nabla^2U^\varepsilon\|_{L_{xy}^2}^{\frac{1}{2}}
       \Big(\|\partial_x \varTheta^\varepsilon\|_{L_{xy}^2}^{\frac{1}{2}}\|\nabla\partial_x\varTheta^\varepsilon\|_{L_{xy}^2}^{\frac{1}{2}}
   + \|\partial_x H^\varepsilon\|_{L_{xy}^2}^{\frac{1}{2}}\|\nabla\partial_xH^\varepsilon\|_{L_{xy}^2}^{\frac{1}{2}} \Big)\|\Delta\partial_x\varTheta^\varepsilon\|_{L_{xy}^2}\\
         \leq {}&\frac{\varepsilon}{8}\|\nabla^2\partial_x H^\varepsilon\|_{L^2_{xy}}^2 + C\Big(\|\nabla\partial_x \varTheta^\varepsilon\|_{L^2_{xy}}^2+ \|\nabla\partial_x H^\varepsilon\|_{L^2_{xy}}^2   + \varepsilon^{\frac{1}{2}}\Big).
   \end{aligned}
   \end{eqnarray*}
Substituting the above estimates for $\mathcal{O}_1,\cdots, \mathcal{O}_6$  into  \eqref{eq_T_varepsilon_nabla_x_1}, we get
\begin{eqnarray}\label{eq_T_varepsilon_nabla_x}
    \begin{aligned}
        & \frac{\mathrm{d}}{\mathrm{d}t} \|\nabla\partial_x \varTheta^\varepsilon\|_{L^2_{xy}}^2 +  \varepsilon\|\nabla^2\partial_x \varTheta^\varepsilon\|_{L^2_{xy}}^2
            \lesssim  \|\nabla\partial_x \varTheta^\varepsilon\|_{L^2_{xy}}^2+ \|\nabla\partial_x H^\varepsilon\|_{L^2_{xy}}^2    + \varepsilon^{\frac{1}{2}}.
    \end{aligned}
\end{eqnarray}

Summing \eqref{eq_H_varepsilon_nabla_x} and \eqref{eq_T_varepsilon_nabla_x} up, we deduce form Gronwall's inequality that
\begin{eqnarray}\label{eq_TH_nabla_x}
    \begin{aligned}
        &   \|\nabla\partial_x (H^\varepsilon, \varTheta^\varepsilon)\|_{L^\infty_TL^2_{xy}}^2 +  \varepsilon \|\nabla^2\partial_x (H^\varepsilon,\varTheta^\varepsilon)\|_{L^2_TL^2_{xy}}^2
            \leq  \tilde{C}_4  T \varepsilon^{\frac{1}{2}}\exp(\tilde{C}_4 T  ).
    \end{aligned}
\end{eqnarray}
\end{proof}
\begin{remark}\label{remark_lemma_HT_nabla_x}
    For readers' convenience, we remark  that the following bounds:
    \begin{align*}
        &    \|(u^I,w,\eta^I,\tau^I)\|_{L^\infty_TH^4_{xy}} +  \|(u^B,\eta^B,\tau^B)\|_{L^\infty_TH^3_xL^2_z}
        +  \|(u^B,\eta^B,\tau^B)\|_{L^\infty_TH^2_xH^1_z}  \leq C,
    \end{align*} and \eqref{eq_source_term} are used in the proof of  Lemma \ref{lemma_HT_nabla_x}.
\end{remark}

\begin{lemma}\label{lemma_UHT_x_t}
    Suppose that $(U^\varepsilon,H^\varepsilon,\varTheta^\varepsilon)$  is the solution of the problem \eqref{eq_UHT_varepsilon}-\eqref{eq_UHT_int_bo} on $[0,T]$ with $T\in (0,\min\{T_\star,T_\varepsilon\}]$. Under the assumptions of  Theorem \ref{theorem_limit}, there exists a generic positive constant C such that
    \begin{eqnarray*}
      \varepsilon^{\frac{1}{2}} \|H^\varepsilon\|_{L^\infty_TH^2_{xy}}^2  +  \varepsilon^{\frac{1}{2}}\|\varTheta^\varepsilon\|_{L^\infty_TH^2_{xy}}^2 \leq C,
    \end{eqnarray*}
    \begin{eqnarray*}
         \|\partial_t \partial_xU^\varepsilon\|_{L^\infty_TL^2_{xy}}^2 + \|\partial_t \partial_xH^\varepsilon\|_{L^\infty_TL^2_{xy}}^2
            + \|\partial_t \partial_x\varTheta^\varepsilon\|_{L^\infty_TL^2_{xy}}^2
           \leq C \varepsilon^{\frac{1}{2}} ,
    \end{eqnarray*}
    and
    \begin{eqnarray*}
         \|\nabla\partial_t\partial_x U^\varepsilon\|_{L^2_tL^2_{xy}}^2  +   \varepsilon\|\nabla\partial_t\partial_x H^\varepsilon\|_{L^2_TL^2_{xy}}^2   + \varepsilon\|\nabla\partial_t\partial_x \varTheta^\varepsilon\|_{L^2_TL^2_{xy}}^2
         \leq C  \varepsilon^{\frac{1}{2}} .
    \end{eqnarray*}
\end{lemma}
\begin{proof}
    To begin wth, from  $\eqref{eq_UHT_varepsilon}_2$ and elliptic estimate, Lemmas \ref{lemma_inequality}, \ref{lemma_inner_outer_layer_est}, \ref{lemma_source_term_1} and \ref{lemma_UHT_infty_L2}, we have
    \begin{eqnarray*}\label{eq_H_varepsilon_H2}
        \begin{aligned}
              \varepsilon^2\|H^\varepsilon\|_{L^\infty_TH^2_{xy}}^2
              \lesssim {}  &    \|\mathfrak{G}\|_{L^\infty_TL^2_{xy}}^2 +\|\partial_t H^\varepsilon\|_{L^\infty_TL^2_{xy}}^2
              + \Big(\|u^a\|_{L^\infty_TL^\infty_{xy}}^2 +1 \Big)\| H^\varepsilon\|_{L^\infty_TH^1_{xy}}^2 \\
              &  + \|U^\varepsilon\|_{L^\infty_TH^2_{xy}}^2\Bigl(\|\nabla \eta^a \|_{L^\infty_TL^2_{xy}}^2  + \|\nabla H^\varepsilon\|_{L^\infty_TL^2_{xy}}^2 + 1 \Bigr)\\
             \lesssim{} & \varepsilon^{\frac{3}{2}}.
         \end{aligned}
    \end{eqnarray*}
  And a similar procedure yields that
  \begin{eqnarray*}\label{eq_T_varepsilon_H2}
    \begin{aligned}
          \varepsilon^2\|\varTheta^\varepsilon\|_{L^\infty_TH^2_{xy}}^2
          \lesssim {}  &    \|\mathfrak{J}\|_{L^\infty_TL^2_{xy}}^2 +\|\partial_t \varTheta^\varepsilon\|_{L^\infty_TL^2_{xy}}^2
          + \Big(\|u^a\|_{L^\infty_TL^\infty_{xy}}^2+  \|\nabla u^a\|_{L^4_{xy}}^2 + 1\Big)\| (H^\varepsilon,\varTheta^\varepsilon)\|_{L^\infty_TH^1_{xy}}^2 \\
          &  + \|U^\varepsilon\|_{L^\infty_TH^2_{xy}}^2\Bigl(\|\nabla \tau^a \|_{L^\infty_TL^2_{xy}}^2 +\|  (\eta^a,\tau^a) \|_{L^\infty_TL^4_{xy}}^2 + \|  (H^\varepsilon,\varTheta^\varepsilon)\|_{L^\infty_TH^1_{xy}}^2 + 1 \Bigr)\\
         \lesssim{} & \varepsilon^{\frac{3}{2}}.
     \end{aligned}
\end{eqnarray*}

The proof of the rest inequalities in Lemma \ref{lemma_UHT_x_t} is similar to Lemma \ref{lemma_UHT_nabla_t}. We just sketch the main steps here.
Multiplying $\partial_t\partial_x$ \eqref{eq_UHT_varepsilon}$_1$ by $\partial_t\partial_x U^\varepsilon$, and integrating by parts over $\mathbb{R}^2_{+}$, we get
        \begin{align*}
            &\frac{1}{2}\frac{\mathrm{d}}{\mathrm{d}t} \|\partial_t \partial_xU^\varepsilon\|_{L^2_{xy}}^2 +  \mu\|\nabla\partial_t\partial_x U^\varepsilon\|_{L^2_{xy}}^2 \\
            ={}&  - \langle \partial_t \mathfrak{F},\partial_t \partial_x^2 U^\varepsilon \rangle
            +\langle U^\varepsilon\otimes \partial_t\partial_xu^a
            +  \partial_xU^\varepsilon\otimes \partial_tu^a
            +  \partial_t U^\varepsilon\otimes \partial_xu^a
            ,\nabla\partial_t\partial_x U^\varepsilon \rangle\\
            &
             +\langle   u^a \otimes\partial_t\partial_x U^\varepsilon
             + \partial_xu^a \otimes\partial_t U^\varepsilon
            +   \partial_tu^a \otimes\partial_x U^\varepsilon
            +\partial_t\partial_xu^a  \otimes U^\varepsilon,\nabla\partial_t\partial_x U^\varepsilon \rangle\\
            &
            +\sqrt{\varepsilon}\langle \partial_x U^\varepsilon \otimes \partial_tU^\varepsilon +\partial_t U^\varepsilon \otimes \partial_xU^\varepsilon , \nabla \partial_t\partial_x U^\varepsilon \rangle
            -\sqrt{\varepsilon}\langle \partial_t\partial_x U^\varepsilon \cdot \nabla U^\varepsilon  ,  \partial_t\partial_x U^\varepsilon \rangle \\
            &
             -\langle \partial_t\partial_x \varTheta^\varepsilon, \nabla \partial_t\partial_x U^\varepsilon \rangle \\
            =:{}& \mathcal{R}_1+ \mathcal{R}_2+ \mathcal{R}_3 + \mathcal{R}_4,
         \end{align*}
where $\mathcal{R}_i$ represents the whole $i$-th line on the right-hand side of the first equality sign in the above equality.

    For the first two terms, using Lemmas \ref{lemma_inequality}, \ref{lemma_inner_outer_layer_est}, \ref{lemma_source_term_1}, \ref{lemma_UHT_infty_L2}  and \ref{lemma_UHT_nabla_t}, we have
        \begin{align*}
             \mathcal{R}_1+ \mathcal{R}_2
             \leq{}&  \Bigl( \|U^\varepsilon\|_{L^\infty_{xy}}\|\partial_t\partial_x u^a\|_{L^2_{xy}}
             +\|\partial_x U^\varepsilon\|_{L^2_{xy}}^{\frac{1}{2}}\|\partial_y\partial_x U^\varepsilon\|_{L^2_{xy}}^{\frac{1}{2}}
             \|\partial_t u^a\|_{L^2_{xy}}^{\frac{1}{2}}\|\partial_t\partial_x u^a\|_{L^2_{xy}}^{\frac{1}{2}}\Bigr)
             \|\nabla\partial_t\partial_x U^\varepsilon\|_{L^2_{xy}}
             \\
             & + \Bigl(\|  \partial_t\partial_xU^\varepsilon\|_{L^2_{xy}}  \|  u^a\|_{L^\infty_{xy}} +\|  \partial_tU^\varepsilon\|_{L^2_{xy}}  \|\partial_x u^a\|_{L^\infty_{xy}} \Bigr) \|\nabla\partial_t\partial_x U^\varepsilon\|_{L^2_{xy}}
             +\|  \partial_t  \mathfrak{F} \|_{L^2_{xy}}\|  \partial_t\partial_x^2 U^\varepsilon\|_{L^2_{xy}}   \\
            \leq{}&  \frac{\mu}{4}\|\nabla\partial_t\partial_x U^\varepsilon\|_{L^2_{xy}}^2    +C  \|\partial_t \partial_xU^\varepsilon\|_{L^2_{xy}}^2   + C\varepsilon^{\frac{3}{2}}.
        \end{align*}
    The rest terms can be estimated as follows:
        \begin{align*}
            &\mathcal{R}_3+ \mathcal{R}_4\\
             \leq{}
             &   \varepsilon^{\frac{1}{2}}\|  \partial_x U^\varepsilon\|_{L^2_{xy}}^{\frac{1}{2}}\|  \partial_y\partial_x U^\varepsilon\|_{L^2_{xy}}^{\frac{1}{2}}\| \partial_x \partial_t  U^\varepsilon\|_{L^2_{xy}}^{\frac{1}{2}}\|  \partial_t U^\varepsilon\|_{L^2_{xy}}^{\frac{1}{2}}\| \nabla\partial_t\partial_x U^\varepsilon\|_{L^2_{xy}}\\
            & + \varepsilon^{\frac{1}{2}}\| \partial_t\partial_x U^\varepsilon\|_{L^2_{xy}}\| \nabla\partial_t\partial_x U^\varepsilon\|_{L^2_{xy}}\|\nabla  U^\varepsilon\|_{L^2_{xy}} +  \| \partial_t \partial_x\varTheta^\varepsilon\|_{L^2_{xy}} \| \nabla\partial_t\partial_x U^\varepsilon\|_{L^2_{xy}}\\
            \leq{}&  \frac{\mu}{4}\|\nabla\partial_t\partial_x U^\varepsilon\|_{L^2_{xy}}^2    +C\Big( \|\partial_t \partial_xU^\varepsilon\|_{L^2_{xy}}^2 +\| \partial_t\partial_x \varTheta^\varepsilon\|_{L^2_{xy}}^2 \Big) + C\varepsilon^{\frac{3}{2}},
        \end{align*}
    which, together with the estimate of $\mathcal{R}_1+ \mathcal{R}_2 $, implies
    \begin{equation}\label{eq_U_tx}
        \begin{split}
            &\frac{\mathrm{d}}{\mathrm{d}t} \|\partial_t \partial_xU^\varepsilon\|_{L^2_{xy}}^2 +  \mu\|\nabla\partial_t\partial_x U^\varepsilon\|_{L^2_{xy}}^2
            \lesssim{}
             \| \partial_t\partial_x (U^\varepsilon,\varTheta^\varepsilon)\|_{L^2_{xy}}^2   + \varepsilon^{\frac{3}{2}}.
        \end{split}
    \end{equation}
    Multiplying $\partial_t\partial_x$ \eqref{eq_UHT_varepsilon}$_2$ by $\partial_t\partial_x H^\varepsilon$, and integrating by parts over $\mathbb{R}^2_{+}$, we get
    \begin{eqnarray*}\label{eq_H_xt_1}
        \begin{aligned}
            &\frac{1}{2}\frac{\mathrm{d}}{\mathrm{d}t} \|\partial_t\partial_x H^\varepsilon\|_{L^2_{xy}}^2 +  \varepsilon\|\nabla\partial_t\partial_x H^\varepsilon\|_{L^2_{xy}}^2 \\
            ={}& - \langle \partial_t  \mathfrak{G},\partial_t \partial_x^2 H^\varepsilon \rangle
            + \langle H^\varepsilon  \partial_t\partial_xu^a
            +  \partial_xH^\varepsilon  \partial_tu^a
            +  \partial_t H^\varepsilon  \partial_xu^a ,\nabla\partial_t\partial_x H^\varepsilon \rangle
            -\langle \partial_t\partial_x U^\varepsilon\cdot\nabla\eta^I ,\partial_t\partial_x H^\varepsilon \rangle\\
            &+\sqrt{\varepsilon}\langle  \eta^B  \partial_t\partial_x U^\varepsilon,\nabla\partial_t\partial_x H^\varepsilon \rangle
            +\langle
                \partial_x\eta^a  \partial_t U^\varepsilon
            +   \partial_t \eta^a \partial_x U^\varepsilon
            + \partial_t\partial_x \eta^a   U^\varepsilon  ,\nabla\partial_t\partial_x H^\varepsilon \rangle\\
            &
            +\sqrt{\varepsilon}\langle \partial_t H^\varepsilon  \partial_x U^\varepsilon
            +  \partial_x H^\varepsilon   \partial_tU^\varepsilon
            +  H^\varepsilon  \partial_t\partial_x U^\varepsilon   , \nabla\partial_t\partial_x H^\varepsilon \rangle
              \\
            =:&~ \mathcal{S}_1 + \mathcal{S}_2 +\mathcal{S}_3,
        \end{aligned}
    \end{eqnarray*}
where $\mathcal{S}_i$ represents the whole $i$-th line on the right hand side of the first equality sign in the above equality.

    Similar to the estimate of $\mathcal{R}_1+ \mathcal{R}_2 $, we have
        \begin{align*}
            &\mathcal{S}_1 + \mathcal{S}_2 \\
            \lesssim{}& \|\partial_t\mathfrak{G}\|_{L^2_{xy}}\|\partial_t\partial_x^2 H^\varepsilon\|_{L^2_{xy}}
            + \|H^\varepsilon\|_{L^2_{xy}}^{\frac{1}{2}}\|\partial_yH^\varepsilon\|_{L^2_{xy}}^{\frac{1}{2}}
            \|\partial_t \partial_x u^a\|_{L^2_{xy}}^{\frac{1}{2}}\|\partial_t\partial_x^2 u^a\|_{L^2_{xy}}^{\frac{1}{2}}
            \|\nabla\partial_t\partial_x H^\varepsilon\|_{L^2_{xy}}\\
            &
            + \Bigl(\|\partial_x H^\varepsilon\|_{L^2_{xy}}
            \|\partial_t u^a\|_{L^\infty_{xy}}
            + \|\partial_t H^\varepsilon\|_{L^2_{xy}}\|\partial_x   u^a\|_{L^\infty_{xy}}
            +\sqrt{\varepsilon} \|\eta^B\|_{L^\infty_{xy}}\|\partial_t\partial_x U^\varepsilon\|_{L^2_{xy}}
            \Bigr)\|\nabla\partial_t\partial_x H^\varepsilon\|_{L^2_{xy}}
            \\
            &  + \|\nabla\eta^{I}\|_{L^\infty_{xy}}\|\partial_t\partial_x U^\varepsilon\|_{L^2_{xy}}\|\partial_t\partial_x H^\varepsilon\|_{L^2_{xy}}
            +\Big(
             \|\partial_x\eta^a\|_{L^\infty_{xy}}\|\partial_t U^\varepsilon\|_{L^2_{xy}}
            +\|\partial_t\partial_x\eta^{a}\|_{L^2_{xy}}\|U^\varepsilon\|_{L^\infty_{xy}}
             \\
             & + \|\partial_t\eta^a\|_{L^2_{xy}}^{\frac{1}{2}}\|\partial_t\partial_x\eta^a\|_{L^2_{xy}}^{\frac{1}{2}} \|\partial_x U^\varepsilon\|_{L^2_{xy}}^{\frac{1}{2}} \|\partial_y\partial_x U^\varepsilon\|_{L^2_{xy}}^{\frac{1}{2}}\Big)\|\nabla\partial_t\partial_x H^\varepsilon\|_{L^2_{xy}}
              \\
             \leq{}&  \frac{\varepsilon}{4}\|\nabla\partial_t\partial_x H^\varepsilon\|_{L^2_{xy}}^2    +C  \|\partial_t \partial_x(U^\varepsilon,H^\varepsilon)\|_{L^2_{xy}}^2   + C\varepsilon^{\frac{1}{2}}.
        \end{align*}
    For $\mathcal{S}_3 $, we have
        \begin{align*}
             \mathcal{S}_3
            \lesssim{} & \varepsilon^{\frac{1}{2}}\|  \partial_t  H^\varepsilon\|_{L^2_{xy}}^{\frac{1}{2}}\| \partial_x \partial_t H^\varepsilon\|_{L^2_{xy}}^{\frac{1}{2}}\|  \partial_x U^\varepsilon\|_{L^2_{xy}}^{\frac{1}{2}}\|  \partial_y\partial_x U^\varepsilon\|_{L^2_{xy}}^{\frac{1}{2}}\| \nabla\partial_t\partial_x H^\varepsilon\|_{L^2_{xy}}\\
            & +\varepsilon^{\frac{1}{2}}\|  \partial_x H^\varepsilon\|_{L^2_{xy}}^{\frac{1}{2}}\|  \partial_y\partial_x H^\varepsilon\|_{L^2_{xy}}^{\frac{1}{2}}\|  \partial_t  U^\varepsilon\|_{L^2_{xy}}^{\frac{1}{2}}\| \partial_x \partial_t U^\varepsilon\|_{L^2_{xy}}^{\frac{1}{2}}\| \nabla\partial_t\partial_x H^\varepsilon\|_{L^2_{xy}}\\
            &
            + \varepsilon^{\frac{1}{2}}\|  H^\varepsilon\|_{L^2_{xy}}^{\frac{1}{2}}\|  \partial_y   H^\varepsilon\|_{L^2_{xy}}^{\frac{1}{2}} \| \partial_t\partial_x U^\varepsilon\|_{L^2_{xy}}^{\frac{1}{2}}\|  \partial_t\partial_x^2 U^\varepsilon\|_{L^2_{xy}}^{\frac{1}{2}}\| \nabla\partial_t\partial_x H^\varepsilon\|_{L^2_{xy}}  \\
            \leq{} &  \frac{\mu}{8}\|\nabla\partial_t\partial_x U^\varepsilon\|_{L^2_{xy}}^2 + \frac{\varepsilon}{4}\|\nabla\partial_t\partial_x H^\varepsilon\|_{L^2_{xy}}^2
            + C\Big(\|\partial_t\partial_x ( U^\varepsilon,H^\varepsilon)\|_{L^2_{xy}}^2
            +\varepsilon^{\frac{1}{2}} \Big) ,
        \end{align*}
    which together with the estimate of $\mathcal{S}_1+ \mathcal{S}_2$, implies
    \begin{eqnarray}\label{eq_H_tx}
        \begin{aligned}
            &\frac{\mathrm{d}}{\mathrm{d}t} \|\partial_t \partial_xH^\varepsilon\|_{L^2_{xy}}^2 +  \varepsilon\|\nabla\partial_t\partial_x H^\varepsilon\|_{L^2_{xy}}^2
            \leq   \frac{\mu}{4}\|\nabla\partial_t\partial_x U^\varepsilon\|_{L^2_{xy}}^2 + C\Big(\|\partial_t\partial_x ( U^\varepsilon,H^\varepsilon)\|_{L^2_{xy}}^2
             +\varepsilon^{\frac{1}{2}}\Big)  .
        \end{aligned}
    \end{eqnarray}

    Multiplying $\partial_t\partial_x$ \eqref{eq_UHT_varepsilon}$_3$ by $\partial_t\partial_x \varTheta^\varepsilon$, and integrating by parts over $\mathbb{R}^2_{+}$, we get
    \begin{align*}
            & \frac{1}{2}\frac{\mathrm{d}}{\mathrm{d}t} \|\partial_t\partial_x \varTheta^\varepsilon\|_{L^2_{xy}}^2
            + \gamma \| \partial_t\partial_x \varTheta^\varepsilon\|_{L^2_{xy}}^2 +  \varepsilon\|\nabla\partial_t\partial_x \varTheta^\varepsilon\|_{L^2_{xy}}^2 \\
            = {}& - \langle \partial_t  \mathfrak{J},\partial_t \partial_x^2 \varTheta^\varepsilon \rangle
            + \langle \varTheta^\varepsilon \otimes  \partial_t\partial_xu^a
            +  \partial_x\varTheta^\varepsilon \otimes   \partial_tu^a
            +  \partial_t \varTheta^\varepsilon \otimes    \partial_xu^a
            ,
            \nabla\partial_t\partial_x \varTheta^\varepsilon \rangle
            \\
            &-\langle \partial_t\partial_x U^\varepsilon\cdot\nabla\tau^I ,\partial_t\partial_x \varTheta^\varepsilon \rangle
             +\langle
                \partial_x\tau^a \otimes \partial_t U^\varepsilon
            +   \partial_t \tau^a \otimes \partial_x U^\varepsilon
            + \partial_t\partial_x \tau^a \otimes  U^\varepsilon  ,\nabla\partial_t\partial_x \varTheta^\varepsilon \rangle\\
            &
            +\sqrt{\varepsilon}\langle  \tau^B\otimes  \partial_t\partial_x U^\varepsilon
             +\partial_t \varTheta^\varepsilon \otimes    \partial_x U^\varepsilon
            +  \partial_x \varTheta^\varepsilon \otimes     \partial_tU^\varepsilon
            +  \varTheta^\varepsilon \otimes    \partial_t\partial_x U^\varepsilon   , \nabla\partial_t\partial_x \varTheta^\varepsilon \rangle
              \\
           &+ \langle \mathcal{Q} (\nabla u^a, \partial_t\partial_x\varTheta^{\varepsilon})
              +  \mathcal{Q} (\nabla \partial_xu^a  , \partial_t\varTheta^{\varepsilon})
               +   \mathcal{Q} (\nabla \partial_tu^a , \partial_x\varTheta^{\varepsilon})
              +  \mathcal{Q} (\nabla \partial_t\partial_xu^a, \varTheta^{\varepsilon}),  \partial_t\partial_x \varTheta^\varepsilon  \rangle
                \\
              &+ \langle \mathcal{Q}(\nabla \partial_t\partial_xU^\varepsilon, \tau^a)
              + \mathcal{Q} (\nabla \partial_xU^\varepsilon, \partial_t\tau^a)
              +  \mathcal{Q} (\nabla \partial_tU^\varepsilon, \partial_x\tau^a  )
               +    \mathcal{Q} (\nabla U^\varepsilon, \partial_t\partial_x\tau^a  ),  \partial_t\partial_x \varTheta^\varepsilon  \rangle\\
               &+ \langle \mathcal{B} (H^{\varepsilon},\nabla \partial_t\partial_xu^a )
              +   \mathcal{B} (\partial_tH^{\varepsilon},\nabla \partial_xu^a )
               +   \mathcal{B} (\partial_xH^{\varepsilon},\nabla \partial_tu^a  )
               +\mathcal{B} (\partial_t\partial_xH^{\varepsilon},\nabla u^a ),  \partial_t\partial_x \varTheta^\varepsilon  \rangle \\
              &
               + \langle \mathcal{B} ( \eta^a,\nabla \partial_t\partial_xU^\varepsilon )
              +   \mathcal{B} (\partial_t\eta^a,\nabla \partial_xU^\varepsilon  )
               +    \mathcal{B} (\partial_x\eta^a,\nabla \partial_tU^\varepsilon )
               +    \mathcal{B} (\partial_t\partial_x\eta^a ,\nabla U^\varepsilon),  \partial_t\partial_x \varTheta^\varepsilon  \rangle \\
             &
               +  \sqrt{\varepsilon}\langle \mathcal{Q} (\nabla \partial_t\partial_xU^\varepsilon, \varTheta^{\varepsilon})
               +   \mathcal{Q} (\nabla \partial_xU^\varepsilon, \partial_t\varTheta^{\varepsilon})
               +   \mathcal{Q} (\nabla \partial_tU^\varepsilon, \partial_x\varTheta^{\varepsilon})
               +  \langle \mathcal{Q} (\nabla U^\varepsilon, \partial_t\partial_x\varTheta^{\varepsilon}),  \partial_t\partial_x \varTheta^\varepsilon  \rangle \\
               &+\sqrt{\varepsilon}\langle \mathcal{B} (H^{\varepsilon},\nabla \partial_t\partial_xU^\varepsilon  )
               +    \mathcal{B} (\partial_tH^{\varepsilon},\nabla \partial_xU^\varepsilon
               +    \mathcal{B} (\partial_xH^{\varepsilon},\nabla \partial_tU^\varepsilon )
               +  \langle \mathcal{B} (\partial_t\partial_xH^{\varepsilon},\nabla U^\varepsilon ),  \partial_t\partial_x \varTheta^\varepsilon  \rangle
                \\
               =:&{} \sum_{i=1}^{9}\mathcal{W}_i,
        \end{align*}
       where $\mathcal{W}_i$ represents the whole $i$-th line on the right hand side of the first equality sign in  the above equality.

    After conducting some steps similar to the estimate of $\mathcal{S}_1 +\mathcal{S}_2$ and $\mathcal{S}_3$, we obtain
        \begin{align*}
             \sum_{i=1}^{3}\mathcal{W}_i\leq{} &  \frac{\varepsilon}{16} \|\nabla\partial_t\partial_x \varTheta^\varepsilon\|_{L^2_{xy}}^2
             +\frac{\mu}{32}\|\nabla\partial_t\partial_x U^\varepsilon\|_{L^2_{xy}}^2
            + C\Big(\|\partial_t\partial_x  U^\varepsilon\|_{L^2_{xy}}^2
            +\|\partial_t\partial_x  \varTheta^\varepsilon\|_{L^2_{xy}}^2
            +\varepsilon^{\frac{1}{2}}
              \Big).
           \end{align*}

          For $\mathcal{W}_4,\mathcal{W}_5,\mathcal{W}_6,\mathcal{W}_7$, using Lemmas \ref{lemma_inequality}, \ref{lemma_inner_outer_layer_est}, \ref{lemma_UHT_infty_L2}  and \ref{lemma_UHT_nabla_t}, we have
           \begin{align*}
            &\mathcal{W}_4 + \mathcal{W}_6\\  \lesssim {}&
            \|\nabla (u^I,w)\|_{L^\infty_{xy}}\Bigl( \|\partial_t\partial_x\varTheta^\varepsilon\|_{L^2_{xy}}^2
            +\|\partial_t\partial_x H^\varepsilon\|_{L^2_{xy}}^2\Bigr)
            + \varepsilon\|\nabla u^B\|_{L^2_{xy}} \|\partial_t\partial_x\varTheta^\varepsilon\|_{L^2_{xy}}
            \|\nabla\partial_t\partial_x\varTheta^\varepsilon\|_{L^2_{xy}} \\
            & + \varepsilon\|\nabla u^B\|_{L^2_{xy}} \|\partial_t\partial_x\varTheta^\varepsilon\|_{L^2_{xy}}^{\frac{1}{2}}
            \|\nabla\partial_t\partial_x\varTheta^\varepsilon\|_{L^2_{xy}}^{\frac{1}{2}} \|\partial_t\partial_xH^\varepsilon\|_{L^2_{xy}}^{\frac{1}{2}}
            \|\nabla\partial_t\partial_xH^\varepsilon\|_{L^2_{xy}}^{\frac{1}{2}}  \\
            &+\|\nabla  \partial_x (u^I,w)\|_{L^2_{xy}}^{\frac{1}{2}}\|\nabla \partial_y\partial_x (u^I,w)\|_{L^2_{xy}}^{\frac{1}{2}}
            \Bigl(\|\partial_t\varTheta^\varepsilon\|_{L^2_{xy}}^{\frac{1}{2}}\|\partial_t\partial_x\varTheta^\varepsilon\|_{L^2_{xy}}^{\frac{1}{2}}
            +\|\partial_tH^\varepsilon\|_{L^2_{xy}}^{\frac{1}{2}}\|\partial_t\partial_xH^\varepsilon\|_{L^2_{xy}}^{\frac{1}{2}}\Bigr)\\
            &\times \|\partial_t\partial_x \varTheta^\varepsilon\|_{L^2_{xy}}
            +\varepsilon\|\nabla\partial_x u^B\|_{L^2_{xy}}^{\frac{1}{2}}\|\nabla\partial_x^2 u^B\|_{L^2_{xy}}^{\frac{1}{2}}
            \|\partial_t (\varTheta^\varepsilon,H^\varepsilon)\|_{L^2_{xy}}
            \|\partial_t\partial_x\varTheta^\varepsilon\|_{L^2_{xy}}^{\frac{1}{2}}
            \|\partial_y\partial_t\partial_x\varTheta^\varepsilon\|_{L^2_{xy}}^{\frac{1}{2}}\\
            &+ \|\nabla\partial_t u^a\|_{L^2_{xy}}^{\frac{1}{2}}\|\nabla\partial_t\partial_x u^a\|_{L^2_{xy}}^{\frac{1}{2}}
            \Bigl(\|\partial_x\varTheta^{\varepsilon}\|_{L^2_{xy}}^{\frac{1}{2}}\|\partial_y\partial_x\varTheta^{\varepsilon}\|_{L^2_{xy}}^{\frac{1}{2}}
            +\|\partial_x H^\varepsilon\|_{L^2_{xy}}^{\frac{1}{2}}\|\partial_y\partial_x H^{\varepsilon}\|_{L^2_{xy}}^{\frac{1}{2}} \Bigr)
            \|\partial_t\partial_x \varTheta^\varepsilon\|_{L^2_{xy}}\\
            &+\|\nabla\partial_t\partial_x u^a\|_{L^2_{xy}}\Bigl(\|(H^\varepsilon,\varTheta^\varepsilon)\|_{H^1_{xy}}
            + \|(H^\varepsilon,\varTheta^{\varepsilon})\|_{L^2_{xy}}^{\frac{1}{2}}
            \|\partial_y\partial_x(H^\varepsilon,\varTheta^{\varepsilon})\|_{L^2_{xy}}^{\frac{1}{2}}\Bigr)
            \|\partial_t\partial_x \varTheta^\varepsilon\|_{L^2_{xy}}\\
           \leq{} &  \frac{\varepsilon}{16}\Bigl(\|\nabla\partial_t\partial_x H^\varepsilon\|_{L^2_{xy}}^2
            +\|\nabla\partial_t\partial_x \varTheta^\varepsilon\|_{L^2_{xy}}^2 \Bigr)
             + C\Big(\|\partial_t\partial_x  (U^\varepsilon,H^\varepsilon,\varTheta^\varepsilon)\|_{L^2_{xy}}^2
             + \varepsilon^{\frac{1}{2}}\Big),
          \end{align*}
          and
          \begin{align*}
            &\mathcal{W}_5 + \mathcal{W}_7\\
             \lesssim {}&
             \|\nabla\partial_t\partial_x U^\varepsilon\|_{L^2_{xy}}\|(\eta^a,\tau^a)\|_{L^\infty_{xy}}\|\partial_t\partial_x\varTheta^\varepsilon\|_{L^2_{xy}}
             +\|\nabla\partial_t U^\varepsilon\|_{L^2_{xy}}\|\partial_x(\eta^a,\tau^a)\|_{L^\infty_{xy}}\|\partial_t\partial_x\varTheta^\varepsilon\|_{L^2_{xy}} \\
             & + \|\nabla\partial_x U^\varepsilon\|_{L^2_{xy}}\|\partial_t(\eta^a,\tau^a)\|_{L^2_{xy}}^{\frac{1}{2}}
             \|\partial_x\partial_t(\eta^a,\tau^a)\|_{L^2_{xy}}^{\frac{1}{2}}
             \|\partial_t\partial_x\varTheta^\varepsilon\|_{L^2_{xy}}^{\frac{1}{2}}\|\partial_y\partial_t\partial_x\varTheta^\varepsilon\|_{L^2_{xy}}^{\frac{1}{2}}
              \\
             &+ \|\nabla U^\varepsilon\|_{H^1_{xy}}\|\partial_t\partial_x(\eta^a,\tau^a)\|_{L^2_{xy}}
             \|\partial_t\partial_x\varTheta^\varepsilon\|_{L^2_{xy}}^{\frac{1}{2}}\|\nabla\partial_t\partial_x\varTheta^\varepsilon\|_{L^2_{xy}}^{\frac{1}{2}}\\
             \leq{}&\frac{\mu}{32}\|\nabla\partial_t\partial_x U^\varepsilon\|_{L^2_{xy}}^2+ \frac{\varepsilon}{16}\|\nabla\partial_t\partial_x \varTheta^\varepsilon\|_{L^2_{xy}}^2
             +C\Big(\| \nabla \partial_t  U^\varepsilon\|_{L^2_{xy}}^2  + \|\partial_t\partial_x  \varTheta^\varepsilon\|_{L^2_{xy}}^2 + \varepsilon^{\frac{1}{2}}\Big).
          \end{align*}
          For $\mathcal{W}_8, \mathcal{W}_9$, we have
          \begin{align*}
            &\mathcal{W}_8 + \mathcal{W}_9\\
             \lesssim {}& \|\nabla\partial_t\partial_x U^\varepsilon\|_{L^2_{xy}}
             \Bigl(\| \varTheta^\varepsilon\|_{L^2_{xy}}^{\frac{1}{2}}\| \nabla\varTheta^\varepsilon\|_{L^2_{xy}}^{\frac{1}{2}}
             +\| H^\varepsilon\|_{L^2_{xy}}^{\frac{1}{2}}\| \nabla H^\varepsilon\|_{L^2_{xy}}^{\frac{1}{2}}\Bigr)
             \| \partial_t\partial_x\varTheta^\varepsilon\|_{L^2_{xy}}^{\frac{1}{2}}\| \nabla\partial_t\partial_x\varTheta^\varepsilon\|_{L^2_{xy}}^{\frac{1}{2}}\\
             & +\|\nabla\partial_x U^\varepsilon\|_{L^2_{xy}}
             \Bigl(\| \partial_t\varTheta^\varepsilon\|_{L^2_{xy}}^{\frac{1}{2}}\| \partial_t\partial_x\varTheta^\varepsilon\|_{L^2_{xy}}^{\frac{1}{2}}
             +\| \partial_tH^\varepsilon\|_{L^2_{xy}}^{\frac{1}{2}}\| \partial_t \partial_x H^\varepsilon\|_{L^2_{xy}}^{\frac{1}{2}}\Bigr)
             \| \partial_t\partial_x\varTheta^\varepsilon\|_{L^2_{xy}}^{\frac{1}{2}}\| \partial_y\partial_t\partial_x\varTheta^\varepsilon\|_{L^2_{xy}}^{\frac{1}{2}}\\
             & +\|\nabla\partial_t U^\varepsilon\|_{L^2_{xy}}
            \Bigl(\| \partial_x\varTheta^\varepsilon\|_{L^2_{xy}}^{\frac{1}{2}}\| \nabla\partial_x\varTheta^\varepsilon\|_{L^2_{xy}}^{\frac{1}{2}}
            +\| \partial_xH^\varepsilon\|_{L^2_{xy}}^{\frac{1}{2}}\| \nabla \partial_x H^\varepsilon\|_{L^2_{xy}}^{\frac{1}{2}}\Bigr)
            \| \partial_t\partial_x\varTheta^\varepsilon\|_{L^2_{xy}}^{\frac{1}{2}}\| \nabla\partial_t\partial_x\varTheta^\varepsilon\|_{L^2_{xy}}^{\frac{1}{2}}\\
            &+ \|\nabla U^\varepsilon\|_{L^2_{xy}}\Bigl( \|\partial_t\partial_x \varTheta^{\varepsilon}\|_{L^2_{xy}}\|\nabla\partial_t\partial_x \varTheta^{\varepsilon}\|_{L^2_{xy}}
             + \|\partial_t\partial_x H^{\varepsilon}\|_{L^2_{xy}}\|\nabla\partial_t\partial_x H^{\varepsilon}\|_{L^2_{xy}}\Bigr)  \\
            \leq{}&\frac{\mu}{32}\|\nabla\partial_t\partial_x U^\varepsilon\|_{L^2_{xy}}^2+ \frac{\varepsilon}{16}\|\nabla\partial_t\partial_x (H^\varepsilon,\varTheta^\varepsilon)\|_{L^2_{xy}}^2
            + C\Big(  \|\partial_t\partial_x  (U^\varepsilon,H^\varepsilon,\varTheta^\varepsilon)\|_{L^2_{xy}}^2
            +\|\nabla\partial_t U^\varepsilon\|_{L^2_{xy}}^2
            + \varepsilon^{\frac{1}{2}}\Big).
          \end{align*}
          Therefore, we have
    \begin{eqnarray}\label{eq_T_tx}
        \begin{aligned}
            &\frac{\mathrm{d}}{\mathrm{d}t} \|\partial_t \partial_x \varTheta^\varepsilon\|_{L^2_{xy}}^2
            + \gamma \| \partial_t\partial_x \varTheta^\varepsilon\|_{L^2_{xy}}^2 +  \varepsilon\|\nabla\partial_t\partial_x \varTheta^\varepsilon\|_{L^2_{xy}}^2   \\
            \leq {}&  \frac{\mu}{4}\|\nabla\partial_t\partial_x U^\varepsilon\|_{L^2_{xy}}^2
            +  \frac{\varepsilon}{2}\|\nabla\partial_t\partial_x H^\varepsilon\|_{L^2_{xy}}^2
            + C\Big(  \|\partial_t\partial_x  (U^\varepsilon,H^\varepsilon,\varTheta^\varepsilon)\|_{L^2_{xy}}^2
            +\|\nabla\partial_t U^\varepsilon\|_{L^2_{xy}}^2
            + \varepsilon^{\frac{1}{2}}\Big).
        \end{aligned}
    \end{eqnarray}
    Summing \eqref{eq_U_tx}, \eqref{eq_H_tx} and  \eqref{eq_T_tx} up, and using Lemma \ref{lemma_UHT_nabla_t} and Gronwall's inequality, we get
    \begin{eqnarray*}
        \begin{aligned}
            &  \|\partial_t\partial_x  (U^\varepsilon,H^\varepsilon,\varTheta^\varepsilon)\|_{L^\infty_TL^2_{xy}}^2
            +  \mu\|\nabla\partial_t\partial_x U^\varepsilon\|_{L^2_TL^2_{xy}}^2
               +  \varepsilon\|\nabla\partial_t\partial_x (H^\varepsilon,\varTheta^\varepsilon)\|_{L^2_TL^2_{xy}}^2
            \leq{}   \tilde{C}_5 T\varepsilon^{\frac{1}{2}}\exp(\tilde{C}_5 T).
        \end{aligned}
    \end{eqnarray*}
    The proof of Lemma \ref{lemma_UHT_x_t} is complete.
\end{proof}
\begin{remark}\label{remark_lemma_UHT_x_t}
    For readers' convenience, we remark  that the following bounds:
    \begin{align*}
        &     \|(u^I,w,\eta^I,\tau^I)\|_{L^\infty_TH^3_{xy}}
        +\|\partial_t(u^I,w)\|_{L^\infty_TH^2_{xy}}
        +\|\partial_t(\eta^I,\tau^I)\|_{L^\infty_TH^1_{xy}}
        +  \|(u^B,\eta^B,\tau^B)\|_{L^\infty_TH^2_xH^1_z} \leq C,\\
        & \| u^B \|_{L^\infty_TH^3_xL^2_z} +\| u^B \|_{L^\infty_TL^2_xH^2_z}
        +\| \partial_t u^B \|_{L^\infty_TH^1_xH^1_z}
        +\| \partial_t u^B \|_{L^\infty_TH^2_xL^2_z}
        +\| \partial_t (\eta^B,\tau^B) \|_{L^\infty_TH^1_xL^2_z} \leq C,
    \end{align*} \eqref{eq_source_term} and \eqref{eq_source_term_t} are used in the proof of Lemma \ref{lemma_UHT_x_t}.
\end{remark}

\begin{lemma}\label{lemma_UHT_x_H_2}
    Suppose that $(U^\varepsilon,H^\varepsilon,\varTheta^\varepsilon)$ is the solution of the problem \eqref{eq_UHT_varepsilon}-\eqref{eq_UHT_int_bo} on $[0,T]$ with $T\in (0,\min\{T_\star,T_\varepsilon\}]$. Under the assumptions of Theorem \ref{theorem_limit}, there exists a generic positive constant C such that
    \begin{eqnarray}\label{eq_UHT_x_infty_H_2}
        \| \partial_x {U^\varepsilon} \|_{L^\infty_TH^2_{xy}}^2
        \leq C \varepsilon^{\frac{1}{2}},~~
          \varepsilon^{\frac{3}{2}}  \|\partial_x(  {H^\varepsilon}, {\varTheta^\varepsilon})\|_{L^\infty_TH^2_{xy}}^2
        \leq C.
     \end{eqnarray}
\end{lemma}

\begin{proof}
   Using $\eqref{eq_UHT_varepsilon}_1$, the Stokes estimate, and Lemmas \ref{lemma_source_term_1},   \ref{lemma_HT_nabla_x} and   \ref{lemma_UHT_x_t}, we have
    \begin{align*}%\label{eq_U_x_H2}
         \|\partial_x U^{\varepsilon}\|_{L^\infty_TH_{xy}^2}^2\lesssim {}& \|\partial_t\partial_x U^\varepsilon\|_{L^\infty_TL_{xy}^2}^2
         + \Big(\|u^a\|_{L^\infty_TL^\infty_{xy}}^2+  \|\nabla u^a\|_{L^\infty_TL^4_{xy}}^2 + \| \nabla\partial_x u^a\|_{L^\infty_TL^2_{xy}}^2 +1\Big)\|  U^\varepsilon\|_{L^\infty_TH_{xy}^2}^2
         \notag\\
        &+ \varepsilon\|  U^\varepsilon\|_{L^\infty_TH_{xy}^2}^4
        + \| \nabla\partial_x \varTheta^\varepsilon\|_{L^\infty_TL_{xy}^2}^2 + \|\partial_x \mathfrak{F} \|_{L^\infty_TL_{xy}^2}^2\\
        \leq{}&  C\varepsilon^{\frac{1}{2}}.\notag
    \end{align*}
From  $\eqref{eq_UHT_varepsilon}_2$ and the standard elliptic estimate, we have
    \begin{align*}%\label{eq_H_x_H2}
        \varepsilon^2\|\partial_x H^{\varepsilon}\|_{L^\infty_TH_{xy}^2}^2\lesssim {}& \|\partial_t\partial_x H^\varepsilon\|_{L^\infty_TL_{xy}^2}^2
        + \|\partial_xu^a\|_{L^\infty_TL^\infty_{xy}}^2\|\nabla  H^\varepsilon\|_{L^\infty_TL_{xy}^2}^2
        + \|u^a\|_{L^\infty_TL^\infty_{xy}}^2\|\nabla \partial_x H^\varepsilon\|_{L^\infty_TL_{xy}^2}^2
        \\
        &+ \|  U^\varepsilon\|_{L^\infty_TH_{xy}^2}^2\Bigl(\|\nabla\eta^a\|_{L^\infty_TL^4_{xy}}^2+ \|\nabla\partial_x\eta^a\|_{L^\infty_TL^2_{xy}}^2  \Bigr)
        +  \varepsilon  \|  U^\varepsilon\|_{L^\infty_TH_{xy}^2}^2 \|  H^\varepsilon\|_{L^\infty_TH_{xy}^2}^2\\
        & +\|\partial_x \mathfrak{G} \|_{L^\infty_TL_{xy}^2}^2 + \|H^\varepsilon\|_{L^\infty_TL_{xy}^2}^2   \\
        \leq{} & C\varepsilon^{\frac{1}{2}}.
    \end{align*}
Similarly, from $\eqref{eq_UHT_varepsilon}_3$, we have
\begin{align*}%\label{eq_T_x_H2}
    \varepsilon^2\|\partial_x \varTheta^\varepsilon\|_{L^\infty_TH_{xy}^2}^2\lesssim {}& \|\partial_t\partial_x \varTheta^\varepsilon\|_{L^\infty_TL_{xy}^2}^2
    + \|\partial_xu^a\|_{L^\infty_TL^\infty_{xy}}^2\|\nabla  \varTheta^\varepsilon\|_{L^\infty_TL_{xy}^2}^2
    + \|u^a\|_{L^\infty_TL^\infty_{xy}}^2\|\nabla \partial_x \varTheta^\varepsilon\|_{L^\infty_TL_{xy}^2}^2
    \notag\\
    &+ \|  U^\varepsilon\|_{L^\infty_TH_{xy}^2}^2\Bigl( \|\nabla\tau^a\|_{L^\infty_TL^4_{xy}}^2+ \|\nabla\partial_x\tau^a\|_{L^\infty_TL^2_{xy}}^2
    +\|  (\eta^a,\tau^a)\|_{L^\infty_TL^\infty_{xy}}^2 \notag\\
    & +\| \partial_x(\eta^a,\tau^a)\|_{L^\infty_TL^4_{xy}}^2 \Bigr)
    + \Bigl(\|\nabla\partial_x u^a\|_{L^\infty_TL^4_{xy}}^2 + \|\nabla  u^a\|_{L^\infty_TL^\infty_{xy}}^2 + 1 \Bigr)
    \\
    &\times \| (H^\varepsilon, \varTheta^\varepsilon)\|_{L^\infty_TH_{xy}^1}^2
     +  \varepsilon  \|  U^\varepsilon\|_{L^\infty_TH_{xy}^2}^2 \| (H^\varepsilon, \varTheta^\varepsilon)\|_{L^\infty_TH_{xy}^2}^2 +\|\partial_x \mathfrak{J} \|_{L^\infty_TL_{xy}^2}^2 \notag\\
    \leq{} & C\varepsilon^{\frac{1}{2}}.\notag
\end{align*}
The proof of Lemma \ref{lemma_UHT_x_H_2} is complete.
\end{proof}

\begin{remark}\label{remark_lemma_UHT_x_H_2}
    For readers' convenience, we remark that the following bounds:
    \begin{align*}
        &   \| u^B \|_{L^\infty_TH^3_xL^2_z} +\| u^B \|_{L^\infty_TH^2_xH^1_z}  +\| u^B \|_{L^\infty_TH^1_xH^2_z}
        +  \|( \eta^B,\tau^B)\|_{L^\infty_TH^2_{xz}}
        \leq C,\\
        &   \|(u^I,w)\|_{L^\infty_TH^3_{xy}} + \|(eta^I,\tau^I )\|_{L^\infty_TH^2_{xy}}
         \leq C,
    \end{align*} \eqref{eq_source_term} and \eqref{eq_source_term_t} are used in the proof of Lemma \ref{lemma_UHT_x_H_2}.
\end{remark}

\subsection{Proof of Theorem \ref{theorem_limit}}

To begin with, using Proposition \ref{local_wellposedness_0}, Lemmas
\ref{lemma_eta_tau_b_1}, \ref{local_wellposedness_outer_2},
 \ref{lemma_eta_tau_b_2}-\ref{lemma_eta_tau_b_3}
  and Corollary \ref{corollary_error},
we obtain that the problem \eqref{OB_main}-\eqref{OB_main_boundary} admits a unique solution $(u^\varepsilon,\eta^\varepsilon,\tau^\varepsilon) \in C([0,T_{\star}];H^2_{xy})$ with $\varepsilon^{\frac{1}{4}}\partial_x(u^\varepsilon,\eta^\varepsilon,\tau^\varepsilon) \in C([0,T_{\star}];H^2_{xy}).$

 Using \eqref{eq_u_infty},  Lemmas \ref{lemma_UHT_infty_L2}, \ref{lemma_UHT_nabla_t} and  \ref{lemma_HT_nabla_x}, we have
\begin{eqnarray}\label{eq_U_infty}
    \begin{aligned}
        \|U^\varepsilon\|_{L^\infty_TL^\infty_{xy}} \lesssim{}& \|\partial_xU^\varepsilon\|_{L^\infty_TL_{xy}^2}^{\frac{1}{2}}\|\partial_yU^\varepsilon\|_{L^\infty_TL_{xy}^2}^{\frac{1}{2}} + \|U^\varepsilon\|_{L^\infty_TL_{xy}^2}^{\frac{1}{2}}\|\partial_x\partial_yU^\varepsilon\|_{L^\infty_TL_{xy}^2}^{\frac{1}{2}}\\
        \lesssim{}&\varepsilon^{\frac{3}{4}}  + \varepsilon^{\frac{1}{2}}\times \varepsilon^{\frac{3}{4}\times\frac{1}{2}}\leq C \varepsilon^{\frac{3}{4}},
    \end{aligned}
\end{eqnarray}
    \begin{eqnarray}\label{eq_H_infty}
        \begin{aligned}
        \|H^\varepsilon \|_{L^\infty_TL^\infty_{xy}}\lesssim {}& \|\partial_xH^\varepsilon\|_{L^\infty_TL_{xy}^2}^{\frac{1}{2}}\|\partial_yH^\varepsilon\|_{L^\infty_TL_{xy}^2}^{\frac{1}{2}} + \|H^\varepsilon\|_{L^\infty_TL_{xy}^2}^{\frac{1}{2}}\|\partial_x\partial_yH^\varepsilon\|_{L^\infty_TL_{xy}^2}^{\frac{1}{2}}\\
        \lesssim{} & \varepsilon^{\frac{3}{4}} +\varepsilon^\frac{1}{2}\times\varepsilon^{\frac{1}{4}\times\frac{1}{2}} \leq C\varepsilon^{\frac{5}{8}},
    \end{aligned}
    \end{eqnarray}
    and
    \begin{eqnarray} \label{eq_T_infty}
        \begin{aligned}
        \|\varTheta^\varepsilon \|_{L^\infty_TL^\infty_{xy}}\lesssim {} & \|\partial_x\varTheta^\varepsilon\|_{L^\infty_TL_{xy}^2}^{\frac{1}{2}}\|\partial_y\varTheta^\varepsilon\|_{L^\infty_TL_{xy}^2}^{\frac{1}{2}} + \|\varTheta^\varepsilon\|_{L^\infty_TL_{xy}^2}^{\frac{1}{2}}\|\partial_x\partial_y \varTheta^\varepsilon\|_{L^\infty_TL_{xy}^2}^{\frac{1}{2}}\\
        \lesssim {}&  \varepsilon^{\frac{3}{4}} +\varepsilon^\frac{1}{2}\times\varepsilon^{\frac{1}{4}\times\frac{1}{2}} \leq C\varepsilon^{\frac{5}{8}}.
    \end{aligned}
\end{eqnarray}

Therefore, from the definition of $ U^\varepsilon$, \eqref{eq_U_infty} and Sobolev   inequalities, we get
 \begin{align*}
    &\|u(x,y,t) - u^{I,0}(x,y,t) \|_{L^\infty_TL^\infty_{xy}} \\
      \lesssim{} & \varepsilon^{\frac{1}{2}}\|U^\varepsilon\|_{L^\infty_TL^\infty_{xy}}  + \varepsilon \Bigl(\|u^{I,2}\|_{L^\infty_TH^2_{xy}} + \|u_1^{B,2}\|_{L^\infty_TH^2_{xz}}\Bigr)
       + \varepsilon^{\frac{3}{2}}\Bigl(\|u^{I,3}\|_{L^\infty_TH^2_{xy}}  + \|u^{B,3}\|_{L^\infty_TH^2_{xz}} \Bigr)\\
      &+ \varepsilon^2\|u^{B,4}\|_{L^\infty_TH^2_{xz}}+ \varepsilon^{\frac{5}{2}}\|u_2^{B,5}\|_{L^\infty_TH^2_{xz}} + \varepsilon^2\|w\|_{L^\infty_TH^2_{xy}}
      \leq   C\varepsilon,
    \end{align*}
    which yields \eqref{eq_u_convergence_rate}. Similarly, using the definition of $ H^\varepsilon$ and \eqref{eq_H_infty}, we obtain
    \begin{align*}
            &\|\eta(x,y,t)   - \eta^{I,0}(x,y,t)   \|_{L^\infty_TL^\infty_{xy}} \\
            \lesssim {} &  \varepsilon^{\frac{1}{2}}\|H^\varepsilon\|_{L^\infty_TL^\infty_{xy}}
            +\varepsilon^{\frac{1}{2}}\|\eta^{B,1}\|_{L^\infty_TH^2_{xz}}
            + \varepsilon \Bigl(\|\eta^{I,2}\|_{L^\infty_TH^2_{xy}} + \|\eta^{B,2}\|_{L^\infty_TH^2_{xz}}\Bigr) + \varepsilon^{\frac{3}{2}} \|\eta^{B,3}\|_{L^\infty_TH^2_{xz}}
            \leq  C\varepsilon^{\frac{1}{2}},
        \end{align*}
        which leads to \eqref{eq_eta_convergence_rate}.

        Next, by the definition of $ \varTheta^\varepsilon$ and  \eqref{eq_T_infty}, we get
        \begin{align*}
            &\|\tau(x,y,t) - \tau^{I,0}(x,y,t) \|_{L^\infty_TL^\infty_{xy}}  \\
            \lesssim {} & \varepsilon^{\frac{1}{2}}\|\varTheta^\varepsilon\|_{L^\infty_TL^\infty_{xy}}
            +\varepsilon^{\frac{1}{2}}\|\tau^{B,1}\|_{L^\infty_TH^2_{xz}}
            + \varepsilon \Bigl(\|\tau^{I,2}\|_{L^\infty_TH^2_{xy}} + \|\tau^{B,2}\|_{L^\infty_TH^2_{xz}}\Bigr) + \varepsilon^{\frac{3}{2}} \|\tau^{B,3}\|_{L^\infty_TH^2_{xz}}
            \leq  C\varepsilon^{\frac{1}{2}},
\end{align*}
which leads to \eqref{eq_tau_convergence_rate}.

Using \eqref{eq_u_infty} and Lemmas  \ref{lemma_HT_nabla_x} and \ref{lemma_UHT_x_H_2}, we have
\begin{eqnarray}\label{eq_U_y_infty}
    \begin{aligned}
        \|\partial_y U^\varepsilon \|_{L^\infty_TL^\infty_{xy}}
        \lesssim {}& \|\partial_x\partial_y U^\varepsilon\|_{L^\infty_TL_{xy}^2}^{\frac{1}{2}}\|\partial_y^2 U^\varepsilon\|_{L^\infty_TL_{xy}^2}^{\frac{1}{2}} + \|\partial_y U^\varepsilon\|_{L^\infty_TL_{xy}^2}^{\frac{1}{2}}\|\partial_x\partial_y^2 U^\varepsilon\|_{L^\infty_TL_{xy}^2}^{\frac{1}{2}}\\
        \lesssim {}&  \|  U^\varepsilon\|_{L^\infty_TH_{xy}^2}  + \|\nabla U^\varepsilon\|_{L^\infty_TL_{xy}^2}^{\frac{1}{2}}\|\partial_x  U^\varepsilon\|_{L^\infty_TH_{xy}^2}^{\frac{1}{2}}\\
        \lesssim {}& \varepsilon^{\frac{3}{4}} + \varepsilon^{\frac{3}{4}\times\frac{1}{2}}\times\varepsilon^{\frac{1}{4}\times\frac{1}{2}}
         \leq C\varepsilon^{  \frac{1}{2}}.
    \end{aligned}
    \end{eqnarray}
    Similarly, we have
\begin{eqnarray}\label{eq_H_y_infty}
    \begin{aligned}
        \|\partial_y H^\varepsilon \|_{L^\infty_TL^\infty_{xy}}
        \lesssim {}& \|\partial_x\partial_y H^\varepsilon\|_{L^\infty_TL_{xy}^2}^{\frac{1}{2}}\|\partial_y^2 H^\varepsilon\|_{L^\infty_TL_{xy}^2}^{\frac{1}{2}} + \|\partial_y H^\varepsilon\|_{L^\infty_TL_{xy}^2}^{\frac{1}{2}}\|\partial_x\partial_y^2 H^\varepsilon\|_{L^\infty_TL_{xy}^2}^{\frac{1}{2}}\\
        \lesssim {}& \varepsilon^{\frac{1}{4}\times\frac{1}{2}}\times\varepsilon^{-\frac{1}{4}\times\frac{1}{2}} + \varepsilon^{\frac{3}{4}\times\frac{1}{2}}\times\varepsilon^{-\frac{3}{4}\times\frac{1}{2}} \leq C,
    \end{aligned}
    \end{eqnarray}
    and
    \begin{eqnarray}\label{eq_T_y_infty}
        \begin{aligned}
            \|\partial_y \varTheta^\varepsilon \|_{L^\infty_TL^\infty_{xy}}
            \lesssim{} & \|\partial_x\partial_y \varTheta^\varepsilon\|_{L^\infty_TL_{xy}^2}^{\frac{1}{2}}\|\partial_y^2 \varTheta^\varepsilon\|_{L^\infty_TL_{xy}^2}^{\frac{1}{2}} + \|\partial_y \varTheta^\varepsilon\|_{L^\infty_TL_{xy}^2}^{\frac{1}{2}}\|\partial_x\partial_y^2 \varTheta^\varepsilon\|_{L^\infty_TL_{xy}^2}^{\frac{1}{2}}\\
            \lesssim {}& \varepsilon^{\frac{1}{4}\times\frac{1}{2}}\times\varepsilon^{-\frac{1}{4}\times\frac{1}{2}} + \varepsilon^{\frac{3}{4}\times\frac{1}{2}}\times\varepsilon^{-\frac{3}{4}\times\frac{1}{2}} \leq C.
        \end{aligned}
        \end{eqnarray}
        Hence, from the definition of $ U^\varepsilon$, \eqref{eq_U_y_infty} and Sobolev   inequalities, we get
        \begin{align*}
           &\|\partial_yu(x,y,t) -\partial_y u^{I,0}(x,y,t) \|_{L^\infty_TL^\infty_{xy}} \\
             \lesssim {}& \varepsilon^{\frac{1}{2}}\|\partial_y U^\varepsilon\|_{L^\infty_TL^\infty_{xy}}
             +\varepsilon^{\frac{1}{2}} \|\partial_zu_1^{B,2}\|_{  L^\infty_TH^1_{x}H^1_{z}}+ \varepsilon\|\partial_y u^{I,2}\|_{  L^\infty_TH^2_{xy}} + \varepsilon \|\partial_zu^{B,3}\|_{  L^\infty_TH^1_{x}H^1_{z}}\\
             &+\varepsilon^{\frac{3}{2}}\|\partial_y u^{I,3}\|_{L^\infty_T H^2_{xy}} +  \varepsilon^\frac{3}{2}\|\partial_zu^{B,4}\|_{  L^\infty_TH^1_{x}H^1_{z}}
             +\varepsilon^2\|\partial_z u_2^{B,5}\|_{  L^\infty_TH^1_{x}H^1_{z}}  + \varepsilon^2\|\partial_y w\|_{L^\infty_TH^2_{xy}}\\
              \leq{}&  C\varepsilon^{\frac{1}{2}},
           \end{align*}
           which yields \eqref{eq_u_y_convergence_rate}. Similarly, using the definition of $ H^\varepsilon$ and \eqref{eq_H_infty}, we obtain
           \begin{align*}
                   &\|\partial_y\eta(x,y,t)   - \partial_y\eta^{I,0}(x,y,t) -  \partial_z\eta^{B,1}(x,\frac{y}{\sqrt{\varepsilon}},t) \|_{L^\infty_TL^\infty_{xy}} \\
                   \lesssim{} & \varepsilon^{\frac{1}{2}}\|\partial_yH^\varepsilon\|_{L^\infty_TL^\infty_{xy}}
                   + \varepsilon^{\frac{1}{2}} \|\partial_z\eta^{B,2}\|_{L^\infty_TH^1_{x}H^1_{z}}
                   +\varepsilon\|\partial_y \eta^{I,2}\|_{L^\infty_TH^2_{xy}} + \varepsilon  \|\partial_z\eta^{B,3}\|_{  L^\infty_TH^1_{x}H^1_{z}}
                   \leq{}  C\varepsilon^{\frac{1}{2}},
               \end{align*}
               which leads to \eqref{eq_eta_y_convergence_rate}.

                Next, by the definition of $ \varTheta^\varepsilon$ and  \eqref{eq_T_infty}, we get
               \begin{align*}
                &\|\partial_y\tau(x,y,t) - \partial_y\tau^{I,0}(x,y,t)-  \partial_z \tau^{B,1}(x,\frac{y}{\sqrt{\varepsilon}},t)\|_{L^\infty_TL^\infty_{xy}}  \\
                \lesssim {}& \varepsilon^{\frac{1}{2}}\|\partial_y\varTheta^\varepsilon\|_{L^\infty_TL^\infty_{xy}}
                + \varepsilon^{\frac{1}{2}} \|\partial_z\tau^{B,2}\|_{ L^\infty_TH^1_{x}H^1_{z}}
                +\varepsilon\|\partial_y \tau^{I,2}\|_{L^\infty_TH^2_{xy}}+ \varepsilon  \|\partial_z\tau^{B,3}\|_{  L^\infty_TH^1_{x}H^1_{z}}
                \leq  C\varepsilon^{\frac{1}{2}},
            \end{align*}
            which leads to \eqref{eq_tau_y_convergence_rate}. The proof of Theorem \ref{theorem_limit} is complete.

\section*{Appendix}

    \appendix

    \section{Derivation of boundary layer systems}\label{section_derivation_eqs}
    In this section, we give a formal derivation of the inner and outer layer profiles via the asymptotic matched expansion method (see for instance Chapter 4 of \cite{Holmes_2013} or \cite{Hou_Zhao_Wang_2016}). \\[2mm]
    {\bfseries Step 1. Initial and boundary conditions.}\\[3mm]
     Substituting \eqref{eq_approx} into the initial conditions \eqref{OB_main_initial}, we have
    \begin{equation}\label{eq_initial}
    \begin{split}
    (u^{I,0},\eta^{I,0},\tau^{I,0})(x,y,0) &=( u_0,\eta_0,\tau_0)(x,y),~~~ (u^{B,0},\eta^{B,0},\tau^{B,0})(x,y,0) = 0,\\[2mm]
    (u^{I,j},\eta^{I,j},\tau^{I,j})(x,y,0)& = (u^{B,j},\eta^{B,j},\tau^{B,j})(x,y,0) = 0,~~~~\text{     for all }j\geq 1.
    \end{split}
    \end{equation}
 Next,  substituting \eqref{eq_approx} into \eqref{OB_main_boundary},   we obtain
    \begin{equation*}
        \begin{split}
        &\sum_{j=0}^{\infty}\varepsilon^{j/2}\Big(u^{I,j}(x,0,t)+u^{B,j}(x,0,t)\Big) = 0,\\
        &\varepsilon^{-1/2}\partial_z\eta^{B,0}(x,0,t)+ \sum_{j=0}^{\infty}\varepsilon^{j/2}\Big(\partial_y\eta^{I,j}(x,0,t)+\partial_z\eta^{B,j+1}(x,0,t)\Big) = 0,\\
        &\varepsilon^{-1/2}\partial_z\tau^{B,0}(x,0,t)+ \sum_{j=0}^{\infty}\varepsilon^{j/2}\Big(\partial_y\tau^{I,j}(x,0,t)+\partial_z\tau^{B,j+1}(x,0,t)\Big) = 0.
        \end{split}
    \end{equation*}
    To fulfill the above conditions, it is required that
    \begin{eqnarray}\label{eq_boundary}
        \begin{aligned}
            &  u^{I,j}(x,0,t)+u^{B,j}(x,0,t)  = 0,\, \partial_z\eta^{B,0}(x,0,t) = 0,\, \partial_z\tau^{B,0}(x,0,t) = 0,\\[2mm]
             \partial_y&\eta^{I,j}(x,0,t)+\partial_z\eta^{B,j+1}(x,0,t)  = 0,\,\partial_y\tau^{I,j}(x,0,t)+\partial_z\tau^{B,j+1}(x,0,t)  = 0,
             \end{aligned}
    \end{eqnarray}
    for all $j\geq 0.$ \\[3mm]
    \noindent{\bfseries Step 2. Equations of $(u^{I,0},\tilde{p}^{I,0},\eta^{I,0},\tau^{I,0})$ and $(u^{B,0},\tilde{p}^{B,0},\eta^{B,0},\tau^{B,0})$.}\\[3mm]
    To begin with, the zeroth order outer layer profiles $(u^{I,0},\tilde{p}^{I,0},\eta^{I,0},\tau^{I,0})$ satisfies the limit problem \eqref{OB_limit_I}-\eqref{OB_limit_I_initial_boundary1}.
    In fact, away from the boundary, we set the  approximate  solutions in the form that
 \begin{eqnarray}\label{eq_approx_outer}
     \begin{aligned}
         u(x,y,t) ={} &\sum_{j=0}^{\infty}\varepsilon^{\frac{1}{2}j}u^{I,j}(x,y,t),~\tilde{p}(x,y,t) = \sum_{j=0}^{\infty}\varepsilon^{\frac{1}{2}j}\tilde{p}^{I,j}(x,y,t),\\
        \eta(x,y,t) ={} &\sum_{j=0}^{\infty}\varepsilon^{\frac{1}{2}j}\eta^{I,j}(x,y,t),~\tau(x,y,t) = \sum_{j=0}^{\infty}\varepsilon^{\frac{1}{2}j}\tau^{I,j}(x,y,t).
     \end{aligned}
 \end{eqnarray}
 Substituting \eqref{eq_approx_outer} into   system \eqref{OB_main}, and matching the $j$-th order terms, we find that $(u^{I,j},\tilde{p}^{I,j},\eta^{I,j},\tau^{I,j})$ satisfies the following
 system:
   \begin{equation}\label{eq_outer_layer_j}
    \begin{cases}
    \partial_tu^{I,j} +  \sum\limits_{\ell=0}^{j}u^{I,\ell}\cdot\nabla u^{I,j-\ell} +\nabla \tilde{p}^{I,j} -\mu \Delta u^{I,j}
    -\mathrm{div} \tau^{I,j} = 0,\\[2mm]
    \partial_t\eta^{I,j}+  \sum\limits_{\ell=0}^{j} u^{I,\ell}\cdot\nabla\eta^{I,j-\ell}  = \Delta \eta ^{I,j-2},\\[2mm]
    \partial_t\tau^{I,j}+ \sum\limits_{\ell=0}^{j} u^{I,\ell}\cdot \nabla \tau^{I,j-\ell}  - \sum\limits_{\ell=0}^{j}\mathcal{Q}(\nabla u^{I,\ell} ,\tau^{I,j-\ell})
    - \sum\limits_{\ell=0}^{j}\mathcal{B}(\eta^{I,\ell},\nabla u^{I,j-\ell})   + \gamma\tau^{I,j} = \Delta \tau ^{I,j-2},\\[2mm]
    \mathrm{div} u^{I,j} = 0,
    \end{cases}
\end{equation}for $j\geq 0,$ where $\eta^{I,-1} =\eta^{I,-2} = 0$ and $\tau^{I,-1} =\tau^{I,-2} = 0$.

Letting $j = 0$ in \eqref{eq_outer_layer_j}, and combining with \eqref{eq_initial}, \eqref{eq_boundary}, we find that
$(u^{I,0},\tilde{p}^{I,0},\eta^{I,0},\tau^{I,0})$ satisfies the limit problem \eqref{OB_limit_I}-\eqref{OB_limit_I_initial_boundary1}.
\begin{lemma}\label{lemma_eq_u_eta_tau_B_0}
    The zeroth order inner profile $(u^{B,0},\tilde{p}^{B,0},\eta^{B,0},\tau^{B,0})$ vanishes identically, i.e.,
    \begin{equation}\label{eq_u_eta_tau_B_0}
        \begin{split}
            &u^{B,0} = 0,\,   \eta^{B,0} = 0,\, \tau^{B,0} = 0,\, \tilde{p}^{B,0}=0.
        \end{split}
    \end{equation}
 \end{lemma}
 \begin{proof}
    Near the boundary, plugging \eqref{eq_approx} into \eqref{OB_main}$_1$,
 subtracting $\eqref{eq_outer_layer_j}_1$ from the resulting equation, and then
 using $y = \sqrt{\varepsilon}z$
    and Taylor's expansion:
    \begin{eqnarray*}
        f(x,y,t) = f(x,\sqrt{\varepsilon}z,t) = \sum_{l=0}^{\infty}\frac{1}{\ell!}(\sqrt{\varepsilon}z)^\ell\partial_y^\ell f(x,0,t),
    \end{eqnarray*}
    for the outer layer profiles $(u^{I,j},\tilde{p}^{I,j},\eta^{I,j},\tau^{I,j})$, we obtain
    \begin{equation*}
        \sum_{j=-2}^\infty \varepsilon^{j/2} F^j(x,z,t) = 0,
    \end{equation*}
    where
    \begin{equation*}
        \begin{cases}
            F^{-2} =  -\mu \partial_z^2 u^{B,0},\\[2mm]
            F^{-1}=   (u_2^{I,0}(x,0,t) + u_2^{B,0})\partial_z u^{B,0} + (0,\partial_z\tilde{p}^{B,0})^\top -\mu \partial_z^2u^{B,1} - (\partial_z \tau_{12}^{B,0},\partial_z \tau_{22}^{B,0})^\top,\\[2mm]
            F^{0}~~=   \partial_t u^{B,0} + u^{B,0}\cdot \nabla u^{I,0}(x,0,t) + (u_1^{I,0}(x,0,t) + u^{B,0}_1)\partial_xu^{B,0} +  (u_2^{I,0}(x,0,t) + u^{B,0}_2)\partial_zu^{B,1}\\
            ~~~~~~~~+(u_2^{I,1}(x,0,t) + u^{B,1}_2)\partial_zu^{B,0} + (\partial_x\tilde{p}^{B,0},\partial_z\tilde{p}^{B,1})^\top - \mu\partial_x^2 u^{B,0} -\mu\partial_z^2u^{B,2}\\
            ~~~~~~~~-(\partial_x\tau_{11}^{B,0}+\partial_z\tau_{12}^{B,1},\partial_x\tau_{21}^{B,0}+\partial_z\tau_{22}^{B,1})^\top + z\partial_yu^{I,0}_2(x,0,t)\partial_zu^{B,0},\\
            \cdots.
        \end{cases}
    \end{equation*}
    Moreover, divergence-free of $u$ yields that
    \begin{equation}\label{eq_divfree_B}
        \partial_x u_1^{B,j} + \partial_z u_2^{B,j+1} = 0,
    \end{equation} for all $j\geq 0$.

     $F^{-2} = 0$ and \eqref{eq_divfree_B} yield
    \begin{eqnarray}\label{eq_u_B_00}
         u^{B,0} = 0,~u_2^{B,1} = 0.
    \end{eqnarray}
      Thus, by the boundary condition \eqref{eq_boundary}, we deduce
      $$u^{I,0}(x,0,t) = 0,u_2^{I,1}(x,0,t) = 0.$$
    By virtue of $F^{-1} = 0,$ we have
    \begin{equation}\label{eq_u_B_1}
        \begin{cases}
            \mu\partial_zu_1^{B,1} + \tau_{12}^{B,0} = 0,\\[2mm]
            \tilde{p}^{B,0} -  \mu\partial_zu_2^{B,1} -\tau_{22}^{B,0} =0.
        \end{cases}
    \end{equation}
    A similar process for $\eta$ yields
    \begin{equation*}
        \sum_{j=-1}^\infty \varepsilon^{j/2} G^j(x,z,t) = 0,
    \end{equation*}
    where
    \begin{equation*}
        \begin{cases}
            G^{-1} =  (u_2^{I,0}(x,0,t) + u^{B,0}_2)\partial_z\eta^{B,0},\\[2mm]
            G^{0}~~ =  \partial_t\eta^{B,0} + u^{B,0}\cdot\nabla\eta^{I,0}(x,0,t) + (u_1^{I,0}(x,0,t) + u^{B,0}_1)\partial_x\eta^{B,0} +(u_2^{I,0}(x,0,t) + u^{B,0}_2)\partial_z\eta^{B,1}\\[2mm]
            ~~~~~~~~~  +(u_2^{I,1}(x,0,t) + u^{B,1}_2)\partial_z\eta^{B,0} -\partial_z^2\eta^{B,0} + z\partial_yu^{I,0}_2(x,0,t)\partial_z\eta^{B,0},\\
            \cdots.
        \end{cases}
    \end{equation*}
 From $G^{0} = 0$, \eqref{eq_u_B_00}, $ \mathrm{div} u^{I,0} = 0, u^{I,0}(x,0,t) = 0,$ the initial condition \eqref{eq_initial} and the boundary condition \eqref{eq_boundary}, we obtain
 \begin{equation}\label{eq_eta_B_0}
    \begin{cases}
        \partial_t \eta^{B,0}-\partial_z^2\eta^{B,0} = 0,\\[2mm]
        \eta^{B,0}(x,z,0) = 0,\partial_z\eta^{B,0}(x,0,t) = 0,
    \end{cases}
 \end{equation}
 which, by the uniqueness of  heat equations, yields
 \begin{equation}
 \label{etab0}
 \eta^{B,0} = 0.
 \end{equation}
 By a similar process for $\tau$, we get
    \begin{equation*}
        \sum_{j=-1}^\infty \varepsilon^{j/2} K^j(x,z,t) = 0,
    \end{equation*}
    where
    \begin{eqnarray*}
        K^{j} =
        \left(\begin{matrix}
             K^{j}_{11} & K^{j}_{12}\\
             K^{j}_{12} & K^{j}_{22}
        \end{matrix}
        \right),
    \end{eqnarray*}
      are $2\times 2$ symmetric matrix for $j=-1,0,1,2,\cdot\cdot\cdot$.

      Noticing that
    \begin{align*}
           K^{-1}_{11}=& (u_2^{I,0}(x,0,t) + u_2^{B,0})\partial_z \tau^{B,0}_{11} - 2(\tau_{12}^{I,0}(x,0,t) + \tau_{12}^{B,0}) \partial_{z}u_1^{B,0},\\[2mm]
           K^{-1}_{12}=&  (u_2^{I,0}(x,0,t) + u_2^{B,0})\partial_z \tau^{B,0}_{12} -(\tau_{22}^{I,0}(x,0,t)
           + \tau_{22}^{B,0})\partial_zu_1^{B,0} \\
           &-(\tau_{12}^{I,0}(x,0,t) + \tau_{12}^{B,0})\partial_zu_2^{B,0}
           -k(\eta^{I,0}(x,0,t) + \eta^{B,0})\partial_zu_1^{B,0},\\[2mm]
           K^{-1}_{22}=& (u_2^{I,0}(x,0,t) + u_2^{B,0})\partial_z \tau^{B,0}_{22} -2(\tau_{22}^{I,0}(x,0,t) + \tau_{22}^{B,0}) \partial_{z}u_2^{B,0}\\
           & -2k(\eta^{I,0}(x,0,t) + \eta^{B,0})\partial_zu_2^{B,0}.
    \end{align*}
    and recalling (\ref{eq_u_eta_tau_B_0}) that $ u^{B,0} = 0,  \eta^{B,0} = 0,$ we find $K^{-1} = 0$.

     For $K^{0}$ we have
        \begin{align*}
            K^0_{11} = ~& \partial_t\tau_{11}^{B,0} + \lambda\tau_{11}^{B,0}-\partial_z^2\tau_{11}^{B,0}+ u^{B,0}\cdot\nabla\tau_{11}^{I,0}(x,0,t) +  [u_1^{I,0}(x,0,t) + u^{B,0}_1]\partial_x\tau_{11}^{B,0} \\
        & +[ u_2^{I,0}(x,0,t) + u^{B,0}_2] \partial_z\tau_{11}^{B,1}+[ u_2^{I,1}(x,0,t) + u^{B,1}_2]\partial_z\tau_{11}^{B,0}  + z\partial_yu^{I,0}_2(x,0,t)\partial_z\tau_{11}^{B,0}\\
        &  -2  [\tau_{11}^{I,0}(x,0,t) + \tau_{11}^{B,0}]\partial_xu_1^{B,0} -2  \tau_{11}^{B,0}\partial_x u_1^{I,0}(x,0,t)  -2 \tau_{12}^{B,0}\partial_y  u_1^{I,0}(x,0,t) \\
        &- 2  [\tau_{12}^{I,0}(x,0,t) + \tau_{12}^{B,0}] \partial_zu_1^{B,1} -2  [\tau_{12}^{I,1}(x,0,t) + \tau_{12}^{B,1}]\partial_zu_1^{B,0}-2z\partial_y\tau_{12}^{I,0}(x,0,t)\partial_z u_1^{B,0} \\
        &- 2k [\eta^{I,0}(x,0,t) + \eta^{B,0}] \partial_x u_1^{B,0}-2k\eta^{B,0}\partial_x u_1^{I,0}(x,0,t),\\[2mm]
            K^0_{12} = ~& \partial_t\tau_{12}^{B,0} + \lambda\tau_{12}^{B,0}-\partial_z^2\tau_{12}^{B,0}+ u^{B,0}\cdot\nabla\tau_{12}^{I,0}(x,0,t) +  [u_1^{I,0}(x,0,t) + u^{B,0}_1 ]\partial_x\tau_{12}^{B,0} \\
        & +[ u_2^{I,0}(x,0,t) + u^{B,0}_2]\partial_z\tau_{12}^{B,1}+ [ u_2^{I,1}(x,0,t) + u^{B,1}_2 ]\partial_z\tau_{12}^{B,0}  + z\partial_yu^{I,0}_2(x,0,t)\partial_z\tau_{12}^{B,0}\\
        &  -  [\tau_{11}^{I,0}(x,0,t) + \tau_{11}^{B,0}]\partial_xu_2^{B,0} - \tau_{11}^{B,0}\partial_x u_2^{I,0}(x,0,t)- \tau_{22}^{B,0}\partial_yu_1^{I,0}(x,0,t) \\
        &- [\tau_{22}^{I,0}(x,0,t) + \tau_{22}^{B,0}]\partial_zu_1^{B,1} - [ \tau_{22}^{I,1}(x,0,t) + \tau_{22}^{B,1}]\partial_zu_1^{B,0}-z\partial_y\tau_{22}^{I,0}(x,0,t)\partial_z u_1^{B,0}\\
        &  - k[ \eta^{I,0}(x,0,t) + \eta^{B,0}]\partial_x u_2^{B,0} -k\eta^{B,0}[ \partial_x u_2^{I,0}(x,0,t) + \partial_y u_1^{I,0}(x,0,t)]\\
        &-k[ \eta^{I,0}(x,0,t) + \eta^{B,0}]\partial_z u_1^{B,1} -k[ \eta^{I,1}(x,0,t) + \eta^{B,1}]\partial_z u_1^{B,0} -kz\partial_y\eta^{I,0}(x,0,t)\partial_z u_1^{B,0},
        \end{align*}
    and
        \begin{align*}
            K^0_{22} = ~& \partial_t\tau_{22}^{B,0} + \lambda\tau_{22}^{B,0}-\partial_z^2\tau_{22}^{B,0}+ u^{B,0}\cdot\nabla\tau_{22}^{I,0}(x,0,t) +  [u_1^{I,0}(x,0,t) + u^{B,0}_1]\partial_x\tau_{22}^{B,0} \\
        & +[ u_2^{I,0}(x,0,t) + u^{B,0}_2] \partial_z\tau_{22}^{B,1}+[ u_2^{I,1}(x,0,t) + u^{B,1}_2]\partial_z\tau_{22}^{B,0}  + z\partial_yu^{I,0}_2(x,0,t)\partial_z\tau_{22}^{B,0}\\
        &  -2  [\tau_{12}^{I,0}(x,0,t) + \tau_{12}^{B,0}]\partial_xu_2^{B,0} -2  \tau_{12}^{B,0}\partial_x u_2^{I,0}(x,0,t)  -2 \tau_{22}^{B,0}\partial_y  u_2^{I,0}(x,0,t) \\
        &- 2  [\tau_{22}^{I,0}(x,0,t) + \tau_{22}^{B,0}] \partial_zu_2^{B,1} -2  [\tau_{22}^{I,1}(x,0,t) + \tau_{22}^{B,1}]\partial_zu_2^{B,0}-2z\partial_y\tau_{22}^{I,0}(x,0,t)\partial_z u_2^{B,0} \\
        &-2k\eta^{B,0}\partial_y u_2^{I,0}(x,0,t)- 2k [\eta^{I,0}(x,0,t) + \eta^{B,0}] \partial_z u_2^{B,1}- 2k [\eta^{I,1}(x,0,t) + \eta^{B,1}] \partial_z u_2^{B,0}\\
        &-2kz\partial_y\eta^{I,0}(x,0,t)\partial_z u_2^{B,0}.
        \end{align*}
        Since the expressions of $K^{j},j\geq 1$ are pretty lengthy, we omit them here for brevity.

     Using $u^{I,0}(x,0,t) = 0,\, u_2^{I,1}(x,0,t) = 0,\, u^{B,0} = 0,\, u_2^{B,1} = 0,\, \eta^{B,0} = 0$ from (\ref{OB_limit_I_initial_boundary1}), (\ref{eq_u_eta_tau_B_0}), and Corollary \ref{corollary_u}, and by virtue of $\mathrm{div} u^{I,0} = 0$, the initial-boundary conditions \eqref{eq_initial}, \eqref{eq_boundary}, we find that $\tau_{22}^{B,0}$ satisfies the following problem:
    \begin{equation}\label{eq_tau_22_B_0}
        \begin{cases}
            \partial_t \tau_{22}^{B,0}+ \gamma\tau_{22}^{B,0} -\partial_z^2\tau_{22}^{B,0} = 0,\\[2mm]
            \tau_{22}^{B,0}(x,z,0) = 0,~\partial_z\tau_{22}^{B,0}(x,0,t) = 0,
        \end{cases}
     \end{equation}
     which, by the uniqueness of solutions, yields
     \begin{eqnarray}\label{eq_tau_B_0_22}
        \tau_{22}^{B,0} = 0,~\text{ and } \tilde{p}^{B,0} = 0,
     \end{eqnarray}
      where $\eqref{eq_u_B_1}_2$ is used.

      Since $\tau_{12}^{B,0} = \tau_{21}^{B,0},$ we obtain
    \begin{equation*}
        \begin{cases}
            \partial_t \tau_{12}^{B,0} -\partial_yu_{1}^{I,0}(x,0,t)\tau_{22}^{B,0} - \partial_zu_1^{B,1}(\tau_{22}^{I,0}(x,0,t) + \tau_{22}^{B,0})\\[1mm]
            ~~-k\eta^{I,0}(x,0,t)\partial_zu_1^{B,1} + \gamma\tau_{12}^{B,0} - \partial_z^2\tau_{12}^{B,0} = 0,\\[2mm]
            \tau_{12}^{B,0}(x,z,0) = 0,~\partial_z\tau_{12}^{B,0}(x,0,t) = 0.
        \end{cases}
    \end{equation*}
    From \eqref{eq_u_B_1}$_1$, we have $\partial_zu^{B,1}_1 = -\frac{1}{\mu} \tau_{12}^{B,0}.$ Thus, $\tau_{12}^{B,0}$ satisfies the following linear equation:
    \begin{equation}\label{eq_tau_12_B_0}
        \begin{cases}
            \partial_t \tau_{12}^{B,0}  +\frac{1}{\mu}(\tau_{22}^{I,0}(x,0,t) + k\eta^{I,0}(x,0,t) + \gamma\mu)\tau_{12}^{B,0}  - \partial_z^2\tau_{12}^{B,0} = 0,\\[2mm]
            \tau_{12}^{B,0}(x,z,0) = 0,~\partial_z\tau_{12}^{B,0}(x,0,t) = 0,
        \end{cases}
    \end{equation}
    which yields
    \begin{eqnarray}\label{eq_tau_u_B_0}
       \tau_{21}^{B,0} = \tau_{12}^{B,0} = 0, \text{ and }   u^{B,1}_1 = 0.
    \end{eqnarray}

Next, using the above results and $K^0_{11} = 0,$ we find that $\tau_{11}^{B,0}$ satisfies
\begin{equation} \label{eq_tau_11_B_0}
    \begin{cases}
        \partial_t \tau_{11}^{B,0}+ \gamma\tau_{11}^{B,0} -\partial_z^2\tau_{11}^{B,0} = 0,\\[2mm]
        \tau_{11}^{B,0}(x,z,0) = 0,\partial_z\tau_{11}^{B,0}(x,0,t) = 0,
    \end{cases}
 \end{equation}
 which yields
 \begin{eqnarray}\label{eq_tau_p_B_0}
     \tau_{11}^{B,0} = 0.
 \end{eqnarray}
  In view of \eqref{eq_u_B_00}, (\ref{etab0}), \eqref{eq_tau_B_0_22}, \eqref{eq_tau_u_B_0}, and \eqref{eq_tau_p_B_0}, the proof of Lemma \ref{lemma_eq_u_eta_tau_B_0} is complete.
\end{proof}

\medskip

\noindent{\bfseries Step 3. Equations of $(u^{I,1},\tilde{p}^{I,1},\eta^{I,1},\tau^{I,1})$ and $(u^{B,1},\tilde{p}^{B,1},\eta^{B,1},\tau^{B,1})$.}

\medskip

As a consequence of Lemma \ref{lemma_eq_u_eta_tau_B_0}, we get the following corollary.
\begin{corollary}\label{corollary_u}
    From \eqref{eq_divfree_B}, (\ref{eq_u_B_00}) and \eqref{eq_tau_u_B_0}, we have
    \begin{eqnarray}\label{eq_u_B_2_2}
        u^{B,1} = 0,~u_2^{B,2} =0,
    \end{eqnarray}
    which combining with \eqref{eq_boundary} yields the following boundary conditions:
    \begin{eqnarray}\label{eq_first_outer_boundary_cond}
        u^{I,1}(x,0,t) = 0,~u^{I,2}_2(x,0,t) = 0.
   \end{eqnarray}
\end{corollary}

\begin{lemma}\label{lemma_u_I_1}
    The first order outer profiles vanish identically, i.e.,
    \begin{eqnarray}
        (u^{I,1},\tilde{p}^{I,1},\eta^{I,1},\tau^{I,1}) = 0.
    \end{eqnarray}
\end{lemma}
\begin{proof}
  Letting $j = 1$ in \eqref{eq_outer_layer_j}, we find that $(u^{I,1},\tilde{p}^{I,1},\eta^{I,1},\tau^{I,1})$ satisfies the following
homogeneous system:
\begin{equation}\label{eq_u_eta_tau_I1}
    \begin{cases}
    \partial_tu^{I,1} +u^{I,0}\cdot\nabla u^{I,1} + u^{I,1}\cdot\nabla u^{I,0} +\nabla \tilde{p}^{I,1} -\mu \Delta u^{I,1}
    -\mathrm{div} \tau^{I,1} = 0,\\[2mm]
    \partial_t\eta^{I,1}+u^{I,0}\cdot\nabla \eta^{I,1}+ u^{I,1}\cdot\nabla \eta^{I,0}  =0,\\[2mm]
    \partial_t\tau^{I,1}+u^{I,0}\cdot \nabla \tau^{I,1} + u^{I,1}\cdot \nabla \tau^{I,0}- \mathcal{Q}(\nabla u^{I,0} ,\tau^{I,1})
    - \mathcal{Q}(\nabla u^{I,1} ,\tau^{I,0})\\
    ~~~ -\mathcal{B}(\eta^{I,0} ,\nabla u^{I,1}) - \mathcal{B}( \eta^{I,1} ,\nabla u^{I,0}) + \gamma\tau^{I,1} =0,\\[2mm]
    \mathrm{div} u^{I,1} = 0.\\
    \end{cases}
\end{equation}
By \eqref{eq_initial} and  \eqref{eq_first_outer_boundary_cond}, we find that
\eqref{eq_u_eta_tau_I1} is equipped with the homogeneous initial-boundary condition:
\begin{eqnarray}
    (u^{I,1}, \eta^{I,1},\tau^{I,1})(x,y,0) = 0,~u^{I,1}(x,0,t) = 0,
\end{eqnarray}
 which implies that $(u^{I,1},\tilde{p}^{I,1},\eta^{I,1},\tau^{I,1})$ vanishes identically. Here the regularities in Proposition \ref{local_wellposedness_0} is used.
\end{proof}
\begin{lemma}\label{lemma_u_eta_tau_B_111}
     The profiles $\eta^{B,1},\tau^{B,1}_{22},\tau^{B,1}_{12}$, and $\tau^{B,1}_{11}$ satisfy \eqref{eq_eta_B1}, \eqref{eq_tau_B1_22}, \eqref{eq_tau_B1_12}, and \eqref{eq_tau_B1_11}, respectively.
    In addition, $\tilde{p}^{B,1} = \tau^{B,1}_{22}$.
\end{lemma}
\begin{proof}
From $F^{0} = 0$,
 Corollary \ref{corollary_u}, Lemmas \ref{lemma_eq_u_eta_tau_B_0} and \ref{lemma_u_I_1}, we get
\begin{equation}\label{eq_u_B2}
   \begin{cases}
       \mu\partial_z u_1^{B,2} + \tau_{12}^{B,1} = 0,\\[2mm]
       \tilde{p}^{B,1} -\tau_{22}^{B,1} = 0,
   \end{cases}
   \implies
   \begin{cases}
        u_1^{B,2} =\frac{1}{\mu} \int_z^{+\infty}\tau_{12}^{B,1}\,dz,\\[2mm]
       \tilde{p}^{B,1} =\tau_{22}^{B,1}.
   \end{cases}
\end{equation}
Combined with $G^{1} = 0,$ \eqref{eq_boundary},
Corollary \ref{corollary_u}, Lemmas \ref{lemma_eq_u_eta_tau_B_0}  and \ref{lemma_u_I_1},\, $\eta^{B,1}$ satisfies \eqref{eq_eta_B1}.

Next, we are in the position to derive the equation of $\tau^{B,1}$. By virtue of $K^{1} = 0$, Corollary \ref{corollary_u}, Lemmas \ref{lemma_eq_u_eta_tau_B_0}  and \ref{lemma_u_I_1}, we get
\begin{equation}\label{eq_tau_B_1}
    \begin{cases}
       \partial_t \tau^{B,1}_{11}+ \gamma\tau^{B,1}_{11} - \partial_z^2\tau^{B,1}_{11} -2\partial_yu_1^{I,0}(x,0,t)\tau_{12}^{B,1} - 2\tau_{12}^{I,0}(x,0,t)\partial_zu_1^{B,2}  = 0,\\[2mm]
       \partial_t \tau^{B,1}_{12}+ \gamma\tau^{B,1}_{12} - \partial_z^2\tau^{B,1}_{12} - \partial_yu_1^{I,0}(x,0,t)\tau_{22}^{B,1} -  \tau_{22}^{I,0}(x,0,t)\partial_zu_1^{B,2}  \\
       ~~~-k\eta^{B,1}\partial_y u_1^{I,0}(x,0,t)- k\eta^{I,0}(x,0,t)\partial_zu_1^{B,2} = 0,\\[2mm]
       \partial_t \tau^{B,1}_{22} + \gamma\tau^{B,1}_{22} - \partial_z^2\tau^{B,1}_{22} = 0.
    \end{cases}
\end{equation}
This, combined with \eqref{eq_boundary} and \eqref{eq_tau_B_1}$_3$, yields that $\tau_{22}^{B,1}$ satisfies \eqref{eq_tau_B1_22}.

By \eqref{eq_boundary}, \eqref{eq_u_B2} and \eqref{eq_tau_B_1}, we have that $\tau_{12}^{B,1}$ and $\tau_{11}^{B,1}$ satisfy \eqref{eq_tau_B1_12} and \eqref{eq_tau_B1_11}, respectively.
\end{proof}
 \noindent{\bfseries Step 4. Equations of $(u^{I,2},\tilde{p}^{I,2},\eta^{I,2},\tau^{I,2})$ and $(u^{B,2},\tilde{p}^{B,2},\eta^{B,2},\tau^{B,2})$.} \\[3mm]
 To begin with, letting $j = 2$ in \eqref{eq_outer_layer_j}, and using \eqref{eq_u_B_2_2}, \eqref{eq_u_B2} and Lemma \ref{lemma_u_I_1}, we find that the second order outer profile $(u^{I,2},\tilde{p}^{I,2},\eta^{I,2},\tau^{I,2})$  solves \eqref{OB_outer_second}.
   \begin{lemma}\label{lemma_u_eta_tau_B_2222}
        The second order inner layer profiles $u_1^{B,2}$ and $u_2^{B,2}$ satisfy that
        \begin{eqnarray*}
            u_1^{B,2} =\frac{1}{\mu} \int_z^{+\infty}\tau_{12}^{B,1}\mathrm{d}z,~ u_2^{B,2} = 0.
        \end{eqnarray*}
        In addition, $\eta^{B,2},\tau^{B,2}_{22},\tau^{B,2}_{12}$ and $\tau^{B,2}_{11}$ satisfy the systems \eqref{eq_eta_B2}, \eqref{eq_tau_B2_22}, \eqref{eq_tau_B2_12} and \eqref{eq_tau_B2_11}.
        Furthermore,  $\tilde{p}^{B,2}$  is determined by \eqref{eq_u_B3}.
    \end{lemma}
    \begin{proof}
        The expression of $u^{B,2}$ is given by \eqref{eq_u_B_2_2} and \eqref{eq_u_B2}.
    Next, combined with $F^{1} = 0,$ \eqref{eq_boundary}, Corollary \ref{corollary_u}, Lemmas \ref{lemma_eq_u_eta_tau_B_0}  and \ref{lemma_u_I_1}, it holds that
    \begin{equation}\label{eq_u_B3_1}
        \begin{cases}
            \partial_x \tilde{p}^{B,1} -\mu \partial_z^2 u_1^{B,3} -\partial_x\tau_{11}^{B,1} -\partial_z\tau_{12}^{B,2}=0,\\[2mm]
            \partial_z\tilde{p}^{B,2} - \mu \partial_z^2 u_2^{B,3} -\partial_x\tau_{12}^{B,1} -\partial_z\tau_{22}^{B,2} = 0.
        \end{cases}
    \end{equation}
    From \eqref{eq_divfree_B} and \eqref{eq_u_B2}, we have
    \begin{equation}\label{eq_u_B3_2}
            \tilde{p}^{B,1}  = \tau_{22}^{B,1},\,
          ~\partial_zu_2^{B,3} = -\partial_x u^{B,2}_1,\,
          \partial_zu_1^{B,2} =- \frac{1}{\mu}\tau_{12}^{B,1},
    \end{equation}
    which together with \eqref{eq_u_B3_1} yields that
    \begin{equation}\label{eq_u_B3}
        \begin{cases}
            \partial_z u_1^{B,3} = - \frac{1}{\mu} \tau_{12}^{B,2} + \frac{1}{\mu} \int_{z}^{+\infty}(\partial_x\tau_{11}^{B,1} - \partial_x\tau_{22}^{B,1}) \mathrm{d}z,\\[3mm]
            \tilde{p}^{B,2} = -\mu \partial_x u_1^{B,2} -\int_z^{+\infty}\partial_x\tau_{12}^{B,1}(x,z,t)\mathrm{d}z + \tau_{22}^{B,2}.
        \end{cases}
    \end{equation}

   By virtue of \eqref{eq_boundary}, Corollary \ref{corollary_u}, Lemmas \ref{lemma_eq_u_eta_tau_B_0}  and \ref{lemma_u_I_1}, the condition $G^{2} = 0$ yields that $\eta^{B,2}$ satisfies \eqref{eq_eta_B2}.

    From $K^{2}_{22} = 0$, Corollary \ref{corollary_u}, Lemmas \ref{lemma_eq_u_eta_tau_B_0}  and \ref{lemma_u_I_1}, we have
        \begin{align*}
              & \partial_t \tau^{B,2}_{22} + \gamma \tau^{B,2}_{22} -\partial_z^2\tau^{B,2}_{22}   + u^{B,2}_1\partial_x\tau^{I,0}_{22}(x,0,t)-2 \tau^{I,0}_{22}(x,0,t)\partial_z u_2^{B,3}
                -2k\eta^{I,0}(x,0,t)\partial_zu_2^{B,3} \\
               &+\frac{1}{2}z^2\partial_y^2 u_2^{I,0}(x,0,t)\partial_z\tau^{B,1}_{22} + z\partial_y u_1^{I,0}(x,0,t)\partial_x\tau_{22}^{B,1} -2z\partial_y^2u_2^{I,0}(x,0,t)\tau^{B,1}_{22}\\
               & - 2kz\partial_y^2u_2^{I,0}(x,0,t)\eta^{B,1}= 0.
    \end{align*}
        This, combined with  \eqref{eq_boundary}, \eqref{eq_divfree_B} and \eqref{eq_u_B3_2}, deduces that $\tau^{B,2}_{22}$ fulfills \eqref{eq_tau_B2_22}.

        From $K^{2}_{12} = 0$, Corollary \ref{corollary_u}, Lemmas \ref{lemma_eq_u_eta_tau_B_0} and \ref{lemma_u_I_1}, we have
            \begin{align*}
                &\partial_t \tau^{B,2}_{12} + \gamma \tau^{B,2}_{12} -\partial_z^2\tau^{B,2}_{12} -\partial_y u^{I,0}_1(x,0,t)\tau^{B,2}_{22} -\tau_{22}^{I,0}(x,0,t)\partial_z u_1^{B,3}
              -k\eta^{I,0}(x,0,t)\partial_z u_1^{B,3} \\
              &  + u^{B,2}_1\partial_x\tau^{I,0}_{12}(x,0,t)  -  \tau_{22}^{B,1} \partial_zu_1^{B,2}  -k\eta^{B,2}\partial_yu_1^{I,0}(x,0,t)
               -k  \eta^{B,1} \partial_z u_1^{B,2}\\
              & +\frac{1}{2}z^2\partial_y^2 u_2^{I,0}(x,0,t)\partial_z\tau^{B,1}_{12} + z\partial_y u_1^{I,0}(x,0,t)\partial_x\tau_{12}^{B,1}- z\partial_y^2u_1^{I,0}(x,0,t)\tau^{B,1}_{22}\\
              & - z\partial_y\tau^{I,0}_{22}(x,0,t)\partial_zu^{B,2}_1-kz\partial_y\eta^{I,0}(x,0,t)\partial_zu^{B,2}_1 - kz\partial_y^2u^{I,0}_1(x,0,t)\eta^{B,1}= 0,
            \end{align*}
        which, combined with \eqref{eq_boundary} and \eqref{eq_divfree_B}, deduces that $\tau^{B,2}_{22}$ satisfies \eqref{eq_tau_B2_12}.

        From $K^{2}_{11} = 0$, Corollary \ref{corollary_u}, Lemmas \ref{lemma_eq_u_eta_tau_B_0}  and \ref{lemma_u_I_1}, we have
        \begin{align*}
            &\partial_t \tau^{B,2}_{11} + \gamma \tau^{B,2}_{11} -\partial_z^2\tau^{B,2}_{11} -2 \partial_yu_1^{I,0}(x,0,t)\tau^{B,2}_{12} -2 \tau^{I,0}_{12}(x,0,t)\partial_z u_1^{B,3}
             -2\tau^{I,0}_{11}(x,0,t)\partial_xu^{B,2}_1 \\
            &  + u^{B,2}_1\partial_x\tau^{I,0}_{11}(x,0,t)-2  \tau_{12}^{B,1} \partial_zu^{B,2}_1 -2 k\eta^{I,0}(x,0,t)\partial_x u^{B,2}_1 +\frac{1}{2}z^2\partial_y^2 u_2^{I,0}(x,0,t)\partial_z\tau^{B,1}_{11}\\
            &  +z\partial_yu_1^{I,0}(x,0,t)\partial_x\tau_{11}^{B,1}-2z\partial_x\partial_yu^{I,0}_1(x,0,t)\tau^{B,1}_{11} -2z\partial_y^2u^{I,0}_1(x,0,t)\tau^{B,1}_{12}
            \\
            & -2z \partial_y\tau^{I,0}_{12}(x,0,t)\partial_zu^{B,2}_1 -2kz\partial_x\partial_yu^{I,0}_1(x,0,t)\eta^{B,1}= 0,
            \end{align*}
   which, combined with  \eqref{eq_boundary}, \eqref{eq_divfree_B}, and \eqref{eq_u_B3}, deduces that $\tau^{B,2}_{11}$ solves \eqref{eq_tau_B2_11}.
\end{proof}

\medskip

\noindent{\bfseries Step 5. Equations of $(u^{I,3},\tilde{p}^{I,3},\eta^{I,3},\tau^{I,3})$ and $(u^{B,3},\tilde{p}^{B,3},\eta^{B,3},\tau^{B,3})$.}  \\[3mm]
To begin with, letting $j = 3$ in \eqref{eq_outer_layer_j}, and using \eqref{eq_u_B2}, \eqref{eq_u_B3}, and Lemma \ref{lemma_u_I_1}, we find that the third order outer profile $(u^{I,3},\tilde{p}^{I,3},\eta^{I,3},\tau^{I,3})$  solves \eqref{OB_outer_third}.
\begin{lemma}\label{lemma_u_eta_tau_B_3333}
    The third inner layer profiles $u_1^{B,3}$ and $u_2^{B,3}$ satisfy
    \begin{eqnarray}\label{eq_u_B_3333}
            \begin{cases}
            u_1^{B,3} =  \frac{1}{\mu}\int_z^{+\infty}\tau_{12}^{B,2}(x,z,t) \mathrm{d}z -\frac{1}{\mu}\int_{z}^{+\infty}\int_{z}^{+\infty}(\partial_x\tau_{11}^{B,1} - \partial_x\tau_{22}^{B,1})(x,\xi,t)\mathrm{d}\xi\mathrm{d}z,\\[3mm]
            u_2^{B,3} = \int_z^{+\infty} \partial_x u_1^{B,2}\mathrm{d}z.
        \end{cases}
         \end{eqnarray}
    In addition, $\eta^{B,3},\tau^{B,3}_{22},\tau^{B,3}_{12}$, and $\tau^{B,3}_{11}$ satisfy the systems \eqref{eq_eta_B3}, \eqref{eq_tau_B3_22}, \eqref{eq_tau_B3_12}, and \eqref{eq_tau_B3_11}, respectively.
    Furthermore, $u^{B,4},u_2^{B,5}$, and $\tilde{p}^{B,3}$ are given by \eqref{eq_u_B4_1}, and \eqref{eq_u_B4}.
\end{lemma}
 \begin{proof}
The equation of $u^{B,3}$ can be deduced by \eqref{eq_divfree_B} and \eqref{eq_u_B3}. Then from $F^{2} = 0,$ \eqref{eq_boundary}, Corollary \ref{corollary_u}, Lemmas \ref{lemma_eq_u_eta_tau_B_0}  and \ref{lemma_u_I_1}, we obtain
\begin{equation}\label{eq_u_B4_1}
    \begin{cases}
        \partial_t u_1^{B,2} + \partial_x \tilde{p}^{B,2} - \mu\partial_x^2 u_1^{B,2}-\mu \partial_z^2 u_1^{B,4} -\partial_x\tau_{11}^{B,2} -\partial_z\tau_{12}^{B,3}=0,\\[2mm]
        \partial_z\tilde{p}^{B,3} - \mu \partial_z^2 u_2^{B,4} -\partial_x\tau_{12}^{B,2} -\partial_z\tau_{22}^{B,3} = 0.
    \end{cases}
\end{equation}
Using \eqref{eq_divfree_B}, \eqref{eq_u_B2} and \eqref{eq_u_B4_1}, we get
\begin{eqnarray}\label{eq_u_B4}
    \begin{cases}
          u_1^{B,4} =  \frac{1}{\mu}\int_z^{+\infty}\tau_{12}^{B,3}\mathrm{d}z  - \frac{1}{\mu}\int_z^{+\infty} \int_\xi^{+\infty}\left(\mu\partial_x^2u_1^{B,2} + \partial_x\tau_{11}^{B,2}  -\partial_t u_1^{B,2} -\partial_x\tilde{p}^{B,2} \right)(x,\zeta,t)\mathrm{d}\zeta\mathrm{d}\xi,\\[3mm]
           u_2^{B,4} = \int_{z}^{+\infty} \partial_x u_1^{B,3}\mathrm{d}z, ~~u_2^{B,5} = \int_{z}^{+\infty} \partial_x u_1^{B,4}\mathrm{d}z,\\[2mm]
           \tilde{p}^{B,3} = \tau_{22}^{B,3} - \mu \partial_x u_1^{B,3} - \int_z^{+\infty} \partial_x \tau_{12}^{B,2}\mathrm{d} z.
    \end{cases}
\end{eqnarray}

In view of \eqref{eq_boundary}, Corollary \ref{corollary_u}, Lemmas \ref{lemma_eq_u_eta_tau_B_0}  and \ref{lemma_u_I_1}, the condition $G^{3} = 0$  yields that $\eta^{B,3}$ satisfies \eqref{eq_eta_B3}.

Similar to \eqref{eq_tau_B2_22}, from $K^{3}_{22} = 0$, Corollary \ref{corollary_u}, Lemmas \ref{lemma_eq_u_eta_tau_B_0}  and \ref{lemma_u_I_1}, we have
    \begin{align*}
        &\partial_t \tau^{B,3}_{22} + \gamma \tau^{B,3}_{22} -\partial_z^2\tau^{B,3}_{22} - \partial_x^2 \tau_{22}^{B,1}
         + u_1^{B,2} \partial_x\tau_{22}^{B,1}+ u_2^{B,3} \partial_z\tau_{22}^{B,1}+ u_{1}^{I,2}(x,0,t) \partial_x\tau_{22}^{B,1} \\
        & + u_{2}^{I,3}(x,0,t) \partial_z\tau_{22}^{B,1}+ u_1^{B,3}\partial_x\tau_{22}^{I,0}(x,0,t)+ u_2^{B,3}\partial_y\tau_{22}^{I,0}(x,0,t) - 2[ \tau_{12}^{I,0}(x,0,t)\partial_x u_2^{B,3}
         \\
        &+\tau_{22}^{B,1}\partial_y u_2^{I,2}(x,0,t) + \tau_{22}^{B,1}\partial_z u_2^{B,3} + \tau_{22}^{I,0}(x,0,t)\partial_z u_2^{B,4} ] -2k [\eta^{B,1}\partial_y u_2^{I,2}(x,0,t) \\
        & +\eta^{I,0}(x,0,t)\partial_z u_2^{B,4}+\eta^{B,1}\partial_z u_{2}^{B,3}]- z^2\partial_y^2\partial_x u_2^{I,0}(x,0,t)\tau_{12}^{B,1} - z^2\partial_y^3 u_2^{I,0}(x,0,t)\tau_{22}^{B,1} \\
        &-k z^2\partial_y^3 u_2^{I,0}(x,0,t)\eta^{B,1}
        + z\partial_y\partial_x \tau_{22}^{I,0}(x,0,t) u_1^{B,2}+ z\partial_y u_1^{I,0}(x,0,t)  \partial_x \tau^{B,2}_{22} \\
         &+ z\partial_y u_2^{I,2}(x,0,t)\partial_z \tau^{B,1}_{22} - 2[z\partial_y^2 u_2^{I,0}(x,0,t)\tau_{22}^{B,2} + z\partial_y\tau_{22}^{I,0}(x,0,t)\partial_z u_2^{B,3}]\\
        &-2k [z\partial_y^2 u_2^{I,0}(x,0,t) \eta^{B,2} + z\partial_y\eta^{I,0}(x,0,t)\partial_z u_2^{B,3} ]
          + \frac{1}{6}z^3\partial_y^3 u_2^{I,0}(x,0,t)\partial_z\tau_{22}^{B,1} \\
          &+ \frac{1}{2}z^2\partial_y^2 u_1^{I,0}(x,0,t)\partial_x\tau_{22}^{B,1}
        +\frac{1}{2}z^2\partial_y^2u_2^{I,0}(x,0,t)\partial_z\tau_{22}^{B,2}= 0.
    \end{align*}
This, combined with \eqref{eq_boundary}, \eqref{eq_divfree_B}, and \eqref{eq_u_B4}, yields that $\tau_{22}^{B,3}$  satisfies \eqref{eq_tau_B3_22} where the source term is given by
    \begin{align} \label{eq_tau_B3_22_h}
        \mathfrak{h}^{B,3}_{22} =&~ \partial_x^2 \tau_{22}^{B,1}  -  2[ k \eta^{I,0}(x,0,t) +  \tau_{22}^{I,0}(x,0,t)]\partial_x u_1^{B,3}
        -u_1^{B,2} \partial_x\tau_{22}^{B,1}- u_2^{B,3} \partial_z\tau_{22}^{B,1} \notag\\
        &  -u_{1}^{I,2}(x,0,t) \partial_x\tau_{22}^{B,1} - u_{2}^{I,3}(x,0,t) \partial_z\tau_{22}^{B,1}- u_1^{B,3}\partial_x\tau_{22}^{I,0}(x,0,t)- u_2^{B,3}\partial_y\tau_{22}^{I,0}(x,0,t)
        \notag\\
        &+ 2[ \tau_{12}^{I,0}(x,0,t)\partial_x u_2^{B,3}+\tau_{22}^{B,1}\partial_y u_2^{I,2}(x,0,t) + \tau_{22}^{B,1}\partial_z u_2^{B,3}  ] +2k [\eta^{B,1}\partial_y u_2^{I,2}(x,0,t)  \notag\\
        &+\eta^{B,1}\partial_z u_{2}^{B,3}]  + z^2\partial_y^2\partial_x u_2^{I,0}(x,0,t)\tau_{12}^{B,1} + z^2\partial_y^3 u_2^{I,0}(x,0,t)\tau_{22}^{B,1}  +k z^2\partial_y^3 u_2^{I,0}(x,0,t)\eta^{B,1} \\
         & + 2[z\partial_y^2 u_2^{I,0}(x,0,t)\tau_{22}^{B,2} + z\partial_y\tau_{22}^{I,0}(x,0,t)\partial_z u_2^{B,3}] +2k [z\partial_y^2 u_2^{I,0}(x,0,t) \eta^{B,2}  \notag\\
        &+ z\partial_y\eta^{I,0}(x,0,t)\partial_z u_2^{B,3} ] -z\partial_y\partial_x \tau_{22}^{I,0}(x,0,t) u_1^{B,2}- z\partial_y u_1^{I,0}(x,0,t) \partial_x \tau^{B,2}_{22} \notag\\
        & - z\partial_y u_2^{I,2}(x,0,t)\partial_z \tau^{B,1}_{22}
          - \frac{1}{6}z^3\partial_y^3 u_2^{I,0}(x,0,t)\partial_z\tau_{22}^{B,1}
          - \frac{1}{2}z^2{\partial_y^2} u_1^{I,0}(x,0,t)\partial_x\tau_{22}^{B,1}
          \notag \\
          &
        -\frac{1}{2}z^2\partial_y^2u_2^{I,0}(x,0,t)\partial_z\tau_{22}^{B,2}. \notag
    \end{align}
   Similar to \eqref{eq_tau_B2_12}, from $K^{3}_{12} = 0$, Corollary \ref{corollary_u}, Lemmas \ref{lemma_eq_u_eta_tau_B_0} and \ref{lemma_u_I_1}, we have
    \begin{align*}
        &\partial_t \tau^{B,3}_{12} + \gamma \tau^{B,3}_{12} -\partial_z^2\tau^{B,3}_{12} - \partial_x^2 \tau_{12}^{B,1}  + u_1^{B,3}\partial_x\tau_{12}^{I,0}(x,0,t)
        + u_2^{B,3}\partial_y\tau_{12}^{I,0}(x,0,t)
        + u_{1}^{I,2}(x,0,t) \partial_x\tau_{12}^{B,1} \\
        &+ u_1^{B,2} \partial_x\tau_{12}^{B,1} + u_{2}^{I,3}(x,0,t) \partial_z\tau_{12}^{B,1}+ u_2^{B,3} \partial_z\tau_{12}^{B,1} - \tau_{11}^{I,0}(x,0,t)\partial_x u_2^{B,3}
        -\tau_{22}^{B,1}\partial_y u_1^{I,2}(x,0,t)\\
        & -\tau_{22}^{B,3}\partial_y u_1^{I,0}(x,0,t)- \tau_{22}^{I,0}(x,0,t)\partial_z u_1^{B,4} -  \tau_{22}^{B,1}\partial_z u_1^{B,3} -\tau_{22}^{I,2}(x,0,t)\partial_z u_1^{B,2} -\tau_{22}^{B,2} \partial_z u_1^{B,2}
         \\
        &-k\eta^{I,0}(x,0,t)\partial_x u_2^{B,3}-k\eta^{B,1}\partial_yu_1^{I,2}(x,0,t) -k\eta^{I,0}(x,0,t)\partial_z u_1^{B,4} - k\eta^{B,3}\partial_y u^{I,0}_1(x,0,t)
        \\
        &- k\eta^{B,1}\partial_z u_1^{B,3}
        - k\eta^{B,2}\partial_z u_1^{B,2}-k\eta^{I,2}(x,0,t)\partial_z u_1^{B,2}  +  z\partial_y\partial_x \tau_{12}^{I,0}(x,0,t) u_1^{B,2} + z\partial_y u_1^{I,0}(x,0,t) \partial_x \tau^{B,2}_{12}
        \\
        & +z\partial_y u_2^{I,2}(x,0,t)\partial_z \tau^{B,1}_{12}-z\partial_y^2u_1^{I,0}(x,0,t)\tau_{22}^{B,2} - z\partial_y\tau_{22}^{I,0}(x,0,t)\partial_z u_1^{B,3} -kz\partial_y^{2}u_1^{I,0}(x,0,t)\eta^{B,2} \\
        &- kz\partial_y\eta^{I,0}(x,0,t)\partial_z u_1^{B,3}
        + \frac{1}{6}z^3\partial_y^3 u_2^{I,0}(x,0,t)\partial_z\tau_{12}^{B,1}
         + \frac{1}{2}z^2{ \partial_y^2} u_1^{I,0}(x,0,t)\partial_x\tau_{12}^{B,1}
      \\
      &+\frac{1}{2}z^2\partial_y^2u_2^{I,0}(x,0,t)\partial_z\tau_{12}^{B,2} - \frac{1}{2}z^2\partial_y^2\partial_x u^{I,0}_2(x,0,t)\tau_{11}^{B,1} - \frac{1}{2}z^2\partial_y^3  u^{I,0}_1(x,0,t)\tau_{22}^{B,1} \\
        &- \frac{1}{2}z^2\partial_y^2  \tau_{22}^{I,0}(x,0,t)\partial_z u_1^{B,2}- \frac{k}{2}z^2\partial_y^2\partial_x u^{I,0}_2(x,0,t)\eta^{B,1} - \frac{k}{2}z^2\partial_y^3  u^{I,0}_1(x,0,t)\eta^{B,1}\\
        & - \frac{k}{2}z^2\partial_y^2  \eta^{I,0}(x,0,t)\partial_z u_1^{B,2} = 0,
    \end{align*}
which, combined with \eqref{eq_boundary}, \eqref{eq_divfree_B}, and \eqref{eq_u_B4}, yields that $\tau_{12}^{B,3}$ solves \eqref{eq_tau_B3_12} where the source term $\mathfrak{h}^{B,3}_{12}$ is given by
    \begin{align}\label{eq_tau_B3_12_h}
        \mathfrak{h}^{B,3}_{12} =~&  \partial_x^2 \tau_{12}^{B,1}  - u_1^{B,3}\partial_x\tau_{12}^{I,0}(x,0,t)
        - u_2^{B,3}\partial_y\tau_{12}^{I,0}(x,0,t)
        - u_{1}^{I,2}(x,0,t) \partial_x\tau_{12}^{B,1} - u_1^{B,2} \partial_x\tau_{12}^{B,1}-u_2^{B,3} \partial_z\tau_{12}^{B,1} \notag\\
        &- u_{2}^{I,3}(x,0,t) \partial_z\tau_{12}^{B,1}+ \tau_{11}^{I,0}(x,0,t)\partial_x u_2^{B,3}
        +\tau_{22}^{B,1}\partial_y u_1^{I,{ 2}}(x,0,t) +\tau_{22}^{B,3}\partial_y u_1^{I,0}(x,0,t) \notag\\
        & +  \tau_{22}^{B,1}\partial_z u_1^{B,3}+\tau_{22}^{I,2}(x,0,t)\partial_z u_1^{B,2} +\tau_{22}^{B,2} \partial_z u_1^{B,2}
        +k\eta^{I,0}(x,0,t)\partial_x u_2^{B,3} +k\eta^{B,1}\partial_z u_1^{B,3} \notag\\
        &+ k\eta^{B,2}\partial_z u_1^{B,2} +k\eta^{B,1}\partial_yu_1^{I,2}(x,0,t) + k\eta^{B,3}\partial_y u^{I,0}_1(x,0,t)
        +k\eta^{I,2}(x,0,t)\partial_z u_1^{B,2}  \notag\\
        &-  z\partial_y\partial_x \tau_{12}^{I,0}(x,0,t) u_1^{B,2}
        - z\partial_y u_1^{I,0}(x,0,t) \partial_x \tau^{B,2}_{12}-z\partial_y u_2^{I,2}(x,0,t)\partial_z \tau^{B,1}_{12} \notag\\
        &+z\partial_y^2u_1^{I,0}(x,0,t)(\tau_{22}^{B,2} + k\eta^{B,2}) + [z\partial_y\tau_{22}^{I,0}(x,0,t)
          + kz\partial_y\eta^{I,0}(x,0,t)]\partial_z u_1^{B,3} \\
        &- \frac{1}{6}z^3\partial_y^3 u_2^{I,0}(x,0,t)\partial_z\tau_{ 1 2}^{B,1}
         - \frac{1}{2}z^2\partial_y^2 u_1^{I,0}(x,0,t)\partial_x\tau_{12}^{B,1}
      -\frac{1}{2}z^2\partial_y^2u_2^{I,0}(x,0,t)\partial_z\tau_{12}^{B,2}\notag\\
      & + \frac{1}{2}z^2\partial_y^2\partial_x u^{I,0}_2(x,0,t)\tau_{11}^{B,1} + \frac{1}{2}z^2\partial_y^3  u^{I,0}_1(x,0,t)\tau_{22}^{B,1} + \frac{1}{2}z^2\partial_y^2  \tau_{22}^{I,0}(x,0,t)\partial_z u_1^{B,2} \notag\\
        &+ \frac{k}{2}z^2\partial_y^2\partial_x u^{I,0}_2(x,0,t)\eta^{B,1} + \frac{k}{2}z^2\partial_y^3  u^{I,0}_1(x,0,t)\eta^{B,1} + \frac{k}{2}z^2\partial_y^2  \eta^{I,0}(x,0,t)\partial_z u_1^{B,2}  \notag\\
        &+\frac{1}{\mu}[\tau_{22}^{I,0}(x,0,t) + k\eta^{I,0}(x,0,t)]\int_z^{+\infty}\left(\mu\partial_x^2u_1^{B,2} + \partial_x\tau_{11}^{B,2} -\partial_t u_1^{B,2} -\partial_x\tilde{p}^{B,2} \right)\mathrm{d}z. \notag
    \end{align}
Similar to \eqref{eq_tau_B2_11}, from $K^{3}_{11} = 0$, Corollary \ref{corollary_u}, Lemmas \ref{lemma_eq_u_eta_tau_B_0}  and \ref{lemma_u_I_1}, we have
    \begin{align*}
        &\partial_t \tau^{B,3}_{11} + \gamma \tau^{B,3}_{11} -\partial_z^2\tau^{B,3}_{11} - \partial_x^2 \tau_{11}^{B,1}  + u_1^{B,3}\partial_x\tau_{11}^{I,0}(x,0,t)
        + u_2^{B,3}\partial_y\tau_{11}^{I,0}(x,0,t)
        + u_{1}^{I,2}(x,0,t) \partial_x\tau_{11}^{B,1} \\
        &+ u_1^{B,2} \partial_x\tau_{11}^{B,1}+ (u_{2}^{I,3}(x,0,t)  + u_2^{B,3}) \partial_z\tau_{11}^{B,1} - 2\tau_{11}^{I,0}(x,0,t)\partial_x u_1^{B,3}  -2\tau_{11}^{B,1}\partial_x u_1^{B,2}
        -2\tau_{11}^{B,1}\partial_x u_1^{I,2}(x,0,t)\\
        &-2\tau_{12}^{B,1}\partial_y u_{1}^{I,2}(x,0,t)-2\tau_{12}^{B,3}\partial_y u_{1}^{I,0}(x,0,t) - 2\tau_{12}^{I,0}(x,0,t)\partial_z u_1^{B,4} - 2\tau_{12}^{B,1}\partial_z u_1^{B,3}
        -2\tau_{12}^{I,2}(x,0,t)\partial_z u_1^{B,2}\\
        &-2\tau_{12}^{B,2}\partial_z u_1^{B,2} -2k \eta^{I,0}(x,0,t) \partial_x u_1^{B,3}  -2k \eta^{B,1}\partial_x u_1^{B,2} - 2k \partial_x u_1^{I,2}(x,0,t)\eta^{B,1}
        + z\partial_y\partial_x \tau_{11}^{I,0}(x,0,t) u_1^{B,2} \\
        &+ z\partial_y u_1^{I,0}(x,0,t) \partial_x \tau^{B,2}_{11}
        +z\partial_y u_2^{I,2}(x,0,t)\partial_z \tau^{B,1}_{11} -2z\partial_y\tau_{11}^{I,0}(x,0,t)\partial_x u_1^{B,2}
        - 2z\partial_y\partial_x u_1^{I,0}(x,0,t) \tau_{11}^{B,2}\\
        & - 2z\partial_y^2 u_1^{I,0}(x,0,t)\tau_{12}^{B,2} - 2z \partial_y \tau_{12}^{I,0}(x,0,t)\partial_z u_1^{B,3}
        -2kz\partial_y \eta^{I,0}(x,0,t) \partial_x u^{B,2}_1-2kz\partial_y \partial_x u_1^{I,0}(x,0,t)\eta^{B,2}\\
        &  -z^2\partial_y^2 \partial_x u_1^{I,0}(x,0,t)\tau_{11}^{B,1} - z^2\partial_y^3 u_1^{I,0}(x,0,t)\tau_{12}^{B,1}
        -z^2\partial_y^2\tau_{12}^{I,0}(x,0,t)\partial_z u_1^{B,2}-kz^2\partial_y^2 \partial_x u_1^{I,0}(x,0,t)\eta^{B,1}\\
        & + \frac{1}{6}z^3\partial_y^3 u_2^{I,0}(x,0,t)\partial_z\tau_{11}^{B,1}
        + \frac{1}{2}z^2\partial_y^2 u_1^{I,0}(x,0,t)\partial_x\tau_{11}^{B,1}
         +\frac{1}{2}z^2\partial_y^2u_2^{I,0}(x,0,t)\partial_z\tau_{11}^{B,2} = 0,
    \end{align*}
which, combined with  \eqref{eq_boundary}, \eqref{eq_divfree_B}, and \eqref{eq_u_B4}, yields that $\tau_{11}^{B,3}$ satisfies \eqref{eq_tau_B3_11},
where the source term $\mathfrak{h}^{B,3}_{11}$  is given by
    \begin{align}\label{eq_tau_B3_11_h}
        \mathfrak{h}^{B,3}_{11} ={}&  \partial_x^2 \tau_{11}^{B,1}  -u_1^{B,3}\partial_x\tau_{11}^{I,0}(x,0,t)
        - u_2^{B,3}\partial_y\tau_{11}^{I,0}(x,0,t)
        - u_{1}^{I,2}(x,0,t) \partial_x\tau_{11}^{B,1} - u_1^{B,2} \partial_x\tau_{11}^{B,1} \notag\\
        &- u_{2}^{I,3}(x,0,t) \partial_z\tau_{11}^{B,1}- u_2^{B,3} \partial_z\tau_{11}^{B,1} + 2\tau_{11}^{I,0}(x,0,t)\partial_x u_1^{B,3}  +2\tau_{11}^{B,1}\partial_x u_1^{B,2}
        \notag\\
        &+2\tau_{11}^{B,1}\partial_x u_1^{I,2}(x,0,t)+2\tau_{12}^{B,1}\partial_y u_{1}^{I,2}(x,0,t)-2\tau_{12}^{B,3}\partial_y u_{1}^{I,0}(x,0,t)
        \notag\\
        &+ 2\tau_{12}^{B,1}\partial_z u_1^{B,3} +2 \tau_{12}^{I,2}(x,0,t) \partial_z u_1^{B,2} + \tau_{12}^{B,2} \partial_z u_1^{B,2} +2k \eta^{I,0}(x,0,t) \partial_x u_1^{B,3}
        \notag \\
        & +2k \eta^{B,1}\partial_x u_1^{B,2} + 2k \partial_x u_1^{I,2}(x,0,t)\eta^{B,1}- z\partial_y\partial_x \tau_{11}^{I,0}(x,0,t) u_1^{B,2}
        \notag\\
        &- z\partial_y u_1^{I,0}(x,0,t) \partial_x \tau^{B,2}_{11}-z\partial_y u_2^{I,2}(x,0,t)\partial_z \tau^{B,1}_{11} +2z\partial_y\tau_{11}^{I,0}(x,0,t)\partial_x u_1^{B,2}
        \notag\\
        &+2z\partial_y\partial_x u_1^{I,0}(x,0,t) \tau_{11}^{B,2}  + 2z\partial_y^2 u_1^{I,0}(x,0,t)\tau_{12}^{B,2}+ 2z \partial_y \tau_{12}^{I,0}(x,0,t)\partial_z u_1^{B,3}
       \\
        & +2kz\partial_y \eta^{I,0}(x,0,t) \partial_x u^{B,2}_1+ 2kz\partial_y \partial_x u_1^{I,0}(x,0,t)\eta^{B,2}  +z^2\partial_y^2 \partial_x u_1^{I,0}(x,0,t)\tau_{11}^{B,1}
        \notag\\
        & + z^2\partial_y^3 u_1^{I,0}(x,0,t)\tau_{12}^{B,1} +z^2\partial_y^2\tau_{12}^{I,0}(x,0,t)\partial_z u_1^{B,2} +kz^2\partial_y^2 \partial_x u_1^{I,0}(x,0,t)\eta^{B,1} \notag\\
        &- \frac{1}{6}z^3\partial_y^3 u_2^{I,0}(x,0,t)\partial_z\tau_{11}^{B,1} - \frac{1}{2}z^2\partial_y^{2} u_1^{I,0}(x,0,t)\partial_x\tau_{11}^{B,1}
         -\frac{1}{2}z^2\partial_y^2u_2^{I,0}(x,0,t)\partial_z\tau_{11}^{B,2}  \notag\\
        & + \frac{2}{\mu} \tau_{12}^{I,0}(x,0,t) \int_z^{+\infty}\left(\mu\partial_x^2u_1^{B,2} + \partial_x\tau_{11}^{B,2} -\partial_t u_1^{B,2} -\partial_x\tilde{p}^{B,2} \right)\mathrm{d}z.\notag
    \end{align}
\end{proof}

\section{Expressions of some source terms}\label{section_source_terms}
In this section, we present the complete expressions of some source terms in (\ref{eq_UHT_varepsilon}), i.e.,
\begin{eqnarray*}
    \mathfrak{F}_{11} =\left(\begin{matrix}
        \mathfrak{F}_{11, 1}  \\
        \mathfrak{F}_{11, 2}
    \end{matrix}\right),
    ~~
    \mathfrak{F}_{12} =\left(\begin{matrix}
        \mathfrak{F}_{12, 1}  \\
        \mathfrak{F}_{12, 2}
    \end{matrix}\right),
    ~~
     \mathfrak{G}_{11},\,  \mathfrak{G}_{12},\,
    ~~
    \mathfrak{J}_{11}
    =\left(\begin{matrix}
        \mathfrak{J}_{11,11} & \mathfrak{J}_{11,12}\\
        \mathfrak{J}_{11,12} & \mathfrak{J}_{11,22}
    \end{matrix}\right),
    ~~
     \mathfrak{J}_{12}
    =\left(\begin{matrix}
        \mathfrak{J}_{12,11} & \mathfrak{J}_{12,12}\\
        \mathfrak{J}_{12,12} & \mathfrak{J}_{12,22}
    \end{matrix}\right).
\end{eqnarray*}
To begin with, for $\mathfrak{F}_{11}$ and $\mathfrak{F}_{12}$, we have
      \begin{align*}
        \mathfrak{F}_{11,1}=&  -\sqrt{\varepsilon}\Big(  u_1^{I,0}(x,y,t) -   u_1^{I,0}(x,0,t)  \Big) \partial_xu_1^{B,2}  - \Big(  u_2^{I,0}(x,y,t) -   u_2^{I,0}(x,0,t)-  y\partial_y  u_2^{I,0}(x,0,t)\Big) \\
        & \times   \Big( \partial_zu_1^{B,2} + \sqrt{\varepsilon}  \partial_zu_1^{B,3}\Big)
          -\sqrt{\varepsilon}u_1^{B,2}\bigl(\partial_x u^{I,0}_1(x,y,t) - \partial_x u^{I,0}_1(x,0,t) \bigr), \\[2mm]
        \mathfrak{F}_{11,2}=& - \sqrt{\varepsilon} \Big(  u_2^{I,0}(x,y,t) -   u_2^{I,0}(x,0,t)  \Big) \partial_zu_2^{B,3}
         -\sqrt{\varepsilon}u_1^{B,2}\bigl(\partial_x u^{I,0}_2(x,y,t) - \partial_x u^{I,0}_2(x,0,t) \bigr),  \\[2mm]
        \mathfrak{F}_{12,1}=& - \varepsilon \Big(\partial_t u_1^{B,3} +  \sqrt{\varepsilon}\partial_t u_1^{B,4}   +  \sqrt{\varepsilon}(u_1^{I,2} + \sqrt{\varepsilon}u_1^{I,3} )   \partial_xu_1^{B,2} +  (  u_2^{I,2} + \sqrt{\varepsilon}u_2^{I,3}   )( \partial_zu_1^{B,2} + \sqrt{\varepsilon}  \partial_zu_1^{B,3})
        \\
        & +    u^{I}_1 (\partial_xu_1^{B,3} + \sqrt{\varepsilon}\partial_x u_1^{B,4}) +   u_2^I\partial_z u_1^{B,4}  + (u_1^{B,3} + \sqrt{\varepsilon} u_1^{B,4})\partial_x u^{I,0}_1 + (u_2^{B,3} + \sqrt{\varepsilon}  u_2^{B,4} + \varepsilon u_2^{B,5})   \\
        &    \times \partial_y u^{I,0}_1 + \sqrt{\varepsilon} u^B_1\partial_x (u_1^{I,2} + \sqrt{\varepsilon} u_1^{I,3}) +  \sqrt{\varepsilon} u^B_2\partial_y (u_1^{I,2}+ \sqrt{\varepsilon} u_1^{I,3}) + \sqrt{\varepsilon} u_1^B\partial_x u_1^{B} + u_2^{B} \partial_z u_1^B + \varepsilon^{ \frac{3}{2}}u_1^{B}   \\
        &    \times \partial_x w_1  + \varepsilon^{ \frac{3}{2}}u_2^{B}\partial_y w_1  +  \varepsilon^{ \frac{3}{2}}w_1\partial_x u^{B}_1  + \varepsilon w_2 \partial_z u_1^B  + \partial_x \tilde{p}^{B,3} - \mu\partial_x^2 (u_1^{B,3} + \sqrt{\varepsilon} u_1^{B,4}) - \partial_x \tau_{11}^{B,3}\Big), \\[2mm]
         \mathfrak{F}_{12,2}=& - \varepsilon \Big(\partial_t u_2^{B,3} +  \sqrt{\varepsilon}\partial_t u_2^{B,4} + \varepsilon \partial_t u_2^{B,5}  +   u_1^{I}\partial_x u^{B,3}_2 +\sqrt{\varepsilon}(u_2^{I,2} + \sqrt{\varepsilon}u_2^{I,3})\partial_z u_2^{B,3} + \sqrt{\varepsilon} u_1^I\partial_x(u_2^{B,4}
         \\
         & +     \sqrt{\varepsilon} u_2^{B,5}) + u_2^{I}\partial_z(u_2^{B,4} + \sqrt{\varepsilon} u_2^{B,5}) + (u_1^{B,3} + \sqrt{\varepsilon} u_1^{B,4})\partial_x u^{I,0}_2 + (u_2^{B,3} + \sqrt{\varepsilon}  u_2^{B,4} + \varepsilon u_2^{B,5})\partial_y u^{I,0}_2    \\
         & +   \sqrt{\varepsilon} u^B_1\partial_x (u_2^{I,2}  + \sqrt{\varepsilon} u_2^{I,3}) +  \sqrt{\varepsilon} u^B_2\partial_y (u_2^{I,2}+ \sqrt{\varepsilon} u_2^{I,3}) + \sqrt{\varepsilon} u_1^B\partial_x u_2^{B} + u_2^{B} \partial_z u_2^B + \varepsilon^{ \frac{3}{2}}u_1^{B}\partial_x w_2    \\
         & +    \varepsilon^{ \frac{3}{2}}u_2^{B}\partial_y w_2    +\varepsilon^{ \frac{3}{2}}w_1\partial_x u^{B}_2  + \varepsilon w_2 \partial_z u_2^B   - \mu\partial_x^2 (u_2^{B,3} + \sqrt{\varepsilon} u_2^{B,4} + \varepsilon u_2^{B,5})
         -\mu\partial_z^2 u^{B,5}_2  - \partial_x \tau_{12}^{B,3}\Big).
           \end{align*}
        Next, for brevity, we introduce the following notation:
        \begin{align*}
            &\mathfrak{M}(u_1^{I,0},u_2^{I,0},u^{I,2}_1,u^{I,2}_2,u^{I,3}_2,u_1^{B,2},u_1^{B,3},u_2^{B,3},\varrho^{I,0},\varrho^{B,1},\varrho^{B,2},\varrho^{B,3}) \\
            := &-\varepsilon\Bigl( u_1^{I,2}(x,y,t) -   u_1^{I,2}(x,0,t) \Bigr)\partial_x\varrho^{B,1} - \varepsilon \Bigl( u_1^{I,0}(x,y,t) -u_1^{I,0}(x,0,t)\Bigr)\partial_x\varrho^{B,3}\\
            & -\varepsilon \Bigl( u_2^{I,3}(x,y,t) -   u_2^{I,3}(x,0,t)\Bigr)\partial_z\varrho^{B,1} - \varepsilon \Bigl( u_2^{I,2}(x,y,t)  -u_2^{I,2}(x,0,t)\Bigr)\partial_z\varrho^{B,2}\\
            &  -\varepsilon u_1^{B,3}\Bigl(\partial_x\varrho^{I,0}(x,y,t) -\partial_x\varrho^{I,0}(x,0,t)\Bigr)   -\varepsilon u_2^{B,3}\Bigl(\partial_y\varrho^{I,0}(x,y,t) -\partial_y\varrho^{I,0}(x,0,t)\Bigr) \\
            &- \Bigl( u_1^{I,0}(x,y,t) -   u_1^{I,0}(x,0,t) -y\partial_y  u_1^{I,0}(x,0,t)  -\frac{1}{2}y^2\partial_y^2  u_1^{I,0}(x,0,t) \Bigr)\partial_x\varrho^{B,1}
            \\
            &-\Bigl( u_2^{I,0}(x,y,t)     -   u_2^{I,0}(x,0,t) -y\partial_y  u_2^{I,0}(x,0,t)  -\frac{1}{2}y^2\partial_y^2  u_2^{I,0}(x,0,t)\Bigr)\partial_z\varrho^{B,2}\\
            & -\sqrt{\varepsilon}\Bigl( u_1^{I,0}(x,y,t)  -   u_1^{I,0}(x,0,t) -y\partial_y  u_1^{I,0}(x,0,t)  \Bigr)\partial_x\varrho^{B,2}
              \\
              & -\sqrt{\varepsilon}\Bigl( u_2^{I,2}(x,y,t)  -     u_2^{I,2}(x,0,t)-y\partial_y  u_2^{I,2}(x,0,t) \Bigr)\partial_z\varrho^{B,1}\\
             &-\sqrt{\varepsilon}\Bigl( u_2^{I,0}(x,y,t)    -     u_2^{I,0}(x,0,t)-y\partial_y  u_2^{I,0}(x,0,t)    \Bigr)\partial_z\varrho^{B,3}\\
             & -\sqrt{\varepsilon} \Bigl(\partial_x\varrho^{I,0}(x,y,t) -      \partial_x\varrho^{I,0}(x,0,t) - y\partial_y \partial_x\varrho^{I,0}(x,0,t)\Bigr) u_1^{B,2} \\
             &-\frac{1}{\sqrt{\varepsilon}}\Bigl( u_2^{I,0}(x,y,t)   -   u_2^{I,0}(x,0,t) -y\partial_y  u_2^{I,0}(x,0,t)  -\frac{1}{2}y^2\partial_y^2  u_2^{I,0}(x,0,t) \\
             &-\frac{1}{6} y^3\partial_y^3 u_2^{I,0}(x,0,t)\Bigr)\partial_z\varrho^{B,1},
        \end{align*}
        and
        \begin{align*}
            &\mathfrak{P}(u^{I,2}_1,u^{I,3}_1,u^{I,2}_2,u^{I,3}_2,u_1^{B,2},u_1^{B,3},u^{B,4}_1,u_2^{B,3},u^{B,4}_2,u^{B,5}_2,w_1,w_2,\varrho^{I,2},\varrho^{I},\varrho^{B,1},\varrho^{B,2},\varrho^{B,3},\varrho^B) \\
            =:{}&\varepsilon^{\frac{3}{2}}   \partial_x^2\Bigl(\varrho^{B,2} + \sqrt{\varepsilon}\varrho^{B,3}\Bigr)
        -\varepsilon^{\frac{3}{2}} u_1^{I,2} \partial_x\Bigl( \varrho^{B,2}  + \sqrt{\varepsilon} \varrho^{B,3} \Bigr) -\varepsilon^{\frac{3}{2}}   u_1^{I,3}\partial_x \varrho^{B}
        -\varepsilon^{\frac{3}{2}}  u_2^{I,2}\partial_z \varrho^{B,3}\\
        &
        -\varepsilon^{\frac{3}{2}} u_2^{I,3}\partial_z\Bigl(  \varrho^{B,2}   + \sqrt{\varepsilon} \varrho^{B,3}\Bigr)
        -\varepsilon^{\frac{3}{2}} \Bigl(  u_1^{B,2} + \sqrt{\varepsilon} u_1^{B,3}\Bigr) \partial_x \varrho^{I,2}
        -\varepsilon^{\frac{3}{2}}  u_1^{B,4}  \partial_x \varrho^I
        -\varepsilon^2 u_2^{B,3}\\
        & \times\partial_y \varrho^{I,2}
        -\varepsilon^{\frac{3}{2}}\Bigl( u_2^{B,4}   +\sqrt{\varepsilon} u_2^{B,5} \Bigr)\partial_y \varrho^I
        -\varepsilon^{\frac{3}{2}} u_1^{B,2}\partial_x\Bigl(\varrho^{B,2} + \sqrt{\varepsilon} \varrho^{B,3}\Bigr)
        -\varepsilon^{\frac{3}{2}} u_1^{B,3} \partial_x\varrho^{B}\\
        & - {\varepsilon}^2  u_1^{B,4}
        \partial_x\varrho^{B}  -\varepsilon^{\frac{3}{2}} u_2^{B,3}
        \partial_z\Bigl(\varrho^{B,2}+ \sqrt{\varepsilon} \varrho^{B,3}\Bigr)
        -\varepsilon^{\frac{3}{2}}\Bigl( u_2^{B,4}  + \sqrt{\varepsilon} u_2^{B,5}    \Bigr)\partial_z\varrho^{B}
     - \varepsilon^2 w_1\partial_x\varrho^B \\
        &   - \varepsilon^{\frac{3}{2}}w_2\partial_z \varrho^B.
        \end{align*}
Then,   replacing ``$\varrho$'' with ``$\eta$'' in the expression of $\mathfrak{M}$ and $\mathfrak{P}$, we get the expression of $\mathfrak{G}_{11}$ and $\mathfrak{G}_{12}$, i.e.,
    \begin{align*}
        \mathfrak{G}_{11 }={}& \mathfrak{M}(u_1^{I,0},u_2^{I,0},u^{I,2}_1,u^{I,2}_2,u^{I,3}_2,u_1^{B,2},u_1^{B,3},u_2^{B,3},\eta^{I,0},\eta^{B,1},\eta^{B,2},\eta^{B,3}),\\[2mm]
        \mathfrak{G}_{12}= {}& \mathfrak{P}(u^{I,2}_1,u^{I,3}_1,u^{I,2}_2,u^{I,3}_2,u_1^{B,2},u_1^{B,3},u^{B,4}_1,u_2^{B,3},u^{B,4}_2,u^{B,5}_2,w_1,w_2,\eta^{I,2},\eta^{B,1},\eta^{B,2},\eta^{B,3},\eta^B).
    \end{align*}
For the components of $\mathfrak{J}_{11}$, we have
    \begin{align*}
        \mathfrak{J}_{11,11}={}&\mathfrak{M}(u_1^{I,0},u_2^{I,0},u^{I,2}_1,u^{I,2}_2,u^{I,3}_2,u_1^{B,2},u_1^{B,3},u_2^{B,3},\tau_{11}^{I,0},\tau_{11}^{B,1},\tau_{11}^{B,2},\tau_{11}^{B,3}) \\
        &+ 2 \varepsilon \Bigl(  \tau_{11}^{I,0}(x,y,t) - \tau_{11}^{I,0}(x,0,t)   \Bigr)\partial_x u_1^{B,3}
        +   2 \varepsilon \Bigl( \tau_{12}^{I,0}(x,y,t) -\tau_{12}^{I,0}(x,0,t)\Bigr) \partial_z u_1^{B,4}  \\
         &
         +2 \varepsilon \Bigl(\tau_{12}^{I,2}(x,y,t) - \tau_{12}^{I,2}(x,0,t)\Bigr)\partial_z u_1^{B,2}
         + 2\varepsilon \Bigl(\partial_x u_1^{I,2}(x,y,t) - \partial_x u_1^{I,2}(x,0,t)\Bigr) \tau_{11}^{B,1}\\
        &
        +2\varepsilon \Bigl(\partial_x u_1^{I,0}(x,y,t) - \partial_x u_1^{I,0}(x,0,t)\Bigr) \tau_{11}^{B,3}
        +2\varepsilon \Bigl(\partial_y u_1^{I,2}(x,y,t) - \partial_y u_1^{I,2}(x,0,t)\Bigr) \tau_{12}^{B,1}\\
         &
        +2\varepsilon \Bigl(\partial_y u_1^{I,0}(x,y,t) - \partial_y u_1^{I,0}(x,0,t)\Bigr) \tau_{12}^{B,3}
        + 2k\varepsilon \Bigl( \eta^{I,0}(x,y,t) - \eta^{I,0}(x,0,t) \Bigr)\partial_x u_1^{B,3}
         \\
        &+2k\varepsilon \Bigl(\partial_x u_1^{I,2}(x,y,t) - \partial_x u_1^{I,2}(x,0,t)\Bigr)\eta^{B,1}
        +2k\varepsilon \Bigl(\partial_x u_1^{I,0}(x,y,t) - \partial_x u_1^{I,0}(x,0,t)\Bigr)  \eta^{B,3}
         \\
        &+  2 \Bigl( \tau_{12}^{I,0}(x,y,t) - \tau_{12}^{I,0}(x,0,t) - y\partial_y \tau_{12}^{I,0}(x,0,t) - \frac{1}{2}y^2\partial_y^2 \tau_{12}^{I,0}(x,0,t)\Bigr)\partial_z u_1^{B,2} \\
        &+ 2\Bigl(\partial_x u_1^{I,0}(x,y,t) -  \partial_x u_1^{I,0}(x,0,t) -y\partial_y\partial_x  u_1^{I,0}(x,0,t)  -\frac{1}{2}y^2\partial_y^2\partial_x  u_1^{I,0}(x,0,t) \Bigr) \tau_{11}^{B,1}
        \\
        &+ 2\Bigl(\partial_y u_1^{I,0}(x,y,t) -  \partial_y u_1^{I,0}(x,0,t) -y\partial_y^2  u_1^{I,0}(x,0,t)  -\frac{1}{2}y^2\partial_y^3   u_1^{I,0}(x,0,t) \Bigr) \tau_{12}^{B,1}
        \\
        &+ 2k\Bigl(\partial_x u_1^{I,0}(x,y,t) -  \partial_x u_1^{I,0}(x,0,t) -y\partial_y\partial_x  u_1^{I,0}(x,0,t)  -\frac{1}{2}y^2\partial_y^2\partial_x  u_1^{I,0}(x,0,t) \Bigr) \eta^{B,1}
        \\
         & + 2\sqrt{\varepsilon}\Bigl( \tau_{11}^{I,0}(x,y,t) - \tau_{11}^{I,0}(x,0,t) - y\partial_y \tau_{11}^{I,0}(x,0,t) \Bigr)\partial_x u_1^{B,2} \\
         &+  2\sqrt{\varepsilon} \Bigl( \tau_{12}^{I,0}(x,y,t) - \tau_{12}^{I,0}(x,0,t) - y\partial_y \tau_{12}^{I,0}(x,0,t) \Bigr)\partial_z u_1^{B,3} \\
        &+2\sqrt{\varepsilon}\Bigl(\partial_x u_1^{I,0}(x,y,t) -  \partial_x u_1^{I,0}(x,0,t) -y\partial_y\partial_x  u_1^{I,0}(x,0,t)    \Bigr) \tau_{11}^{B,2} \\
        &+2k\sqrt{\varepsilon}\Bigl(\partial_x u_1^{I,0}(x,y,t) -  \partial_x u_1^{I,0}(x,0,t) -y\partial_y\partial_x  u_1^{I,0}(x,0,t)    \Bigr) \eta^{B,2}  \\
        &+2\sqrt{\varepsilon}\Bigl(\partial_y u_1^{I,0}(x,y,t) -  \partial_y u_1^{I,0}(x,0,t) -y\partial_y^2  u_1^{I,0}(x,0,t)    \Bigr) \tau_{12}^{B,2} \\
        &+ 2k\sqrt{\varepsilon}\Bigl(\eta^{I,0}(x,y,t) - \eta^{I,0}(x,0,t) -y\partial_y \eta^{I,0}(x,0,t)\Bigr) \partial_x u_1^{B,2}, \\[2mm]
         \mathfrak{J}_{11,12}= {}&\mathfrak{M}(u_1^{I,0},u_2^{I,0},u^{I,2}_1,u^{I,2}_2,u^{I,3}_2,u_1^{B,2},u_1^{B,3},u_2^{B,3},\tau_{12}^{I,0},\tau_{12}^{B,1},\tau_{12}^{B,2},\tau_{12}^{B,3}) \\
         & +  \varepsilon \Bigl( \tau_{11}^{I,0}(x,y,t) - \tau_{11}^{I,0}(x,0,t)  \Bigr)\partial_x u_2^{B,3}
         +  \varepsilon \Bigl( \tau_{22}^{I,2}(x,y,t) - \tau_{22}^{I,2}(x,0,t)  \Bigr)\partial_z u_1^{B,2}\\
         &+   \varepsilon \Bigl( \tau_{22}^{I,0}(x,y,t) - \tau_{22}^{I,0}(x,0,t)  \Bigr)\partial_z u_1^{B,4}
         + \varepsilon\Bigl( \partial_x u_2^{I,2}(x,y,t) - \partial_x u_2^{I,2}(x,0,t)\Bigr)  \tau_{11}^{B,1}\\
         &+ \varepsilon\Bigl( \partial_x u_2^{I,0}(x,y,t) - \partial_x u_2^{I,0}(x,0,t)\Bigr)  \tau_{11}^{B,3}
          + \varepsilon\Bigl( \partial_y u_1^{I,2}(x,y,t) - \partial_y u_1^{I,2}(x,0,t)\Bigr)  \tau_{22}^{B,1}\\
         &+ \varepsilon\Bigl( \partial_y u_1^{I,0}(x,y,t) - \partial_y u_1^{I,0}(x,0,t)\Bigr) \tau_{22}^{B,3}
         + k\varepsilon \Bigl(\eta^{I,0}(x,y,t) - \eta^{I,0}(x,0,t)\Bigr)\partial_x u_2^{B,3}\\
        &+ k\varepsilon \Bigl(\eta^{I,2}(x,y,t) - \eta^{I,2}(x,0,t)\Bigr)\partial_z u_1^{B,2}
           +  k\varepsilon \Bigl(\eta^{I,0}(x,y,t) - \eta^{I,0}(x,0,t)\Bigr)\partial_z u_1^{B,4} \\
         & + k \varepsilon\Bigl( \partial_x u_2^{I,2}(x,y,t) - \partial_x u_2^{I,2}(x,0,t)\Bigr)  \eta^{B,1}
         + k  \varepsilon\Bigl( \partial_x u_2^{I,0}(x,y,t) - \partial_x u_2^{I,0}(x,0,t)\Bigr)\eta^{B,3} \\
         & + k\varepsilon\Bigl( \partial_y u_1^{I,2}(x,y,t) - \partial_y u_1^{I,2}(x,0,t)\Bigr)\eta^{B,1}
         +k\varepsilon\Bigl( \partial_y u_1^{I,0}(x,y,t) - \partial_y u_1^{I,0}(x,0,t)\Bigr) \eta^{B,3} \\
         &+   \Bigl( \tau_{22}^{I,0}(x,y,t) - \tau_{22}^{I,0}(x,0,t) - y\partial_y \tau_{22}^{I,0}(x,0,t) - \frac{1}{2}y^2\partial_y^2 \tau_{22}^{I,0}(x,0,t)\Bigr)\partial_z u_1^{B,2} \\
        &+  k \Bigl( \eta^{I,0}(x,y,t) - \eta^{I,0}(x,0,t) - y\partial_y \eta^{I,0}(x,0,t) - \frac{1}{2}y^2\partial_y^2 \eta^{I,0}(x,0,t)\Bigr)\partial_z u_1^{B,2} \\
        &+ \Bigl( \partial_x u_2^{I,0}(x,y,t) - \partial_x u_2^{I,0}(x,0,t) - y\partial_y \partial_x u_2^{I,0}(x,0,t) - \frac{1}{2}y^2\partial_y^2 \partial_x u_2^{I,0}(x,0,t)\Bigr)\tau_{11}^{B,1} \\
        &+ k\Bigl( \partial_x u_2^{I,0}(x,y,t) - \partial_x u_2^{I,0}(x,0,t) - y\partial_y \partial_x u_2^{I,0}(x,0,t) - \frac{1}{2}y^2\partial_y^2 \partial_x u_2^{I,0}(x,0,t)\Bigr)\eta^{B,1} \\
        &+ \Bigl( \partial_y u_1^{I,0}(x,y,t) - \partial_y u_1^{I,0}(x,0,t) - y\partial_y^2   u_1^{I,0}(x,0,t) - \frac{1}{2}y^2\partial_y^3   u_1^{I,0}(x,0,t)\Bigr)\tau_{22}^{B,1} \\
        &+ k\Bigl( \partial_y u_1^{I,0}(x,y,t) - \partial_y u_1^{I,0}(x,0,t) - y\partial_y^2   u_1^{I,0}(x,0,t) - \frac{1}{2}y^2\partial_y^3   u_1^{I,0}(x,0,t)\Bigr)\eta^{B,1} \\
       &+  \sqrt{\varepsilon}\Bigl( \tau_{22}^{I,0}(x,y,t) - \tau_{22}^{I,0}(x,0,t) - y\partial_y \tau_{22}^{I,0}(x,0,t) \Bigr)\partial_z u_1^{B,3} \\
        &+ \sqrt{\varepsilon}\Bigl( \partial_x u_2^{I,0}(x,y,t) - \partial_x u_2^{I,0}(x,0,t) - y\partial_y \partial_x u_2^{I,0}(x,0,t)  \Bigr) \tau_{11}^{B,2}   \\
         &+ \sqrt{\varepsilon}\Bigl( \partial_y u_1^{I,0}(x,y,t) - \partial_y u_1^{I,0}(x,0,t) - y\partial_y^2 u_1^{I,0}(x,0,t)  \Bigr) \tau_{22}^{B,2}   \\
        &+  k\sqrt{\varepsilon} \Bigl( \eta^{I,0}(x,y,t) - \eta^{I,0}(x,0,t) - y\partial_y \eta^{I,0}(x,0,t)  \Bigr)\partial_z u_1^{B,3} \\
        &+ k\sqrt{\varepsilon}\Bigl( \partial_x u_2^{I,0}(x,y,t) - \partial_x u_2^{I,0}(x,0,t) - y\partial_y \partial_x u_2^{I,0}(x,0,t)  \Bigr) \eta^{B,2}  \\
        &+ k\sqrt{\varepsilon}\Bigl( \partial_y u_1^{I,0}(x,y,t) - \partial_y u_1^{I,0}(x,0,t) - y\partial_y^2 u_1^{I,0}(x,0,t)  \Bigr)   \eta^{B,2}, \\[2mm]
         \mathfrak{J}_{11,22}= {} &\mathfrak{M}(u_1^{I,0},u_2^{I,0},u^{I,2}_1,u^{I,2}_2,u^{I,3}_2,u_1^{B,2},u_1^{B,3},u_2^{B,3},\tau_{ 22}^{I,0},\tau_{22}^{B,1},\tau_{22}^{B,2},\tau_{22}^{B,3}) \\
          &+ 2 \varepsilon \Bigl(\tau_{12}^{I,0}(x,y,t) - \tau_{12}^{I,0}(x,0,t)\Bigr)\partial_x u_2^{B,3}
         +2 \varepsilon \Bigl(\tau_{22}^{I,0}(x,y,t) - \tau_{22}^{I,0}(x,0,t)\Bigr)\partial_z u_2^{B,4} \\
        &+2\varepsilon \Bigl(\partial_x u_2^{I,2}(x,y,t) - \partial_x u_2^{I,2}(x,0,t)\Bigr)  \tau_{12}^{B,1}
        +2\varepsilon \Bigl(\partial_x u_2^{I,0}(x,y,t) - \partial_x u_2^{I,0}(x,0,t)\Bigr) \tau_{12}^{B,3}
          \\
        &+2\varepsilon \Bigl(\partial_y u_2^{I,2}(x,y,t)  - \partial_y u_2^{I,2}(x,0,t)\Bigr) \Bigl(\tau_{22}^{B,1} +k \eta^{B,1}\Bigr)
        +2\varepsilon \Bigl(\partial_y u_2^{I,0}(x,y,t) \\
        & -\partial_y u_2^{I,0}(x,0,t) \Bigr)\Bigl(\tau_{22}^{B,3}+k\eta^{B,3}\Bigr)
         + 2k\varepsilon\Bigl(\eta^{I,0}(x,y,t)- \eta^{I,0}(x,0,t)\Bigr)\partial_z u_2^{B,4} \\
        &+ 2\Bigl(\partial_x u_2^{I,0}(x,y,t) -  \partial_x u_2^{I,0}(x,0,t) -y\partial_y\partial_x  u_2^{I,0}(x,0,t)  -\frac{1}{2}y^2\partial_y^2\partial_x  u_2^{I,0}(x,0,t) \Bigr) \tau_{12}^{B,1}
        \\
        &+ 2\Bigl(\partial_y u_2^{I,0}(x,y,t) -  \partial_y u_2^{I,0}(x,0,t) -y\partial_y^2  u_2^{I,0}(x,0,t)  -\frac{1}{2}y^2\partial_y^3   u_2^{I,0}(x,0,t) \Bigr) \tau_{22}^{B,1}
        \\
        &+ 2k\Bigl(\partial_y u_2^{I,0}(x,y,t) -  \partial_y u_2^{I,0}(x,0,t) -y\partial_y^2  u_2^{I,0}(x,0,t)  -\frac{1}{2}y^2\partial_y^3   u_2^{I,0}(x,0,t) \Bigr) \eta^{B,1}
        \\
         & + 2\sqrt{\varepsilon}\Bigl( \tau_{22}^{I,0}(x,y,t) - \tau_{22}^{I,0}(x,0,t) - y\partial_y \tau_{22}^{I,0}(x,0,t) \Bigr)\partial_z u_2^{B,3} \\
         &+2\sqrt{\varepsilon}\Bigl(\partial_x u_2^{I,0}(x,y,t) -  \partial_x u_2^{I,0}(x,0,t) -y\partial_y\partial_x  u_2^{I,0}(x,0,t)    \Bigr)  \tau_{12}^{B,2}  \\
         &+2\sqrt{\varepsilon}\Bigl(\partial_y u_2^{I,0}(x,y,t) -  \partial_y u_2^{I,0}(x,0,t) -y\partial_y^2  u_2^{I,0}(x,0,t)    \Bigr)  \tau_{22}^{B,2}   \\
         &+2k\sqrt{\varepsilon}\Bigl(\partial_y u_2^{I,0}(x,y,t) -  \partial_y u_2^{I,0}(x,0,t) -y\partial_y^2  u_2^{I,0}(x,0,t)    \Bigr)   \eta^{B,2}  \\
         &+ 2k\sqrt{\varepsilon}\Bigl(\eta^{I,0}(x,y,t) - \eta^{I,0}(x,0,t) -y\partial_y \eta^{I,0}(x,0,t)\Bigr) \partial_z u_2^{B,3}.
         \end{align*}
  Finally, for the components of $\mathfrak{J}_{12}$, we have
    \begin{align*}
        \mathfrak{J}_{12,11}={} &  \mathfrak{P}(u^{I,2}_1,u^{I,3}_1,u^{I,2}_2,u^{I,3}_2,u_1^{B,2},u_1^{B,3},u^{B,4}_1,u_2^{B,3},u^{B,4}_2,u^{B,5}_2,w_1,w_2,\tau_{11}^{I,2},\tau_{11}^{I},\tau_{11}^{B,1},\tau_{11}^{B,2},\tau_{11}^{B,3},\tau_{11}^B) \\
         &
        + 2\varepsilon^{\frac{3}{2}} \Bigl(\tau_{11}^{I,0} \partial_x  u_1^{B,4}
        +   \tau_{11}^{I,2}\partial_x u_1^B   \Bigr)
        + 2  \varepsilon^{\frac{3}{2}}\tau_{12}^{I,2}\partial_z\Bigl( u_1^{B,3} +  \sqrt{\varepsilon}  u_1^{B,4} \Bigr)
        + 2\varepsilon^{\frac{3}{2}}   \tau_{11}^{B,1} \partial_x u_1^{I,3} \\
        &
        + 2\varepsilon^{\frac{3}{2}} \tau_{12}^{B,1}\partial_y u_1^{I,3} +2\varepsilon^{\frac{3}{2}}\Bigl(\tau_{11}^{B,2} + \sqrt{\varepsilon}\tau_{11}^{B,3}\Bigr)\Bigl(\partial_x u_1^{I,2} +\sqrt{\varepsilon} \partial_x u_1^{I,3}\Bigr)
        + 2\varepsilon^{\frac{3}{2}}\Bigl(\tau_{12}^{B,2} + \sqrt{\varepsilon}\tau_{12}^{B,3}\Bigr)\\
        & \times \partial_y\Bigl( u_1^{I,2} +\sqrt{\varepsilon}   u_1^{I,3}\Bigr)
        +  2\varepsilon^{\frac{3}{2}}\tau_{11}^{B,1}\partial_x\Bigl(  u_1^{B,3}
        + \sqrt{\varepsilon}    u_1^{B,4} \Bigr)
        + 2\varepsilon^{\frac{3}{2}}\Bigl( \tau_{11}^{B,2}+ \sqrt{\varepsilon} \tau_{11}^{B,3}\Bigr)  \partial_x u_1^{B}\\
        &
        + 2 \varepsilon^{\frac{3}{2}} \tau_{12}^{B,1}   \partial_z u_1^{B,4} + 2\varepsilon^{\frac{3}{2}} \tau_{12}^{B,2}\partial_z\Big(u_1^{B,3} +  \sqrt{\varepsilon}  u_1^{B,4}\Big)
         +2\varepsilon^{\frac{3}{2}}\tau_{12}^{B,3}\partial_z u_1^{B}
        + 2\varepsilon^2  \tau_{11}^B\partial_x w_1  \\
        &    + 2\varepsilon^2 \tau_{12}^B \partial_y w_1
        +2k\varepsilon^{\frac{3}{2}}\eta^{I,0} \partial_x u_1^{B,4}
        +2k\varepsilon^{\frac{3}{2}}\eta^{I,2}\partial_x u_1^B
         + 2k\varepsilon^{\frac{3}{2}}\eta^{B,1} \partial_x u_1^{I,3}+ 2k\varepsilon^{\frac{3}{2}}\Bigl(\eta^{B,2}\\
        &  + \sqrt{\varepsilon } \eta^{B,3}\Bigr)\partial_x\Bigl(  u_1^{I,2} + \sqrt{\varepsilon}  u_1^{I,3}\Bigr)
        + 2k\varepsilon^{\frac{3}{2}} \eta^{B,1}\partial_x\Bigl(  u_1^{B,3}
        + \sqrt{\varepsilon}     u_1^{B,4} \Bigr)
        + 2k\varepsilon^{\frac{3}{2}}\Bigl(\eta^{B,2}+ \sqrt{\varepsilon} \eta^{B,3}\Bigr)\\
        &
         \times \partial_x u_1^B +  2k\varepsilon^2\eta^B\partial_x w_1, \\[2mm]
        \mathfrak{J}_{12,12}= {}&   \mathfrak{P}(u^{I,2}_1,u^{I,3}_1,u^{I,2}_2,u^{I,3}_2,u_1^{B,2},u_1^{B,3},u^{B,4}_1,u_2^{B,3},u^{B,4}_2,u^{B,5}_2,w_1,w_2,\tau_{12}^{I,2},\tau_{12}^{I},\tau_{12}^{B,1},\tau_{12}^{B,2},\tau_{12}^{B,3},\tau_{12}^B) \\
        &
          +   \varepsilon^{\frac{3}{2}} \tau_{11}^{I,0}\partial_x\Bigl(  u_2^{B,4} + \sqrt{\varepsilon} u_2^{B,5} \Bigr)
        +\varepsilon^\frac{3}{2}\tau_{11}^{I,2}\partial_x u_2^B
        +  \varepsilon^{\frac{3}{2}}    \tau_{11}^{B,1}\partial_x u_2^{I,3}
        +  \varepsilon^{\frac{3}{2}} \Bigl( \tau_{11}^{B,2}  + \sqrt{\varepsilon} \tau_{11}^{B,3}\Bigr)  \\
        & \times \partial_x \Bigl(u_2^{I,2}+  \sqrt{\varepsilon}  u_2^{I,3}\Bigr)
        +   \varepsilon^{\frac{3}{2}}    \tau_{11}^{B} \partial_x\Bigl(  u_2^{B,3} +  \sqrt{\varepsilon}  u_2^{B,4}  + \varepsilon u_2^{B,5} \Bigr)
        +   \varepsilon^{\frac{3}{2}}\tau_{22}^{I,2} \partial_z \Bigl( u_1^{B,3}\\
        &
         + \sqrt{\varepsilon} u_1^{B,4}  \Bigr)
        +\varepsilon^{\frac{3}{2}}    \tau_{22}^{B,1}\partial_y u_1^{I,3}
        + \varepsilon^{\frac{3}{2}}    \Big(\tau_{22}^{B,2}+ \sqrt{\varepsilon}\tau_{22}^{B,3}\Big)\partial_y\Bigl( u_1^{I,2} +\sqrt{\varepsilon}  u_1^{I,3} \Bigr)
        + \varepsilon^{\frac{3}{2}}\tau_{22}^B\partial_z u_1^{B,4}  \\
         &
         +  \varepsilon^{\frac{3}{2}}    \tau_{22}^{B,3}\partial_z\Bigl( u_1^{B,2} + \sqrt{\varepsilon}   u_1^{B,3}  \Bigr)
          +  \varepsilon^{\frac{3}{2}}    \tau_{22}^{B,2} \partial_z u_1^{B,3}
         +\varepsilon^2\Bigl(\tau_{11}^B \partial_x w_2 +\tau_{22}^B \partial_y w_1 \Bigr)
        +k\varepsilon^{\frac{3}{2}}\eta^{I} \\
        &    \times\partial_x \Bigl(u_2^{B,4} + \sqrt{\varepsilon} u_2^{B,5}\Bigr)     + k\varepsilon^{\frac{3}{2}}\eta^{I,2}\partial_z u_1^{B,3}
        + k\varepsilon^2 \eta^{I,2}\Bigl( \partial_x u_2^{B,3}  +  \partial_z u_1^{B,4} \Bigr)
        +k\varepsilon^{\frac{3}{2}}\eta^{B}\Bigl(\partial_x u_2^{I,3}\\
        &
         + \partial_y u_1^{I,3}\Bigr)
        +k\varepsilon^{\frac{3}{2}}\Big(\eta^{B,2} +\sqrt{\varepsilon}\eta^{B,3}\Big)\Bigl(\partial_x u_2^{I,2} + \partial_y u_1^{I,2}\Bigr)
         +\varepsilon^{\frac{3}{2}}k\eta^{B } \partial_x \Bigl(u_2^{B,3} + \sqrt{\varepsilon}  u_2^{B,4}  \\
        & +    \varepsilon u_2^{B,5}  \Bigr)
        +k\varepsilon^{\frac{3}{2}}\eta^{B,2}\partial_z u_1^{B,3} +\varepsilon^{\frac{3}{2}}k\eta^{B,3}\Bigl(\partial_z u_1^{B,2} + \sqrt{\varepsilon}\partial_z u_1^{B,3}\Bigr)
          + k\varepsilon^{\frac{3}{2}}  \eta^B \partial_z u_1^{B,4} \\
        &  +\varepsilon^2k\eta^B \Bigl(  \partial_x w_2 +  \partial_y w_1 \Bigr) ,
    \end{align*}
and
    \begin{align*}
        \mathfrak{J}_{12,22}={}&  \mathfrak{P}(u^{I,2}_1,u^{I,3}_1,u^{I,2}_2,u^{I,3}_2,u_1^{B,2},u_1^{B,3},u^{B,4}_1,u_2^{B,3},u^{B,4}_2,u^{B,5}_2,w_1,w_2,\tau_{22}^{I,2},\tau_{22}^{I},\tau_{22}^{B,1},\tau_{22}^{B,2},\tau_{22}^{B,3},\tau_{22}^B) \\
        &
          +  2\varepsilon^{\frac{3}{2}}\tau_{12}^{I}\partial_x \Bigl(u_2^{B,4} + \sqrt{\varepsilon} u_2^{B,5}\Bigr)
         + 2\varepsilon^{2} \tau_{12}^{I,2}\partial_x u^{B,3}_2
         + 2\varepsilon^{\frac{3}{2}}\tau_{12}^{B}\partial_x u_2^{I,3}
         + 2 \varepsilon^{\frac{3}{2}} \Bigl(\tau_{12}^{B,2} + \sqrt{\varepsilon}\tau_{12}^{B,3}\Bigr) \\
         &\times \partial_x u_2^{I,2}
         + 2 \varepsilon^{\frac{3}{2}}\tau_{12}^{B} \partial_x\Bigl( u^{B,3}_2  + \sqrt{\varepsilon} u^{B,4}_2 +  \varepsilon u_2^{B,5} \Bigr)
            + 2\varepsilon^{\frac{3}{2}}\tau_{22}^{I,2} \Bigl(\partial_z u_2^{B,3} + \sqrt{\varepsilon}\partial_z u_2^{B,4}\Bigr) \\
         & + 2 \varepsilon^{\frac{3}{2}} \tau_{22}^{I} \partial_z u_2^{B,5}
           + 2\varepsilon^{\frac{3}{2}}\tau_{22}^{B}\partial_y u_2^{I,3}
        + 2 \varepsilon^{\frac{3}{2}}\Bigl(\tau_{22}^{B,2}
         + \sqrt{\varepsilon}\tau_{22}^{B,3}\Bigr)\partial_y u_2^{I,2}
         + 2 \varepsilon^{\frac{3}{2}} \Bigl(\tau_{22}^{B,2} + \sqrt{\varepsilon}\tau_{22}^{B,3}\Bigr)\partial_z u_2^{B,3} \\
        &
         + 2 \varepsilon^{\frac{3}{2}}\tau_{22}^{B} \partial_z( u_2^{B,4} + \sqrt{\varepsilon}  u_2^{B,5}  \Bigr)
        + 2\varepsilon^2\Bigl(\tau_{12}^{B}\partial_x w_2 +\tau_{22}^{B}\partial_y w_2 \Bigr)
          +2k\varepsilon^{\frac{3}{2}}\eta^{I,0}\partial_z u_2^{B,5} \\
        &
          {
          + 2k\varepsilon^{\frac{3}{2}}\eta^{I,2}\partial_z\Bigl( u^{B,3}_2  + \sqrt{\varepsilon}  u^{B,4}_2 +\varepsilon u_2^{B,5}\Bigr)}
           + 2k\varepsilon^{\frac{3}{2}}\Bigl(\eta^{B,2}+ \sqrt{\varepsilon}\eta^{B,3}\Bigr)\partial_y u_2^{I,2}\\
        &
        + 2k\varepsilon^{\frac{3}{2}}\eta^B\partial_y u_2^{I,3}
        + 2k\varepsilon^{\frac{3}{2}}\eta^{B}\partial_z\Bigl( u_2^{B,4} + \sqrt{\varepsilon} u_2^{B,5}\Bigr)
        +2k \varepsilon^{\frac{3}{2}}\Bigl(\eta^{B,2}
         + \sqrt{\varepsilon}\eta^{B,3}\Bigr)\partial_z   u_2^{B,3}
         \\
       & + 2k\varepsilon^2  \eta^{B}\partial_y w_2.
       \end{align*}

     \section*{Acknowledgements}
The authors would like to thank Professor Zhifei Zhang for some inspiring discussions about the boundary layer for the Oldroyd-B model. H. Wen would like to thank Professors Matthias Hieber, Yong Lu, and Ruizhao Zi for some early discussions about the Oldroyd-B model.
 This work was supported by the National Natural Science Foundation of China (12071152), by the Guangdong Basic and Applied Basic Research Foundation (2022A1515012112, 2020B1515310015), and by the Guangdong Provincial Key Laboratory of Human Digital Twin (2022B1212010004).

        \end{document}